\pdfoutput=1

\documentclass{amsart}

\usepackage[mathscr]{eucal}
\usepackage{amssymb}
\usepackage[usenames,dvipsnames]{xcolor} 
\usepackage[normalem]{ulem}
\usepackage{amsthm}
\usepackage{bbold}
\usepackage{comment}
\usepackage{enumitem}
\usepackage{amsmath}
\usepackage{tikz}
\usepackage{tikz-cd}
\usetikzlibrary{arrows}

\usepackage{etoolbox} 

\usepackage[unicode]{hyperref} 

\definecolor{dark-red}{rgb}{0.5,0.15,0.15}
\definecolor{dark-blue}{rgb}{0.15,0.15,0.6}
\definecolor{dark-green}{rgb}{0.15,0.6,0.15}

\hypersetup{
    colorlinks, linkcolor=Blue,
    citecolor=Blue, urlcolor=Blue
}

\usepackage{microtype}
\usepackage{quiver}
\usepackage[nameinlink,capitalise,noabbrev]{cleveref}

\tikzcdset{scale cd/.style={every label/.append style={scale=#1},
    cells={nodes={scale=#1}}}}

\numberwithin{equation}{section}
\usepackage{graphicx}
\usepackage{mathtools}

\newtheorem*{thmx*}{Theorem}

\swapnumbers 

\newtheorem{Thm}[equation]{Theorem}
\newtheorem*{Thm*}{Theorem}
\newtheorem{Prop}[equation]{Proposition}
\newtheorem{Lem}[equation]{Lemma}
\newtheorem{Cor}[equation]{Corollary}

\newtheorem*{Que*}{Question}

\theoremstyle{remark}
\newtheorem{Def}[equation]{Definition}
\newtheorem{Ter}[equation]{Terminology}
\newtheorem{Not}[equation]{Notation}
\newtheorem{Exa}[equation]{Example}

\newtheorem{Hyp}[equation]{Hypothesis}
\newtheorem*{Hyp*}{Hypothesis}

\newtheorem{Rem}[equation]{Remark}
\newtheorem{Que}[equation]{Question}

\tikzset{
    labelrotatebelow/.style={anchor=north, rotate=90, inner sep=1.0mm}
}
\tikzset{
    labelrotateabove/.style={anchor=south, rotate=90, inner sep=1.0mm}
}

\usetikzlibrary{decorations.markings}
\tikzset{negated/.style={
        decoration={markings,
            mark= at position 0.5 with {
                \node[transform shape] (tempnode) {$\backslash \! \! \backslash$};
            }
        },
        postaction={decorate}
    }
}

\newcommand{\nc}{\newcommand}
\nc{\dmo}{\DeclareMathOperator}

\renewcommand{\emptyset}{\varnothing}

\nc{\Beren}[1]{{\color{MidnightBlue}#1}}
\nc{\Drew}[1]{{\color{Orange}#1}}
\nc{\Tobi}[1]{{\color{Green}#1}}
\nc{\Natalia}[1]{{\color{Yellow}#1}}
\nc{\Dout}[1]{\Drew{\sout{#1}}}
\nc{\Bout}[1]{\Beren{\sout{#1}}}
\nc{\Tout}[1]{\Tobi{\sout{#1}}}
\nc{\Nout}[1]{\Natalia{\sout{#1}}}

\usepackage{todonotes}

\nc{\Ybarcon}{\overline{Y}^{\mathrm{con}}}
\nc{\mfrak}{\mathfrak}
\nc{\overbar}[1]{\mkern 1.5mu\overline{\mkern-1.5mu#1\mkern-1.5mu}\mkern 1.5mu}
\renewcommand{\L}{\mathrm{L}}

\nc{\weaklyfinite}{weakly closed}
\nc{\finite}{strongly closed}
\nc{\BCdual}[1]{{#1}^*}
\nc{\LCore}{\mathrm{LCore}}
\nc{\Stovicek}{\v{S}\v{t}ov\'{i}\v{c}ek}
\nc{\ftriple}{f_{\natural}}
\nc{\unitC}{\unit_{\cat C}}%
\nc{\unitD}{\unit_{\cat D}}%
\nc{\unitS}{\unit_{\cat S}}%
\nc{\unitT}{\unit_{\cat T}}%
\nc{\Pone}{{\mathbb{P}^1}}
\nc{\InvSupp}[1]{\Supp^{-1}(#1)}
\nc{\InvCosupp}[1]{\Cosupp^{-1}(#1)}
\nc{\closureP}{\overbar{\{\cat P\}}}
\nc{\closureQ}{\overbar{\{\cat Q\}}}
\nc{\singP}{\{\cat P\}}
\nc{\singQ}{\{\cat Q\}}
\nc{\singm}{\{\mathfrak m\}}
\dmo{\Inj}{Inj}
\dmo{\Rep}{Rep}
\dmo{\Res}{Res}
\dmo{\KInjdmo}{K}
\dmo{\Dbdmo}{mod}
\nc{\KInj}[1]{\KInjdmo(\Inj #1)}
\nc{\Dbmod}[1]{\Der^b(\Dbdmo #1)}
\dmo{\Viss}{vis}
\nc{\Vis}{\Viss}
\nc{\vis}{\Vis}
\nc{\kappaaux}{g}
\nc{\kappaCh}{{\kappaaux(\cat C_h)}}
\nc{\kappam}{{\kappaaux({\mathfrak m})}}
\nc{\kappaP}{{\kappaaux(\cat P)}}
\nc{\kappaQ}{{\kappaaux(\cat Q)}}
\nc{\kappaCP}{{\kappaaux_{\cat C}(\cat P)}}
\nc{\kappaDP}{{\kappaaux_{\cat D}(\cat P)}}
\nc{\kappaCQ}{{\kappaaux_{\cat C}(\cat Q)}}
\nc{\kappaDQ}{{\kappaaux_{\cat D}(\cat Q)}}
\nc{\kappaphiB}{{\kappaaux(\phi(\cat B))}}
\nc{\kappaphiQ}{{\kappaaux(\varphi(\cat Q))}}

\dmo{\Sub}{Sub}
\nc{\SpEn}{\cat S_{E(n)}}
\nc{\SpEnf}{\cat S_n}
\nc{\Lcomp}{L^{\mathrm{com}}} 
\nc{\Ucomp}{U^{\mathrm{com}}}
\nc{\bbullet}{{\scriptscriptstyle\hspace{-1pt}\bullet}}
\nc{\bullett}{{\scriptscriptstyle\bullet}\hspace{-1pt}}
\nc{\LF}{L\hspace{-0.2ex}F}
\dmo{\StMod}{StMod}
\dmo{\Proj}{Proj}
\dmo{\Ind}{Ind}
\nc{\SpG}{\Sp^G}
\nc{\EG}{\bbE_G}
\nc{\DEG}{\Der(\EG)}
\nc{\DE}{\Der(\bbE)}
\nc{\Prst}{{\cat P}\mathrm{r^{st}}}
\nc{\Mack}[2]{\mathrm{Mack}_{#1}(#2)}
\nc{\SC}{S\cat C}
\dmo{\fin}{{fin}}
\dmo{\DM}{DM}
\dmo{\fp}{fp}
\nc{\DMQ}{\DM_Q}
\dmo{\DerKal}{DMack}
\dmo{\Perf}{Perf}
\dmo{\coh}{coh}
\dmo{\Der}{D}
\dmo{\DMot}{DMot}
\dmo{\rmH}{H}
\dmo{\piu}{\underline{\pi}}
\dmo{\Sphere}{\mathbb{S}}
\nc{\HR}{{H \hspace{-0.1em}R}}
\nc{\HA}{{\rmH \hspace{-0.2em}\bbA}}
\nc{\HZ}{{H \hspace{-0.1em}\bbZ}}
\nc{\HZbar}{{\rmH \hspace{-0.2em}\underline{\bbZ}}}
\nc{\Fp}{{\bbF_{\hspace{-0.1em}p}}}
\nc{\HFp}{{\rmH \hspace{-0.15em}\bbF_{\hspace{-0.1em}p}}}
\nc{\DHZG}{\Der(\HZ_G)}
\nc{\DHZH}{\Der(\HZ_H)}
\nc{\DHZK}{\Der(\HZ_K)}
\nc{\DHZGN}{\Der(\HZ_{G/N})}
\nc{\DHZGG}{\Der(\HZ_{G/G})}
\nc{\DHZCp}{\Der(\HZ_{C_p})}
\nc{\DHZGprime}{\Der(\HZ_{G'})}
\nc{\DHZ}{\Der(\HZ)}
\nc{\mathfrakp}{\mathfrak{p}}
\nc{\mathfrakq}{\mathfrak{q}}
\nc{\mathfrakS}{\mathfrak{S}}
\nc{\mathfrakT}{\mathfrak{T}}
\nc{\Z}{\mathbb{Z}}
\nc{\SSG}{\text{sSet}_*^G}
\nc{\sSet}{\text{sSet}}
\newcommand{\cosuppa}{\operatorname{co-supp}}

\dmo{\csupp}{csupp}
\dmo{\Con}{Conj}
\dmo{\Id}{Id}
\dmo{\rmK}{\textrm{\rm K}}
\dmo{\Spc}{Spc}
\dmo{\thick}{thick}
\dmo{\thickid}{thickid}
\nc{\thicko}[1]{\thickid\langle #1 \rangle}
\nc{\thickt}[1]{\thickid\langle #1 \rangle}
\dmo{\cone}{cone}
\dmo{\End}{End}
\dmo{\Derperf}{D_{perf}}
\dmo{\Mor}{Mor}
\dmo{\id}{id}
\dmo{\incl}{incl}
\dmo{\Img}{Im}
\dmo{\im}{im}
\dmo{\Ker}{Ker}
\dmo{\ind}{ind}
\dmo{\CoInd}{coind}
\dmo{\res}{res}
\dmo{\infl}{infl}
\dmo{\Derqc}{D_{qc}}
\nc{\DbcohX}{\Der^b(\coh X)}
\dmo{\triv}{triv}
\dmo{\Tel}{Tel} 
\dmo{\grMod}{grMod}%
\dmo{\Mod}{Mod}%
\dmo{\opname}{op}
\dmo{\SH}{SH}
\dmo{\smallb}{b}
\dmo{\Spec}{Spec}
\dmo{\supp}{supp}
\dmo{\Supp}{Supp}
\dmo{\cosupp}{cosupp}
\dmo{\Cosupp}{Cosupp}
\nc{\SuppT}{\Supp_{\cat T}}
\nc{\SuppS}{\Supp_{\cat S}}
\nc{\CosuppT}{\Cosupp_{\cat T}}
\nc{\CosuppS}{\Cosupp_{\cat S}}
\nc{\SHc}{{\SH^c}}
\nc{\SHp}{{\SH_{(p)}}}
\nc{\SHcp}{{\SH^c_{(p)}}}
\nc{\SHG}{\SH(G)}
\nc{\SHGp}{\SH(G)_{(p)}}
\nc{\SHGc}{\SHG^c}
\nc{\SHGcp}{\SHG^c_{(p)}}
\nc{\quadtext}[1]{\quad\textrm{#1}\quad}
\nc{\qquadtext}[1]{\qquad\textrm{#1}\qquad}
\nc{\adj}{\dashv}
\nc{\adjto}{\rightleftarrows}
\nc{\bbL}{\mathbb{L}}
\nc{\bbA}{\mathbb{A}}
\nc{\bbE}{\mathbb{E}}
\nc{\bbN}{\mathbb{N}}
\nc{\bbQ}{\mathbb{Q}}
\nc{\bbZ}{\mathbb{Z}}
\nc{\bbF}{\mathbb{F}}
\nc{\bbT}{\mathbb{T}}
\nc{\cat}[1]{\mathscr{#1}}
\nc{\ie}{{\sl i.e.}, }
\nc{\into}{\mathop{\rightarrowtail}}
\nc{\inv}{^{-1}}
\nc{\isoto}{\mathop{\overset{\sim}\to}}
\nc{\isotoo}{\mathop{\overset{\sim}\too}}
\nc{\onto}{\mathop{\twoheadrightarrow}}
\nc{\too}{\mathop{\longrightarrow}\limits}
\nc{\mapstoo}{\longmapsto}
\nc{\adh}[1]{\overline{#1}}
\nc{\adhpt}[1]{\adh{\{#1\}}}
\nc{\aka}{{a.\,k.\,a.}\ }
\nc{\calF}{\mathcal{F}}
\nc{\eg}{{\sl e.\,g.}}
\nc{\hook}{\hookrightarrow}
\nc{\borel}[2]{b_{#1}{#2}}
\nc{\ideal}[1]{\langle #1\rangle}

\dmo{\Hom}{Hom}
\nc{\Homcat}[1]{\Hom_{\cat #1}}
\nc{\ihomname}{\mathsf{hom}}
\nc{\ihom}[1]{\mathsf{hom}(#1)}
\nc{\ihomC}[1]{\mathsf{hom}_{\cat C}(#1)}
\nc{\ihomD}[1]{\mathsf{hom}_{\cat D}(#1)}
\nc{\ihomT}[1]{\mathsf{hom}_{\cat T}(#1)}
\nc{\ihomS}[1]{\mathsf{hom}_{\cat S}(#1)}
\nc{\ihomTU}[1]{\mathsf{hom}_{\cat T(U)}(#1)}
\nc{\ihomTV}[1]{\mathsf{hom}_{\cat T(V)}(#1)}
\nc{\ihomsub}[2]{\mathsf{hom}_{#1}(#2)}

\nc{\Mid}{\,\big|\,}
\nc{\MMod}{\,\text{-}\Mod}%
\nc{\GrMMod}{\,\text{-}\grMod}%
\nc{\op}{^{\opname}}
\nc{\oto}[1]{\overset{#1}\to}
\nc{\otoo}[1]{\overset{#1}{\,\too\,}}
\nc{\sminus}{\!\smallsetminus\!}
\nc{\poplus}[1]{^{\oplus #1}}%
\nc{\potimes}[1]{^{\otimes #1}}
\nc{\sbull}{{\scriptscriptstyle\bullet}}
\nc{\SET}[2]{\big\{\,#1\Mid#2\,\big\}}
\nc{\SpcK}{\Spc(\cat K)}
\nc{\then}{\Rightarrow}
\nc{\unit}{\mathbb{1}}
\nc{\xra}{\xrightarrow}
\nc{\phigeom}[1]{\widetilde{\Phi}^{#1}}
\dmo{\Oname}{O}
\dmo{\proper}{proper}
\dmo{\lenormal}{\unlhd}
\dmo{\lnormal}{\lhd}
\nc{\normal}{\trianglelefteq}
\nc{\Op}{\Oname^p}
\nc{\Oq}{\Oname^q}
\dmo{\Sp}{Sp}
\dmo{\Ho}{Ho}
\dmo{\Fin}{Fin}
\dmo{\add}{add}
\dmo{\Fun}{Fun}
\dmo{\Ext}{Ext}
\dmo{\CAlg}{CAlg}
\dmo{\CMon}{CMon}
\dmo{\CC}{\cat C} 
\dmo{\DD}{\cat D}
\dmo{\OO}{\mathcal{O}}
\dmo{\Map}{Map}
\dmo{\Span}{Span}
\dmo{\N}{N}
\dmo{\Cat}{Cat}
\dmo{\colim}{colim}
\dmo{\hocolim}{hocolim}
\dmo{\Ch}{Ch}
\dmo{\A}{\mathbb{A}^{eff}}
\nc{\AGeff}{\mathbb{A}_G^{\mathrm{eff}}}
\nc{\BGeff}{\mathcal{B}_G^{\mathrm{eff}}}
\nc{\BG}{{\mathcal{B}_G}}
\nc{\NBGeff}{{\N}{\BGeff}}
\dmo{\Ab}{Ab}
\dmo{\Set}{Set}
\dmo{\ev}{ev}
\dmo{\Spcl}{Spcl}
\nc{\Funadd}{\Fun_{\add}}
\dmo{\proj}{proj}
\dmo{\cof}{cof}

\dmo{\Coideal}{Coideal}
\dmo{\gen}{gen}
\nc{\auxcoidealsymb}{\vartriangleleft}

\dmo{\Loc}{Loc}
\dmo{\Coloc}{Coloc}
\dmo{\Locideal}{Locid}
\dmo{\Colocideal}{Colocid}
\nc{\LOCO}{\Locideal}
\nc{\COLOCO}{\Colocideal}
\nc{\Loco}[1]{\LOCO\langle #1 \rangle}
\nc{\Coloco}[1]{\COLOCO\langle #1 \rangle}

\nc{\Thickidset}[1]{\mathcal{THICK}_{\otimes}(#1)}
\nc{\Locidset}[1]{\mathcal{LOC}_{\otimes}(#1)}
\nc{\Colocidset}[1]{\mathcal{C}\mathrm{olocid}(#1)}

\nc{\LambdaN}{\Lambda^{\hspace{-0.2ex}\bbN}}
\nc{\LambdaP}{\Lambda^{\hspace{-0.2ex}\cat P}} 
\nc{\LambdaQ}{\Lambda^{\hspace{-0.2ex}\cat Q}} 
\nc{\GammaP}{\Gamma_{\hspace{-0.3ex}\cat P}} 
\nc{\GammaQ}{\Gamma_{\hspace{-0.3ex}\cat Q}} 
\nc{\GammaphiQ}{\Gamma_{\hspace{-0.3ex}\varphi(\cat Q)}} 
\nc{\LambdaW}{\Lambda^{\hspace{-0.3ex}W}} 
\nc{\GammaW}{\Gamma_{\hspace{-0.3ex}W}} 
\nc{\GammainvW}{\Gamma_{\hspace{-0.3ex}\varphi^{-1}(W)}}
\nc{\LambdainvW}{\Lambda^{\varphi^{-1}(W)}}
\nc{\LambdainvphiP}{\Lambda^{\varphi^{-1}(\{\varphi(\cat P)\})}}
\nc{\GammaphiP}{\Gamma_{\hspace{-0.3ex}\varphi(\cat P)}}
\nc{\LambdaphiP}{\Lambda^{\varphi(\cat P)}}
\nc{\gW}{g_W}
\nc{\gP}{g_{\cat P}}
\nc{\gQ}{g_{\cat Q}}
\nc{\gphiQ}{g_{\varphi(\cat Q)}}
\nc{\cC}{{\cat C}}
\nc{\cT}{{\cat T}}
\nc{\cD}{{\cat D}}

\newcommand\noloc{%
  \nobreak
  \mspace{6mu plus 1mu}
  {:}
  \nonscript\mkern-\thinmuskip
  \mathpunct{}
  \mspace{2mu}
}

\nc{\mT}{\kern-0.5em\mod\kern-0.1em\text{-}\cat{T}^c}
\nc{\mTc}{\kern-0.5em\mod\kern-0.1em\text{-}\cat{T}^c}
\nc{\MTc}{\Mod\kern-0.1em\text{-}\cat{T}^c}
\nc{\MT}{\Mod\kern-0.1em\text{-}\cat{T}}
\newcounter{enum-resume-hack}
\Crefname{Thm}{Theorem}{Theorems}
\Crefname{Prop}{Proposition}{Propositions}
\Crefname{Rem}{Remark}{Remarks}
\Crefname{thmx}{Theorem}{Theorems}

\counterwithout{footnote}{part}

\begin{document}

\title{Cosupport in tensor triangular geometry}

\author{Tobias Barthel}
\author{Nat{\`a}lia Castellana}
\author{Drew Heard}
\author{Beren Sanders}
\date{\today}

\makeatletter
\patchcmd{\@setaddresses}{\indent}{\noindent}{}{}
\patchcmd{\@setaddresses}{\indent}{\noindent}{}{}
\patchcmd{\@setaddresses}{\indent}{\noindent}{}{}
\patchcmd{\@setaddresses}{\indent}{\noindent}{}{}
\makeatother

\address{Tobias Barthel, Max Planck Institute for Mathematics, Vivatsgasse 7, 53111 Bonn, Germany}
\email{tbarthel@mpim-bonn.mpg.de}
\urladdr{https://sites.google.com/view/tobiasbarthel/home}

\address{Nat{\`a}lia Castellana, Departament de Matem\`atiques, Universitat Aut\`onoma de Barcelona, 08193 Bellaterra, Spain}
\email{natalia@mat.uab.cat}
\urladdr{http://mat.uab.cat/$\sim$natalia}

\address{Drew Heard, Department of Mathematical Sciences, Norwegian University of Science and Technology, Trondheim}
\email{drew.k.heard@ntnu.no}
\urladdr{https://folk.ntnu.no/drewkh/}

\address{Beren Sanders, Mathematics Department, UC Santa Cruz, 95064 CA, USA}
\email{beren@ucsc.edu}
\urladdr{http://people.ucsc.edu/$\sim$beren/}

\begin{abstract}
We develop a theory of cosupport and costratification in tensor triangular geometry. We study the geometric relationship between support and cosupport, provide a conceptual foundation for cosupport as categorically dual to support, and discover surprising relations between the theory of \mbox{costratification} and the theory of stratification. We prove that many categories in algebra, topology and geometry are costratified by developing and applying descent techniques. An overarching theme is that cosupport is relevant for diverse questions in tensor triangular geometry and that a full understanding of a category requires knowledge of both its support and its cosupport.
\mbox{}\vspace*{-2\baselineskip}
\end{abstract}



\vspace*{-0.1cm}
\maketitle

{
\hypersetup{linkcolor=black}
\tableofcontents
}

\section{Introduction}

One approach for comparing the objects of a given category is via a support theory. The prototypical example is given by the support of a module in commutative algebra, which inspired similar constructions in related areas such as algebraic geometry, representation theory, and chromatic homotopy theory, to name just a few. The common perspective is to view the objects of the category as ``bundles'' over a geometric base space and record the points at which an object does not vanish. These theories of support have been remarkably successful in organizing the objects being studied, especially in derived and homotopical contexts \cite{BensonCarlsonRickard97, Thomason97, HopkinsSmith98}.

Tensor triangular geometry \cite{BalmerICM} puts these developments in a unified framework: It regards a tensor-triangulated category~$\cat T$ as a bundle over a certain space, its Balmer spectrum, and constructs the universal theory of support for the compact objects of $\cat T$. This universal theory of support classifies the compact objects up to how they build each other using the naturally available categorical structure. The theory was extended beyond compact objects to ``infinite-dimensional'' contexts in \cite{Neeman92a,HoveyPalmieriStrickland97,BalmerFavi11,BensonIyengarKrause08,Stevenson13,bhs1}. The theory of support is more subtle for non-compact objects but in desirable situations extends the classification of compact objects to all objects. This is characterized by a property of the category called stratification.

The goal of this paper is to systematically develop a dual theory of cosupport in tensor triangular geometry as well as the accompanying notion of costratification. Conceptually, this may be motivated from three complementary points of view:
	\begin{itemize}
		\item (Geometric) In algebraic geometry, the support of a quasi-coherent sheaf is captured by its local cohomology. The cosupport corresponds to local homology. Grothendieck's local duality expresses the relation between the two notions through an adjunction, which can be formalized and provides a definition of cosupport in more general settings. This is how cosupport has traditionally been introduced into the literature.
		\item (Constructive) Theories of support provide an approach to understanding infinite-dimensional objects in terms of how they build each other using the naturally available ``finite'' structure together with infinite coproducts. This amounts to studying the localizing ideals of the category. While the idea of generating objects using coproducts is instilled in us from birth, it is not the only choice: We can also gain insight by considering how objects build each other using other constructions, such as infinite products. This leads to studying the colocalizing coideals of the category, and is the organizational framework that cosupport provides.
		\item (Categorical) The cosupport of a compactly generated category $\cat T$ can be understood as the support associated to the opposite category $\cat T\op$. This provides a very conceptual understanding of cosupport which has been missing from the literature. It also highlights part of the subtlety of the theory as $\cat T\op$ fails to be compactly generated in all but trivial cases.
	\end{itemize}
Superficially, this seems to suggest that cosupport and its properties are merely a formal variant of the already established theory of support. However, this conclusion would be false, as we will demonstrate throughout this work. Moreover, it turns out that cosupport appears naturally even in situations where one is only interested in support, and only the combination of both provides a full picture.

Our construction of cosupport takes place in the context of a rigidly-compactly generated tensor-triangulated (``tt'') category $\cat T$ whose Balmer spectrum of compact objects $\Spc(\cat T^c)$ is weakly noetherian (a mild point-set topological assumption introduced in \cite{bhs1} which simultaneously generalizes noetherian and profinite). It takes the form of an assignment
	\[
		\Cosupp\colon \big\{ \text{objects of $\cat T$} \big\} \longrightarrow \big\{ \text{subsets of $\Spc(\cat T^c)$}\big\}
	\]
satisfying a number of compatibilities with respect to the tensor-triangulated structure of $\cat T$. We briefly and informally summarize the main features of our theory of cosupport as follows:
\begin{itemize}
	\item Cosupport illuminates key properties of support. For instance, the local-to-global principle for support is equivalent to the detection property for cosupport, and stratification can be characterized in terms of the combined behaviour of support and cosupport. The behaviour of cosupport also reflects the topology of the Balmer spectrum and thus provides insight into the structure of compact objects. More generally, we study the geometric relationship between support and cosupport and how they are intertwined through intrinsic dualities in $\cat T$.
    \item Cosupport affords an accompanying notion of costratification, which attempts to parameterize colocalizing coideals of $\cat T$ in terms of subsets of the Balmer spectrum. The basic properties of this theory are dual to those of stratification as developed in \cite{bhs1}. For example, we prove that our theory of cosupport is the universal choice for classifying colocalizing coideals. In particular, this implies that if $\cat T$ is costratified in the sense of Benson--Iyengar--Krause \cite{BensonIyengarKrause12} then it is also costratified in our sense, but there are many classes of examples for which the converse fails. 
	\item We clarify the relation between stratification and costratification, discovering a surprising asymmetry between the corresponding notions of (co)detection and (co)local-to-global principle. Nevertheless, while costratification is an \emph{a priori} stronger property than stratification, we verify that all known examples of stratified categories are also costratified. These results are obtained as applications of general descent techniques. Our methods provide streamlined proofs of all known classifications of colocalizing coideals in the literature, and also establish new ones which were not previously accessible.
	\item We unify support and cosupport by showing that they both arise as particular instances of a more general notion of support defined at a level of generality which encompasses both $\cat T$ and $\cat T\op$. This leads to a deeper conceptual understanding of cosupport as simply the support of the opposite category.
\end{itemize}

\subsection*{Content and summary of main results} 
We now proceed to give a more detailed outline of the main results of the paper. These can be loosely organized into six interconnected themes, as follows. Note however that the story being told here does not faithfully reflect the linear structure of the document, for which we instead refer to the end of the introduction.

\begin{Hyp*}
	Throughout the introduction, $(\cat T,\otimes,\unit)$ will denote a rigidly-compactly generated tt-category with weakly noetherian spectrum $\Spc(\cat T^c)$. We will denote the internal hom of $\cat T$ by $\ihom{-,-}$.
\end{Hyp*}

\subsection*{Theme I.~Cosupport and costratification}\label{themeI:cosupport}
The common starting point for the definition of support and cosupport in our tt-geometric setting is the existence of a suitable supply of idempotents $\gP \in \cat T$ which can be used to isolate attention at each point $\cat P \in \Spc(\cat T^c)$. The construction of these idempotents relies on our topological assumption that $\Spc(\cat T^c)$ is weakly noetherian. We then define the support and cosupport of an object $t \in \cat T$ (\cref{def:BF-support-and-cosupport}) as the following subsets of $\Spc(\cat T^c)$:
	\[
		\Supp(t) = \SET{\cat P}{g_{\cat P} \otimes t \neq 0} \quad \text{and} \quad \Cosupp(t) = \SET{\cat P}{\ihom{g_{\cat P},t} \neq 0}.
	\]
This definition of support is due to Balmer--Favi \cite{BalmerFavi11} and was studied in \cite{bhs1}. The definition of cosupport is inspired by the constructions of \cite{HoveyStrickland99}, \cite{Neeman11} and especially \cite{BensonIyengarKrause12}, where a similar definition is considered in the context of a triangulated category equipped with an auxiliary action of a commutative noetherian ring. We first extract the elementary properties the function $\Cosupp(t)$ satisfies and thereby formulate the axiomatic notion of a cosupport theory (\cref{def:axiomaticcosupp} and \cref{prop:tt_cosupport}). This is summarized as follows:

\begin{thmx*}\label{thmx:cosupport-is-cosupport}
	Cosupport satisfies the conditions of an axiomatic cosupport theory:
		\begin{enumerate}
			\item $\Cosupp(0) = \emptyset$ and $\Cosupp(\cat T)=\Spc(\cat T^c)$;
			\item $\Cosupp(\Sigma t) = \Cosupp(t)$ for every $t \in\cat T$;
			\item $\Cosupp(c) \subseteq \Cosupp(a) \cup \Cosupp(b)$ for any exact triangle $a \to b \to c$ in~$\cat T$;
			\item $\Cosupp(\prod_{i\in I} t_i) = \bigcup_{i \in I} \Cosupp(t_i)$ for any set of objects $t_i$ in $\cat T$;
			\item $\Cosupp(\ihom{s,t}) \subseteq \Cosupp(t)$ for all $s,t \in \cat T$.
		\end{enumerate}
\end{thmx*}
Although basic, formulating the axioms correctly is subtle due to the asymmetric interaction between cosupport and support discussed in \hyperref[themeII:asymmetry]{Theme~II} below, which is hinted at by the unusual form of axiom (e). The key principle which guides our choice of axioms is the observation that, while the collection of objects supported on a given set forms a localizing ideal, the collection of objects cosupported on a given set forms a colocalizing coideal (\cref{def:subcategories}). The above axioms (excluding $\Cosupp(\cat T)=\Spc(\cat T^c)$) are in fact equivalent to the statement that the collection of objects cosupported on a given set form a colocalizing coideal. Cosupport is thus intimately related to the study of colocalizing coideals in the same way that support is intimately related to the study of localizing ideals.

In \cite{bhs1} we defined the category $\cat T$ to be stratified if the map
	\[
		\Supp\colon \big\{ \text{localizing ideals of $\cat T$} \big\} \longrightarrow \big\{ \text{subsets of $\Spc(\cat T^c)$}\big\}
	\]
induced by the Balmer--Favi support is a bijection. Although one could consider the analogous statement for any support theory on $\cat T$, we proved that the Balmer--Favi notion of support provides the universal choice of support theory for the purposes of stratifying $\cat T$. This justifies defining stratification as a property of the category as above, rather than as a notion relative to a choice of auxiliary support theory.

Similarly, we say that $\cat T$ is costratified (\cref{def:costratification}) if the map
	\[
		\Cosupp\colon \big\{ \text{colocalizing coideals of $\cat T$} \big\} \longrightarrow \big\{ \text{subsets of $\Spc(\cat T^c)$}\big\}
	\]
induced by our tensor triangular cosupport theory is bijective. This intrinsic definition of costratification as a property of $\cat T$ is justified by a corresponding universality result for our cosupport theory:

\begin{thmx*}[Informal version]
	Cosupport is the universal choice among costratifying \mbox{cosupport} theories for $\cat T$ which is compatible with the usual classification of compact objects in~$\cat T$. In particular, if $\cat T$ is costratified in the sense of \cite{BensonIyengarKrause12}, then it is also costratified in our sense. 
\end{thmx*}

\vspace{-0.25ex}
More precise statements are in \cref{cor:Top-uniqueness} and \cref{cor:BIK-costrat}. We merely remark in passing that this universality result for cosupport is more subtle than the corresponding result for support, since the restriction of $\Cosupp$ to the compact objects does not in general coincide with Balmer's universal support for compact objects. Examples abound of tt-categories which are costratified in our sense but not in the sense of \cite{BensonIyengarKrause12}, and so---in our tensor triangular setting---our theory of costratification is strictly more general than that of \cite{BensonIyengarKrause12}.

In order to discuss our results further, we need additional preparation. Recall that the spectrum of $\cat T$ controls its global geometric structure through a divide and conquer approach: First decompose $\cat T$ into pieces $\GammaP \cat T$ that capture those objects with support concentrated at a single point $\cat P \in \Spc(\cat T^c)$ and secondly study these individual pieces (or ``stalks''). Intuitively, the local-to-global principle stipulates that $\cat T$ can be reconstructed from its local pieces. In other words, each object of~$\cat T$ can be built from its stalks. In particular, it implies the detection property: $\Supp(t)=\emptyset$ if and only if $t=0$. If $\cat T$ satisfies the local-to-global principle, there is then a condition on the local pieces $\GammaP\cat T$ which characterizes when the category~$\cat T$ is stratified, namely that the stalks $\GammaP\cat T$ are minimal as localizing ideals. This is the beginning to the theory of stratification developed systematically in \cite{bhs1}.

Now we can consider the analogue of these notions based on cosupport: We define the colocalizing coideal $\LambdaP\cat T$ of $\cat T$ consisting of all objects which are cosupported at a single point $\cat P\in \Spc(\cat T^c)$ and we have a corresponding colocal-to-global principle (\cref{def:LGP}) which morally states that every object can be built from these ``costalks''. This in turn implies the codetection property (\cref{def:codetection}) which states that $\Cosupp(t) = \emptyset$ if and only if $t=0$. In complete analogy with the theory of stratification we have (\cref{thm:equiv-costrat}):

\begin{thmx*}\label{thmx:costratification}
	The following conditions are equivalent:
		\begin{enumerate}
			\item $\cat T$ is costratified;
			\item $\cat T$ satisfies the colocal-to-global  principle, and the colocalizing coideal $\LambdaP\cat T$ is minimal for each $\cat P \in \Spc(\cat T^c)$.
		\end{enumerate}
\end{thmx*}

\vspace{-0.25ex}
In this way, the theory of costratification has the same basic features as the theory of stratification developed in \cite{bhs1}.

\vspace{-0.25ex}
\subsection*{Theme II.~Asymmetry between stratification and costratification}\label{themeII:asymmetry}

Although stratification and costratification have analogous characterizations, as described above, there is a remarkable asymmetry in the relationship between the two properties. Surprisingly, the local-to-global principle is equivalent to the codetection property and these are both equivalent to the colocal-to-global principle (\cref{thm:LGP-equiv}):

\begin{thmx*}\label{thmx:local-to-global}
	The following conditions are equivalent:
		\begin{enumerate}
			\item $\cat T$ satisfies the local-to-global principle;
			\item $\cat T$ satisfies the colocal-to-global principle;
			\item $\cat T$ satisfies the codetection property.
		\end{enumerate}
	These conditions imply the detection property, but the converse does not hold in general.
\end{thmx*}

Moreover, costratification implies stratification (\cref{thm:costrat_implies_strat} and \cref{cor:costrat-perps}):

\begin{thmx*}\label{thmx:costrat-implies-strat}
	If $\cat T$ is costratified, then it is also stratified. Moreover, in this case the map sending a localizing ideal $\cat L$ to its right orthogonal $\cat L^{\perp} = \SET{t \in \cat T}{\ihom{\cat L,t} = 0}$ induces a bijection
		\[
			\{ \text{localizing ideals of $\cat T$} \} \xrightarrow{\sim} \{\text{colocalizing coideals of $\cat T$}\}
		\]
	with inverse given by the left orthogonal.
\end{thmx*}

We thus have a hierarchy of properties that a tt-category can possess, summarized by the following diagram:
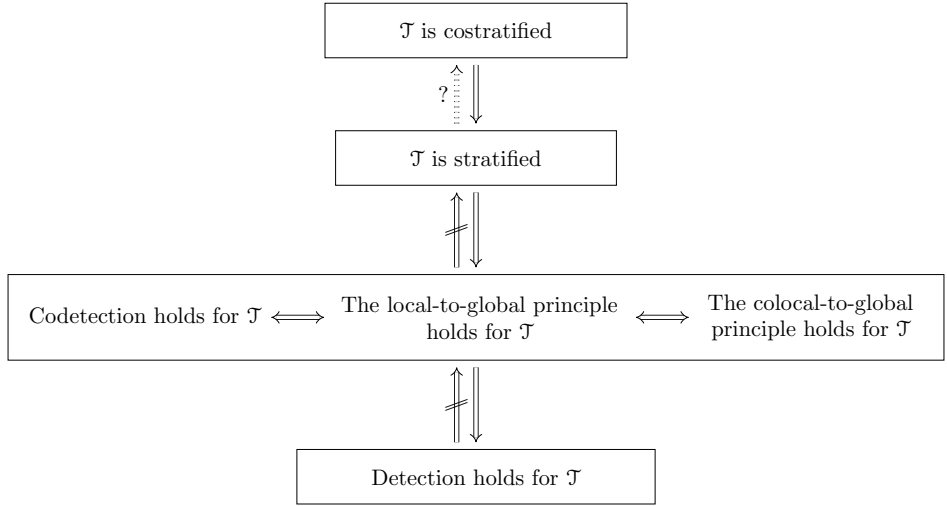
\begin{equation*}\label{eq:big-figure}\resizebox{\columnwidth}{!}{
		\begin{tikzpicture}[mybox/.style={draw, inner sep=5pt}]
		\node[mybox] (row1) at (0,-0.5){%
		  \begin{tikzcd}[ampersand replacement=\&]
			\& \text{$\cat T$ is costratified} \&
		  \end{tikzcd}
		};
		\node[mybox] (row2) at (0,-2.5) {
		  \begin{tikzcd}[ampersand replacement=\&]
			\& \text{$\cat T$ is stratified} \&
		  \end{tikzcd}
			};
		\node[mybox] (row3) at (0,-5) {
		  \begin{tikzcd}[ampersand replacement=\&]
			  {\text{Codetection holds for } \cat T} \ar[Rightarrow, 2tail reversed, r] \& {	{\parbox{13em}{\centering The local-to-global principle \\ holds for $\cat T$}} }\ar[Rightarrow, 2tail reversed, r] \& {{\parbox{10em}{\centering The colocal-to-global \\ principle holds for $\cat T$}}}
		  \end{tikzcd}
			};
		\node[mybox] (row4) at (0,-7.5) {
		  \begin{tikzcd}[ampersand replacement=\&]
			\& \text{Detection holds for $\cat T$} \&
		  \end{tikzcd}
		};
		\draw[-{Implies},double equal sign distance] ([yshift=-3pt]row1.south) -- ([yshift=3pt]row2.north);
		\draw[-{Implies},dotted,double equal sign distance] ([xshift=-2ex,yshift=3pt]row2.north) -- ([xshift=-2ex,yshift=-3pt]row1.south) node[midway,left] {?};
		\draw[-{Implies},double equal sign distance] ([yshift=-3pt]row2.south) -- ([yshift=3pt]row3.north);
		\draw[-{Implies},negated,double equal sign distance] ([xshift=-2ex,yshift=3pt]row3.north) -- ([xshift=-2ex,yshift=-3pt]row2.south);
		\draw[-{Implies},double equal sign distance] ([yshift=-3pt]row3.south) -- ([yshift=3pt]row4.north);
		\draw[-{Implies},negated,double equal sign distance] ([xshift=-2ex,yshift=3pt]row4.north) -- ([xshift=-2ex,yshift=-3pt]row3.south);
		\end{tikzpicture}
}\end{equation*}

It remains an open question whether stratification implies costratification in general. This is related to questions concerning the existence of arbitrary Bousfield localizations and is thereby related to set-theoretic concerns such as Vopĕnka's principle (see \cref{rem:vopenka} and \crefrange{rem:does-strat-imply-costrat}{rem:vopenka-positive}). In light of this discrepancy, costratification remains an \emph{a priori} deeper property than stratification. From a more practical point of view, the proofs of the known instances of costratification---such as \cite{Neeman11} or \cite{BensonIyengarKrause12}---are significantly more involved than their stratification counterparts. We will return to this topic later in \hyperref[themeVI:descent]{Theme~VI}.

We now turn to another asymmetry between support and cosupport. A basic property of support is that 
	\begin{equation*}\label{eq:intro-supp}
		\Supp(s\otimes t)\subseteq \Supp(s) \cap \Supp(t)
		\tag{$\dagger$}
	\end{equation*}
for any $s, t \in \cat T$. The ``half-$\otimes$ formula'' states that this inclusion is an equality when the object~$s$ is compact. Moreover, this is promoted to a ``full-$\otimes$ formula'' (that is,~\eqref{eq:intro-supp} is an equality for all objects) when the category is stratified. In contrast, the behaviour for cosupport is as follows:

\begin{thmx*}
	The following statements hold:
	\begin{enumerate}
		\item For any $s,t \in \cat T$, $\Cosupp(\ihom{s,t}) \subseteq \Supp(s) \cap \Cosupp(t)$.
		\item For any $x \in \cat T^c$ and $t \in \cat T$, $\Cosupp(\ihom{x,t}) = {\Supp(x) \cap \Cosupp(t)}$.
		\item Assume the local-to-global principle holds for $\cat T$. The following conditions are equivalent:
			\begin{enumerate}
				\item[(i)] $\cat T$ is stratified;
				\item[(ii)] For all $s,t \in \cat T$, we have $\Cosupp(\ihom{s,t})=\Supp(s)\cap \Cosupp(t)$.
			\end{enumerate}
	\end{enumerate}
\end{thmx*}

This is established in \cref{lem:cosupp_hom}, \cref{prop:halfhom}, and \cref{thm:strat_cosupport}. Note how the inclusion in part $(a)$, which is the analogue of property \eqref{eq:intro-supp}, involves both cosupport \emph{and} support. Part $(b)$ of the theorem provides the ``half-$\ihomname$ formula''. Part $(c)$ establishes that the half-$\ihomname$ formula promotes to a ``full-$\ihomname$ formula'' if \emph{and only if} the category is stratified (provided the local-to-global principle holds). This fundamental relationship between stratification and the behaviour of \emph{cosupport} is a discovery due to Benson--Iyengar--Krause \cite{BensonIyengarKrause12}. One might have expected \emph{a priori} that the full-$\ihomname$ formula would instead be more closely related to \emph{co}stratification.

\subsection*{Theme III.~The geometric relationship between support and cosupport}\label{themeIII:geometric}

The above motivates the search for a systematic geometric description of the relation between support and cosupport as subsets of the spectrum. This turns out to be a subtle question, but there are several conceptual results we can prove. For example, we prove (\cref{cor:min}): 

\begin{thmx*}\label{thmx:min-theorem}
	If $\cat T$ has noetherian spectrum, then $\min\Supp(t) = \min\Cosupp(t)$ for any $t \in \cat T$. In general, if the spectrum is not noetherian, this identity can fail.
\end{thmx*}

In the situation of the theorem, one might wonder whether it is possible for the support and cosupport functions to coincide. We prove (\cref{cor:finite-discrete-noetherian}):
\begin{thmx*} 
	If $\cat T$ has noetherian spectrum, then the following are equivalent:
		\begin{enumerate}
			\item $\Supp(t)=\Cosupp(t)$ for all $t \in \cat T$.
			\item $\Spc(\cat T^c)$ is a finite discrete space.
		\end{enumerate}
\end{thmx*}

From a more general perspective, the question of whether support coincides with cosupport is related to the vanishing of the Tate construction and thereby reflects the topology of the Balmer spectrum. For example, we have (\cref{cor:cosupp-spc-disjoint}):

\begin{thmx*}
	Assume that the codetection property holds. Let $Y \subseteq \Spc(\cat T^c)$ be a Thomason subset and let $e_Y \in \cat T$ be the associated left idempotent with $\Supp(e_Y)=Y$. Then $\Cosupp(e_Y) \subseteq Y$ if and only if there is a decomposition $\Spc(\cat T^c) = Y \sqcup Y^c$ as a disjoint union of closed sets.
\end{thmx*}

More refined statements could be made, but the significant point is that understanding cosupport is relevant even for questions purely about the structure of the category $\cat T^c$ of compact objects.

It would be desirable to find a process for computing the cosupport of a given object in terms of its support or vice versa. This turns out to be too optimistic: We provide explicit counterexamples showing that in general the support of an object does not determine its cosupport (\cref{exa:supp-no-cosupp}) and, vice versa, the cosupport of an object does not determine its support (\cref{exa:cosupp-no-supp}). There are, however, more refined ways to relate support and cosupport; for example, by considering pairs of objects related by some notion of duality. To this end, we undertake a general study of dualities in tt-categories (\cref{def:dualizes}) and analyze how support and cosupport transform under them. For example, Brown--Comenetz duality provides a way to construct objects with prescribed cosupport (\cref{prop:all-objects-dualize}):

\begin{thmx*}
	For any $t \in \cat T$, $\Cosupp(t^*) = \Supp(t)$, where $t^*$ denotes the Brown--Comenetz dual of $t$.
\end{thmx*}

Our results on dualities can be applied to a wide variety of examples. The following theorem (\cref{prop:cosupp-for-dualizing}) provides an illustrative example of the type of result which can be obtained. To understand the statement, we say that an object $t \in \cat T$ has small cosupport if $\Cosupp(t) \subseteq \Supp(t)$. Examples include compact objects, and also dualizing complexes in algebraic geometry.

\begin{thmx*}
	Let $\cat T$ be stratified and suppose $\kappa \in \cat T$ dualizes the subcategory $\cat T_0 \subseteq \cat T$. If the objects of $\cat T_0$ have small cosupport, then for any $t \in \cat T_0$ we have
		\[
			\Cosupp(t) = \Supp(t)\cap \Cosupp(\kappa).
		\]
\end{thmx*}

Applied to Spanier--Whitehead duality, the statement of the theorem reduces to the half-$\ihomname$ formula. Applied to the derived category of a commutative noetherian ring which admits a dualizing complex, it extends the half-$\ihomname$ formula to bounded complexes of coherent sheaves. These results clarify a number of results in the literature concerning the relation between cosupport and completion.

\subsection*{Theme IV.~Cosupport is support}\label{themeIV:opposite}

A significant contribution of this paper is a unification of support and cosupport which shows that they are both manifestations of the same construction. This unification is not obvious. Attempts to find a deeper connection between support and cosupport are clouded by the fact that colocalizing coideals are not obviously categorically dual to localizing ideals. The key insight that leads to their unification is the realization that $\cat T$ and $\cat T\op$ share the same rigid tt-category $\cat T^d \cong (\cat T\op)^d$ of dualizable objects and that localizing $\cat T^d$-submodules of~$\cat T$ are precisely the localizing ideals of $\cat T$, while the localizing $\cat T^d$-submodules of $\cat T\op$ are precisely the colocalizing coideals of $\cat T$. Technicalities arise because the opposite category $\cat T\op$ is never compactly generated, but a framework which covers both $\cat T$ and $\cat T\op$ is provided by the notion of a perfectly generated non-closed tt-category (\cref{def:non-closed-tt-category}). Summarizing results from \cref{sec:cosupp-is-supp}, we have:

\begin{thmx*}
	There is a theory of support for perfectly generated non-closed tt-categories~$\cat T$ whose spectrum $\Spc(\cat T^d)$ of dualizable objects is weakly noetherian, and a corresponding notion of stratification for classifying the localizing $\cat T^d$-submodules~of~$\cat T$.
\end{thmx*}

\Cref{exa:supp-for-T}, \cref{exa:supp-for-Top}, and \cref{thm:cosupp-is-supp} then provide:

\begin{thmx*} Let $\cat T$ be a rigidly-compactly generated tt-category.
	\begin{enumerate}
		\item Applied to $\cat T$ the above theory reduces to the theory of support and stratification for localizing ideals of $\cat T$ developed in \cite{bhs1}.
		\item Applied to $\cat T\op$ the above theory reduces to the theory of cosupport and costratification for colocalizing coideals of $\cat T$ developed in this paper.
		\item In particular, 
			\[\Cosupp_{\cat T}(t) = \Supp_{\cat T\op}(t)\]
			for all $t \in \cat T$, the colocalizing coideals of $\cat T$ are precisely the localizing \mbox{$\cat T^d$-submodules} of $\cat T\op$, and $\cat T$ is costratified precisely when $\cat T\op$ is stratified.
	\end{enumerate}
\end{thmx*}

This provides sound conceptual foundations for the construction of cosupport.

\subsection*{Theme V.~Base change of support and cosupport}\label{themeV:base-change}
Another major theme in this paper is the study of the behaviour of support and cosupport in a relative setting, that is under base change along functors between tt-categories. Such results are of fundamental importance in the subject, as they allow us to reduce problems in a category of interest to simpler categories. This is vital in our applications to classification problems discussed below. It also emphasizes the value of basing our theory of cosupport on the Balmer spectrum, which readily affords a geometric perspective for relating tt-categories.

The basic setup is a geometric functor $f^*\colon \cat T\to \cat S$, that is, a tt-functor which preserves coproducts. The functor $f^*$ induces a continuous map on spectra
	\[
		\varphi \coloneqq \Spc(f^*)\colon \Spc(\cat S^c) \to \Spc(\cat T^c).
	\]
Our assumptions also guarantee the existence of two layers of right adjoints:
	\[
		f^* \dashv f_* \dashv f^!.
	\]
We obtain a variety of results concerning the image and preimage under $\varphi$ of support and cosupport. An interesting feature is the prominent role that the double right adjoint $f^!$ plays concerning base change for cosupport. For example, one highlight (\cref{cor:globalavruninscott}) establishes the Avrunin--Scott identities \cite{avruninscott} in a general tensor triangular context:

\begin{thmx*}
	If $f^*\colon \cat T \to \cat S$ a geometric functor with $\cat T$ stratified, then for any $t \in \cat T$:
	\[
		\Supp_{\cat S}(f^*(t)) = \varphi^{-1}(\Supp_{\cat T}(t)) \quad  \text{and} \quad \Cosupp_{\cat S}(f^!(t)) = \varphi^{-1}(\Cosupp_{\cat T}(t)).
	\]
\end{thmx*}

\noindent
These identities play a vital role in the study of descent properties for stratification, a topic discussed in more detail below.

Another result (\cref{cor:img-of-spc}) gives an unconditional description of the image:
\begin{thmx*}
	If $f^*\colon \cat T \to \cat S$ is a geometric functor, then $\im \varphi = \SuppT(f_*(\unitS))$.
\end{thmx*}
A number of more refined statements are obtained. For example, we prove that if $f^*$ satisfies Grothendieck--Neeman duality (in the sense of \cite{BalmerDellAmbrogioSanders16}) then the induced map on spectra $\varphi$ is a closed map (\cref{rem:map-is-closed}). We also exhibit a close relationship between the surjectivity of $\varphi$ and the conservativity of the functors $f^*$ and $f^!$; see, e.g., \cref{prop:weaklyfinite-surjective-conservative} and \cref{cor:conservative-iff-surjective}.

These are just the highlights of our base change results; they already demonstrate that cosupport arises naturally in the study of any geometric functor $f^*\colon \cat T \to \cat S$, providing insight into the behaviour of the double right adjoint $f^!\colon \cat T\to \cat S$.

\subsection*{Theme VI.~Descent and applications}\label{themeVI:descent}
Another theme is the development of general techniques which allow us to establish costratification for numerous categories of interest. These techniques will provide streamlined proofs for all known classifications of colocalizing coideals in the literature and also apply to new classes of examples that were not previously accessible. To provide proper context for our results, we briefly review a blueprint for proving stratification, and then explain how to bootstrap this process to establish costratification.

Suppose $\cat T$ is a tt-category that we wish to prove is stratified. This proceeds naturally in three steps:

\begin{enumerate}[labelindent=\parindent,leftmargin=!]
	\item[\textit{Step 1}] Construct geometric functors $f_i^*\colon \cat T \to \cat S_i$ to simpler tt-categories $\cat S_i$ such that the images of the maps $\varphi_i\coloneqq\Spc(f_i^*)$ jointly cover $\Spc(\cat T^c)$.
	\item[\textit{Step 2}] Prove stratification for the categories $\cat S_i$. 
	\item[\textit{Step 3}] Descend stratification along the functors $f_i^*$.
\end{enumerate}
Progress usually hinges on improvements in Step 3, i.e., finding more general criteria for descending stratification. To this end, we add to the existing toolbox for establishing stratification in the form of \emph{quasi-finite descent} (\cref{thm:etaledescent-strat}) and \emph{nil-descent} (\cref{thm:nilpotentdescent}).

Now the question presents itself of whether a similar approach works for costratification. We show that this is indeed the case in the strongest possible way, provided we already have stratification. In other words, we prove that whenever we can descend stratification, we can also descend costratification. The key insight is the following `bootstrap theorem':

\begin{thmx*}\label{thmx:bootstrap}
	Suppose $f^*\colon \cat T \to \cat S$ is a conservative geometric functor such that $\cat S$ is costratified. If $\cat T$ is stratified, then it is also costratified. 
\end{thmx*}

This is \cref{cor:detectingcomin} and is one of our main results. In order to run the descent strategy and apply the theorem, we also need a sufficient supply of costratified categories. One source is provided by our next result which combines \cref{cor:stratttfields} and \cref{thm:costratification_for_weakly_affine}:

\begin{thmx*}\label{thmx:costrat-base-cases}
	The following tt-categories $\cat T$ are costratified (and hence also stratified):
		\begin{enumerate}
			\item $\cat T$ is pure-semisimple (e.g., a tt-field in the sense of \cite{BalmerKrauseStevenson19}). In this case, $\Spc(\cat T^c)$ is a finite discrete space.
			\item $\cat T$ is affine weakly regular in the sense of \cite{DellAmbrogioStanley16}. In this case, there is a canonical homeomorphism $\Spc(\cat T^c) \cong \Spec^h(\End_{\cat T}^*(\unit))$.
		\end{enumerate}
\end{thmx*}

Our methods are flexible and have wide applicability to diverse classes of examples in algebraic geometry, representation theory and homotopy theory. The next theorem collects a sample of our applications; further examples can be found in the main body of the paper.

\begin{thmx*}\label{thmx:costrat-examples}
	\hspace{1em}
	\begin{enumerate}
		\item Let $X$ be a quasi-compact and quasi-separated scheme which is topologically weakly noetherian. The derived category $\Derqc(X)$ is stratified if and only if it is costratified. In particular, it follows that  $\Derqc(X)$ is costratified for any noetherian scheme $X$. (See \cref{thm:qc_shseaves}.)
		\item For $X$ a $p$-good connected space with noetherian mod $p$ cohomology, the category of modules over the cochain algebra $\Mod(C^*(X;\Fp))$ is stratified if and only if it is costratified if and only if $X$ satisfies Chouinard's condition. In particular, these conditions hold for connected noetherian $H$-spaces. (See \cref{thm:cochains}.)
		\item The category of $E_n$-local spectra is costratified. (See \cref{thm:chromatic_costratification}.)
		\item Let $G$ be a finite group and let $\bbE_G \in \CAlg(\Sp_G)$ be a commutative equivariant ring spectrum such that the non-equivariant derived categories $\Der(\Phi^H\bbE_G)$ are costratified with noetherian spectrum for each $H \le G$, where $\Phi^H\bbE_G$ denotes the geometric fixed points. Then $\Der(\bbE_G)$ is costratified. (See \cref{thm:equivariant_costratification}.)
		\item As a special case of $(d)$, the category of derived Mackey functors is costratified for any finite group $G$. (See \cref{cor:derived_mackey_costratified}.)
		\item As a special case of $(d)$ and \cite{BCHNP1}, the category of equivariant modules over Borel-equivariant Morava $E$-theory is costratified. (See \cref{thm:borel-E_n}.)
		\item The category of rational $G$-spectra is costratified for compact Lie groups $G$. (See \cref{thm:rational-spectra}.)
		\item The derived category of permutation modules $\mathrm{DPerm}(G,k)$ is costratified for any finite group $G$ and field $k$. (See \cref{thm:dpermstratification}.)
	\end{enumerate}
	In each of these cases, we get a classification of localizing ideals and colocalizing coideals in terms of the (known) underlying set of the corresponding Balmer spectrum.
\end{thmx*}

For each of these examples, stratification was already known, and we give precise attributions and references in the main text.

Another new class of costratified categories, which uses the techniques of this paper but whose proof lies outside its scope, are stable module categories $\StMod(G,R)$ of finite groups with coefficients in noetherian commutative rings; see \cite{BBIKP_stratification}. In the present paper, we include a streamlined proof of the classical $R=k$ case, handling elementary abelian groups via Galois ascent (\cref{prop:galoisascent}) à la Mathew; see \cref{thm:rep_g_k_costratifcation} and \cref{thm:rep_e_k_costratifcation}.

In summary, we have exhausted the list of all stratified tt-categories we are aware of and have shown that each of them is also costratified. This in particular includes the case of $X = S^3\langle 3\rangle$ in part $(b)$ of the previous theorem, which was not accessible to previous technology; see \cref{exa:3-connected-cover}.

\subsection*{Outline of the document}
The paper consists of four parts.

\Cref{part:cosupport-costratification} begins with \cref{sec:colocalizing-coideals} where we state our terminological conventions, introduce notation, and recall fundamental facts about colocalizing coideals and Bousfield localization. We define the notion of a cosupport theory in \cref{sec:cosupport-theories}. In \cref{sec:BF-cosupport}, we define the tensor triangular cosupport of a rigidly-compactly generated category and investigate its elementary properties. This study continues in \cref{sec:stalk-and-costalk} where we define the stalk and costalk at each point and introduce the detection and codetection properties. We then proceed to a study of the local-to-global principle and the colocal-to-global principle in \cref{sec:LGP} where we show that they are equivalent and in fact also equivalent to the codetection property. In \cref{sec:costratification}, we define costratification, establish the conditions which characterize when a category is costratified, and also prove that costratification implies stratification. We conclude this part with \cref{sec:geometry-of-cosupport} which is a study of the geometric behaviour of cosupport and its relation to support.

In \cref{part:perfection-duality} we switch from the world of rigidly-compactly generated tt-categories to a more general setting. We recall the notion of a perfectly generated triangulated category in \cref{sec:perfect-generation} and recall how the opposite category of a compactly generated category is perfectly generated. In \cref{sec:cosupp-is-supp}, we set up a theory of support for perfectly generated tt-categories, and show that it provides a unification of support and cosupport. In \cref{sec:universality}, we prove a universality result for our general theory of support, obtaining a universality result for cosupport as a special case. Finally, in \cref{sec:duality} we study how support and cosupport are related under intrinsic dualities, such as Spanier--Whitehead duality and Brown--Comenetz duality.

In \cref{part:morphisms-and-descent} we study base change and descent results for (co)support and (co)stratification. In particular, we study the image of a geometric functor, including critera for it to be surjective, in \cref{sec:image}, and base change results for (co)support in \cref{sec:base-change}. In \cref{sec:descending-LGP}, we apply these results to study descent for the local-to-global principle. In \cref{sec:local-cogeneration}, we obtain local cogenerators for our costalk categories, which is an important technical ingredient for the results which follow. This culminates in \cref{sec:bootstrap} where we establish our key `bootstrap' theorem which allows us to descend costratification whenever we can descend stratification. We also establish descent techniques for stratification to power the theorem.

In \cref{part:applications-examples} we turn to applications and examples. We discuss abstract tensor triangular examples in \cref{sec:ttexamples}, algebraic examples in \cref{sec:algebraicexamples}, and homotopical examples in \cref{sec:homotopicalexamples}. We conclude the paper in \cref{sec:open-questions} with a list of open questions, which we hope will stimulate further research.

\subsection*{Acknowledgements}

We thank Scott Balchin, Paul Balmer, David Rubinstein and Changhan Zou for their interest. We would also like to thank the Max Planck Institute and the Hausdorff Research Institute for Mathematics for their hospitality in the context of the Trimester program \emph{Spectral Methods in Algebra, Geometry, and Topology} funded by the Deutsche Forschungsgemeinschaft (DFG, German Research Foundation) under Germany’s Excellence Strategy – EXC-2047/1 – 390685813.
The first-named author is supported by the European Research Council (ERC) under Horizon Europe (grant No.~101042990). The second-named author is partially supported by Spanish State Research Agency project PID2020-116481GB-I00, the Severo Ochoa and María de Maeztu Program for Centers and Units of Excellence in R$\&$D (CEX2020-001084-M), and the CERCA Programme/Generalitat de Catalunya. The third-named author is supported by grant number TMS2020TMT02 from the Trond Mohn Foundation. The fourth-named author is supported by NSF grant~DMS-1903429.

\newpage
\part{Cosupport and costratification}\label{part:cosupport-costratification}

\section{Colocalizing coideals}\label{sec:colocalizing-coideals}

\begin{Ter}
	We follow the notation and terminology from \cite{bhs1}. For the majority of the paper $\cat T$ will denote a rigidly-compactly generated tensor-triangulated category (with exceptions in \crefrange{sec:perfect-generation}{sec:universality}). 
	We will denote the tensor by $-\otimes-$ and the unit object by $\unit$.
	We will denote the internal hom by $\ihom{-,-}$ and the abelian group of morphisms by $\Hom_{\cat T}(-,-)$ or $\cat T(-,-)$. We also write $t^\vee \coloneqq \ihom{t,\unit}$ for the dual of an object.
\end{Ter}

\begin{Def}\label{def:subcategories}
	A \emph{localizing subcategory} of $\cat T$ is a thick subcategory which is closed under coproducts. A \emph{localizing ideal} is a localizing subcategory $\cat L$ which is also a tensor-ideal: $\cat L \otimes \cat T \subseteq \cat L$ where $\cat L \otimes \cat T =\SET{s\otimes t}{s\in \cat L ,t\in \cat T}$. A \emph{colocalizing subcategory} of $\cat T$ is a thick subcategory which is closed under products. A \emph{colocalizing coideal} is a colocalizing subcategory $\cat C$ with the property that $\ihom{\cat T,\cat C} \subseteq \cat C$ where $\ihom{\cat T,\cat C} =\SET{\ihom{t,s}}{{s\in \cat C \text{, } t\in \cat T}}$.
\end{Def}

\begin{Not}
	We write $\Loc\langle \cat E\rangle$, $\Loco{\cat E}$, $\Coloc\langle \cat E\rangle$, and $\Coloco{\cat E}$ for the localizing subcategory, localizing ideal, colocalizing subcategory, and colocalizing coideal generated by a collection of objects $\cat E \subseteq \cat T$.
\end{Not}

\begin{Rem}\label{rem:monogenic}
	If $\cat T=\Loc\langle \cat E\rangle$ for a collection of objects $\cat E\subseteq \cat T$, then a colocalizing subcategory $\cat C$ is a coideal if and only if $\ihom{\cat E,\cat C}\subseteq \cat C$. Thus, if~$\cat T$ is monogenic (meaning $\cat T=\Loc\langle \unit\rangle$) then every colocalizing subcategory is automatically a coideal.
\end{Rem}

\begin{Rem}
	For any object $t \in \cat T$, consider the three functors
		\[
			t\otimes - \colon \cat T \to \cat T, \quad
			\ihom{t,-}\colon \cat T \to \cat T, \quad \text{ and }\quad  \ihom{-,t}\colon \cat T\op \to \cat T.
		\]
	Localizing ideals pull back under the first functor to localizing ideals; colocalizing coideals pull back under the second functor to colocalizing coideals; and colocalizing coideals pull back under the third functor to localizing ideals. It follows that for any collection of objects $\cat E \subseteq \cat T$, we have
		\begin{align}
			t \otimes \Loco{\cat E} &\subseteq \Loco{t\otimes \cat E},\label{eq:t@loc}\\
			\ihom{t,\Coloco{\cat E}} &\subseteq \Coloco{\ihom{t,\cat E}}, \text{ and}\label{eq:[t,coloc]}\\
			\ihom{\Loco{\cat E},t} &\subseteq \Coloco{\ihom{\cat E,t}}.\label{eq:[loc,t]}
		\end{align}
\end{Rem}

\begin{Def}[Orthogonal subcategories]\label{def:orthogonal}
	If $\cat E \subseteq \cat T$ is a collection of objects, we define the \emph{right orthogonal} of $\cat E$ to be the full subcategory
		\[ 
			{\cat E}^\perp \coloneqq \SET{t \in \cat T}{\ihom{s,t}=0 \text{ for all } s\in \cat E}.
		\]
	Note that ${\cat E}^\perp$ is a colocalizing coideal of $\cat T$. Moreover, it follows from \eqref{eq:[loc,t]} that ${\cat E}^\perp = \Loco{\cat E}^\perp$. Similarly, the \emph{left orthogonal} is defined by
		\[
			{}^\perp {\cat E} \coloneqq \SET{t \in \cat T}{\ihom{t,s} = 0 \text{ for all } s \in \cat E}.
		\]
	It is a localizing ideal of $\cat T$ and it follows from \eqref{eq:[t,coloc]} that ${}^\perp \cat E = {}^\perp \Coloco{\cat E}$.
\end{Def}

\begin{Rem}
	If $\cat E \subseteq \cat T$ is an ideal then 
		\[
			{\cat E}^\perp = \SET{t \in \cat T}{\cat T(s,t) = 0 \text{ for all } {s \in \cat E}}.
		\]
	Similarly, if $\cat E \subseteq \cat T$ is a coideal then
		\[
			{}^\perp \cat E = \SET{t \in \cat T}{\cat T(t,s) = 0 \text{ for all } {s \in \cat E}}.
		\]
\end{Rem}

\begin{Rem}\label{rem:strictly-localizing}
	A localizing ideal $\cat L$ is \emph{strictly localizing} if the inclusion $\cat L \hookrightarrow \cat T$ has a right adjoint. This is the case if and only if $\cat L$ is the kernel of a Bousfield localization on $\cat T$. In this situation, the subcategory of local objects is $\cat L^\perp$. Significantly, if $\cat L$ is strictly localizing, then $\cat L={}^\perp(\cat L^\perp)$, see for example \cite[Proposition 4.9.1]{Krause10}. Similarly, a colocalizing coideal $\cat C$ is \emph{strictly colocalizing} if the inclusion $\cat C \hookrightarrow \cat T$ has a left adjoint. This is the case if and only if $\cat C$ is the image of a Bousfield localization.  Moreover, in this case, $\cat C=({}^\perp \cat C)^\perp$. We always have a function $\cat L \mapsto \cat L^\perp$ between the strictly localizing ideals of $\cat T$ and the strictly colocalizing coideals of $\cat T$. This assignment is a bijection modulo the set-theoretic question of whether there is only a set of such Bousfield localizations.
\end{Rem}

\begin{Rem}
	The question of whether Bousfield localizations always exist has an interesting history, starting with \cite{Bousfield79}. Of particular note for our purposes is the following:
\end{Rem}

\begin{Thm}[Neeman]\label{thm:set-generated-are-strictly-localizing}
	If $\cat T$ is a well generated triangulated category then every set-generated localizing subcategory is strictly localizing.
\end{Thm}

\begin{proof}
	The key point is that a set-generated localizing subcategory of a well generated triangulated category is itself well generated, hence the inclusion has a right adjoint by Brown representability; see \cite[Remark~1.16 and Proposition~1.21]{Neeman01} or \cite[Proposition~4.9.1 and Theorem~7.2.1]{Krause10}.
\end{proof}

\begin{Rem}\label{rem:vopenka}
	There is no analogous result for colocalizing subcategories (morally, because the proof techniques do not apply to $\cat T\op$). This is the heart of the issue for why there could be ``more'' colocalizing subcategories and why one might expect a classification of colocalizing coideals to be ``harder'' than a classification of localizing ideals. For example, if there is only a set of localizing ideals (e.g., if $\cat T$ is stratified in the sense of \cite{bhs1}) then every localizing ideal is set-generated by \cite[Lemma~3.3]{KrauseStevenson19} and hence is strictly localizing by \cref{thm:set-generated-are-strictly-localizing}. Hence in this case the assignment $\cat L \mapsto \cat L^\perp$ provides an injection from the set of localizing ideals into the collection of colocalizing coideals. A priori, there might be more colocalizing coideals. The difference would disappear if we knew that all colocalizing coideals were strictly colocalizing, but therein lie set-theoretic dragons. For example, Casacuberta--Gutiérrez--Rosický \cite{CasacubertaGutierrezRosicky14} prove that if a large cardinal axiom known as Vopĕnka's principle holds (see \cite[Chapter 6]{AdamekRosicky94}) then every colocalizing subcategory of $\cat T$ is strictly colocalizing provided $\cat T=\Ho(\cat M)$ is the homotopy category of a stable combinatorial model category. It remains an open question whether this is true for arbitrary $\cat T$ and without assuming axioms beyond ZFC.
\end{Rem}

\begin{Rem}
	The modified version of \cite[Lemma 3.3.1]{KrauseStevenson19} provided by \cite[Proposition 3.5]{bhs1} actually establishes that if there is a set of \emph{set-generated} localizing ideals then all localizing ideals are set-generated. With this in hand, a variant of the above argument runs as follows: If there is a set of colocalizing coideals then there is a set of strictly colocalizing coideals, hence a set of strictly localizing ideals, hence a set of set-generated localizing ideals (by invoking \cref{thm:set-generated-are-strictly-localizing}), and hence all localizing ideals are set-generated (and strictly localizing). In summary: If there is a set of colocalizing coideals then there is a set of localizing ideals and they correspond to the strictly colocalizing coideals of $\cat T$.
\end{Rem}

\section{Cosupport theories}\label{sec:cosupport-theories}
Our goal is to classify colocalizing coideals using a suitable notion of ``cosupport'' for the objects of $\cat T$. First we axiomatize the properties such a cosupport theory should satisfy.

\begin{Def}\label{def:axiomaticcosupp}
	Let $X$ be a topological space and let $\mathfrak C\colon \cat T \to \mathcal{P}(X)$ be a function, where $\mathcal{P}(X)$ denotes the power set of (the underlying set of) $X$. This function extends to collections $\cat E$ of objects in $\cat T$ by setting $\mathfrak C(\cat E) = \bigcup_{t\in \cat E}\mathfrak C(t)$. The pair $(X,\mathfrak C)$ is called a \emph{cosupport theory} if it satisfies the following conditions:
    \begin{enumerate}
	    \item $\mathfrak C(0) = \emptyset$ and $\mathfrak C(\cat T)=X$;
	    \item $\mathfrak C(\Sigma t) = \mathfrak C(t)$ for every $t \in\cat T$;
	    \item $\mathfrak C(c) \subseteq \mathfrak C(a) \cup \mathfrak C(b)$ for any exact triangle $a \to b \to c \to \Sigma a$ in~$\cat T$;
	    \item $\mathfrak C(\prod_{i\in I} t_i) = \bigcup_{i \in I} \mathfrak C(t_i)$ for any set of objects $t_i$ in $\cat T$;
	    \item $\mathfrak C(\ihom{s,t}) \subseteq \mathfrak C(t)$ for all $s,t \in \cat T$.
    \end{enumerate}
	We also refer to $\mathfrak C$ as a \emph{cosupport function}. 
\end{Def}

\begin{Rem}\label{rem:axioms-equiv}
	These axioms (excluding $\mathfrak C(\cat T)=X$) are equivalent to the statement that for any subset $Y \subseteq X$, the subcategory $\SET{t \in \cat T}{\mathfrak C(t) \subseteq Y}$ is a colocalizing coideal of $\cat T$. In particular, $\mathfrak C(\Coloco{\cat E})=\mathfrak C(\cat E)$ for any collection of objects $ \cat E \subseteq \cat T$. For example, if $\Coloco{t_1}=\Coloco{t_2}$ then $\mathfrak C(t_1)=\mathfrak C(t_2)$.
\end{Rem}

\begin{Rem}
	Note that we have not yet made use of the topology on $X$. For now, it could therefore be dropped from the definition. We also remark in passing that our axiomatization is similar but not exactly the same as the one given by \cite[Definition~3.2]{Verasdanis22bpp}. Further discussion on this point will be given in \cref{rem:comp-veras}.
\end{Rem}

\begin{Exa}\label{exa:DR-cosupport}
    Let $\Der(R)$ be the derived category of a commutative noetherian ring~$R$. Neeman \cite{Neeman11} has given a classification of the colocalizing coideals of $\Der(R)$ using the assignment
		\[
			\big\{ \text{colocalizing coideals of $\Der (R)$} \big\} \xrightarrow{B(-)} \big\{ \text{subsets of $\Spec(R)$}\big\}
		\]
	where $B(\cat C) \coloneqq \SET{\mathfrak p \in \Spec(R) }{\kappa(\mathfrak p) \in \cat C}$ and $\kappa(\mathfrak p)$ is the residue field of $R$ associated to the prime ideal $\mathfrak p$. We claim that the function $B(t) \coloneqq B(\Coloco{t})$ defined on objects $t \in \Der(R)$ is a cosupport theory in the sense of \Cref{def:axiomaticcosupp}. Indeed, $B(0) = \SET{\mathfrak p \in \Spec(R)}{\kappa(\mathfrak p) \in (0)} = \emptyset$, while $\kappa(\mathfrak p) \in B(\kappa(\mathfrak p))$ by definition so that $\bigcup_{\mathfrak p \in \Spec(R)} B(\kappa(\mathfrak p)) = \Spec(R)$. Condition (b) is clear, while (c) follows because if $a \to b \to c \to \Sigma a$ is an exact triangle, then~$c$ is in the colocalizing subcategory generated by $a$ and $b$. Finally, $(e)$ is a consequence of the inclusion $\Coloco{\ihom{s,t}} \subseteq \Coloco{t}$. Note that this cosupport theory is defined for any commutative ring $R$, although Neeman's theorem requires $R$ noetherian.
\end{Exa}
 
\begin{Exa}\label{exa:SWW-cosupport}
     An alternative approach is given in \cite{SatherWagstaffWicklein17}, where the cosupport of $t \in \Der(R)$ is defined by
		 \[
			\cosuppa_R(t) = \SET{\mathfrak p \in \Spec(R) }{\ihom{k(\mathfrak p),t} \neq 0}. 
		 \]
	 Most of the axioms for cosupport are verified in \cite[Prop.~4.7--4.9]{SatherWagstaffWicklein17}. Axiom (b) is not verified, but is clear from the definition. The axiom $\cosuppa_R(\Der(R)) = \Spec(R)$ follows, for example, from the observation that $\ihom{k(\mathfrak p),k(\mathfrak p)} \ne 0$. The only remaining axiom to check is (e). To this end, suppose $\ihom{k(\mathfrak p),\ihom{s,t}} \ne 0$. Then $\ihom{k(\mathfrak p) \otimes s,t} \ne 0$. But $k(\mathfrak p) \otimes s$ is a coproduct of suspensions of $k(\mathfrak p)$, so that $\ihom{k(\mathfrak p),t} \ne 0$, as well.
\end{Exa}

\begin{Exa}\label{Exa:chromatic_cosupport}
	Let $\cat S_{E(n)}$ denote the category of $E(n)$-local spectra; see \cite{HoveyStrickland99} and \cite[Section 10]{bhs1}. For $t \in \cat S_{E(n)}$, Hovey and Strickland \cite[Section 6]{HoveyStrickland99} consider the chromatic cosupport, defined by 
		\[
			\cosuppa(t)\coloneqq \SET{m \in \{0,\ldots, n\}}{\ihom{K(m),t} \ne 0},
		\]
    where $K(m)$ is the $m$-th Morava $K$-theory. Arguments similar to those used for \cref{exa:SWW-cosupport} establish that this defines a theory of cosupport on $\cat S_{E(n)}$.
\end{Exa}

\begin{Exa}
	If a rigidly-compactly generated tt-category $\cat T$ is equipped with a central action by a graded commutative noetherian ring $R$, then Benson--Iyengar--Krause \cite{BensonIyengarKrause12} provide a cosupport theory $(\cosupp_R(\cat T), \cosupp_R)$ whose space of cosupports $\cosupp_R(\cat T) \subseteq \Spec^h(R)$ lies in the homogeneous spectrum of the acting ring. The cosupport axioms are established in Sections 4, 8 and 9 of their paper. For example, we could take the derived category~$\Der(R)$ of a noetherian commutative ring acted upon by $R$ itself, or we could take the stable module category $\StMod(kG)$ of a finite group $G$ over a field $k$ acted upon by the group cohomology ring $H^*(G;k)$. In the latter example the space of cosupports is $\Proj(H^*(G,k))\subsetneq \Spec^h(H^*(G,k))$. 
\end{Exa}

\section{Tensor triangular cosupport}\label{sec:BF-cosupport}

In this section we introduce the main cosupport theory of interest to us, which is related to the Balmer--Favi support (a.k.a.~small support) introduced in \cite{BalmerFavi11} and studied in depth in \cite{bhs1}. Throughout~$\cat T$ will denote a rigidly-compactly generated tensor-triangulated category whose spectrum $\Spc(\cat T^c)$ is weakly noetherian. We will briefly recall what the latter topological condition means, before proceeding with the definition of cosupport. Further discussion is found in \cite[Section~2]{bhs1}. First we recall some details about smashing and finite localizations and their idempotent triangles.

\begin{Rem}\label{rem:smashing-recollement}
	Let $\cat T$ be a rigidly-compactly generated tt-category. Recall from~\cite[Theorem 2.13]{BalmerFavi11} that a \emph{smashing ideal} is a strictly localizing ideal $\cat L$ which satisfies the following equivalent conditions:
		\begin{enumerate}
			\item $\cat L^\perp$ is a localizing subcategory of $\cat T$;
			\item $\cat L^\perp$ is a localizing ideal of $\cat T$;
			\item $\cat L^\perp$ is an ideal of $\cat T$.
		\end{enumerate}
	Associated to a smashing ideal is an idempotent exact triangle
		\[
			e \to \unit \to f \to \Sigma e
		\]
	and we have the following diagram of adjunctions
		\begin{equation}\label{eq:smashing-diagram}
				\begin{tikzcd}[column sep=-2em,row sep=2em,scale cd=1]
				\cat L = e\otimes \cat T 
			\ar[rr,bend left=7,"\ihom{e,-}"]
			\ar[rr,phantom,"\cong"description]
			\ar[rr,<-,bend right=7,"e\otimes-"']
			\ar[ddr,shift={(-5pt,-5pt)},hook,"\ihom{e,-}"{yshift=-5pt},shift left=2.5ex]
				\ar[ddr,shift={(-5pt,-5pt)},hook,"\mathrm{incl}"',shift right=2.5ex]
				&& \ihom{e,\cat T}=\cat L^{\perp \perp}
				\ar[ddl,shift={(5pt,-5pt)},hook,"\mathrm{incl}"{yshift=-0pt},shift left=2.5ex]
				\ar[ddl,shift={(5pt,-5pt)},hook,"e\otimes -"'{yshift=-5pt},shift right=2.5ex]
				\\
			& & \\
		& \cat T 
		\ar[uul,two heads,shift={(-5pt,-5pt)},"e\otimes-"description]
		\ar[uur,two heads,shift={(5pt,-5pt)},"\ihom{e,-}"description]
		\ar[dd,two heads,"\ihom{f,-}"{yshift=-0pt},shift left=2.5ex]
			\ar[dd,two heads,"f\otimes -"',shift right=2.5ex]
		& \\
		&&\\
		&\cat L^\perp = f\otimes \cat T=\ihom{f,\cat T}
		\ar[uu,hook,"\mathrm{incl}"description]
			\end{tikzcd}
		\end{equation}
	in which each of the six ``vertical" sequences $\bullet\hookrightarrow \cat T \onto \bullet$ is a Bousfield localization. See \cite[Theorem 3.5]{BalmerFavi11}, \cite[Remark 5.3]{BalmerSanders17}, \cite[Section 2]{BarthelCastellanaHeardValenzuela18} and \cite[Section 3.3]{HoveyPalmieriStrickland97} for further discussion.
\end{Rem}

\begin{Exa}[Finite localizations]
	Let $\cat T_Y^c\coloneqq \SET{x \in \cat T^c}{\supp(x)\subseteq Y}$ denote the thick ideal of compact objects corresponding to a Thomason subset $Y \subseteq \Spc(\cat T^c)$. The localizing ideal $\cat T_Y\coloneqq \Loco{\cat T_Y^c}=\Loc\langle\cat T_Y^c\rangle$ is a smashing ideal; see \cite[Remark 1.23]{bhs1}. We write
		\[
			e_Y \to \unit \to f_Y \to \Sigma e_Y
		\]
	for the corresponding idempotent triangle.
\end{Exa}

\begin{Exa}\label{exa:DRsmashing}
	If $\cat T=\Der(R)$ is the derived category of a commutative ring $R$ and $I \subseteq R$ is a finitely generated ideal, we can take $Y\coloneqq V(I)$, the set of prime ideals of $R$ containing $I$. In this case, $e_Y \otimes \cat T$ is the category of $I$-torsion complexes and $\ihom{e_Y,\cat T}$ is the category of $I$-adically complete complexes; see \cite[Section 5]{Greenlees01} and \cite[Example 2.24]{Stevenson18}. The equivalence $e_Y\otimes \cat T \cong \ihom{e_Y,\cat T}$ first arose in the work of Matlis \cite{Matlis78} and Greenlees--May \cite{GreenleesMay92}; cf.~\cite{PortaShaulYekutieli14}. The functors $e_Y\otimes-$ and $\ihom{e_Y,-}$ can be interpreted in terms of local cohomology and local homology, respectively; see, e.g., \cite{AlonsoJeremiasLipman97,DwyerGreenlees02, BarthelCastellanaHeardValenzuela18}.
\end{Exa}

\begin{Lem}\label{lem:key-observation}
	For each Thomason subset $Y \subseteq \Spc(\cat T^c)$,
		\[
			\Coloc \langle \cat T_Y^c \otimes \cat T\rangle = \ihom{e_Y,\cat T} = (\cat T_Y)^{\perp \perp}.
		\]
\end{Lem}

\begin{proof}
	This is proved in \cite[Theorem 3.3.5(e)]{HoveyPalmieriStrickland97}.
\end{proof}

\begin{Def}
	A subset $W \subseteq \Spc(\cat T^c)$ is said to be \emph{weakly visible} if it can be written as the intersection of a Thomason subset and the complement of a Thomason subset: $W=Y_1 \cap Y_2^c$. We can then define an idempotent 
		\[	
			\gW \coloneqq e_{Y_1}\otimes f_{Y_2}.
		\]
	This object of $\cat T$ only depends, up to isomorphism, on the subset $W$; see \cite[Corollary~7.5]{BalmerFavi11}. We say that a point $\cat P\in\Spc(\cat T^c)$ is weakly visible if the singleton subset $\{\cat P\}$ is weakly visible, and we define
		\[
			\gP \coloneqq g_{\{\cat P\}}.
		\]
	That is, $\gP = e_{Y_1}\otimes f_{Y_2}$ for any choice of Thomason subsets $Y_1,Y_2\subseteq \Spc(\cat T^c)$ such that $\{\cat P\}=Y_1\cap Y_2^c$.
\end{Def}

\begin{Rem}\label{rem:good-Thomasons}
	For a weakly visible point $\cat P$, we can always take $Y_2 = \gen(\cat P)^c$, where $ \gen(\cat P)=\SET{\cat Q}{\cat P \subseteq \cat Q}$ and $Y_1=\supp(a)$ for some $a\in \cat T^c$; see \cite[Remark 2.8]{bhs1}.
\end{Rem}

\begin{Rem}\label{rem:W-intersection}
	The intersection $W_1 \cap W_2$ of two weakly visible subsets is again weakly visible and $g_{W_1} \otimes g_{W_2} = g_{W_1 \cap W_2}$. Moreover, $\gW=0$ if and only if $W =\emptyset$.  These facts are proved in \cite[Lemma 1.27]{bhs1}.
\end{Rem}

\begin{Exa}\label{ex:W-intersection}
	If $Y$ is a Thomason subset then both $Y$ and $Y^c$ are weakly visible subsets. We have $g_Y = e_Y$ and $g_{Y^c}=f_Y$. \Cref{rem:W-intersection} thus specializes to give $e_{Y_1} \otimes e_{Y_2} = e_{Y_1 \cap Y_2}$ and $f_{Y_1} \otimes f_{Y_2} = f_{Y_1 \cup Y_2}$ for Thomason subsets $Y_1$ and $Y_2$; see also \cite[Theorem 5.18]{BalmerFavi11} and \cite[Lemma 1.27]{bhs1}.
\end{Exa}

\begin{Def}
	A spectral space $X$ is \emph{weakly noetherian} if every point is weakly visible. This is the topological condition we will require of $\Spc(\cat T^c)$ in order to construct our cosupport theory.
\end{Def}

\begin{Rem}
	A spectral space $X$ is weakly noetherian if and only if its Hochster dual $X^*$ has the property that every point is locally closed. The latter is a separation axiom between $T_0$ and $T_1$ called $T_D$ and shows up in work on the analogue of these constructions for the smashing spectrum; see \cite{BalchinStevenson21pp,Verasdanis22pp} and \cite[Section~4.5]{DickmannSchwartzTressl19}.
\end{Rem}

\begin{Rem}\label{rem:spec-order}
	We define the specialization order on $X$ by $x \le y$ if and only if $x$ is a specialization of $y$, that is, $x \in \overbar{\{y\}}$. Be warned that this is opposite to how the specialization order is defined in \cite{DickmannSchwartzTressl19}. According to our convention, closed points are minimal for the specialization order.
\end{Rem}

\begin{Prop}\label{prop:weakly-noetherian-implies-DCC}
	Weakly noetherian spectral spaces satisfy the descending chain condition (DCC) on irreducible closed sets.
\end{Prop}

\begin{proof}
	It suffices to prove the result for the Balmer spectrum $X=\Spc(\cat K)$ of an essentially small tensor-triangulated category $\cat K$ since every spectral space arises in this way.\footnote{Every spectral space arises as the Zariski spectrum of a commutative ring by Hochster's theorem \cite{Hochster69} and it follows from Thomason's theorem \cite{Thomason97} that the Zariski spectrum of a commutative ring coincides with the Balmer spectrum of its derived category of perfect complexes.}
	We claim that if $X$ is weakly noetherian then $\cat K$ satisfies the descending chain condition on prime ideals. Indeed, if $\cat P_1 \supseteq \cat P_2 \supseteq \cat P_3 \supseteq \cdots$ is a descending chain of prime ideals, then consider the prime ideal $\cat P\coloneqq \bigcap_{n=1}^\infty \cat P_n$. If $\cat P$ is weakly visible then $\{\cat P\}=\supp(a)\cap \gen(\cat P)$ for some $a \in \cat K$ (\cref{rem:good-Thomasons}). Then $a \not\in \cat P$ implies $a \not\in \cat P_n$ for some $n$, so that $\cat P_n \in \supp(a)$. Since $\cat P_n \in \gen(\cat P)$ we conclude that $\cat P=\cat P_n$.
\end{proof}

\begin{Rem}\label{rem:DCC-point-set}
	There is also a purely point-set topological proof of the previous proposition. The result is equivalent to the statement that any weakly noetherian spectral space $X$ has the descending chain condition for the specialization order on $X$ (\cref{rem:spec-order}). To establish this, let $Y = (y_1 \ge y_2 \ge \cdots)$ be a descending specialization chain in $X$. Let $\Ybarcon$ denote the closure of $Y$ in the constructible topology. By \cite[Theorem~4.2.6]{DickmannSchwartzTressl19}, $Y$ has an infimum $y_\infty \in X$ which is contained in $\Ybarcon$. Since $X$ is weakly noetherian, $\{y_\infty\} = Z \cap \gen(y_\infty)$ for some Thomason closed subset $Z\subseteq X$. To establish that $y_\infty =y_n$ for some $n\ge 1$, it suffices to prove that $Y \cap  Z \neq \emptyset$ since $Y\subseteq \gen(y_\infty)$. If $Y \cap Z = \emptyset$ then $Y \subseteq Z^c$, but $Z^c$ is closed in the constructible topology. Thus it would follow that $\Ybarcon \subseteq Z^c$ which is a contradiction since $y_\infty \in \Ybarcon \cap Z$.
\end{Rem}

\begin{Exa}
	Let $X$ denote the Hochster dual of the Zariski spectrum $\Spec(\bbZ)$. Specialization chains in $X$ are of length at most one, hence $X$ satisfies the DCC on irreducible closed sets. However, the point of $X$ corresponding to the generic point of $\Spec(\bbZ)$ is contained in every Thomason subset of $X$. Thus, this point is not weakly visible. This example shows that the converse to \cref{prop:weakly-noetherian-implies-DCC} is false.
\end{Exa}

\begin{Not}\label{not:min}
	Let $X$ be a spectral space. For any subset $V \subseteq X$, we write $\min V$ for the collection of points in $V$ which are minimal for the specialization order ($x\leq y$ iff $x\in \overbar{\{y\}}$) among the points of $V$: 
		\[
			\min V \coloneqq \SET{ x \in V }{\overbar{\{x\}} \cap V = \{x\}}.
		\]
	For example, $\min X$ is the set of closed points of $X$, and $V \cap \min X \subseteq \min V$.
\end{Not}

\begin{Rem}
	\Cref{prop:weakly-noetherian-implies-DCC} establishes that if $X$ is weakly noetherian then 
		\[
			V \neq \emptyset \Longrightarrow \min V \neq \emptyset.
		\]
\end{Rem}

\begin{Def}
	We say that a point $\cat P$ is \emph{visible} if $\overbar{\{\cat P\}}$ is Thomason; in other words, it is a weakly visible point for which we can take $Y_1=\overbar{\singP}$. Recall that a spectral space is noetherian if and only if every point is visible; see \cite[Corollary 7.14]{BalmerFavi11}. Also note that the closed point of a local category is weakly visible if and only if it is visible; see \cite[Remark 2.9]{bhs1}.
\end{Def}

\begin{Not}\label{not:vis}
	We write $\Vis X$ for the set of visible points in $X$. We will also abuse notation slightly and write, for example, $\Vis \cat T$ for $\vis(\Spc(\cat T^c))$.
\end{Not}

\begin{Not}\label{not:the-functors}
	For a weakly visible point $\cat P \in \Spc(\cat T^c)$, we have the functor
		\[
			\GammaP\coloneqq- \otimes \gP: \cat T \rightarrow \cat T
		\]
	and its right adjoint
		\[
			\LambdaP\coloneqq \ihom{\gP,-}:\cat T\to \cat T.
		\]
	More generally, we have functors
		\[
			\GammaW\coloneqq - \otimes \gW \dashv \ihom{{\gW},-} \eqqcolon \LambdaW
		\]
	for any weakly visible subset $W \subseteq \Spc(\cat T^c)$.
\end{Not}

\begin{Rem}\label{rem:formulalambdaHom}
	Given $t_1,t_2\in \cat T$, adjunction provides the following useful formulas:
		\begin{equation}\label{eq:formulalambdaHom}
			\begin{split}
				\LambdaP\ihom{t_1,t_2} \simeq \ihom{\GammaP t_1,t_2}
				\simeq \ihom{t_1,\LambdaP t_2}.
			\end{split}
		\end{equation}
\end{Rem}

\begin{Def}[The Balmer--Favi support and cosupport]\label{def:BF-support-and-cosupport}
	Let $\cat T$ be a rigidly-compactly generated tt-category whose spectrum $\Spc(\cat T^c)$ is weakly noetherian. The \emph{support} of an object $t \in \cat T$ is defined as
		\begin{align*}
			\Supp(t) \coloneqq& \;\SET{\cat P \in \Spc(\cat T^c)}{\GammaP t \ne 0}\\
			=& \;\SET{\cat P \in \Spc(\cat T^c)}{\gP\otimes t \neq 0}.
	\intertext{The \emph{cosupport} of an object $t \in
			\cat T$ is defined as}
			\Cosupp(t) \coloneqq& \;\SET{\cat P \in \Spc(\cat T^c)}{\LambdaP t \ne 0}\\
				=& \;\SET{\cat P \in \Spc(\cat T^c)}{\ihom{\gP,t} \neq 0}.
		\end{align*}
\end{Def}

\begin{Rem}
	As discussed in \cite[Remark 2.12]{bhs1}, the function $\Supp$ defines a support theory on $\cat T$ in the sense of \cite[Definition~7.1]{bhs1}. Our present goal is to study the properties of the function $\Cosupp$ and the relationship between the two theories.
\end{Rem}

\bgroup\crefname{Def}{Def.}{Defs.}
\begin{Prop}\label{prop:tt_cosupport}
	The pair $(\Spc(\cat T^c), \Cosupp)$ defines a cosupport theory.
\end{Prop}
\egroup

\begin{proof}
	For any subset $Y \subseteq \Spc(\cat T^c)$, observe that
		\[
			\SET{t \in \cat T}{\Cosupp(t) \subseteq Y} = \bigcap_{\cat P \in Y^c} \ker(\LambdaP(-))
		\]
	is a colocalizing subcategory since each $\LambdaP= \ihom{\gP,-}$ is a product preserving exact functor. Moreover, it is a coideal by the formula \eqref{eq:formulalambdaHom}. This establishes, by \cref{rem:axioms-equiv}, that $\Cosupp$ satisfies all properties of a cosupport theory except $\Cosupp(\cat T)=\Spc(\cat T^c)$. For this just note that since the object~$\gP$ is nonzero (\cref{rem:W-intersection}), the internal hom $\ihom{\gP,\gP}$ is nonzero. Thus, $\cat P \in \Cosupp(\gP)$ and hence $\Spc(\cat T^c) = \bigcup_{\cat P \in \Spc(\cat T^c)} \Cosupp(\gP)$.
\end{proof}

\begin{Lem}\label{lem:cosupp-of-gP}
	Let $\cat P \in \Spc(\cat T^c)$. Then	
		\[
			\singP \subseteq \Cosupp(\gP) \subseteq \gen(\cat P).
		\]
\end{Lem}

\begin{proof}
	The first inclusion $\singP \subseteq \Cosupp(\gP)$ is established in the proof of \cref{prop:tt_cosupport}. For the second inclusion we need to show that if $\cat P \not\in \closureQ$ then $\ihom{\gQ,\gP}=0$. Write $\singQ = U_1 \cap U_2^c$ and $\singP = V_1 \cap V_2^c$ for Thomason subsets $U_1,U_2,V_1,V_2$. Then we have equalities 
		\begin{equation*}
			\begin{split}
				\ihom{\gQ,\gP} &= \ihom{e_{U_1} \otimes f_{U_2},e_{V_1}\otimes f_{V_2}} \\
				&= \ihom{e_{U_1}\otimes f_{U_2} \otimes f_{V_2},e_{V_1}\otimes f_{V_2}} \\
				&= \ihom{e_{U_1}\otimes f_{U_2} \otimes f_{V_2},e_{U_1}\otimes e_{V_1}\otimes f_{V_2}} \\
				&= \ihom{e_{U_1}\otimes f_{U_2\cup V_2},e_{U_1 \cap V_1}\otimes f_{V_2}},
			\end{split}
		\end{equation*}
	where the last step uses \cref{ex:W-intersection}. It follows that $\ihom{\gQ,\gP}$ vanishes if either of the following two conditions hold:
		\[
			U_1 \subseteq U_2 \cup V_2 \quad\text{ or }\quad U_1 \cap V_1 \subseteq V_2.
		\]
	By \cref{rem:good-Thomasons} we can always take $U_2 \coloneqq \gen(\cat Q)^c$ and $V_2\coloneqq \gen(\cat P)^c$. If $\cat P \not\in \closureQ$ then $\cat Q \in V_2$ hence $U_1 \cap U_2^c \subseteq V_2$ which means $U_1 \subseteq U_2 \cup V_2$ and the proof is complete.
\end{proof}

\begin{Rem} 
	In general, $\Cosupp(\gP) \neq \singP$; see \cref{exa:two-connected-points}.
\end{Rem}

\begin{Rem}
	The geometric behavior of cosupport is slightly counterintuitive. For example, it is often the case that $\Cosupp(\unit)\subsetneq \Spc(\cat T^c)$ is a proper subset; see \cref{exa:KInj}. We will study the geometry of cosupport, as well as its subtle geometric relationship with support, in \cref{sec:geometry-of-cosupport}. At present we focus on general properties. 
\end{Rem}

\begin{Lem}\label{lem:cosupp_hom}
	For $t_1,t_2 \in \cat T$, there is an inclusion
		\begin{equation}\label{eq:cosupp_hom}
			\Cosupp(\ihom{t_1,t_2}) \subseteq \Supp(t_1) \cap \Cosupp(t_2). 
		\end{equation}
\end{Lem}

\begin{proof}
	If $\cat P \in \Cosupp(\ihom{t_1,t_2})$ then $\LambdaP\ihom{t_1,t_2} \neq 0$ by definition, and hence $\ihom{\GammaP t_1,t_2} \ne 0$ and $\ihom{t_1,\LambdaP t_2} \ne 0$ by \cref{rem:formulalambdaHom}. This implies that $\GammaP t_1\neq 0$ and $\LambdaP t_2\neq 0$, and the result follows. 
\end{proof}

\begin{Rem}
	The inclusion \eqref{eq:cosupp_hom} is our most fundamental relationship between support and cosupport. It is an equality (for all objects $t_1$ and~$t_2$) if and only if~$\cat T$ is stratified; see \cref{thm:strat_cosupport}. This demonstrates the significance of cosupport even if one is only interested in localizing ideals. In general, \eqref{eq:cosupp_hom} is an equality when $t_1=e_Y$ or $t_1=f_Y$ for $Y \subseteq \Spc(\cat T^c)$ a Thomason subset:
\end{Rem}

\begin{Lem}\label{lem:cosuppcoloc}
	Let $Y \subseteq \Spc(\cat T^c)$ be a Thomason subset and let $t \in \cat T$. Then:
		\begin{enumerate}
			\item $\Cosupp(\ihom{e_Y,t}) = Y \cap \Cosupp(t)$;
			\item $\Cosupp(\ihom{f_Y,t}) = Y^c \cap \Cosupp(t)$;
			\item $\Cosupp(\LambdaP t)=\{ \cat P\} \cap \Cosupp(t)$ for any $\cat P \in \Spc(\cat T^c)$.
		\end{enumerate}
\end{Lem}

\begin{proof}
	First note that for any Thomason subset $Y \subseteq \Spc(\cat T^c)$ and $\cat P \in Y$ we have $\gP\otimes e_Y=\gP$ by \cref{rem:W-intersection}.

	Now, by \Cref{lem:cosupp_hom} and \cite[Lemma 2.13]{bhs1} we have
		\[
			\Cosupp(\ihom{e_Y,t}) \subseteq Y \cap \Cosupp(t).
		\] 
	To establish the reverse inclusion, let $\cat P \in Y \cap \Cosupp(t)$. By the previous paragraph, we have 
		\[
			\ihom{\gP,\ihom{e_Y,t}} \simeq \ihom{\gP,t} \ne 0
		\]
	so that $\cat P \in \Cosupp(\ihom{e_Y,t})$ as well. This establishes the equality in part~$(a)$. The equality in part $(b)$ can be proved similarly. The final statement is a consequence of $(a)$ and $(b)$ by writing $\{\cat P\}=Y_1\cap Y_2^c$:
		\begin{align*}
			\Cosupp(\LambdaP t) &= \Cosupp(\ihom{e_{Y_1}\otimes f_{Y_{2}},t})\\
			&=\Cosupp(\ihom{e_{Y_1},\ihom{f_{Y_{2}},t}}) \\
			&= Y_1 \cap \Cosupp(\ihom{f_{Y_2},t}) \\
			&= \{ \cat P\} \cap \Cosupp(t). \qedhere
		\end{align*}
\end{proof}

\begin{Rem}
	We next establish a cosupport-theoretic analogue of the ``half \mbox{$\otimes$-theorem}'' of \cite[Theorem~7.22]{BalmerFavi11}; see also \cite[Lemma~2.18]{bhs1}. This requires the following lemma:
\end{Rem}

\begin{Lem}\label{lem:eqlococosupp}
	Let $s_1,s_2,t \in \cat T$.
	\begin{enumerate}
		\item If $s_1 \in \Loco{s_2}$, then $\Cosupp\ihom{s_1,t} \subseteq \Cosupp\ihom{s_2,t}$.
		\item If $s_1 \in \Coloco{s_2}$, then $\Cosupp\ihom{t,s_1} \subseteq \Cosupp\ihom{t,s_2}$.
	\end{enumerate}
\end{Lem}

\begin{proof}
	Bearing in mind \cref{rem:axioms-equiv}, part~$(a)$ follows from \eqref{eq:[loc,t]} while part~$(b)$ follows from \eqref{eq:[t,coloc]}.
\end{proof}

\begin{Prop}[Half-$\mathsf{hom}$ Theorem]\label{prop:halfhom}
	If $t \in \cat T$ and $x \in \cat T^c$, then
		\[
			\Cosupp\ihom{x,t} = \supp(x) \cap \Cosupp(t). 
		\]
\end{Prop}

\begin{proof}
	The support $\supp(x) \subseteq \Spc(\cat T^c)$ of the compact object $x \in \cat T^c$ is a Thomason subset and we have an equality $\Loco{x} =  \Loco{e_{\supp(x)}}$ of compactly generated localizing ideals. By \cref{lem:eqlococosupp}, we obtain
		\[
			\Cosupp\ihom{x,t} = \Cosupp\ihom{e_{\supp(x)},t} = \supp(x) \cap \Cosupp(t),
		\]
	where the last equality uses \cref{lem:cosuppcoloc}.
\end{proof}

\begin{Exa}\label{exa:cosupp-compact}
	For any compact $x \in \cat T^c$, we have
		\[
			\Cosupp(x) = \supp(x) \cap \Cosupp(\unit).
		\]
	Indeed, since $x$ and its dual $x^\vee=\ihom{x,\unit}$ generate the same thick ideal of compact objects, they have the same cosupport. Thus $\Cosupp(x)=\Cosupp(\ihom{x,\unit})=\supp(x) \cap \Cosupp(\unit)$ by \cref{prop:halfhom}. It follows that, although
		\[
			\Cosupp(\cat T) = \Spc(\cat T^c)
		\]
	(as established in  \cref{prop:tt_cosupport}), we have
		\[
			\Cosupp(\cat T^c) = \Cosupp(\unit).
		\]
	It is possible for the latter to be a proper subset of $\Spc(\cat T^c)$; see  \cref{exa:KInj}.
\end{Exa}

\section{The stalk and costalk}\label{sec:stalk-and-costalk}

\begin{Def}
	We let $\GammaP \cat T$ and $\LambdaP \cat T$ denote the essential images of the functors $\GammaP\colon\cat T\to \cat T$ and $\LambdaP\colon\cat T \to \cat T$, respectively (\cref{not:the-functors}). We call $\GammaP \cat T$ the \emph{stalk} of~$\cat T$ at $\cat P$ and call $\LambdaP \cat T$ the \emph{costalk} of $\cat T$ at $\cat P$.
\end{Def}

\begin{Prop}\label{prop:local-cats-are-ok}
	Let $\cat P \in \Spc(\cat T^c)$.
	\begin{enumerate}
		\item The full subcategory $\GammaP \cat T$ is a localizing ideal and 
				\[
					\GammaP \cat T = \SET{t\in \cat T}{t \simeq \GammaP t}.
				\]
		\item The full subcategory $\LambdaP \cat T$ is a colocalizing coideal and
				\[
					\LambdaP \cat T = \SET{t\in\cat T}{t \simeq \LambdaP t}.
				\]
	\end{enumerate}
\end{Prop}

\begin{proof}
	We prove part $(b)$; the proof of $(a)$ is similar. Since $\cat P$ is weakly visible we can write $\singP = Y_1 \cap Y_2^c$ with $Y_1$ and $Y_2$ Thomason subsets.  Consider the exact triangle
		\[
			e_{Y_1 \cap Y_2} \to e_{Y_1} \to \gP \to \Sigma e_{Y_1 \cap Y_2}
		\]
	obtained from $e_{Y_2} \to \unit \to f_{Y_2} \to \Sigma e_{Y_2}$ by tensoring with $e_{Y_1}$. We first establish that $t \in \LambdaP \cat T$ if and only if the canonical maps
		\begin{equation}\label{eq:canonical-maps}
			\ihom{\gP,t} \to \ihom{e_{Y_1},t} \leftarrow t
		\end{equation}
	induced from $e_{Y_1} \to \gP$ and $e_{Y_1} \to \unit$ are isomorphisms. Indeed, if $t\in \LambdaP \cat T$ then $t \simeq \ihom{\gP,t'}$ for some $t' \in \cat T$. Since $\cat P \not\in Y_1 \cap Y_2$ we have $e_{Y_1 \cap Y_2} \otimes \gP =0$ and hence ${\ihom{e_{Y_1 \cap Y_2},t}=\ihom{e_{Y_1 \cap Y_2}\otimes g_\cat P,t'}=0}$. Also, $\ihom{f_{Y_1},t}=\ihom{f_{Y_1}\otimes g_\cat P,t'}=0$ since $\cat P \in Y_1$ implies $f_{Y_1} \otimes \gP = 0$. (Here we have repeatedly invoked \cref{rem:W-intersection}.)
	
	We have thus established that $t \in \LambdaP \cat T$ if and only if the canonical maps in~\eqref{eq:canonical-maps} are isomorphisms if and only if there exists an isomorphism $t \simeq \ihom{\gP,t}=\LambdaP t$. Then, it follows from the latter characterization that $\LambdaP \cat T$ is closed under exact triangles and hence is a triangulated subcategory of $\cat T$. Now, $\LambdaP \cat T$ is closed under products since $\LambdaP = \ihom{\gP,-}$ is a right adjoint. Moreover, it is a coideal: if $t_1 \in \cat T$ and $t_2 \in \LambdaP \cat T$ then
		\[
			\ihom{t_1,t_2} \simeq \ihom{t_1,\ihom{\gP,t_2}} \simeq \ihom{\gP,\ihom{t_1,t_2}}=\LambdaP \ihom{t_1,t_2}
		\]
	by adjunction.
\end{proof}

\begin{Rem}\label{rem:supp-of-local-small}
	It follows from \cref{lem:cosuppcoloc} and \cite[Lemma~2.13]{bhs1} that 
		\[
			\Supp(\GammaP \cat T) \subseteq \singP \quad\text{ and }\quad \Cosupp(\LambdaP \cat T) \subseteq \singP.
		\]
	This leads to the following:
\end{Rem}

\begin{Def}\label{def:codetection}
	We say that  $\cat T$ satisfies
	\begin{enumerate} 
		\item the \emph{detection property} if $\Supp(t) = \emptyset$ implies $t=0$ for all $t\in \cat T$;
		\item the \emph{codetection property} if $\Cosupp(t)=\emptyset$ implies $t=0$ for all $t \in \cat T$.
	\end{enumerate}
\end{Def}

\begin{Lem}\label{lem:local-cat-when-detect}
	Let $\cat P \in \Spc(\cat T^c)$.
	\begin{enumerate}
		\item If $\cat T$ has the detection property then
			\[ \GammaP\cat T = \SET{t\in \cat T}{\Supp(t) \subseteq \singP}.\]
		\item If $\cat T$ has the codetection property then
			\[ \LambdaP\cat T=\SET{t \in \cat T}{\Cosupp(t) \subseteq \singP}.\]
	\end{enumerate}
\end{Lem}

\begin{proof}
	The inclusion $\subseteq$ in both $(a)$ and $(b)$ follows from \cref{rem:supp-of-local-small}. We now establish the reverse inclusion of part $(b)$; the proof of the reverse inclusion of part~$(a)$ is similar. We can write $\singP = Y_1 \cap \gen(\cat P)$ with $Y_1 \subseteq \Spc(\cat T^c)$ Thomason, so that $\gP = e_{Y_1} \otimes f_{\gen(\cat P)^c}$. If $\Cosupp(t) \subseteq \singP$ then $\ihom{e_{\gen(\cat P)^c},t}$ and $\ihom{f_{Y_1},t}$ have empty cosupport by \cref{lem:cosuppcoloc}. It follows that $\ihom{e_{\gen(\cat P)^c},t}=0$ and $\ihom{f_{Y_1},t}=0$ since $\cat T$ has the codetection property. Hence $t \simeq \ihom{f_{\gen(\cat P)^c},t}$ and $t \simeq \ihom{e_{\gen(\cat P)^c},t}$. Together, this implies $t \simeq \ihom{\gP,t} \in \LambdaP\cat T$ and we are done.
\end{proof}

\begin{Prop}\label{prop:codetection-implies-detection}
	If $\cat T$ satisfies the codetection property then it also satisfies the detection property.
\end{Prop}

\begin{proof}
	Let $t \in \cat T$ and suppose $\Supp(t) = \emptyset$. Then
		\[
			\Cosupp(\ihom{t,t}) \subseteq \Supp(t) \cap \Cosupp(t)
		\]
	by \cref{lem:cosupp_hom} which implies that $\Cosupp(\ihom{t,t})=\emptyset$. The codetection property then implies $\ihom{t,t}=0$. Hence $\cat T(t,t)=\cat T(\unit,\ihom{t,t})=0$, so that $t=0$.
\end{proof}

\begin{Exa}\label{exa:detect-but-not-codetect}
	In contrast, the following example shows that the detection property does not imply the codetection property. Let $R$ be an absolutely flat ring and let $\cat T=\Der(R)$. Since $\Spc(\cat T^c)\cong \Spec(R)$ has dimension zero, $\{\mathfrak p\} = \Spec(R) \cap \gen(\mathfrak p)$. Hence $g_{\mathfrak p} \cong f_{\gen(\mathfrak p)^c} \cong R_{\mathfrak p} \cong k(\mathfrak p)$ is the residue field of the prime ideal $\mathfrak p$; see~\cite[Lemma 4.2 and Equation (4.1)]{Stevenson14} and \cite[Example 1.36]{bhs1}. Hence, we have
		\[
			\Cosupp(t) = \SET{\mathfrak p \in \Spec(R)}{\ihom{k(\mathfrak p),t} \ne 0}
		\]
	for any $t \in \Der(R)$. The detection property holds for any absolutely flat ring~$R$ by \cite[Lemma 4.1]{Stevenson14}. On the other hand, the codetection property does not hold if $R$ is not semi-artinian. This follows from the argument in the proof of \cite[Theorem 4.8]{Stevenson14}: If $R$ is not semi-artinian then there exists a superdecomposable (pure-)injective $R$-module $E\neq 0$ by \cite[Lemma~1.2]{Trlifaj96}. For each $\mathfrak p\in \Spec(R)$, we must have $\Hom_R(k(\mathfrak p),E)=0$ for otherwise~$E$ would contain the (injective) simple module $k(\mathfrak p)$ as a direct summand, contradicting the superdecomposability of $E$. It follows that  $\Cosupp(E) = \emptyset$ even though $E \neq 0$. An explicit example of an absolutely flat ring which is not semi-artinian is an infinite product of fields, such as $R=\prod_{\bbN} \Fp$. \Cref{thm:LGP-equiv} below will provide further insight into this example and the relationship between the detection and codetection properties.
\end{Exa}

\begin{Lem}\label{lem:pointsuppcosupp}
	Let $\cat P \in \Spc(\cat T^c)$.
	\begin{enumerate}
		\item If $t \in \GammaP\cat T$, then $\Supp(t) \subseteq \Cosupp(t)$.
		\item If $t \in \LambdaP \cat T$, then $\Cosupp(t) \subseteq \Supp(t)$.
	\end{enumerate}
\end{Lem}

\begin{proof}
	If $t \in \GammaP \cat T$ then $t \simeq \GammaP t$ by \cref{prop:local-cats-are-ok} and hence $\Supp(\GammaP t) \subseteq \singP$ by \cref{rem:supp-of-local-small}. The claim is vacuously true if $t=0$; otherwise
		\[
			0 \neq \cat T(t,t)\simeq \cat T(\GammaP t,t)\simeq \cat T(t,\LambdaP t)
		\]
	shows that $\LambdaP t \neq 0$, i.e., $\cat P \in \Cosupp(t)$. This establishes $(a)$. The proof of~$(b)$ uses a similar argument. If $t \in \LambdaP \cat T$ then $t \simeq \LambdaP t$ by \cref{prop:local-cats-are-ok} and hence $\Cosupp(\LambdaP t) \subseteq \singP$ by \cref{rem:supp-of-local-small}. If $t=0$, the inclusion holds. If not, the same displayed equation applies to show that $\GammaP t \neq 0$.
\end{proof}

\begin{Cor}\label{cor:cosupp_surj}
	For any $\cat P \in \Spc(\cat T^c)$, we have $\Cosupp (\LambdaP \GammaP \unit)=\{\cat P\}$. Moreover,	for every subset $Y \subseteq \Spc(\cat T^c)$, there exists an object $t \in \cat T$ with $\Cosupp(t) = Y$.
\end{Cor}

\bgroup\crefname{Lem}{Lem.}{Lems.}
\begin{proof}
	The first statement is a consequence of \cref{lem:pointsuppcosupp}(a) and \cref{lem:cosuppcoloc}(c), together with the observation that $\Gamma_p \unit = \gP \ne 0$ (\Cref{rem:W-intersection}). Then, given $Y \subseteq \Spc(\cat T^c)$ the object $t\coloneqq \prod_{\cat P \in Y} \LambdaP\GammaP \unit$ has cosupport equal to $Y$.
\end{proof}
\egroup

\begin{Rem}\label{rem:stalk-costalk-equivalence}
	Let $Y_1, Y_2 \subseteq \Spc(\cat T^c)$ be Thomason subsets. Recall from \eqref{eq:smashing-diagram} that for any smashing localization $e\to\unit \to f\to \Sigma e$ we have an equivalence $e\otimes \cat T\cong\ihom{e,\cat T}$ and an equality $f\otimes \cat T=\ihom{f,\cat T}$. In particular, it follows that $\ihom{f,f\otimes t}\simeq f\otimes t$ and $f\otimes\ihom{f,t}\simeq \ihom{f,t}$ for any $t \in \cat T$. From this, and the fact that idempotent functors are fully faithful on their essential image, we deduce that we have a diagram of adjunctions
		\[\begin{tikzcd}[row sep=3em]
			\ihom{e_{Y_1},\cat T} 
			\ar[d,shift left=1.25ex,"\ihom{f_{Y_2},-}"]
			\ar[r,phantom,"\cong"description]
			\ar[r,shift left=1.25ex]
			\ar[r,<-,shift right=1.25ex] 
			& e_{Y_1} \otimes \cat T
			\ar[d,shift right=1.25ex,"-\otimes f_{Y_2}"']
			\\
			\ihom{f_{Y_2},\ihom{e_{Y_1},\cat T},}
			\ar[u,phantom,"\dashv"description] \ar[u,hook,shift left=1.25ex,"-\otimes f_{Y_2}"]
			& f_{Y_2}\otimes e_{Y_1}\otimes \cat T \ar[u,hook,shift right=1.25ex,"\ihom{f_{Y_2},-}"'] \ar[u,phantom,"\dashv"description]
		\end{tikzcd}\]
	where the hooked arrows are fully faithful. In particular, writing $W\coloneqq Y_1 \cap Y_2^c$, the $\GammaW : \cat T \adjto \cat T:\LambdaW$ adjunction restricts to an adjoint equivalence
		\[\begin{tikzcd}
			\LambdaW \cat T
			\ar[r,bend left=10,shift left=1.25ex]
			\ar[r,<-,bend right=10,shift right=1.25ex]
			\ar[r,phantom,"\cong"description]& \GammaW \cat T.
		\end{tikzcd} \]
	This type of equivalence has been previously observed in \cite[Prop.~5.1]{BensonIyengarKrause12}. For $Y_2=\emptyset$ it specializes to the Matlis--Greenlees--May style equivalence mentioned in \cref{exa:DRsmashing}; cf.~\cite[Theorem 2.1]{DwyerGreenlees02}. In particular, we obtain an equivalence
		\[ 
			\LambdaP\cat T \cong \GammaP \cat T
		\]
	between the costalk and stalk for any weakly visible point $\cat P$. This equivalence provides another perspective on \cref{lem:pointsuppcosupp}.
\end{Rem}

\begin{Exa}
	Consider a smashing ideal $\cat L$ with idempotent triangle
		\[
			e \to \unit \to f\to\Sigma e.
		\]
	Then $\cat L=e\otimes \cat T$ and it follows from \cref{rem:smashing-recollement} that its right orthogonal $\cat L^\perp = f\otimes \cat T = \ihom{f,\cat T}$ is both a localizing ideal and colocalizing coideal, and $\cat L^{\perp \perp} = \ihom{e,\cat T}$ is a colocalizing coideal. Using these descriptions we can readily check that
		\[
			\Supp(e) = \Supp(\cat L)=\Cosupp(\cat L^{\perp \perp})
		\]
	and
		\[
			\Supp(f) = \Supp(\cat L^\perp)=\Cosupp(\cat L^\perp).
		\]
	For example, $\Cosupp(\cat L^\perp) = \Cosupp(\ihom{f,\cat T}) \subseteq \Supp(f)$ by \cref{lem:cosupp_hom}. Conversely, if $\cat P \in \Supp(f)$ then $\cat P \in \Cosupp(\cat L^\perp)$ since $\gP\otimes f \neq 0$ implies 
		\[
			\ihom{\gP,\ihom{f,\gP\otimes f}} = \ihom{\gP\otimes f,\gP\otimes f}\neq 0.
		\]
	A similar argument gives $\Cosupp(\cat L^{\perp\perp}) = \Supp(e)$.
\end{Exa}

\section{The local-to-global principle}\label{sec:LGP}

\begin{Def}[Local-to-global principle]\label{def:LGP}
	Let $\cat T$ be a rigidly-compactly generated \mbox{tt-category} with $\Spc(\cat T^c)$ weakly noetherian. Recall from \cite[Definition~3.8]{bhs1} that $\cat T$ satisfies the \emph{local-to-global principle for localizing ideals} (or simply the \emph{local-to-global principle}) if for every object $t \in \cat T$, we have
		\begin{equation}\label{eq:LGP}
			\Loco{t} = \Loco{\GammaP t \mid \cat P \in \Spc(\cat T^c)}.
		\end{equation}
	Note that this implies the detection property (\cref{def:codetection}) since the right-hand side of~\eqref{eq:LGP} is the same as $\Loco{\GammaP t\mid \cat P \in \Supp(t)}$. Similarly, we say that $\cat T$ satisfies the \emph{local-to-global principle for colocalizing coideals} (or simply the \emph{colocal-to-global principle}) if for every object $t \in \cat T$, we have
		\begin{equation}\label{eq:co-LGP}
			\Coloco{t} = \Coloco{\LambdaP t \mid \cat P \in \Spc(\cat T^c)}.
		\end{equation}
	Note that this implies the codetection property since the right-hand side of~\eqref{eq:co-LGP} is the same as $\Coloco{\LambdaP t \mid \cat P \in \Cosupp(t)}$.
\end{Def}

\begin{Thm}\label{thm:LGP-equiv}
	Let $\cat T$ be a rigidly-compactly generated tt-category with weakly noetherian spectrum. The following are equivalent:
	\begin{enumerate}
		\item $\cat T$ satisfies the codetection property;
		\item $\cat T$ satisfies the local-to-global principle for localizing ideals;
		\item $\cat T$ satisfies the local-to-global principle for colocalizing coideals.
	\end{enumerate}
\end{Thm}

\begin{proof}
	Let $\cat L \coloneqq \Loco{\gP \mid \cat P \in \Spc(\cat T^c)}$. An object $t \in \cat T$ satisfies $\Cosupp(t) = \emptyset$ if and only if $t \in \cat L^\perp$. Thus, the codetection property is equivalent to $\cat L^\perp = 0$. On the other hand, the local-to-global principle for localizing ideals is equivalent to~$\cat L=\cat T$ since if $\unit \in \cat L$ then for any $t \in \cat T$ we have 
		\[
			t=t\otimes \unit \in t\otimes \cat L \subseteq \Loco{t\otimes \gP \mid \cat P \in \Spc(\cat T^c)}
		\]
	by \eqref{eq:t@loc}. Clearly $\cat L=\cat T$ implies $\cat L^\perp=0$. On the other hand, $\cat L$ is strictly localizing by \cref{thm:set-generated-are-strictly-localizing}, hence we have $\cat L={}^\perp(\cat L^\perp)$ (\cref{rem:strictly-localizing}) so that $\cat L^\perp=0$ implies $\cat L=\cat T$. Thus, $(a)$ and $(b)$ are equivalent. Condition $(c)$ evidently implies $(a)$ so it remains to show that $(b)$ implies $(c)$. This is also immediate since if $\unit \in \cat L$ then for any $t \in \cat T$ we have \[ t=\ihom{\unit,t} \in \ihom{\cat L,t} \subseteq \Coloco{\ihom{\gP,t} \mid \cat P \in \Spc(\cat T^c)}\] by~\eqref{eq:[loc,t]}.
\end{proof}

\begin{Cor}\label{cor:noetherian_local_global}
	If $\Spc(\cat T^c)$ is noetherian, then $\cat T$ satisfies the local-to-global principle for colocalizing coideals.
\end{Cor}

\begin{proof}
	The local-to-global principle for localizing ideals holds by \cite[Theorem~3.21]{bhs1} and hence the result follows from \cref{thm:LGP-equiv}.
\end{proof}

\begin{Rem}\label{rem:tantalizing}
	Although the proof of \cref{thm:LGP-equiv} is not difficult in hindsight, we regard the statement as quite surprising. The local-to-global and colocal-to-global principles are equivalent, and in fact are equivalent to the codetection property. In particular, this shows that the theory of cosupport is highly relevant even for the task of classifying localizing ideals via a theory of support. The theorem also provides another way to see that the codetection property does not always hold. Indeed, \cite[Theorem 4.8]{Stevenson14} establishes that if $R$ is an absolutely flat ring which is not semi-artinian then $\Der(R)$ does not satisfy the local-to-global principle; cf.~\cref{exa:detect-but-not-codetect}. This example also shows that the detection property does not imply the codetection property. It remains a tantalizing possibility that the detection property is \emph{always}\footnote{When the spectrum is not weakly noetherian, the definition of support requires modification, as the example of $p$-local spectra shows. For some remarks in this direction, see \cite[Remark~5.14]{bhs2}. A general approach to extending the Balmer--Favi support to points which are not weakly visible is considered by William Sanders \cite{BillySanders2017pp} and has been further developed in recent work of Changhan Zou \cite{Zou23bpp}.} satisfied. We do not know of any counterexamples.
\end{Rem}

\section{Costratification}\label{sec:costratification}

\begin{Def}[Costratification]\label{def:costratification}
	Let $\cat T$ be a rigidly-compactly generated tensor-triangulated category with $\Spc(\cat T^c)$ weakly noetherian. We say that~$\cat T$ is \emph{costratified} if the cosupport theory $(\Spc(\cat T^c), \Cosupp)$ (\cref{def:BF-support-and-cosupport}) induces a bijection
		\begin{equation}\label{eq:costrat-map} 
			\big\{ \text{colocalizing coideals of $\cat T$} \big\} \xra{\Cosupp} \big\{ \text{subsets of $\Spc(\cat T^c)$}\big\}.
		\end{equation}
	The inclusion-preserving function \eqref{eq:costrat-map} is always surjective by \cref{cor:cosupp_surj}, so costratification amounts to its injectivity.
\end{Def}

\begin{Rem}
	For a subset $Y \subseteq \Spc(\cat T)$, we use the notation
		\[
			\InvSupp{Y} \coloneqq \SET{t \in \cat T}{\Supp(t) \subseteq Y}
		\]
	for the localizing ideal of objects supported on $Y$, and 
		\[
			\InvCosupp{Y} \coloneqq \SET{t \in \cat T}{\Cosupp(t) \subseteq Y}
		\]
	for the colocalizing coideal of objects cosupported on $Y$. If $\cat T$ is costratified, then the inverse of \eqref{eq:costrat-map} is necessarily given by $Y \mapsto \InvCosupp{Y}$. Indeed, we always have the inclusion $\cat C \subseteq \InvCosupp{\Cosupp(\cat C)}$ and costratification is equivalent to these inclusions being equalities.
\end{Rem}

\begin{Rem}\label{rem:set-generated}
    If a colocalizing coideal is generated by a set of objects $\cat E$ then it is also generated by a single object: $\Coloco{\cat E} = \Coloco{\prod_{t \in \cat E}t}$. We refer to such colocalizing coideals as the \emph{set-generated} colocalizing coideals.
\end{Rem}

\begin{Prop}\label{prop:set-generated}
	If the class of set-generated colocalizing coideals of $\cat T$ forms a set, then every colocalizing coideal of $\cat T$ is generated by a set.
\end{Prop}

\begin{proof}
	The argument in \cite[Lemma 3.3.1]{KrauseStevenson19} goes through verbatim with ``localizing subcategory'' and ``$\Loc$'' replaced by ``colocalizing coideal'' and ``$\Colocideal$''; cf.~\cite[Proposition~3.5]{bhs1}.
\end{proof}

\begin{Rem}\label{rem:cominimality}
	It follows from \cref{cor:cosupp_surj} and \cref{lem:cosuppcoloc} that the colocalizing coideal $\LambdaP \cat T$ is nonzero and in fact has cosupport equal to the single point~$\singP$. Thus, a necessary condition for costratification is that $\LambdaP \cat T$ contains no nontrivial proper colocalizing coideals, i.e., that it is a \emph{minimal} colocalizing coideal.
\end{Rem}

\begin{Thm}\label{thm:equiv-costrat}
	Let $\cat T$ be a rigidly-compactly generated tensor-triangulated category with weakly noetherian spectrum. The following are equivalent:
	\begin{enumerate}
	\item	The local-to-global principle for colocalizing coideals holds for $\cat T$, and for all $\cat P\in \Spc(\cat T^c)$, $\LambdaP \cat T$ is a minimal colocalizing coideal of $\cat T$.
	\item	For all $t \in \cat T$, $\Coloco{t} = \Coloco{\LambdaP \cat T \mid \cat P \in \Cosupp(t)}$.
	\item	The function
				\[ 
					\big\{ \text{colocalizing coideals of $\cat T$} \big\} \xra{\Cosupp} \big\{ \text{subsets of $\Spc(\cat T^c)$}\big\}
				\]
			is injective (and hence a bijection by \Cref{cor:cosupp_surj}); that is, $\cat T$ is costratified.
	\end{enumerate}
\end{Thm}

\begin{proof}
	$(a)\Rightarrow(b)$: By the local-to-global principle for colocalizing coideals we have 
		\[ 
			\Coloco{t} = \Coloco{\LambdaP t \mid \cat P \in \Cosupp(t)}.
		\]
	Since $\LambdaP \cat T$ is a minimal colocalizing coideal, for any $t\in \cat T$ we have equalities $\Coloco{\LambdaP t} = \LambdaP \cat T$ for any $\cat P \in \Cosupp(t)$. Therefore, 
		\[ 
			\Coloco{t} = \Coloco{\LambdaP \cat T \mid \cat P \in \Cosupp(t)}.
		\]

	$(b) \Rightarrow (c)$: We will actually show that the map
        \begin{equation}\label{eq:supp-set-gen}
			\resizebox{0.9 \textwidth}{!}{
				$\big\{\text{set-generated colocalizing coideals of $\cat T$}\big\} \xra{\Cosupp} \big\{\text{subsets of $\Spc(\cat T^c)$}\big\}$
			}
        \end{equation}
	is injective. It will follow that there is only a set of set-generated colocalizing coideals. Hence by \cref{prop:set-generated}, every colocalizing coideal is set-generated, so we will have established $(c)$. Note that the map is always surjective by \cref{cor:cosupp_surj}. Recall that every set-generated colocalizing coideal of $\cat T$ is generated by a single object (\cref{rem:set-generated}) and that $\Cosupp(\Coloco{t}) = \Cosupp(t)$ (\cref{rem:axioms-equiv}). Thus the injectivity of \eqref{eq:supp-set-gen} is equivalent to 
		\[ 
			\forall t_1,t_2 \in \cat T, \Cosupp(t_1)=\Cosupp(t_2) \Longrightarrow \Coloco{t_1} = \Coloco{t_2}.
		\]
	If $\Cosupp(t_1) = \Cosupp(t_2)$ then $(b)$ implies that $\Coloco{t_1}=\Coloco{t_2}$, so we are done.
    
	$(c)\Rightarrow(a)$: 
	Suppose the function $\Cosupp$ is injective. Applied to the zero coideal, we obtain the codetection property and hence the local-to-global principle by \cref{thm:LGP-equiv}. Suppose then that $0 \ne \cat C \subseteq \LambdaP\cat T$ is a colocalizing coideal. Then
		\[
			\emptyset \ne \Cosupp(\cat C) \subseteq \Cosupp(\LambdaP \cat T) = \{ \cat P \}
		\]
	using \Cref{lem:cosuppcoloc}. Hence $\Cosupp(\cat C) = \Cosupp(\LambdaP\cat T)$ and so $\cat C = \LambdaP \cat T$. This establishes the minimality of $\LambdaP \cat T$. 
\end{proof}

\begin{Ter}
	We say that $\cat T$ satisfies (or has) ``cominimality at $\cat P$'' if~$\LambdaP \cat T$ is a minimal colocalizing coideal of $\cat T$ as in \cref{rem:cominimality} and part $(a)$ of \cref{thm:equiv-costrat}.
\end{Ter}

\begin{Rem}
	Recall from \cite{bhs1} that $\cat T$ is said to be \emph{stratified} if the Balmer--Favi support theory $(\Spc(\cat T^c),\Supp)$ provides a bijection
	\[ 
		\big\{ \text{localizing ideals of $\cat T$} \big\} \xra{\Supp} \big\{ \text{subsets of $\Spc(\cat T^c)$}\big\}.
	\]
	\Cref{thm:equiv-costrat} should be compared with \cite[Theorem 4.1]{bhs1} which provides a directly analogous characterization of stratification. Note that the statement of part~$(b)$ in \cref{thm:equiv-costrat} is slightly different than the formulation of part~$(b)$ in \cite[Theorem 4.1]{bhs1}. Morally this is because cosupport is not controlled by the unit $\unit$ in the same way that support is.
\end{Rem}

\begin{Rem}
	Our next goal is to prove that costratification implies stratification. To this end, we will utilize the following crucial result which is a minor modification of \cite[Lemma 3.9]{BensonIyengarKrause11a}.
\end{Rem}

\begin{Lem}[Benson--Iyengar--Krause]\label{lem:minimality_crit}
	Let $\cat T$ be a rigidly-compactly generated tt-category and let $\cat L$ be a nonzero localizing ideal of $\cat T$. The following statements are equivalent:
	\begin{enumerate}
		\item The localizing ideal $\cat L$ is minimal. 
		\item $\ihom{t_1,t_2}\neq 0$ for any two nonzero objects $t_1,t_2\in \cat L$.
	\end{enumerate}
\end{Lem}

\begin{proof}
	$(a)\Rightarrow(b)$: If $t_1 \neq 0$ then minimality implies $\cat L=\Loco{t_1}$. Hence if $\ihom{t_1,t_2}=0$ then $\ihom{\cat L,t_2}=0$. In particular $\ihom{t_2,t_2}=0$ so that~$t_2 =0$.

	$(b)\Rightarrow (a)$: Let $0 \neq t_1\in \cat L$. Then $\cat L_1\coloneqq \Loco{t_1} = \Loc\langle t_1\otimes \cat T^c\rangle$ is a set-generated localizing subcategory and hence is strictly localizing by \cref{thm:set-generated-are-strictly-localizing}. By \Cref{rem:strictly-localizing}, $\cat L_1\subseteq \cat L$ is the kernel of a Bousfield localization on $\cat T$. For any $t_2 \in \cat L$ consider the associated Bousfield localization triangle $\Gamma t_2 \to t_2 \to \L t_2\to\Sigma\Gamma t_2$. The first two objects are in $\cat L$, thus so is $\L t_2$. However, $(b)$ implies that no nonzero object of $\cat L$ can be contained in $(\cat L_1)^\perp$; that is, $(\cat L_1)^\perp \cap \cat L=\{0\}$. Hence $\L t_2=0$,  since $\Gamma t_2\in \cat L_1$, so that $\Gamma t_2 \simeq  t_2$ is contained in $\cat L_1$. This proves that $\cat L_1=\cat L$.
\end{proof}

\begin{Rem}
	We highlight that the proof of \cref{lem:minimality_crit} crucially depends on the existence of Bousfield localizations for arbitrary set-generated localizing ideals. In light of \cref{rem:vopenka}, one cannot establish minimality of colocalizing coideals by a similar argument.
\end{Rem}

\begin{Rem}\label{rem:relative_minimality_crit}
	\Cref{lem:minimality_crit} admits the following relative version: Let $\cat L$ be a localizing ideal of $\cat T$ and let $\cat J\subsetneq \cat L$ be a set-generated localizing subcategory. Then~$\cat L$ is minimal among those localizing ideals which properly contain $\cat J$ if and only if $\ihom{t_1,t_2}\neq 0$ for any two objects $t_1,t_2\in \cat L\setminus \cat J$. In the proof one needs to consider $\cat L_1\coloneqq \Loco{\cat J \cup \{t_1\}}$.
\end{Rem}

\begin{Thm}\label{thm:strat_cosupport}
	Let $\cat T$ be a rigidly-compactly generated tt-category whose spectrum is weakly noetherian. Assume that the local-to-global principle for localizing ideals holds. The following conditions are equivalent:
	\begin{enumerate}
		\item $\cat T$ is stratified. 
		\item $\Cosupp \ihom{t_1,t_2} = \Supp(t_1) \cap \Cosupp(t_2)$ for all $t_1,t_2 \in \cat T$.
		\item $\ihom{t_1,t_2} = 0$ implies $\Supp(t_1) \cap \Cosupp(t_2) = \emptyset$ for all $t_1,t_2 \in \cat T$.
	\end{enumerate}
\end{Thm}

\begin{proof}
	Using \cref{lem:pointsuppcosupp} we can follow the proof of Theorem 9.5 of \cite{BensonIyengarKrause12}, which we spell out for the reader. 

	$(a)\Rightarrow(b)$: One inclusion ($\subseteq$) is \cref{lem:cosupp_hom}. For the other inclusion, let $\cat P \in \Supp(t_1) \cap \Cosupp(t_2)$. In particular, $\Gamma_{\cat P} t_1 \ne 0$. Since $\Gamma_{\cat P}\cat T$ is a nontrivial minimal localizing ideal,  $\gP \in \Loco{\Gamma_{\cat P} t_1}$. Using \cite[Lemma~8.4]{BensonIyengarKrause12} and the assumption $\cat P \in \Cosupp(t_2)$ we see that
		\[
			0 \ne \LambdaP t_2=\ihom{\gP,t_2} \in \Coloco{\ihom{\Gamma_{\cat P} t_1,t_2}}.
		\]
	Therefore $ 0 \ne \ihom{\Gamma_{\cat P} t_1,t_2} \simeq \ihom{\gP,\ihom{t_1,t_2}}$ and as a consequence we have that $\cat P \in \Cosupp\ihom{t_1,t_2}$. 

	$(b) \Rightarrow(c)$: This follows from $\Cosupp(0) = \emptyset$.

	$(c) \Rightarrow(a)$: The local-to-global principle for localizing ideals holds by hypothesis. We show that the criteria of \cref{lem:minimality_crit} are satisfied for $\GammaP \cat T$. Let $t_1,t_2 \in \GammaP \cat T$ be nonzero objects. Then $\cat P \in \Supp(t_1)$ and $\cat P\in\Supp(t_2)$. By \cref{lem:pointsuppcosupp}, we also have $\cat P \in \Cosupp(t_2)$. Hence $(c)$ implies $\ihom{t_1,t_2} \ne 0$.
\end{proof}

\begin{Rem}
	The appearance of cosupport in the above characterization of \emph{stratification} should be regarded as significant: it is a fundamental characterization of stratification which utilizes the notion of cosupport. This motivates the study of cosupport even for questions purely about localizing ideals.
\end{Rem}

\begin{Rem}
	Recall that stratification implies that the Balmer--Favi support satisfies the full tensor product property: $\Supp(t_1 \otimes t_2)=\Supp(t_1)\cap \Supp(t_2)$ for all $t_1,t_2\in \cat T$; see \cite[Theorem 8.2]{bhs1}. It would be natural to guess that costratification would promote the half-hom theorem (\cref{prop:halfhom}) to a full hom-theorem: $\Cosupp(\ihom{t_1,t_2})=\Supp(t_1)\cap \Cosupp(t_2)$. Perhaps surprisingly, \cref{thm:strat_cosupport} shows that this full hom-theorem is actually equivalent to \emph{stratification}.
\end{Rem}

\begin{Cor}\label{cor:cosupp-of-perp}
	If $\cat T$ is stratified then:
	\begin{enumerate}
		\item $\cat L^\perp=\InvCosupp{\Supp(\cat L)^c}$ for every localizing ideal $\cat L$.
		\item ${}^\perp \cat C = \InvSupp{\Cosupp(\cat C)^c}$ for every colocalizing coideal $\cat C$.
	\end{enumerate}
\end{Cor}

\begin{proof}
	Stratification implies the local-to-global principle by \cite[Theorem~4.1]{bhs1} and hence also the codetection property by \cref{thm:LGP-equiv}. Hence we can use \cref{thm:strat_cosupport} to observe that
	\begin{align*}
		\cat L^\perp &= \SET{t \in \cat T}{\ihom{s,t}=0 \;\text{ for all }  s \in \cat L}\\
		&= \SET{t \in \cat T}{\Cosupp \ihom{s,t}=\emptyset \;\text{ for all }  s \in \cat L}\\
		&= \SET{t \in \cat T}{\Supp(s) \cap \Cosupp(t) = \emptyset \; \text{ for all }  s \in \cat L}\\
		&= \InvCosupp{\Supp(\cat L)^c}
	\end{align*}
	and similarly
	\begin{align*}
		{}^\perp \cat C &= \SET{s\in \cat T}{\ihom{s,t} = 0 \;\text{ for all }  t \in \cat C}\\
		&= \SET{s\in \cat T}{\Cosupp \ihom{s,t} = \emptyset \;\text{ for all }  t \in \cat C}\\
		&= \SET{ s \in \cat T}{\Supp(s)\cap \Cosupp(t) = \emptyset \; \text{ for all }  t \in \cat C}\\
		&= \InvSupp{\Cosupp(\cat C)^c}.\qedhere
	\end{align*}
\end{proof}

\begin{Thm}\label{thm:costrat_implies_strat}
	If $\cat T$ is costratified, then it is also stratified. 
\end{Thm}

\begin{proof}
	Since the colocal-to-global principle is equivalent to the local-to-global principle by \cref{thm:LGP-equiv}, it suffices by \cite[Theorem 4.1]{bhs1} and \cref{thm:equiv-costrat} to establish that $\GammaP\cat T$ is minimal for all $\cat P \in \Spc(\cat T^c)$. For this we follow the proof of \cite[Theorem 9.7]{BensonIyengarKrause12} by invoking \cref{lem:minimality_crit}. To this end, we fix two nonzero objects $t_1,t_2 \in \GammaP\cat T$ and show that $\ihom{t_1,t_2}\neq 0$. Since $t_1\in \GammaP\cat T$, we have that $\GammaP t_1\simeq t_1$ by \Cref{prop:local-cats-are-ok}. Then by \Cref{rem:formulalambdaHom}, $\ihom{t_1,t_2}\simeq \ihom{t_1,\LambdaP t_2}$. Since $t_2\in \GammaP \cat T$, by \cref{lem:pointsuppcosupp} we have that $\cat P \in \Cosupp(t_2)$, that is, $\LambdaP t_2\neq 0$.  Because~$\cat T$ is costratified, the cominimality of $\LambdaP \cat T$ implies that $\LambdaP t_1\in \Coloco{\LambdaP t_2}$. To conclude, we observe that because $t_1 \ne 0$, we have ${0 \ne \ihom{t_1,t_1}} \simeq \ihom{t_1,\LambdaP t_1}$, which by \eqref{eq:[t,coloc]} is only possible if $0 \neq \ihom{t_1,\LambdaP t_2}\simeq \ihom{t_1,t_2}$. 
\end{proof}

\begin{Cor}\label{cor:costrat-perps}
	If $\cat T$ is costratified, then the map sending a subcategory $\cat L$ to $\cat L^{\perp}$ induces a bijection
	\[
		\{ \text{localizing ideals of $\cat T$} \} \xrightarrow{\sim} \{\text{colocalizing coideals of $\cat T$}\}
	\]
	The inverse map sends $\cat C$ to ${}^{\perp} \cat C$. 
\end{Cor}

\begin{proof}
	Costratification and stratification (\cref{thm:costrat_implies_strat}) provide bijections of both sets with the set of all subsets of $\Spc(\cat T^c)$. The map which sends a subset to its complement then induces a bijection from the set of localizing ideals to the set of colocalizing ideals which, by \cref{cor:cosupp-of-perp}, is given by $\cat L \mapsto \cat L^\perp$ with inverse~$\cat C\mapsto { }^\perp \cat C$.
\end{proof}

\begin{Prop}\label{prop:not-strictly-colocalizing}
	If $\cat T$ is stratified but not costratified, then $\cat T$ has a colocalizing coideal which is not strictly colocalizing.
\end{Prop}

\begin{proof}
	A strictly colocalizing coideal must be of the form $\cat L^\perp$ for a localizing ideal~$\cat L$ (see \Cref{rem:strictly-localizing}). If $\cat T$ is stratified then for any $Y \subseteq \Spc(\cat T^c)$, we have
		\[
			(\InvSupp{Y})^\perp = (\Loco{\gP \mid \cat P \in Y})^\perp = \SET{\gP}{\cat P \in Y}^\perp = \InvCosupp{Y^c}.
		\]
	Now, since $\cat T$ is not costratified, there is some colocalizing coideal $\cat C$ such that $\cat C \neq \InvCosupp{\Cosupp(\cat C)}$. But if $\cat C$ was strictly colocalizing then
		\[
			\cat C = \cat L^\perp = (\InvSupp{\Supp(\cat L)})^\perp = \InvCosupp{\Supp(\cat L)^c}
		\]
		and $\Supp(\cat L)^c=\Supp({}^\perp\cat C)^c=\Cosupp(\cat C)$ by \cref{cor:cosupp-of-perp}(b), so that  $\cat C=\InvCosupp{\Cosupp(\cat C)}$.
\end{proof}

\begin{Rem}\label{rem:does-strat-imply-costrat}
	We do not know whether stratification implies costratification. This seems to be a difficult question. For example, suppose we could find an example~$\cat T$ which is stratified but not costratified. By \cref{prop:not-strictly-colocalizing}, this category would have a colocalizing coideal which is not strictly colocalizing. Assuming the counterexample arises from a combinatorial model, this would prove (\cref{rem:vopenka}) that Vopĕnka's principle is inconsistent with ZFC. Whether that is likely or unlikely is beyond our expertise.
\end{Rem}

\begin{Rem}\label{rem:vopenka-positive}
	On the other hand, if Vopĕnka's principle is taken as an axiom then stratification implies costratification for any example that arises from a stable combinatorial model. Indeed, in this case every colocalizing coideal $\cat C$ of $\cat T=\Ho(\cat M)$ is strictly colocalizing (\cref{rem:vopenka}). Hence $\cat C=({}^\perp\cat C)^\perp$. If $\cat T$ is stratified then $({ }^\perp \cat C)^\perp =\InvCosupp{\Cosupp(\cat C)}$ by \cref{cor:cosupp-of-perp} and costratification follows.
\end{Rem}

\begin{Rem}
	We can summarize the situation with the diagram displayed on page~\pageref{eq:big-figure}. It assembles the implications of \cref{prop:codetection-implies-detection}, \cref{thm:LGP-equiv}, \cref{thm:costrat_implies_strat}, along with \cref{exa:detect-but-not-codetect} and \cite[Example 4.6]{bhs1}.
\end{Rem}

\section{The geometry of cosupport}\label{sec:geometry-of-cosupport}

We have already seen some relations between support and cosupport in the previous sections. We now dig deeper into the somewhat mysterious relationship between the support and cosupport of a given object. Throughout this section~$\cat T$ will denote a rigidly-compactly generated tt-category whose spectrum~$\Spc(\cat T^c)$ is weakly noetherian.

\begin{Prop}\label{prop:intersect-with-thomason}
	Assume that codetection holds. Let $Y \subseteq \Spc(\cat T^c)$ be a Thomason subset. Then for any $t \in \cat T$ we have:
		\[
			\Supp(t) \cap Y = \emptyset \text{ if and only if } \Cosupp(t) \cap Y = \emptyset.
		\]
\end{Prop}
\begin{proof}
	The codetection property implies the detection property by \cref{prop:codetection-implies-detection}. Then observe that
		\begin{align*}
			\Cosupp(t) \cap Y = \emptyset	&\Longleftrightarrow \Cosupp(\ihom{e_Y,t}) = \emptyset  && \text{(\cref{lem:cosuppcoloc})}\\
											&\Longleftrightarrow \ihom{e_Y,t}=0 &&\text{(codetection)}\\
											&\Longleftrightarrow \ihom{x,t}=0 \text{ for all } x \in \cat T^c_Y &&\\
											&\Longleftrightarrow x\otimes t=0 \text{ for all } x \in \cat T^c_Y &&\\
											&\Longleftrightarrow e_Y \otimes t=0 &&\\
											&\Longleftrightarrow \Supp(e_Y \otimes t)=\emptyset &&\text{(detection)}\\
											&\Longleftrightarrow \Supp(t)\cap Y=\emptyset &&
		\end{align*}
	which establishes the claim.
\end{proof}

\begin{Rem}
	The next theorem (and its corollary) expresses a fundamental relation between support and cosupport, generalizing \cite[Theorem 4.13]{BensonIyengarKrause12}. The statement uses \cref{not:vis} and \cref{not:min}.
\end{Rem}

\begin{Thm}\label{thm:general-min}
	Assume that codetection holds. Then for any $t \in \cat T$ we have
		\[
			\Vis \cat T \cap \min \Supp(t) = \Vis \cat T \cap \min \Cosupp(t).
		\]
\end{Thm}

\begin{proof}
	Let $\cat P$ be a visible point. Then $\singP = \closureP \cap \gen(\cat P)$ and $\closureP \setminus \singP = {\closureP \cap \gen(\cat P)^c}$ is the intersection of two Thomason subsets and hence is itself Thomason. Applying \cref{prop:intersect-with-thomason} to the Thomason $\closureP \setminus \singP$ yields
		\[
			\Supp(t) \cap \overbar{\{\cat P\}} \subseteq \{\cat P\} \Longleftrightarrow \Cosupp(t) \cap \overbar{\{\cat P\}} \subseteq \{\cat P\}.
		\]
	It follows, by invoking \cref{prop:intersect-with-thomason} again for the Thomason $\closureP$, that
		\[
			\Supp(t) \cap \overbar{\{\cat P\}} = \{\cat P\} \Longleftrightarrow \Cosupp(t) \cap \overbar{\{\cat P\}} = \{\cat P\}.
		\]
	Recall that $\cat P \in \min V$ means, by definition, that $V \cap \overbar{\{\cat P\}} = \{\cat P\}$ so we are done.
\end{proof}

\begin{Cor}\label{cor:min}
	If $\cat T$ has noetherian spectrum and $t \in \cat T$, then 
		\[
			\min\Supp(t) = \min\Cosupp(t).
		\]
\end{Cor}

\begin{proof}
	Since the spectrum is noetherian, the colocal-to-global principle holds (by \cref{cor:noetherian_local_global}) and hence the codetection property holds (by \cref{thm:LGP-equiv}). Moreover, a spectral space is noetherian if and only if each of its points is visible, by \cite[Proposition~7.13]{BalmerFavi11}; see also \cite[Remark~2.2]{bhs1}. Hence the result follows immediately from \cref{thm:general-min}.
\end{proof}

\begin{Rem}\label{rem:hausdorff}
	Recall that a spectral space has Krull dimension zero if and only if its specialization order is trivial if and only if it is $T_1$ if and only if it is Hausdorff if and only if it is profinite (see, e.g., \cite[Remark 2.4]{bhs1}). In this case we have:
\end{Rem}

\begin{Cor}\label{cor:zero-dimensional}
	If the spectrum of $\cat T$ is zero-dimensional and $\cat T$ satisfies the codetection property then for any $t \in \cat T$ we have
		\[
			\Vis \cat T \cap \Supp(t) = \Vis \cat T \cap \Cosupp(t).
		\]
\end{Cor}

\begin{proof}
	Since the specialization order is trivial, we have $\min W = W$ for any subset $W \subseteq \Spc(\cat T^c)$. Hence the statement follows from \cref{thm:general-min}.
\end{proof}

\begin{Exa}\label{exa:KInj}
	Let $\cat T=\KInj{kG}$ denote the homotopy category of complexes of injective $kG$-modules with $G$ a finite group and $k$ a field whose characteristic divides the order of $G$. In \cite[Example 11.1]{BensonIyengarKrause12} the authors prove that $\Cosupp(\unit)=\{\mathfrak m\}$ is a singleton, where $\mathfrak m$ is the unique maximal ideal in $H^*(G,k)$. This shows that the statement of \cref{cor:min} is, in the given generality, optimal.
\end{Exa}

\begin{Exa}
	Consider $\bbN^+$, the one-point compactification of the discrete set $\bbN$ of natural numbers with $\infty$ denoting the accumulation point, and let $R = C(\bbN^+,k)$ be the commutative ring of locally constant functions on $\bbN^+$ with values in a field~$k$. If $\Der(R)$ denotes the derived category of $R$-modules, we have homeomorphisms
		\[
			\Spc(\Der(R)^c) \cong \Spec(R) \cong \bbN^+.
		\]
	Note that this space is profinite, hence weakly noetherian and generically noetherian\footnote{Recall that a spectral space $X$ is said to be \emph{generically noetherian} if $\gen(x)$ is noetherian for every $x \in X$; see \cite[Definition 9.5]{bhs1}.} but not noetherian. The canonical inclusion $\bbN \subset \bbN^+$ exhibits $\bbN$ as a Thomason subset and we can consider $\LambdaN \unit = \ihom{e_{\bbN},\unit}$. We claim:
    \begin{enumerate}
        \item $\infty \in \Supp \LambdaN \unit$; and
        \item $\Cosupp \LambdaN \unit \subseteq \bbN$.
    \end{enumerate}
	Since $\min W = W$ for any subset $W \subseteq \bbN^+$, we deduce that 
		\[
			\min \Supp t \neq \min \Cosupp t 
		\]
	for $t = \LambdaN \unit \in \Der(R)$. Thus \cref{cor:min} does not hold without the noetherian assumption, leading to the more precise statement~\cref{thm:general-min}. Moreover, since $\Vis \Der(R)^c = \bbN$, this example also shows that the restriction to visible points
	in \cref{cor:zero-dimensional}
	is necessary.

	Claim (b) follows from \cref{lem:cosuppcoloc}, so it remains to prove Claim~(a). To this end, note that since the space is Hausdorff, $\gen(\infty)=\{\infty\}$ and $g_{\infty} = f_{\gen(\infty)^c} = f_{\bbN}$. Thus, $\infty \not\in\Supp \LambdaN\unit$ would mean that $f_{\bbN}\otimes \LambdaN\unit=0$. This would imply that the homotopy pullback square 
		\[\begin{tikzcd}
			\unit \ar[r] \ar[d] & \LambdaN \unit \ar[d] \\
			f_{\bbN} \ar[r] & f_{\bbN} \otimes \LambdaN \unit
		\end{tikzcd}\]
	degenerates to an isomorphism $\unit \simeq f_{\bbN} \oplus \LambdaN \unit$. By naturality, this would force $\LambdaN \unit \simeq e_{\bbN}$ and consequently a decomposition $\bbN^+ = \bbN \sqcup \{\infty\}$. This is a contradiction, so $\infty \in \Supp \LambdaN \unit$ as claimed.
	
	In fact, $\Cosupp \LambdaN \unit = \bbN$, since for any $n \in \bbN$, one can use the delta function $\delta_n \in C(\bbN^+,k)$ to construct a nontrivial map $g_n \to \unit$ in $\Der(R)$, bearing in mind \cref{exa:detect-but-not-codetect}. It then follows from \cref{thm:general-min} and (a) that $\Supp \LambdaN \unit = \bbN^+$.
\end{Exa}

\begin{Rem}
	The above examples show that $\Supp(t) \neq \Cosupp(t)$ in general. In fact, it turns out that the relationship between $\Supp(t)$ and $\Cosupp(t)$ has an interesting connection with the Tate construction.
\end{Rem}

\begin{Prop}\label{prop:tate}
	Let $Y\subseteq \Spc(\cat T^c)$ be a Thomason subset and let
		\[
			e_Y \to \unit \to f_Y \xrightarrow{\theta} \Sigma e_Y
		\]
	be the associated finite (co)localization. The following are equivalent:
	\begin{enumerate}
		\item	The internal hom $\ihom{f_Y,\Sigma e_Y}$ vanishes.
		\item	The map $\theta\colon f_Y \to \Sigma e_Y$ vanishes.
		\item	The Thomason $Y$ is both open and closed and the spectrum decomposes
					\[
						\Spc(\cat T^c) = Y \sqcup Y^c
					\]
				as a disjoint union of closed subsets.
		\item	The complement $Y^c$ is Thomason.
	\end{enumerate}
\end{Prop}

\begin{proof}
	This is established by \cite[Proposition 2.29]{PatchkoriaSandersWimmer22}.
\end{proof}

\begin{Lem}\label{lem:cosupp-tate}
	Let $Y \subseteq \Spc(\cat T^c)$ be a Thomason subset. The following are equivalent:
	\begin{enumerate}
		\item $\Cosupp(e_Y) \subseteq Y$.
		\item $\Cosupp(\ihom{f_Y,\Sigma e_Y}) = \varnothing$.
	\end{enumerate}
\end{Lem}

\begin{proof}
	By \cref{lem:cosuppcoloc}, we have
		\[
			\Cosupp(\ihom{f_Y,\Sigma e_Y}) = Y^c \cap \Cosupp(e_Y)
		\]
	and hence $(a)$ is equivalent to $(b)$.
\end{proof}

\begin{Cor}\label{cor:cosupp-spc-disjoint}
	Assume that the codetection property holds and let $Y \subseteq \Spc(\cat T^c)$ be a Thomason subset. Then $\Cosupp(e_Y) \subseteq Y$ if and only if there is a decomposition $\Spc(\cat T^c) = Y \sqcup Y^c$ as a disjoint union of closed subsets.
\end{Cor}

\begin{proof}
	Since we are assuming that the codetection property holds, \cref{lem:cosupp-tate} says that $\Cosupp(e_Y) \subseteq Y$ if and only if $\ihom{f_Y,\Sigma e_Y}=0$. The result then follows from \cref{prop:tate}.
\end{proof}

\begin{Prop}\label{prop:finite-discrete}
	Let $\cat T$ be a rigidly-compactly generated tt-category with $\Spc(\cat T^c)$ weakly noetherian. The following are equivalent:
	\begin{enumerate}
		\item The codetection property holds and $\Supp(t)=\Cosupp(t)$ for all $t \in \cat T$.
		\item $\Spc(\cat T^c)$ is a finite discrete space.
	\end{enumerate}
\end{Prop}

\begin{proof}
	$(a)\Rightarrow(b)$: Recall that $\Supp(e_Y) = Y$ for any Thomason subset $Y \subseteq \Spc(\cat T^c)$. Hence \cref{cor:cosupp-spc-disjoint} implies that $Y$ is both open and closed. Since $\Spc(\cat T^c)$ is assumed weakly noetherian, every point is the intersection of a Thomason and the complement of a Thomason. Hence every point is open and so the topology is discrete. Since $\Spc(\cat T^c)$ is quasi-compact, it is then finite and discrete.
	$(b)\Rightarrow(a):$ This is provided by \cref{cor:min}, \cref{thm:LGP-equiv} and \cref{cor:noetherian_local_global} since a finite discrete space is both noetherian and zero-dimensional.
\end{proof}

\begin{Cor}\label{cor:finite-discrete-noetherian}
	If $\Spc(\cat T^c)$ is noetherian then the following are equivalent:
	\begin{enumerate}
		\item $\Supp(t)=\Cosupp(t)$ for all $t \in \cat T$.
		\item $\Spc(\cat T^c)$ is a finite discrete space.
	\end{enumerate}
\end{Cor}

\begin{proof}
	This follows from \cref{prop:finite-discrete} and \cref{cor:noetherian_local_global}.
\end{proof}

\begin{Exa}
	If $\Spc(\cat T^c)=*$ is a single point then $\Supp=\Cosupp$. Indeed, both just detect whether an object is nonzero.
\end{Exa}

\begin{Exa}[Two connected points]\label{exa:two-connected-points}
	Let $\cat T$ be a rigidly-compactly generated tt-category whose spectrum 
		\[
			\Spc(\cat T^c) =\; \begin{tikzcd}
			\bullet \ar[d,dash] &[-25pt] \mathfrak m \\
			\bullet & \eta
			\end{tikzcd}
		\]
	consists of two connected points: a closed point $\mathfrak m$ and a generic point $\eta$. There is only one non-trivial Thomason subset, namely $\{\mathfrak m\}$, hence there is only one non-trivial finite localization
		\[
			e_{\singm} \to \unit \to f_{\singm} \to \Sigma e_{\singm}.
		\]
	Moreover, $g(\mathfrak m)=e_{\singm}$ and $g(\eta)=f_{\singm}$. We can directly observe that the codetection property holds (i.e., without invoking \cref{cor:noetherian_local_global}). In fact we can directly observe that the local-to-global principle and the colocal-to-global principle hold in this example. From \cref{lem:cosupp-of-gP} we have that $\Cosupp(g(\eta)) = \{\eta\}$ is a single point. On the other hand, since the two points are connected, we know by \cref{cor:cosupp-spc-disjoint} that $\Cosupp(g(\mathfrak m)) = \{\mathfrak m,\eta\}$ is the whole space. The nature of $\Cosupp(\unit)$ is more subtle. We always have $\mathfrak m \in \Cosupp(\unit)$, either by a direct argument or  by \cref{cor:min}. On the other hand, both $\eta \in \Cosupp(\unit)$ and $\eta \not\in \Cosupp(\unit)$ are possible, as the next sub-example demonstrates.
\end{Exa}

\begin{Exa}\label{exa:DVR}
	Take $\cat T=\Der(R)$ to be the derived category of a discrete valuation ring $(R,\mathfrak m,k)$. In this example, $\Cosupp(\unit) \subseteq \{\mathfrak m\}$ if and only if $R$ is complete. This follows from 
	the triangle
	\[
	\ihom{f_{\{\mathfrak m\}},\unit} \to \unit \to  \ihom{e_{\{\mathfrak m\}},\unit}
	\]
	and $g(\eta)=f_{\{\mathfrak m\}}$, once one recognizes that in this case $\ihom{e_{\{\mathfrak m\}},-}$ is the derived functor of $\mathfrak m$-adic completion; see \cref{exa:DRsmashing} and the references therein.
\end{Exa}

\begin{Rem}
	If a category is stratified then its support function determines and is determined by its cosupport function. Indeed, for any $t \in \cat T$, we have $\Supp(t) = \Cosupp(\{t\}^\perp)^c$ and $\Cosupp(t) = \Supp({}^\perp \{t\})^c$ by \cref{cor:cosupp-of-perp}. However, the question remains whether the support of an individual object $\Supp(t)$ determines its cosupport $\Cosupp(t)$, or vice versa. The next two examples demonstrate that in general this is false.
\end{Rem}

\begin{Exa}\label{exa:supp-no-cosupp}
	Let $Y\coloneqq \supp(x)$ for a compact object $x \in \cat T^c$. Then $\Supp(x) = \Supp(e_Y)=Y$. In particular, $\Cosupp(x) \subseteq Y$ by \cref{exa:cosupp-compact}. On the other hand, we have seen in \cref{cor:cosupp-spc-disjoint} that if $\Cosupp(e_Y)\subseteq Y$ then the spectrum disconnects into $Y$ and its complement (provided the codetection property holds). Therefore there are many examples where we have two objects $t_1$ and $t_2$ with $\Supp(t_1) = \Supp(t_2)$ and yet $\Cosupp(t_1) \neq \Cosupp(t_2)$. For an explicit example, this occurs for any category $\cat T$ whose spectrum is two connected points (\cref{exa:two-connected-points}) and $t_1=e_{\{\mathfrak m\}}$ and $t_2=x$ any compact object with $\supp(x)=\{\mathfrak m\}$. Moreover, there are such examples which are costratified; e.g.,~the derived category $\cat T=\Der(R)$ of a discrete valuation ring (see \cref{prop:costratification_for_dr} below). We conclude that an object's support does not in general determine its cosupport, even in costratified examples.
\end{Exa}

\begin{Exa}\label{exa:cosupp-no-supp}
	Let $\cat T$ be a local category where $\Cosupp(\unit)=\{\mathfrak m\}$ just consists of the closed point. For example, we could take $\cat T=\Der(R)$ for~$R$ a complete discrete valuation ring (\cref{exa:DVR}) or $\cat T= \KInj{kG}$ (\cref{exa:KInj}). Then for any nonzero compact object $x$, we have $\Cosupp(x) = \supp(x) \cap \Cosupp(\unit) = \{\mathfrak m\}$ (\cref{exa:cosupp-compact}). Thus, any two nonzero compact objects have the same cosupport: $\Cosupp(x)=\Cosupp(y)$. On the other hand, provided the spectrum is not a single point, there exist nonzero compact objects with differing support. We conclude that an object's cosupport does not in general determine its support, even in costratified examples.
\end{Exa}

\begin{Rem}
	In \cref{sec:duality} we will establish a further important relationship between support and cosupport, namely that the support of an object coincides with the cosupport of its Brown--Comenetz dual; see \cref{prop:cosupp-BC-duals}. The discussion there will also elaborate on the relationship with completion.
\end{Rem}

\newpage
\part{Perfection and duality}\label{part:perfection-duality}

The opposite category $\cat T\op$ of a triangulated category $\cat T$ inherits a triangulated structure (with $\Sigma_{\cat T\op} = \Sigma^{-1}_{\cat T}$) and the notion of a thick subcategory is self-dual. In fact, the thick subcategories of a triangulated category $\cat T$ coincide with the thick subcategories of its opposite triangulated category~$\cat T\op$. Moreover, the localizing subcategories of $\cat T$ are precisely the colocalizing subcategories of $\cat T\op$ and vice versa. However, localizing ideals and colocalizing coideals are not obviously dual notions. As it stands, the notion of a colocalizing coideal is somewhat mysterious, besides its appearance in right orthogonals of localizing ideals. Similarly, definitions of cosupport in the literature --- including the notion we have studied above --- are not very conceptually motivated. It is not too far a stretch to say that one scratches one's head the first time one sees the definition of cosupport. Fortunately, in this part of the paper we will provide a conceptual understanding for it all, exhibit localizing ideals as dual to colocalizing coideals, support as dual to cosupport, and stratification as dual to costratification. In order to achieve this, we will work in a more general setting that encompasses both $\cat T$ and~$\cat T\op$. This leads us to perfection:

\section{Perfect generation}\label{sec:perfect-generation}

The opposite category $\cat T\op$ of a compactly generated triangulated category~$\cat T$ is never compactly generated.\footnote{Boardman \cite{Boardman70} proved this for $\cat T=\SH$. The general result is due to Neeman \cite[Appendix~E.1]{Neeman01}.} The more general notion of ``perfectly generated'' triangulated category is more flexible in this respect, encompassing both compactly generated triangulated categories and their opposites. This notion has a three-fold role to play in this work. Firstly, it will provide an adequate level of generality in which the Balmer--Favi approach to support can be constructed, which will facilitate a comparison between support and cosupport. Secondly, it will provide another way to construct objects with prescribed cosupport. Thirdly, it will be important for controlling descent of costratification.

\begin{Def}[Krause]\label{def:perfectly-generated}
	Recall from \cite[Section 5]{Krause10} that a triangulated category~$\cat T$ is \emph{perfectly generated} if it has coproducts and there exists a set of objects~$\cat G \subseteq \cat T$ satisfying the following two conditions:
	\begin{itemize}
		\item[(PG1)]	$\Loc\langle \cat G\rangle = \cat T$.
		\item[(PG2)]	Given any family $(x_i \to y_i)_{i \in I}$ of morphisms in $\cat T$ such that the induced map $\cat T(g,x_i) \to \cat T(g,y_i)$ is surjective for all $g \in \cat G$ and $i\in I$, the induced map
						\[
							\cat T(g,\coprod_i x_i) \to \cat T(g,\coprod_i y_i)
						\]
						is surjective.
	\end{itemize}
	Dually, we say that a triangulated category $\cat T$ is \emph{perfectly cogenerated} if $\cat T\op$ is perfectly generated.
\end{Def}

\begin{Rem}\label{rem:weak-generation}
	As noted in \cite[Remark 5.1.2]{Krause10}, in the presence of (PG2), the generation condition (PG1) is equivalent to the (\emph{a priori} weaker) generation condition that an object $t \in \cat T$ vanishes if $\cat T(\Sigma^n g,t) =0$ for all $g \in \cat G$ and $n \in \bbZ$.
\end{Rem}

\begin{Exa}
	A compactly generated triangulated category is perfectly generated. The (PG2) condition is automatic when $\cat G$ is a set of compact objects. More generally, a well generated triangulated category is perfectly generated; see \cite{Krause01}.
\end{Exa}

\begin{Def}\label{def:compact-BC}
	Let $\cat T$ be a compactly generated triangulated category. For any compact object $c \in \cat T^c$, the functor
		\begin{equation}\label{eq:homological}
			\Hom_{\bbZ}(\cat T(c,-),\bbQ/\bbZ)\colon\cat T\op \to \Ab
		\end{equation}
	is homological and sends coproducts in $\cat T$ to products in $\Ab$, hence by Brown representability is represented by an object $I_c \in \cat T$. Consequently, we have a natural isomorphism
		\begin{equation}\label{eq:BC-homs}
			\cat T(t,I_c) \cong \Hom_{\bbZ}(\cat T(c,t),\bbQ/\bbZ)
		\end{equation}
	for any $t \in \cat T$ and $c \in \cat T^c$.
\end{Def}

\begin{Rem}\label{rem:Q/Z-cogenerator}
	The significance of $\bbQ/\bbZ$ in the above construction is that it is an injective cogenerator of the category of abelian groups. Injectivity ensures that the functor \eqref{eq:homological} is homological, while cogeneration amounts to the fact that every nonzero abelian group admits a nonzero homomorphism to $\bbQ/\bbZ$. It follows that an object $t=0$ if and only if $\cat T(c,t)=0$ for all $c \in \cat T^c$ if and only if $\cat T(t,I_c)=\Hom_{\bbZ}(\cat T(c,t),\bbQ/\bbZ)=0$ for all $c \in \cat T^c$. 
	Moreover, we have:
\end{Rem}

\begin{Lem}\label{lem:Q/Z-cogenerator}
    For any morphism $x \to y$ in $\cat T$ and $c \in \cat T^c$, the induced map $\cat T(y,I_c)\to \cat T(x,I_c)$ is surjective if and only if the induced map $\cat T(c,x)\to \cat T(c,y)$ is injective. 
\end{Lem}
\begin{proof}
	One implication follows from the injectivity of $\bbQ/\bbZ$: The induced map $\cat T(y,I_c)\to \cat T(x,I_c)$ is surjective if the induced map $\cat T(c,x)\to \cat T(c,y)$ is injective. The other implication uses the cogeneration property of $\bbQ/\bbZ$: If $\cat T(y,I_c)\to \cat T(x,I_c)$ is surjective then by the natural isomorphism \eqref{eq:BC-homs} any homomorphism of abelian groups $\cat T(c,x)\to \mathbb Q/\mathbb Z$ factors through $\cat T(c,x)\to \cat T(c,y)$. This implies that $\cat T(c,x) \to \cat T(c,y)$ is injective because if $f \in \cat T(c,x)$ is nonzero we can extend a nonzero homomorphism of abelian groups $\langle f \rangle \to \mathbb Q/\mathbb Z$ to a homomorphism $\cat T(c,x)\to\mathbb Q/\mathbb Z$ which does not annihilate~$f$ and hence $\cat T(c,x)\to \cat T(c,y)$ cannot annihilate $f$.
\end{proof}

\begin{Prop}\label{prop:Top-is-perfectly-generated}
	If $\cat T$ is compactly generated then $\cat T\op$ is perfectly generated. More precisely, if $\cat G \subseteq \cat T$ is a set of compact generators, then $\cat T\op$ is perfectly generated by $I(\cat G)\coloneqq \SET{ I_c }{c \in \cat G}$. In particular, $\cat T=\Coloc\langle I(\cat G)\rangle$.
\end{Prop}

\begin{proof}
	Suppose $t\in \cat T$ is such that $\cat T(t,I_c)=0$ for all $c \in \cat G$. Since $\bbQ/\bbZ$ is a cogenerator of $\Ab$, this is equivalent to the statement that $\cat T(c,t)=0$ for all $c \in \cat G$ (\cref{rem:Q/Z-cogenerator}). By assumption, $\cat G$ is a generating set for $\cat T$, hence $t=0$. This shows that $I(\cat G)$ is a set of cogenerators for $\cat T$ (\cref{rem:weak-generation}). In order to verify condition~(PG2), let $f_i\colon x_i \to y_i$ be a collection of maps in $\cat T$ with $\cat T(f_i,I_c)$ surjective for all $c \in \cat G$. In light of \cref{lem:Q/Z-cogenerator}, this translates to the statement that the maps $\cat T(c,f_i)$ are injective for all $c \in \cat G$. It follows that $\prod_i \cat T(c,f_i) = \cat T(c,\prod_i f_i)$ is injective in $\Ab$ for each $c \in \cat G$, which in turn is equivalent to the surjectivity of $\cat T(\prod_i f_i, I_c)$ for all $c \in \cat G$. This establishes that $I(\cat G)$ is a set of perfect cogenerators for $\cat T$.
\end{proof}

\begin{Rem}
	The Brown representability theorem holds for perfectly generated categories; see \cite{Krause02,Krause10}. In particular, if $\cat T$ is perfectly generated then any coproduct-preserving exact functor $f^*\colon\cat T\to \cat S$ admits a right adjoint~$f_*$.
\end{Rem}

\begin{Prop}\label{prop:perfect-conservation}
	Suppose the triangulated category $\cat T$ is perfectly generated by~${\cat G \subseteq \cat T}$. Let $f^*\colon \cat T \to \cat S$ be a coproduct-preserving exact functor whose right adjoint $f_*$ preserves coproducts. Then the following are equivalent:
	\begin{enumerate}
		\item The category $\cat S$ is perfectly generated by $f^*(\cat G)$.
		\item The right adjoint $f_*$ is conservative.
	\end{enumerate}
\end{Prop}

\begin{proof}
	One verifies directly that $f^*(\cat G)$ satisfies (PG2). Here one uses the assumption that $f_*$ preserves coproducts. Then using \cref{rem:weak-generation}, the equivalence of $(a)$ and~$(b)$ just follows from observing that $f_*(s) = 0$ if and only if $\cat T(\Sigma^n \cat G,f_*s)=0$ if and only if $\cat S(\Sigma^n f^*(\cat G),s)=0$.
\end{proof}

\begin{Rem}\label{rem:replace-with-compact}
	With ``perfectly generated'' replaced by ``compactly generated'', \cref{prop:perfect-conservation} is well-known. In this case, the assumption that $f_*$ preserves coproducts can be removed since it is implied by (and is in fact equivalent to) the condition in part $(a)$ that $f^*$ preserves the compactness of the generators; see \cite[Theorem 5.1]{Neeman96}.
\end{Rem}

\section{Cosupport is support}\label{sec:cosupp-is-supp}

By a tensor-triangulated category we usually mean a triangulated category equipped with a compatible \emph{closed} symmetric monoidal structure as spelled out in \cite[Appendix A]{HoveyPalmieriStrickland97}. Indeed, the closed structure is used in a fundamental way in the definition of cosupport and is featured in the very definition of colocalizing coideal. Although the opposite category of a symmetric monoidal category inherits a symmetric monoidal structure, it does not inherit a \emph{closed} symmetric monoidal structure. For this reason, we will need the following weaker notion:

\begin{Ter}\label{def:non-closed-tt-category}
	In this paper, by a \emph{non-closed tensor-triangulated category} we mean a triangulated category $\cat T$ equipped with a symmetric monoidal structure such that $-\otimes t\colon\cat T\to \cat T$ is an exact functor for each $t \in \cat T$ (with the usual compatibility between the associated suspension isomorphisms; see \cite[Definition~3]{BalmerICM}). We will not assume that $\cat T$ has an internal hom, and even when we assume that $\cat T$ has coproducts, we do not assume that $-\otimes t\colon\cat T\to \cat T$ preserves them. However we will require that the full subcategory $\cat T^d \subseteq \cat T$ of dualizable objects is an essentially small triangulated subcategory and that $(-)^\vee\colon\cat T^d \to (\cat T^d)\op$ preserves exact triangles.
\end{Ter}

\begin{Rem}
	The requirement that $\cat T^d$ is a triangulated subcategory amounts to the assumption that an extension of dualizable objects is again dualizable. If $\cat T$ is idempotent-complete (for example, if it has countable coproducts) then $\cat T^d$ is a thick subcategory.
\end{Rem}

\begin{Rem}
	If $\cat T$ is a non-closed tensor-triangulated category then $\cat T^d$ is a rigid essentially small tensor-triangulated subcategory of $\cat T$.
\end{Rem}

\begin{Rem}
	If $\cat T$ is a rigidly-compactly generated tensor-triangulated category in the usual sense then $\cat T^d=\cat T^c$. Nevertheless, we will often write~$\cat T^d$ when emphasizing dualizability over compactness. The results which follow provide evidence that it is~$\cat T^d$ which is more fundamental to tensor triangular geometry and that $\Spc(\cat T^d)$ is the correct definition of ``the'' Balmer spectrum of $\cat T$.
\end{Rem}

\begin{Exa}
	Let $\cat T$ be a rigidly-compactly generated tensor-triangulated category in the usual sense. Its opposite category $\cat T\op$ is a non-closed tensor-triangulated category and $\cat T^d \cong (\cat T\op)^d$ are equivalent tensor-triangulated categories, hence have the same Balmer spectrum. In fact, since thick ideals of dualizable objects are closed under taking duals, $\cat T$ and $\cat T\op$ have exactly the same thick ideals of dualizable objects, the same Balmer spectrum of dualizable objects, the same universal support for dualizable objects, and so on. For all intents and purposes, $\cat T$ and $\cat T\op$ are just two extensions of ``the same'' rigid tensor-triangulated category that we will sometimes denote by~$\cat K$.
\end{Exa}

\begin{Rem}
	The following definition is the key to reconciling localizing ideals and colocalizing coideals.
\end{Rem}

\begin{Def}
	Let $\cat T$ be a non-closed tensor-triangulated category which has small coproducts. A \emph{localizing $\cat T^d$-submodule} of $\cat T$ is a localizing subcategory $\cat L\subseteq \cat T$ such that ${\cat T^d \otimes \cat L \subseteq \cat L}$.
\end{Def}

\begin{Rem}\label{rem:submodule-generated-by}
	The localizing $\cat T^d$-submodule generated by a collection of objects $\cat E \subseteq \cat T$ coincides with $\Loc\langle \cat E \otimes \cat T^d\rangle$.
\end{Rem}

\begin{Exa}\label{exa:submodules-of-T}
	Let $\cat T$ be a rigidly-compactly generated tt-category (in the usual sense). A localizing subcategory $\cat L$ is a localizing ideal if and only if $\cat L \otimes \cat T \subseteq \cat L$ if and only if $\cat L \otimes \cat T^d \subseteq \cat L$ if and only if $\cat L$ is a $\cat T^d$-submodule of $\cat T$. That is, the localizing ideals of $\cat T$ are precisely the localizing $\cat T^d$-submodules of $\cat T$. The key reason for this is that $\cat T=\Loc\langle \cat T^d \rangle$ is generated by the subcategory of dualizable objects.
\end{Exa}

\begin{Exa}\label{exa:submodules-of-Top}
	Let $\cat T$ be a rigidly-compactly generated tt-category (in the usual sense). A colocalizing subcategory $\cat C$ of $\cat T$ is a colocalizing coideal if and only if $\ihom{\cat T,\cat C}\subseteq \cat C$ if and only if $\ihom{\cat T^d,\cat C}\subseteq \cat C$ if and only if $\cat T^d \otimes \cat C \subseteq \cat C$ if and only if $\cat C$ is a localizing $\cat T^d$-submodule of $\cat T\op$. That is, the colocalizing coideals of~$\cat T$ are precisely the localizing $\cat T^d$-submodules of $\cat T\op$. Thus, the task of classifying localizing ideals and colocalizing coideals is, in both cases, the question of classifying the localizing $\cat T^d$-submodules of $\cat T$ and $\cat T\op$, respectively.
\end{Exa}

\begin{Rem}
	Our next goal is to explain how cosupport is related to support.
	Recall from \cref{prop:Top-is-perfectly-generated} that if $\cat T$ is compactly generated then~$\cat T\op$ is perfectly generated.
\end{Rem}

\begin{Lem}\label{lem:perfectly-generated}
	Let $\cat T$ be a non-closed tensor-triangulated category and let $\cat E \subseteq \cat T^d$ be a set of dualizable objects. If $\cat T$ is perfectly generated by $\cat G \subseteq \cat T$ then $\Loc\langle \cat E \otimes \cat T\rangle$ is perfectly generated by $\cat E \otimes \cat G$.
\end{Lem}

\begin{proof}
	Let $\cat L\coloneqq \Loc\langle \cat E \otimes \cat G\rangle$. Then $\SET{t\in \cat T}{\cat E \otimes t \in \cat L}$ is a localizing subcategory since the objects in $\cat E$ are dualizable. It contains $\cat G$ and hence since $\cat T=\Loc\langle \cat G\rangle$ we conclude that $\cat E \otimes \cat T \subseteq \Loc\langle \cat E \otimes \cat G\rangle$. This establishes~(PG1). Condition (PG2) can be readily checked from the definition. The key point here is that for a dualizable object $d\in \cat T^d$, the functor $d\otimes -\colon\cat T \to \cat T$ has a right adjoint $d^\vee\otimes -\colon\cat T\to \cat T$ which itself preserves coproducts; cf.~\cref{prop:perfect-conservation}.
\end{proof}

\begin{Prop}
	Let $\cat T$ be a perfectly generated non-closed tensor-trian\-gulated category and write $\cat K \coloneqq \cat T^d$. For every Thomason subset $Y \subseteq \Spc(\cat K)$, the localizing subcategory $\Loc\langle \cat K_Y \otimes \cat T\rangle$ is strictly localizing, where $\cat K_Y \coloneqq \SET{ a \in \cat K}{\supp(a) \subseteq Y}$. We write
		\begin{equation}\label{eq:bousfield-localization-triangle}
			\Gamma_Y(t) \to t \to \L_Y(t) \to \Sigma\Gamma_Y(t)
		\end{equation}
	for the associated Bousfield localization triangle.
\end{Prop}

\begin{proof}
	By \cref{lem:perfectly-generated}, the localizing subcategory $\Loc\langle \cat K_Y \otimes \cat T\rangle$ is perfectly generated. It follows from \cite[Proposition~5.2.1]{Krause10} that it is strictly localizing.
\end{proof}

\begin{Rem}
	The subcategory $\Gamma_Y(\cat T)=\Loc\langle\cat K_Y \otimes \cat T\rangle$ of colocal objects is a localizing $\cat T^d$-submodule of $\cat T$. The proof is a standard thick subcategory argument which uses the fact that for a dualizable object $a \in \cat T^d$, the functor $a \otimes -\colon\cat T \to \cat T$ preserves coproducts. Note that we are not assuming that $t\otimes-$ preserve coproducts in general (for an arbitrary $t \in \cat T$). This is related to the fact that $\Gamma_Y(\cat T)$ is not necessarily a localizing ideal of $\cat T$.
\end{Rem}

\begin{Rem}\label{rem:smashing}
	The subcategory of local objects $\L_Y(\cat T) = \Loc\langle \cat K_Y \otimes \cat T\rangle^\perp$ coincides with $\bigcap_{a \in \cat K_Y} \ker(a\otimes -)$ and hence is itself a localizing $\cat T^d$-submodule of $\cat T$. This implies that the Bousfield localization \eqref{eq:bousfield-localization-triangle} is smashing in the sense that $\Gamma_Y$ and~$\L_Y$ preserve coproducts. It follows that
		\begin{equation}\label{eq:smashing}
			\Gamma_Y(a \otimes t) \simeq a\otimes \Gamma_Y(t) \quad\text{ and }\quad \L_Y(a\otimes t) \simeq a \otimes \L_Y(t)
		\end{equation}
	for any $a \in \cat T^d$ and $t \in \cat T$.
\end{Rem}

\begin{Rem}\label{rem:not-generated-by}
	The localizing $\cat T^d$-submodule generated by the thick ideal~$\cat K_Y$ is given by $\Loc\langle \cat K_Y \rangle =\Loc\langle \cat K_Y \otimes \cat T^d\rangle$; cf.~\cref{rem:submodule-generated-by}. The localizing $\cat T^d$-submodule $\Gamma_Y(\cat T)=\Loc\langle \cat K_Y \otimes \cat T\rangle$ is potentially larger.
\end{Rem}

\begin{Lem}\label{lem:inflation-is-lattice-hom}
	The assignment $Y \mapsto \Gamma_Y\cat T$, regarded as a map from the lattice of Thomason subsets of $\Spc(\cat T^d)$ to the lattice of localizing $\cat T^d$-submodules of~$\cat T$, preserves arbitrary joins and finite meets.
\end{Lem}

\begin{proof}
	For notational simplicity, write $\cat K \coloneqq \cat T^d$. It is immediate that the map preserves the greatest and least elements. We first prove that it preserves arbitrary joins. To this end, consider a union $\bigcup_i Y_i$ of Thomason subsets of $\Spc(\cat K)$. The inclusion $\bigvee_i \Gamma_{Y_i} \cat T \subseteq \Gamma_{\bigcup_i Y_i}\cat T$ is trivial. To establish the other inclusion we need to check that $\cat K_{\bigcup_i Y_i} \otimes \cat T \subseteq \bigvee_i \Gamma_{Y_i} \cat T$. Observe that
		\[
			\SET{a \in \cat T^d}{a\otimes\cat T\subseteq \bigvee\nolimits_{\hspace{-0.5ex}i} \Gamma_{Y_i} \cat T}
		\]
	is a thick ideal of $\cat T^d$. It contains $\cat K_{Y_i}$ for each $i$. Hence it contains $\thickt{\bigcup_i \cat K_{Y_i}}$ which --- by the classification of thick ideals of the rigid tt-category $\cat K$ --- coincides with $\cat K_{\bigcup_i Y_i}$. This establishes the desired claim.

	It remains to prove that the function preserves finite meets. That is, 
		\[
			\Gamma_{Y_1}\cat T \cap \Gamma_{Y_2}\cat T = \Gamma_{Y_1 \cap Y_2}\cat T
		\]
	for any pair of Thomason subsets $Y_1,Y_2$. The $\supseteq$ inclusion is immediate. For the $\subseteq$ inclusion we need to establish that 
		\begin{equation}\label{eq:finite-meet}
			 \Loc\langle \cat K_{Y_1}\otimes \cat T\rangle \cap \Loc\langle\cat K_{Y_2}\otimes \cat T\rangle \subseteq \Loc\langle\cat K_{Y_1 \cap Y_2} \otimes \cat T\rangle.
		\end{equation}
	First note that $\cat K_{Y_1 \cap Y_2} = \cat K_{Y_1} \cap \cat K_{Y_2} = \thick\langle\cat K_{Y_1} \otimes \cat K_{Y_2}\rangle$ where the last equality uses the rigidity of $\cat K$. Then consider an object $t$ contained in the left-hand side of~\eqref{eq:finite-meet}. Since $t \in \Loc\langle\cat K_{Y_2}\otimes \cat T\rangle$, we have
		\[
			\Gamma_{Y_1}(t) \in \Gamma_{Y_1}(\Loc\langle\cat K_{Y_2}\otimes \cat T\rangle) \subseteq \Loc\langle\Gamma_{Y_1}(\cat K_{Y_2}\otimes \cat T)\rangle,
		\]
	where the inclusion uses that $\Gamma_{Y_1}$ preserves coproducts (\cref{rem:smashing}). But also $t\simeq \Gamma_{Y_1}(t)$ since $t \in \Loco{\cat K_{Y_1}\otimes \cat T}$; hence
		\[
			t \in \Loc\langle \Gamma_{Y_1}(\cat K_{Y_2}\otimes \cat T)\rangle.
		\]
	Now using \eqref{eq:smashing} and the fact that tensoring with dualizable objects preserves coproducts, we have
		\begin{align}\label{eq:locY1Y2}
			\begin{split}
			 \Loc\langle\Gamma_{Y_1}(\cat K_{Y_2} \otimes \cat T)\rangle = \Loc\langle\cat K_{Y_2} \otimes \Gamma_{Y_1}\cat T\rangle
			&= \Loc\langle\cat K_{Y_2} \otimes \Loc\langle\cat K_{Y_1} \otimes \cat T\rangle\rangle\\
			&= \Loc\langle\cat K_{Y_2} \otimes \cat K_{Y_1} \otimes \cat T\rangle\\
			&= \Loc\langle\cat K_{Y_1 \cap Y_2} \otimes \cat T\rangle
			\end{split}
		\end{align}
	which establishes \eqref{eq:finite-meet}.
\end{proof}

\begin{Lem}\label{lem:properties-of-localization}
	 Let $Y_1, Y_2 \subseteq \Spc(\cat T^d)$ be two Thomason subsets. Then we have:
		 \begin{enumerate}
			\item $\Gamma_{Y_1}\Gamma_{Y_2} = \Gamma_{Y_1 \cap Y_2} = \Gamma_{Y_2}\Gamma_{Y_1}$;
			\item $\L_{Y_1} \L_{Y_2} = \L_{Y_1 \cup Y_2} = \L_{Y_2} \L_{Y_1}$;
			\item $\Gamma_{Y_1} \L_{Y_2} = \L_{Y_2} \Gamma_{Y_1}$.
		 \end{enumerate}
\end{Lem}

\begin{proof}
	First note that if $Y \subseteq Y'$ then $\Gamma_{Y} \simeq \Gamma_{Y} \Gamma_{Y'}$. Hence, $\Gamma_{Y_1 \cap Y_2} \simeq \Gamma_{Y_1 \cap Y_2} \Gamma_{Y_2} \simeq \Gamma_{Y_1 \cap Y_2} \Gamma_{Y_1} \Gamma_{Y_2}$. To prove $(a)$ it suffices to verify that $\Gamma_{Y_1} \Gamma_{Y_2} t \in \Loc\langle\cat K_{Y_1 \cap Y_2} \otimes \cat T\rangle$ so that $\Gamma_{Y_1 \cap Y_2} \Gamma_{Y_1} \Gamma_{Y_2}(t) \simeq \Gamma_{Y_1} \Gamma_{Y_2}(t)$. Note that, by the classification of thick ideals of the rigid tt-category $\cat K$, $\cat K_{Y_1 \cap Y_2} = \cat K_{Y_1} \cap \cat K_{Y_2}$. Thus, we want to show that
		\[
			\Gamma_{Y_1} \Gamma_{Y_2} t \in \Loc\langle\cat K_{Y_1} \otimes \cat K_{Y_2} \otimes \cat T\rangle)
		\]
	for any $t \in \cat T$. Since $\Gamma_{Y_2} t \in \Loc\langle\cat K_{Y_2}\otimes \cat T\rangle$, we have
		\[
			\Gamma_{Y_1}\Gamma_{Y_2} t \in \Gamma_{Y_1}(\Loc\langle \cat K_{Y_2}\otimes \cat T\rangle) \subseteq \Loc\langle\Gamma_{Y_1}(\cat K_{Y_2}\otimes \cat T)\rangle,
		\]
		where the last inclusion uses that $\Gamma_Y$ preserves coproducts (\cref{rem:smashing}). Now using the same \cref{rem:smashing}, we see
	as in \eqref{eq:locY1Y2} above that
		\begin{align*}
			\Loc\langle\Gamma_{Y_1}(\cat K_{Y_2} \otimes \cat T)\rangle 
			= \Loc\langle\cat K_{Y_2} \otimes \cat K_{Y_1} \otimes \cat T\rangle
		\end{align*}
	which is what we wanted to show. This establishes $(a)$.

	For part $(b)$, first observe that if $Y \subseteq Y'$ then $\L_{Y'} \simeq \L_{Y'} \L_Y$. Hence $\L_{Y_1 \cup Y_2} \simeq \L_{Y_1 \cup Y_2} \L_{Y_2} \simeq \L_{Y_1 \cup Y_2} \L_{Y_1} \L_{Y_2}$. The proof will be finished if we can show that
		\[
			\L_{Y_1} \L_{Y_2} t \in \Loc\langle\cat K_{Y_1 \cup Y_2} \otimes \cat T\rangle^\perp.
		\]
	For this, it is enough to show that 
		\begin{equation}\label{eq:deepblah}
			\cat K_{Y_1 \cup Y_2} \otimes \L_{Y_1} \L_{Y_2} t = 0
		\end{equation}
	which, since $\cat K_{Y_1 \cup K_Y} = \thickt{\cat K_{Y_1} \cup \cat K_{Y_2}}$, is equivalent to 
		\[
			{\cat K_{Y_1} \otimes \L_{Y_1} \L_{Y_2} t = 0} \quad\text{ and }\quad \cat K_{Y_2} \otimes \L_{Y_1} \L_{Y_2} t = 0
		\]
	both of which follow from \cref{rem:smashing}.

	Finally, we prove part $(c)$. It follows from $(a)$ that $\Gamma_{Y_1}\L_{Y_2}t$ is $\L_{Y_2}$-local. Hence $\Gamma_{Y_1}\L_{Y_2} t \xra{\sim} \L_{Y_2}\Gamma_{Y_1}\L_{Y_2}t$. On the other hand, it also follows from $(a)$ that $\Gamma_{Y_1}t \to \Gamma_{Y_1}\L_{Y_2}t$ is an $\L_{Y_2}$-equivalence, hence 
		\[
			\Gamma_{Y_1}\L_{Y_2} t \xra{\sim} \L_{Y_2}\Gamma_{Y_1}\L_{Y_2}t \xleftarrow{\sim} \L_{Y_2}\Gamma_{Y_1}t.\qedhere
		\]
\end{proof}

\begin{Def}[Support for perfectly generated categories]\label{def:support-for-perfect}
	Let $\cat T$ be a perfectly generated non-closed tensor-triangulated category whose spectrum $\Spc(\cat T^d)$ is weakly noetherian. Write $\{\cat P\} = Y_1 \cap Y_2^c$ for Thomason subsets $Y_1, Y_2\subseteq \Spc(\cat T^d)$ and define $\GammaP \coloneqq \Gamma_{Y_1}L_{Y_2}$. Armed with \cref{lem:properties-of-localization}, one shows that this definition does not depend on the choice of $Y_1$ and $Y_2$ by following the ideas of \cite[Lemma~7.4]{BalmerFavi11}. The \emph{Balmer--Favi support} of an object $t \in \cat T$ is defined as
		\[
			\Supp_{\cat T}(t) \coloneqq \SET{\cat P \in \Spc(\cat T^d) }{\GammaP(t) \neq 0} \subseteq \Spc(\cat T^d).
		\]
\end{Def}

\begin{Rem}\label{rem:supp-not-supp}
	We warn the reader that it is not necessarily true that $\Supp_{\cat T}(x) = \supp(x)$ for $x \in \cat T^d$. That is, the small support does not necessarily recover the universal support of dualizable objects.
\end{Rem}

\begin{Def}\label{def:support-function-for-submodules}
	Let $\cat T$ be a non-closed tensor-triangulated category which has small coproducts. A function $\sigma \colon \cat T \to \mathcal{P}(X)$ is a \emph{support theory for localizing $\cat T^d$-submodules} if it satifies the following conditions:
		\begin{enumerate}
			\item $\sigma(0) = \emptyset$;
			\item $\sigma(\Sigma t) = \sigma(t)$ for every $t \in \cat T$;
			\item $\sigma(c) \subseteq \sigma(a) \cup \sigma(b)$ for any exact triangle $a \to b \to c \to \Sigma a$ in $\cat T$.
			\item $\sigma(\coprod_i t_i) = \bigcup_i \sigma(t_i)$ for any set of objects $t_i$ in $\cat T$.
			\item $\sigma(x \otimes t) \subseteq \sigma(t)$ for any $x \in \cat T^d$ and $t \in \cat T$.
		\end{enumerate}
	These properties are precisely equivalent to the statement that for every subset $Y \subseteq X$, the subcategory $\SET{ t\in \cat T}{ \sigma(t) \subseteq Y}$ is a localizing $\cat T^d$-submodule of~$\cat T$.
\end{Def}

\begin{Def}\label{def:stratifying-submodules}
	We say that a support theory $(X,\sigma)$ \emph{stratifies} $\cat T$ if the map $\cat L \mapsto \bigcup_{t\in \cat L} \sigma(t)$ is a bijection from the collection of localizing $\cat T^d$-submodules of $\cat T$ to the set of all subsets of $X$.
\end{Def}

\begin{Rem}\label{rem:inverse-of-strat}
	If $(X,\sigma)$ stratifies $\cat T$ then the inverse must necessarily be given by $Y \mapsto \SET{t \in \cat T}{\sigma(t) \subseteq Y}$. It follows that $\sigma$ provides a lattice isomorphism between the lattice of localizing $\cat T^d$-submodules and the lattice of all subsets of $X$. In particular, it preserves joins and meets.
\end{Rem}

\begin{Exa}\label{exa:perf-support}
	If $\Spc(\cat T^d)$ is weakly noetherian, then $\Supp_{\cat T}$ defined in \cref{def:support-for-perfect} is a support function in the sense of \cref{def:support-function-for-submodules}. Axioms (d) and (e) follow from \cref{rem:smashing}.
\end{Exa}

\begin{Def}\label{def:perfectly-stratified}
	We say that a perfectly generated non-closed tt-category~$\cat T$ with~$\Spc(\cat T^d)$ weakly noetherian is \emph{stratified} if it is stratified by $\Supp_{\cat T}$.
\end{Def}

\begin{Exa}\label{exa:supp-for-T}
	Let $\cat T$ be a rigidly-compactly generated tensor-triangulated category in the usual sense and let $\cat K\coloneqq \cat T^d$. The localizing $\cat K$-submodules of $\cat T$ are precisely the localizing ideals of $\cat T$ (\cref{exa:submodules-of-T}). If we apply the construction of \cref{def:support-for-perfect} to $\cat T$, we are just considering for each Thomason subset $Y \subseteq \Spc(\cat K)$, the Bousfield localization of $\cat T$ whose acyclics are
		\[
			\Loc\langle \cat K_Y \otimes \cat T\rangle = \Loco{\cat K_Y} = \Loco{e_Y} = \Loc\langle \cat K_Y\rangle \eqqcolon \cat T_Y.
		\]
	This is the usual idempotent triangle
		\[
			e_Y \otimes t \to t \to f_Y \otimes t \to \Sigma e_Y \otimes t
		\]
	on $\cat T$, and the support is
		\[
			\Supp_{\cat T}(t) = \SET{\cat P \in \Spc(\cat K)}{ \gP \otimes t \neq 0}.
		\]
	That is, we obtain the usual Balmer--Favi support of \cref{def:BF-support-and-cosupport}. Stratification for $\cat T$ in the sense of \cref{def:perfectly-stratified} recovers the notion of stratification studied in \cite{bhs1}. In this example, $\Supp_{\cat T}(x) = \supp(x)$ for $x \in \cat T^d$.
\end{Exa}

\begin{Exa}\label{exa:supp-for-Top}
	Let $\cat T$ be a rigidly-compactly generated tensor-triangulated category in the usual sense. Then $\cat T\op$ is a perfectly generated non-closed tensor-triangulated category with subcategory of dualizable objects $\cat K \coloneqq \cat (\cat T\op)^d \cong \cat T^d$. The localizing $\cat K$-submodules of $\cat T\op$ are precisely the colocalizing coideals of $\cat T$ (\cref{exa:submodules-of-Top}). If we apply the construction of \cref{def:support-for-perfect} to $\cat T\op$, we are then considering a Bousfield localization in $\cat T\op$ for each Thomason subset $Y \subset \Spc(\cat K)$,
		\begin{equation}\label{eq:BLinTop}
			\Gamma_Y\op(t) \to t \to \L_Y\op(t) \to \Sigma \Gamma_Y\op(t),
		\end{equation}
	whose acyclics are $\Loc_{\cat T\op}\langle \cat K_Y \otimes \cat T\op\rangle$. This corresponds to a Bousfield localization in $\cat T$,
		\begin{equation}\label{eq:BLinT}
			\L_Y\op(t) \to t \to \Gamma_Y\op(t) \to \Sigma \L_Y\op(t),
		\end{equation}
	whose \emph{local} objects are 
		\[
			\Loc_{\cat T\op}\langle \cat K_Y \otimes \cat T\op \rangle = \Coloc_{\cat T}\langle\cat K_Y \otimes \cat T\rangle.
		\]
	That is, expressed in $\cat T$, the subcategory of local objects is $\Coloc_{\cat T}\langle \cat K_Y \otimes \cat T\rangle = \ihom{e_Y,\cat T} = (\cat T_Y)^{\perp \perp}$ using \cref{lem:key-observation}. Thus the Bousfield localization~\eqref{eq:BLinT} in~$\cat T$ is the Bousfield localization
		\[
			\ihom{f_Y,t} \to t \to \ihom{e_Y,t} \to \Sigma \ihom{f_Y,t}
		\]
	whose acyclics are $(\cat T_Y)^\perp$. Thus, the original Bousfield localization \eqref{eq:BLinTop} in $\cat T\op$ is given by
		\[
			\Gamma_Y\op(t) = \ihom{e_Y,t}  \qquad\text{ and }\qquad \L_Y\op(t) = \ihom{f_Y,t}
		\]
	and the support function on $\cat T\op$ is given by
		\[
			\Supp_{\cat T\op}(t) \coloneqq \SET{ \cat P \in \cat K}{ \ihom{\gP,t} \neq 0}.
		\]
	In conclusion, we have established that the support (\cref{def:support-for-perfect}) of the opposite category $\cat T\op$ is precisely Balmer--Favi cosupport (\cref{def:BF-support-and-cosupport}):
\end{Exa}

\begin{Thm}\label{thm:cosupp-is-supp}
	Let $\cat T$ be a rigidly-compactly generated tensor-triangulated category with $\Spc(\cat T^d)$ weakly noetherian. Then 
		\[
			\Cosupp_{\cat T}(t) = \Supp_{\cat T\op}(t)
		\]
	for all $t \in \cat T$, the colocalizing coideals of $\cat T$ are precisely the localizing \mbox{$\cat T^d$-submodules} of $\cat T\op$, and $\cat T$ is costratified (\cref{def:costratification}) precisely when~$\cat T\op$ is stratified (\cref{def:perfectly-stratified}).
\end{Thm}

\begin{Rem}
	Our perspective has been to regard a non-closed tt-category~$\cat T$ as a module over its rigid tt-subcategory of dualizable objects $\cat T^d$. There are similarities but also differences between the support theory we have constructed compared to what is achieved in \cite{Stevenson13}. Our setting is both more general and more specialized. When $\cat T$ is a rigidly-compactly generated \mbox{tt-category} acting on a compactly generated triangulated category $\cat S$, Stevenson constructs a notion of support for localizing $\cat T$-submodules of $\cat S$ taking values in $\Spc(\cat T^c)$. In particular, one can let $\cat T$ act on itself, in which case one recovers the Balmer--Favi support theory. On the other hand, Stevenson's theory applies to non-tensor examples like singularity categories (see, e.g., \cite{Stevenson14b}). Our setting is more specialized because we require a monoidal structure, but within the monoidal setting our construction is more general in that it only requires a non-closed perfectly generated tt-category; moreover, it zeros in on the significance of the $\cat T^d$ part of the action. These ideas also provide new examples beyond cosupport, as we now explain:
\end{Rem}

\begin{Exa}
	If $\cat L$ is any strictly localizing ideal of a rigidly-compactly generated tt-category $\cat T$ then the corresponding Bousfield localization $\cat T/\cat L$ inherits the structure of a tensor-triangulated category and the generators of $\cat T$ localize to a set of dualizable generators of $\cat T/\cat L$. In general, these generators are not compact when $n>0$. For example, the unit of the $K(n)$-local category is not compact. Nevertheless, the local category~$\cat T/\cat L$ is well generated hence perfectly generated (see \cite[Theorem 7.2.1]{Krause10}) and $(\cat T/\cat L)^d$ is essentially small (see \mbox{\cite[p.~413]{Mathew16}}). Thus, \cref{def:support-for-perfect} provides a support theory for any Bousfield localization of a rigidly-compactly generated tt-category which lies in the spectrum of dualizable objects $\Spc( (\cat T/\cat L)^d)$ provided this space is weakly noetherian. We will not pursue this class of examples in this work, but it would be an interesting support theory to study, e.g., for the $K(n)$-local category. Some work in this direction has been carried out in \cite{BarthelHeardNaumann20pp}.	
\end{Exa}

\section{Universality}\label{sec:universality}

Our next goal is to establish that our approach to costratification via the cosupport theory defined in \cref{sec:BF-cosupport} is in a certain sense the unique or universal such choice. We obtained analogous results for stratification in \cite[Section 7]{bhs1}. In fact, given the connection between stratification and costratification established in \cref{sec:cosupp-is-supp}, we will proceed by proving a generalized version of that uniqueness result which holds for support theories for perfectly generated categeories.

\begin{Hyp}
	Unless otherwise specified, $\cat T$ will denote a perfectly generated non-closed tensor-triangulated category (\cref{def:non-closed-tt-category}) and support theory will mean the notion of \cref{def:support-function-for-submodules}.
\end{Hyp}

\begin{Lem}\label{lem:sigma-forced}
	Let $\sigma\colon \cat T\rightarrow \mathcal{P}(\Spc(\cat T^d))$ be a support theory (\cref{def:support-function-for-submodules}) which satisfies $\sigma(\Gamma_Y \cat T) \subseteq Y$ and $\sigma(\L_Y \cat T) \subseteq Y^c$ for a Thomason subset $Y \subseteq \Spc(\cat T^d)$. Then 
		\[
			\sigma(\Gamma_Y t) = Y \cap \sigma(t) \quad\text{ and }\quad \sigma(\L_Y t) = Y^c \cap \sigma(t)
		\]
	for any $t \in \cat T$.
\end{Lem}

\begin{proof}
	This follows in a routine manner from the exact triangle~\eqref{eq:bousfield-localization-triangle}.
\end{proof}

\begin{Lem}
	Assume $\Spc(\cat T^d)$ is weakly noetherian. For any Thomason subset $Y \subseteq \Spc(\cat T^d)$ and $t \in \cat T$, we have
		\begin{enumerate}
			\item $\Supp_{\cat T}(\Gamma_Yt) = Y \cap \Supp_{\cat T}(t)$, and
			\item $\Supp_{\cat T}(\L_Yt) = Y^c \cap \Supp_{\cat T}(t)$.
		\end{enumerate}
\end{Lem}
\begin{proof}
	This is a routine exercise using \cref{lem:properties-of-localization} and the exact triangle~\eqref{eq:bousfield-localization-triangle}. Note that in light of \cref{lem:sigma-forced} and \cref{exa:perf-support}, it suffices to establish the $\subseteq$ inclusions in $(a)$ and $(b)$.
\end{proof}

\begin{Prop}\label{prop:unique}
	Let $\cat T$ be a perfectly generated non-closed tensor-triangu\-lated category with $\Spc(\cat T^d)$ weakly noetherian. The notion of support defined in \cref{def:support-for-perfect} is the only assignment of a subset $\sigma(t) \subseteq \Spc(\cat T^d)$ to each object $t \in \cat T$ which can satisfy the following two properties:
		\begin{enumerate}
			\item For every $t \in \cat T$, $\sigma(t)=\emptyset$ implies $t=0$.
			\item For every Thomason subset $Y \subseteq \Spc(\cat T^d)$, $\sigma(\Gamma_Y t)=\sigma(t)\cap Y$ and $\sigma(\L_Y t)=\sigma(t)\cap Y^c$.
		\end{enumerate}
\end{Prop}

\begin{proof}
	If $\{\cat P\}=Y_1 \cap Y_2^c$ then $\Gamma_{Y_1} \L_{Y_2} t = \{\cat P\} \cap \sigma(t)$. Thus $\cat P \in \sigma(t)$ if and only if $\sigma(\Gamma_{Y_1} \L_{Y_2} t) \neq \emptyset$ if and only if $\Gamma_{Y_1} \L_{Y_2} t\neq 0$ if and only if $\cat P \in \Supp_{\cat T}(t)$.
\end{proof}

\begin{Thm}\label{thm:perfectly-generated-uniqueness}
	Let $\cat T$ be a perfectly generated non-closed tensor-triangulated category. 
	Suppose $\cat T$ is stratified (\cref{def:stratifying-submodules}) by a support function $\sigma$ (\cref{def:support-function-for-submodules}) taking values in a spectral space $X$. Suppose $\sigma$ induces a bijection
		\[
			Y \mapsto \sigma(\Gamma_Y \cat T)
		\]
	between the set of Thomason subsets of $\Spc(\cat T^d)$ and the set of Thomason subsets of $X$. Then there is a unique homeomorphism $f\colon X\xra{\sim}\Spc(\cat T^d)$ such that $\sigma(\Gamma_{\supp(a)}\cat T) = f^{-1}(\supp(a))$ for all $a \in \cat T^d$. Moreover, if the space $X$ is weakly noetherian then $\sigma(t) = f^{-1}(\Supp_{\cat T}(t))$ for all $t \in \cat T$.
\end{Thm}

\begin{proof}
	Recall from \cref{rem:inverse-of-strat} that $\sigma$ provides a lattice isomorphism. Together with \cref{lem:inflation-is-lattice-hom} this ensures that the hypothesized bijection $Y \mapsto \sigma(\Gamma_Y \cat T)$ is a lattice isomorphism. Both lattices are coherent frames (see, e.g., \cite{KockPitsch17}) hence by Stone duality \cite{Johnstone82} it corresponds to a homeomorphism $f\colon X \xra{\sim} \Spc(\cat T^d)$. By construction this homeomorphism satisfies $f^{-1}(Y) = \sigma(\Gamma_Y \cat T)$ for every Thomason subset $Y \subseteq \Spc(\cat T^d)$. Since $\sigma$ is a lattice isomorphism, $\Gamma_Y \cat T \cap \L_Y \cat T = 0$ and $\Gamma_Y \cat T \vee \L_Y \cat T=\cat T$ which together imply that $\sigma(\L_Y \cat T)=f^{-1}(Y)^c$. The fact that $f^{-1}(\Supp_{\cat T}(t)) = \sigma(t)$ for each $t \in \cat T$ then follows from \cref{lem:sigma-forced} and \cref{prop:unique}. The uniqueness statement is standard; see \cite[Lemma~3.3]{Balmer05a} for example.
\end{proof}

\begin{Rem}
	Recall from \cref{rem:not-generated-by} that $\Gamma_Y \cat T$ is not necessarily the localizing $\cat T^d$-submodule generated by $\cat T^d_Y$. For example, the localizing $\cat T^d$-submodule generated by a dualizable object $a \in \cat T^d$ is $\Loc\langle a\otimes \cat T^d\rangle$ while $\Gamma_{\supp(a)} \cat T = \Loc\langle a\otimes \cat T\rangle$. This may be helpful in understanding \cref{thm:perfectly-generated-uniqueness}: $\sigma(a)$ is the support of the localizing $\cat T^d$-submodule generated by $a$, while $\sigma(\Gamma_{\supp(a)} \cat T)$ is the support of a potentially larger localizing $\cat T^d$-submodule; cf.~\cref{rem:supp-not-supp}.
\end{Rem}

\begin{Rem}
	The lattice homomorphism $\cat C \mapsto \Loc\langle\cat C\otimes \cat T\rangle$ from thick ideals of~$\cat T^d$ to localizing $\cat T^d$-submodules of $\cat T$ (cf.~\cref{lem:inflation-is-lattice-hom}) is not obviously injective. In situations where it is known to be injective, the hypothesis of \cref{thm:perfectly-generated-uniqueness} simplifies slightly. This will be the case for our two main examples of interest below.
\end{Rem}

\begin{Cor}\label{cor:usual-uniqueness}
	Let $\cat T$ be a rigidly-compactly generated tensor-triangulated category. Suppose $\cat T$ is stratified by a support theory $\sigma$ in a spectral space~$X$. Suppose that~$\sigma(a)$ is Thomason closed for every $a \in \cat T^d$ and that every Thomason closed subset of~$X$ arises in this way. If $X$ is weakly noetherian then there is a homeomorphism $X \cong \Spc(\cat T^d)$ under which $\sigma(t)\cong \Supp_{\cat T}(t)$.
\end{Cor}

\begin{proof}
	Note that $\Gamma_Y\cat T = \Loc\langle\cat K_Y \otimes \cat T\rangle = \Loc\langle\cat K_Y \otimes \Loc\langle\cat T^d\rangle\rangle = \Loc\langle\cat K_Y\rangle$. Hence $\sigma(\Gamma_Y\cat T) = \sigma(\cat K_Y)$. In particular, $\sigma(\Gamma_{\supp(a)}\cat T) =\sigma(a)$ for any ${a \in \cat T^d}$. Similarly, for a Thomason subset $Y=\bigcup_i \supp(a_i)$ we have (invoking \cref{lem:inflation-is-lattice-hom}) 
		\[
			\sigma(\Gamma_Y \cat T)=\sigma(\bigvee_i \Gamma_{\supp(a_i)} \cat T) = \bigcup_i \sigma(\Gamma_{\supp(a_i)}\cat T) = \bigcup_i \sigma(a_i).
	 	\]
	Thus, the hypothesis that $\sigma(a)\subseteq X$ is Thomason for all $a \in \cat T^d$ implies that $\sigma(\Gamma_Y \cat T)$ is also Thomason for each Thomason subset $Y\subseteq X$. Moreover, every Thomason closed of $X$ is hypothesized to be of the form $\sigma(a)$ for some $a \in \cat T^d$. Hence every Thomason subset of $X$ is of the form $\bigcup_i \sigma(a_i) = \bigcup_i \sigma(\Gamma_{\supp(a_i)}\cat T) = \sigma(\bigvee_i \Gamma_{\supp(a_i)}\cat T) = \sigma(\Gamma_{\cup_i \supp(a_i)} \cat T)$. Moreover, the map $Y \mapsto \Gamma_Y \cat T$ is injective, because $\Gamma_Y \cat T \cap \cat T^d =  \Loc\langle\cat K_Y\rangle\cap \cat T^d = \cat K_Y$ by \cite[Lemma 2.2]{Neeman92b}. Then the hypotheses of \cref{thm:perfectly-generated-uniqueness} hold and we are done.
\end{proof}

\begin{Rem}
	\Cref{cor:usual-uniqueness} is essentially the uniqueness theorem of \cite[Theorem~7.6]{bhs1}. We have included the above arguments to show how it is deduced from the primordial uniqueness \cref{thm:perfectly-generated-uniqueness}. The statement in \cite[Theorem~7.6]{bhs1} includes some further conditions equivalent to the hypotheses.
\end{Rem}

\begin{Cor}\label{cor:Top-uniqueness}
	Let $\cat T$ be a rigidly-compactly generated tensor-triangulated category. Suppose $\cat T$ is costratified by a cosupport theory $\mathfrak C$ taking values in a spectral space~$X$. The following are equivalent:
	\begin{enumerate}
		\item	For every Thomason subset $Y\subseteq \Spc(\cat T^d)$, $\mathfrak C(\ihom{e_Y,\cat T})$ is a Thomason subset of $X$, and every Thomason subset of $X$ arises in this way.
		\item	The map $Y \mapsto \mathfrak C(\ihom{e_Y,\cat T})$ is a bijection from the lattice of Thomason subsets of $\Spc(\cat T^d)$ onto the lattice of Thomason subsets of~$X$.
		\item	For every Thomason subset $Y \subseteq \Spc(\cat T^d)$, $\mathfrak C(\Loc\langle\cat K_Y\rangle^\perp)^c$ is a Thomason subset of~$X$, and every Thomason subset of $X$ arises in this way.
		\item	The map $\cat C \mapsto \mathfrak C(\Loc\langle\cat C\rangle^\perp)^c$ is a bijection between the thick ideals of~$\cat T^d$ and the Thomason subsets of $X$.
		\item	There is a unique homeomorphism $f\colon X \xra{\sim} \Spc(\cat T^d)$ such that 
					\[
						f^{-1}(\supp(a)) = \mathfrak C(\ihom{a,\cat T})
					\]
				for each $a \in \cat T^d$.
	\end{enumerate}
	If these conditions hold and the spectral space $X$ is weakly noetherian then under the homeomorphism $X \xra{\sim} \Spc(\cat T^d)$ the cosupport theory $\mathfrak C$ coincides with the Balmer--Favi cosupport $\Cosupp$.
\end{Cor}

\begin{proof}
	We wish to apply \cref{thm:perfectly-generated-uniqueness} to $\cat T\op$. Recall from \cref{exa:supp-for-Top} that $\Gamma_Y \cat T\op = \ihom{e_Y,\cat T}$. Thus, the hypothesis of \cref{thm:perfectly-generated-uniqueness} (applied to~$\cat T\op$) is that $Y \mapsto \mathfrak C(\ihom{e_Y,\cat T})$ is a bijection between the set of Thomason subsets of $\Spc(\cat T^d)$ and the set of Thomason subsets of $X$; that is, statement $(b)$. Note that this map is necessarily injective: If $\mathfrak C(\ihom{e_{Y_1},\cat T})=\mathfrak C(\ihom{e_{Y_2},\cat T})$ then $\hom({e_{Y_1},\cat T}) = \hom({e_{Y_2},\cat T})$ since $\mathfrak C$ costratifies $\cat T$. Hence 
		\[
			\LOCO(e_{Y_1}) = {}^{\perp\perp} \ihom{e_{Y_1},\cat T} = {}^{\perp\perp} \ihom{e_{Y_2},\cat T} = \LOCO(e_{Y_2})
		\]
	so that 
		\[
			\cat K_{Y_1} = \cat T^d \cap \LOCO(e_{Y_1})\\ = \cat T^d \cap \LOCO(e_{Y_2}) = \cat K_{Y_2}
		\]
	and hence $Y_1 = Y_2$. Therefore, statement $(b)$ is equivalent to statement $(a)$.

	Now recall from \cref{rem:inverse-of-strat} that the costratifying cosupport $\mathfrak C$ is a lattice isomorphism; in particular it preserves joins and meets. If $\cat C_1 \coloneqq \Loc\langle\cat K_Y\rangle^\perp = f_Y \otimes \cat T = \ihom{f_Y,\cat T}$ and $\cat C_2\coloneqq \ihom{e_Y,\cat T} = \Loc\langle\cat K_Y\rangle^{\perp \perp}$, we have $\cat C_1 \cap \cat C_2 = 0$. On the other hand, for any $t\in \cat T$, there is an exact triangle $\ihom{f_Y,t} \to t \to \ihom{e_Y,t} \to \Sigma\ihom{f_Y,t}$ which shows that $\cat C_1 \vee \cat C_2=\cat T$. Since $\mathfrak C$ preserves joins and meets, we conclude that $\mathfrak C(\cat C_1)^c = \mathfrak C(\cat C_2)$. Thus statement $(c)$ is equivalent to statement $(a)$, and statement $(d)$ is equivalent to statement $(b)$.

	Bearing in mind \cref{lem:eqlococosupp} which gives $\ihom{a,\cat T}=\ihom{e_{\supp(a)},\cat T}$, \cref{thm:perfectly-generated-uniqueness} establishes that $(b)$ implies $(e)$. Conversely, suppose $(e)$ holds. For an arbitrary Thomason $Y=\bigcup_{i\in I} \supp(a_i)$ we have $\ihom{e_Y,\ihom{a_i,t}}=\ihom{a_i,t}$ which implies that $f^{-1}(Y) \subseteq \mathfrak C(\ihom{e_Y,\cat T})$. Moreover, since $e_Y \in \Loco{{a_i \mid i \in I}}$, we have
		\[
			\ihom{e_Y,t} \in \ihom{\Loco{a_i \mid i \in I},t}\subseteq \Coloco{\ihom{a_i,t} \mid i\in I}
		\]
	by \eqref{eq:[loc,t]} so that $\mathfrak C(\ihom{e_Y,\cat T}) \subseteq f^{-1}(Y)$. Thus, $\mathfrak C(\ihom{e_Y,\cat T})=f^{-1}(Y)$ which establishes that $(a)$ holds.

	Finally note that the last statement (when $X$ is weakly noetherian) is also provided by \cref{thm:perfectly-generated-uniqueness}.
\end{proof}

\begin{Rem}
	As with \cref{cor:usual-uniqueness} (and the original \cite[Theorem~7.6]{bhs1}), the point of \cref{cor:Top-uniqueness} is that if $\cat T$ is costratified by a notion of cosupport in a weakly noetherian spectral space such that the costratification is ``compatible'' with the usual classification of thick ideals of dualizable objects, then it must be achieved by the tensor triangular cosupport theory which lies in the Balmer spectrum.
\end{Rem}

\begin{Cor}[BIK costratification]\label{cor:BIK-costrat}
	Let $\cat T$ be a rigidly-compactly generated tt-category which, in the terminology of \cite{BensonIyengarKrause12} is noetherian and costratified by the action of a graded-noetherian ring $R$. Then the BIK space of cosupports $\cosupp_R(\cat T)$ is homeomorphic to $\Spc(\cat T^c)$ and the BIK notion of cosupport in $\cosupp_R(\cat T)$ coincides with the notion of cosupport in $\Spc(\cat T^c)$ defined in \cref{def:BF-support-and-cosupport}.
\end{Cor}

\begin{proof}
	We first note that because $\cat T$ is costratified by $R$, it is also stratified by $R$, see \cite[Theorem~9.7]{BensonIyengarKrause12}. Moreover, $\cosupp_R(\cat T)=\supp_R(\cat T)$ and for any localizing ideal $\cat L$, $\cosupp_R(\cat L^\perp) = \supp_R(\cat L)^c$ (see the proof of \cite[Corollary~9.9]{BensonIyengarKrause12}). Invoking \cite[Theorem~6.1]{BensonIyengarKrause11b} we see that $\cat C \mapsto \cosupp_R(\Loc\langle\cat C\rangle^\perp)^c$ is a bijection between the thick ideals and the Thomason subsets of $\supp_R \cat T$. Moreover, by \cite[Theorem~5.5]{BensonIyengarKrause08}, $\supp_R(\cat T)=\Supp_{R}(\unit)$ is a closed subspace of $\Spec(R)$, hence is itself a noetherian spectral space. Hence, we can invoke \cref{cor:Top-uniqueness} with hypothesis~$(d)$. Alternatively, we could invoke \cite[Corollary~7.11]{bhs1} to conclude that $\Spc(\cat T^c)\cong \supp_R(\cat T)\cong \cosupp_R(\cat T)$ and invoke \cref{prop:unique} noting that the BIK theory of cosupport satisfies conditions $(a)$ and $(b)$ (see in particular \cite[Theorem~4.5 and Proposition~4.7]{BensonIyengarKrause12}).
\end{proof}

\begin{Rem}\label{rem:axiom-differences}
	The axioms of a support function for localizing submodules given in \cref{def:support-function-for-submodules} look slightly weaker than the axioms  for a support function for localizing ideals given in \cite[Definition 7.1]{bhs1}. The former's axiom (e) states that $\sigma(x\otimes t)\subseteq \sigma(t)$ for any $x \in \cat T^d$ and $t\in \cat T$ while the latter's axiom (e) states that $\sigma(t_1 \otimes t_2) \subseteq \sigma(t_1) \cap \sigma(t_2)$ for all $t_1,t_2 \in \cat T$. For~$\cat T$ rigidly-compactly generated, this stronger axiom is forced by the axioms of \cref{def:support-function-for-submodules}. Indeed, for any $t \in \cat T$, $\SET{s\in \cat T}{{\sigma(s\otimes t)\subseteq \sigma(t)}}$ is a localizing subcategory which contains the dualizable objects. Similarly, axiom (e) of \cref{def:axiomaticcosupp} looks slightly stronger than what is provided by \cref{def:support-function-for-submodules} when applied to $\cat T\op$, namely it requires $\mathfrak C(\hom(s,t)) \subseteq \mathfrak C(t)$ for all $s,t \in \cat T$. Again it is implied by the \emph{a priori} weaker axioms by considering $\SET{s \in \cat T}{\mathfrak C(\ihom{s,t})\subseteq \mathfrak C(t)}$. On the other hand, \cref{def:support-function-for-submodules} does not include the second part of axiom (a) of \cref{def:axiomaticcosupp}.
\end{Rem}

\begin{Rem}\label{rem:comp-veras}
	The reader may find it interesting to compare our axioms for cosupport theory with the axioms of \cite[Definition~3.2]{Verasdanis22bpp} which explicitly involve Brown--Comenetz duals (which will make an appearance in our work in the next section). Our axioms are \emph{a priori} weaker. In particular, a cosupport datum in the sense of \cite[Definition~3.2]{Verasdanis22bpp} is a cosupport function in the sense of \cref{def:axiomaticcosupp}. We claim that our formulation of axiom (e) is morally ``correct'' due to \cref{rem:axioms-equiv}; also cf.~\cref{def:support-function-for-submodules}(e). In any case, it follows from \cref{prop:cosupp-BC-duals} below that the Balmer--Favi cosupport function satisfies the stronger axioms stated in \cite[Definition~3.2]{Verasdanis22bpp}. Hence, by \cref{cor:Top-uniqueness} any cosupport theory which costratifies the colocalizing coideals in a weakly noetherian space (in the sense of the corollary) satisfies the stronger axioms of \cite[Definition~3.2]{Verasdanis22bpp}.
\end{Rem}

\section{Duality}\label{sec:duality}

\begin{Def}\label{def:dualizes}
	Let $\kappa$ be an object in a closed symmetric monoidal category~$\cat T$. We say that $\kappa$ \emph{dualizes} a full subcategory $\cat T_0 \subseteq \cat T$ if $\ihom{-,\kappa}$ restricts to an equivalence $\cat T_0 \xra{\simeq} \cat T_0\op$.
\end{Def}

\begin{Rem}\label{rem:Deltakappa}
	Writing  $\Delta_\kappa(x) \coloneqq \ihom{x,\kappa}$, observe that the evaluation morphism $\Delta_\kappa(x)\otimes x \to \kappa$ is adjoint to a morphism $\phi_x \colon x\to \Delta_\kappa\Delta_\kappa(x)$. Dinaturality of evaluation implies that $\phi_x$ is natural in $x$. Unwinding the definition, we see that $\kappa$ dualizes $\cat T_0$ if and only if $\Delta_{\kappa}$ preserves $\cat T_0$ and $x \to \Delta_{\kappa}\Delta_\kappa x$ is an isomorphism for all $x \in \cat T_0$. 
\end{Rem}

\begin{Exa}
	The monoidal unit $\unit$ dualizes the full subcategory $\cat T^d \subseteq \cat T$ of dualizable objects and also the (potentially larger) full subcategory of reflexive objects of $\cat T$, i.e., those $x\in \cat T$ for which the canonical map $x \to (x^\vee)^\vee=\Delta_\unit\Delta_\unit(x)$ is an isomorphism. More generally, any $\otimes$-invertible object~$\kappa$ dualizes precisely the same subcategories as the monoidal unit~$\unit$.
\end{Exa}

\begin{Prop}\label{prop:all-objects-dualize}
	Let $\cat T$ be a triangulated category with a compatible closed symmetric monoidal structure (as in \cite[Appendix A]{HoveyPalmieriStrickland97}). Let $\kappa \in \cat T$ be an arbitrary object. It dualizes the full subcategory
		\[
			\cat T_\kappa \coloneqq \SET{x \in \cat T}{ \phi_x \colon x\to \Delta_\kappa\Delta_\kappa(x) \text{ is an equivalence}}.
		\]
	The full subcategory $\cat T_\kappa$ is a thick $\cat T^d$-submodule of $\cat T$ (that is, a thick subcategory satisfying $\cat T^d \otimes \cat T_\kappa \subseteq \cat T_\kappa$). It contains any full subcategory of $\cat T$ dualized by $\kappa$.
\end{Prop}

\begin{proof}
	The functor $\Delta_\kappa \coloneqq\ihom{-,\kappa}$ is adjoint to itself
		\begin{equation}\label{eq:delta-adj}
			\Delta_\kappa : \cat T \adjto \cat T\op : \Delta_\kappa
		\end{equation}
	with the map $\phi_x$ of \cref{rem:Deltakappa} serving as both the unit and the counit of the adjunction. The fact that $\kappa$ dualizes the subcategory $\cat T_{\kappa}$ is then just a manifestation of the fact that any adjunction induces an equivalence between the full subcategory of objects on which the unit is an isomorphism and the full subcategory of objects on which the counit is an isomorphism. Since~\eqref{eq:delta-adj} is a triangulated adjunction,~$\cat T_\kappa$ is a thick subcategory of $\cat T$. To see that it is a $\cat T^d$-submodule just note that for any~$x \in \cat T$ and $y \in \cat T^d$, we have
		\[
			\Delta_\kappa\Delta_\kappa(x) \otimes y \simeq \Delta_\kappa\Delta_\kappa(x \otimes y)
		\]
	and $\phi_x \otimes \text{id}_y \simeq \phi_{x \otimes y}$. The final statement is clear from \cref{rem:Deltakappa}.
\end{proof}

\begin{Rem}
	The fact that \emph{every} object $\kappa \in \cat T$ dualizes a certain canonical thick $\cat T^d$-submodule of $\cat T$ seems not to have been previously noticed (or appreciated) and leads to new approaches to organizing the objects of a tt-category. For example, the Picard group of a rigid tt-category $\cat K$ describes precisely those objects which dualize $\cat K$ itself. A more in-depth study of this notion will be the subject of future work. Our present focus is on how this connects with cosupport.
\end{Rem}

\begin{Exa}[Brown--Comenetz duality]\label{exa:BC-duality}
	Recall that in \cref{def:compact-BC} we constructed an object $I_c$ for each compact object $c$ in a compactly generated triangulated category $\cat T$. Now suppose $\cat T$ is a rigidly-compactly generated \mbox{tt-category}. It follows from \eqref{eq:BC-homs} that $I_c = c \otimes I_\unit$ for any compact object~${c \in \cat T^c}$. We define the \emph{Brown--Comenetz dual} of an arbitrary object~${t \in \cat T}$, as
		\[
			\BCdual{t} \coloneqq \ihom{t,I_\unit}.
		\]
	This object represents the functor
		\[
			\Hom_\bbZ(\cat T(\unit,t\otimes -),\bbQ/\bbZ)\colon\cat T\op \to \Ab.
		\]
	Note that for a compact object $c \in \cat T^c$, we have $\BCdual{c} = I_{c^{\vee}}$. In particular, $\BCdual{\unit}=I_\unit$. This observation, together with \eqref{eq:BC-homs}, implies that an object $t=0$ if and only if~$\BCdual{t}=0$.
\end{Exa}

\begin{Exa}
	When $\cat T=\SH$ is the category of spectra, the Brown--Comenetz dual of the sphere $I_\unit$ dualizes the full subcategory of spectra whose homotopy groups are finite; see \cite{BrownComenetz76}.
\end{Exa}

\begin{Prop}\label{prop:cosupp-BC-duals}
	For any compact object $c \in \cat T^c$, we have
		\[
			\Cosupp(I_c) = \supp(c).
		\]
	For an arbitrary object $t \in \cat T$, we have
		\[
			\Cosupp(\BCdual{t})= \Supp(t).
		\]
	If $t \in \cat T_{I_{\unit}}$, that is, $t \xra{\sim} t^{**}$ is an isomorphism, then
		\[
			\Cosupp(t) = \Supp(t^*).
		\]
\end{Prop}

\begin{proof}
	Let $c \in \cat T^c$. Then observe that
		\begin{align*}
			\ihom{\gP,I_c} = 0 &\Longleftrightarrow \cat T(d\otimes \gP,I_c)=0 \;\text{ for all } d\in \cat T^c\\
			&\Longleftrightarrow \cat T(c,d\otimes \gP) =0 \;\text{ for all }  d\in \cat T^c\\
			&\Longleftrightarrow \cat T(d^\vee,c^\vee \otimes \gP) = 0 \;\text{ for all }  d \in \cat T^c\\
			&\Longleftrightarrow c^\vee \otimes \gP = 0.
		\end{align*}
	Here, the second equivalence relies on the discussion in \Cref{rem:Q/Z-cogenerator}. Thus,
		\[
			\Cosupp(I_c)=\Supp(c^\vee)=\supp(c^\vee)=\supp(c).
		\]
	In particular, $\Cosupp(I_\unit)=\Cosupp(\BCdual{\unit})=\Spc(\cat T^c)$. For the second statement, \cref{lem:cosupp_hom} provides the inclusion
		\[
			\Cosupp(t^*)=\Cosupp(\ihom{t,\unit^*}) \subseteq \Supp(t) \cap \Cosupp(\unit^*) = \Supp(t).
		\]
	On the other hand, suppose $\cat P\in \Supp(t)$, so that $\gP \otimes t\neq 0$. Hence since the $I_c$ cogenerate we have $\cat T(\gP\otimes t,I_c)\neq 0$ for some compact $c$. Using $I_c \simeq c \otimes I_{\unit}$, we see that $\cat T(c^{\vee},\ihom{\gP,\ihom{t,I_\unit}})\neq 0$ so that $\cat P \in \Cosupp(t^*)$.
\end{proof}

\begin{Exa}[Dualizing complexes]\label{exa:dualizing-complex}
	A \emph{dualizing complex}\footnote{The classical literature sometimes also requires a ``dualizing complex'' to have finite injective dimension. Neeman \cite{Neeman10} emphasized that this should not be included as part of the definition.} for a separated noetherian scheme $X$ is an object $t\in \DbcohX$ which dualizes $\DbcohX$.
\end{Exa}

\begin{Def}
	We say that an object $t\in \cat T$ has \emph{small cosupport} if $\Cosupp(t) \subseteq \Supp(t)$.
\end{Def}

\begin{Exa}
	Compact objects have small cosupport by \cref{exa:cosupp-compact}.
\end{Exa}

\begin{Rem}
	If $\Supp(t)$ is Thomason and the codetection property holds then~$t$ has small cosupport if and only if the canonical map $t \to\ihom{e_{\Supp(t)},t}$ is an isomorphism. In this way, having small cosupport is related to derived notions of being adically complete.
\end{Rem}

\begin{Exa}\label{exa:coherent-have-small-cosupport}
	If $X=\Spec(R)$ is a noetherian affine scheme then any $t \in \DbcohX$ has small cosupport. Indeed, the Balmer--Favi support $\Supp(t)\subseteq \Spec(R)$ coincides with the ordinary cohomological support $\bigcup_{i\in \mathbb Z} \Supp(H^i(t))$ which is just a finite union of closed sets, hence is itself closed. Denoting this closed set by $V(I)$, the complex $t$ is cohomologically $I$-adically complete by \cite[Theorem 1.21]{PortaShaulYekutieli15} (see also \cite[Theorem 6.7]{SatherWagstaffWicklein17} and \cite[Proposition 4.19]{BensonIyengarKrause12}) meaning that the canonical map $t \to \ihom{e_{V(I)},t}$ is an isomorphism. Hence $\Cosupp(t) \subseteq V(I) = \Supp(t)$.
\end{Exa}

\begin{Prop}\label{prop:cosupp-for-dualizing}
	Let $\cat T$ be stratified and suppose $\kappa \in \cat T$ dualizes the subcategory~$\cat T_0 \subseteq \cat T$. For any $t \in \cat T_0$, we have
		\[
			\Cosupp(t) = \Supp(\Delta_\kappa(t)) \cap \Cosupp(\kappa).
		\]
	In particular, if $\unit \in \cat T_0$ (that is, if $\cat T^d \subseteq \cat T_0$) then
		\[
			\Cosupp(\unit) = \Supp(\kappa) \cap \Cosupp(\kappa).
		\]
	If the objects of $\cat T_0$ have small cosupport, then
		\[
			\Cosupp(t) = \Supp(t) \cap \Cosupp(\kappa)
		\]
	for each $t \in \cat T_0$.
\end{Prop}

\begin{proof}
	The first statement follows directly from the isomorphism $t\simeq \Delta_\kappa \Delta_\kappa(t)$ and \cref{thm:strat_cosupport}. The second statement is an immediate consequence of the first statement. It follows that if $t \in \cat T_0$ has small cosupport then $\Cosupp(t) \subseteq \Supp(t) \cap \Cosupp(\kappa)$. On the other hand, if $\Delta_\kappa(t) \in \cat T_0$ also has small cosupport, then $\Supp(t) \cap \Cosupp(\kappa)=\Cosupp(\Delta_\kappa(t)) \subseteq \Supp(\Delta_\kappa(t))$ and it follows that ${\Supp(t) \cap \Cosupp(\kappa)} \subseteq \Supp(\Delta_\kappa(t))\cap \Cosupp(\kappa) = \Cosupp(t)$.
\end{proof}

\begin{Rem}
	In general, we know that $\Cosupp(x) =\Supp(x) \cap \Cosupp(\unit)$ for any compact object $x \in \cat T^c$. The point of \cref{prop:cosupp-for-dualizing} is that it gives us relations between the support and cosupport for bigger (not necessarily compact) objects. For example:
\end{Rem}

\begin{Cor}\label{cor:scheme-dualizing-complex}
	Let $X$ be a noetherian affine scheme which admits a dualizing complex. Then
		\[
			\Cosupp(t) = \Supp(t) \cap \Cosupp(\unit).
		\]
	for any $t \in \DbcohX$.
\end{Cor}

\begin{proof}
	Let $\kappa \in \DbcohX$ be a dualizing complex for $X$ (\cref{exa:dualizing-complex}). We apply \cref{prop:cosupp-for-dualizing}. Since bounded complexes of coherent sheaves have small cosupport (\cref{exa:coherent-have-small-cosupport}), we obtain $\Cosupp(t)=\Supp(t)\cap \Cosupp(\kappa)$ for any $t \in \DbcohX$. Note that $\DbcohX \supseteq \Derqc(X)^c$ contains~$\unit$. Moreover, a dualizing complex is (by definition) itself contained in $\DbcohX$, hence itself has small cosupport. It follows that $\Cosupp(\unit)=\Cosupp(\kappa)$ and we are done.
\end{proof}

\begin{Rem}
	Many schemes admit dualizing complexes; for example, complete noetherian local rings, Dedekind domains, and any scheme of finite type over a field (see, e.g., \cite[Section 10]{Hartshorne66}). \Cref{prop:cosupp-for-dualizing} and \cref{cor:scheme-dualizing-complex} provide a general perspective on (and generalization of) results such as \cite[Proposition~4.18]{BensonIyengarKrause12}; see also \cite[Theorem~1.2]{SatherWagstaffWicklein17}. On the other hand, not every scheme admits a dualizing complex (see, e.g., \cite{Sharp79,Kawasaki02}).
\end{Rem}

\begin{Exa}[Matlis duality]
	Let $(R,\mathfrak m,k)$ be a noetherian local ring. Matlis duality implies that the injective hull $E(k)$ dualizes the full subcategory $\thick\langle k\rangle \subset \Der(R)\eqqcolon\cat T$. Indeed, this duality can be ``lifted'' from the ordinary duality on $\Der(k)$ along the map $R \to k$; see \cite[Example~7.2]{BalmerDellAmbrogioSanders16}. Recall that $\thick\langle k\rangle$ consists of those complexes of $R$-modules which have only finitely many nonzero homology modules, each of which is a module of finite length (see, e.g., \cite{DwyerGreenleesIyengar06}). The objects of $\thick\langle k \rangle$ have small cosupport since a module of finite length is $\mathfrak m$-adically complete. Consider $\thick\langle k\rangle \subseteq \cat T_{E(k)}$. Since $\ihom{E(k),E(k)}$ is isomorphic to the $\mathfrak m$-adic completion of $R$, we see that $\unit \in \cat T_{E(k)}$ if and only if $R$ is complete, while $\unit \in \thick\langle k\rangle$ if and only if $R$ is artinian (and hence complete). Note that the dualizing object $E(k)$ is typically not contained in the subcategory $\thick\langle k\rangle$ which it dualizes. For example, if $R=\mathbb{Z}_{(p)}$ then $E(k)=\mathbb{Z}(p^\infty)$ which is artinian but not noetherian. Note that if $R$ is regular (or, more generally, Gorenstein) of dimension $d$, then $E(k)=\Sigma^d e_{\{\mathfrak m\}}$, so in light of \cref{cor:cosupp-spc-disjoint}, $E(k)$ cannot have small cosupport unless $\dim(R)=0$, that is, unless $R$ is artinian (in which case $E(k) \in \thick\langle k\rangle$).

\end{Exa}

\newpage
\part{Morphisms and descent}\label{part:morphisms-and-descent}

We now return to the world of rigidly-compactly generated tt-categories and study base change results for support and cosupport, as well as descent techniques for establishing stratification and costratification.

\section{The image of a geometric functor}\label{sec:image}

We begin with a general observation about the image of the map on Balmer spectra induced by a geometric functor. This seems to have been overlooked in the literature; it gives further evidence for the distinguished role played by the Balmer--Favi notion of support.

\begin{Ter}
	A coproduct-preserving tt-functor $f^*\colon\cat T \to \cat S$ between rigidly-compactly generated tt-categories is called a \emph{geometric functor}. As explained in \cite{BalmerDellAmbrogioSanders16}, such a functor admits a right adjoint $f_*$ which itself admits a right adjoint~$f^!$. The object $\omega_f \coloneqq f^!(\unitT) \in \cat S$ is called the relative dualizing object for $f^*$.
\end{Ter}

\begin{Hyp}\label{hyp:geometric}
	Throughout this section $f^*\colon\cat T \to \cat S$ will denote a geometric functor between rigidly-compactly generated tt-categories and
		\[
			\varphi\colon\Spc(\cat S^c) \to \Spc(\cat T^c)
		\]
	will denote the induced map on Balmer spectra. We will further assume throughout that $\Spc(\cat S^c)$ and $\Spc(\cat T^c)$ are weakly noetherian.
\end{Hyp}

\begin{Rem}\label{rem:preimageofloco}
	The adjoints $f^* \dashv f_* \dashv f^!$ of a geometric functor $f^*$ are related by a number of useful formulas, spelled out in \cite[Proposition 2.15]{BalmerDellAmbrogioSanders16}. In particular, the projection formula
		\[
			f_*(f^*(t)\otimes s)\cong t\otimes f_*(s)
		\]
	holds for any $s \in \cat S$ and $t \in \cat T$. It follows, using the fact that $f_*$ preserves coproducts, that
		\begin{equation}\label{eq:f_*-inclusion}
			f_* \Loco{f^*(\cat E)} \subseteq \Loco{\cat E}
		\end{equation}
	for any collection of objects $\cat E \subseteq \cat T$.
\end{Rem}

\begin{Rem}
	While the kernel of $f^*$ is a localizing ideal, the kernel of $f^!$ is a colocalizing coideal. This follows from another useful formula:
		\begin{equation}\label{eq:f!hom}
			f^!\ihom{t_1,t_2}\cong\ihom{f^* t_1,f^! t_2}
		\end{equation}
	for any $t_1, t_2 \in \cat T$; see \cite[(2.19)]{BalmerDellAmbrogioSanders16}.
\end{Rem}

\begin{Rem}\label{rem:fiber-is-support}
	For any Thomason subset $Y \subseteq \Spc(\cat T^c)$, it is established in \cite[Proposition 5.11]{BalmerSanders17} that $f^*(e_Y)=e_{\varphi^{-1}(Y)}$ and $f^*(f_Y) = f_{\varphi^{-1}(Y)}$. In particular, if $\cat P \in \Spc(\cat T^c)$, then writing $\singP = Y_1 \cap Y_2^c$ for Thomason subsets $Y_1,Y_2$, we see that the preimage $\varphi^{-1}(\singP)=\varphi^{-1}(Y_1)\cap\varphi^{-1}(Y_2)^c$ is a weakly visible subset, and $f^*(\gP) = e_{\varphi^{-1}(Y_1)}\otimes f_{\varphi^{-1}(Y_2)} = g_{\varphi^{-1}(\singP)}$. Hence
	    \begin{equation}\label{eq:supportbasechangegp}
			\Supp_{\cat S}(f^*(\gP)) = \varphi^{-1}(\{\cat P \}).
	    \end{equation}
    Combined with \cite[Lemma 2.13]{bhs1}, this generalizes to  
	    \[
	    \Supp_{\cat S}(f^*(\gP)\otimes s)  = \varphi^{-1}(\singP) \cap \Supp_{\cat S}(s)
	    \]
	for any $s \in \cat S$. These observations will be used repeatedly in the proofs. 
\end{Rem}

\begin{Rem}\label{rem:conservative-on-adjoint}
	Note that the unit-counit equation
		\[\begin{tikzcd}[column sep=scriptsize]
			f^*(t) \ar[rr,bend left=13,"\id"] \ar[r]& f^*f_*f^*(t) \ar[r]& f^*(t) 
		\end{tikzcd}\]
	implies that a right adjoint $f_*$ is conservative on the essential image of its left adjoint $f^*$. The other unit-counit equation shows that, similarly, a left adjoint is conservative on the essential image of its right adjoint.
\end{Rem}

\begin{Def}
	A \emph{weak ring object} in a symmetric monoidal category is an object~$w$ which admits a map $\eta\colon\unit \to w$ such that $w\otimes \eta\colon w\to w\otimes w$ is split monic. Similarly, a \emph{weak coring object} is an object $c$ which admits a map $\epsilon\colon c\to \unit$ such that $c\otimes\epsilon\colon c\otimes c\to c$ is split epi. Note that these notions are left-right agnostic.
\end{Def}

\begin{Exa}
	The left idempotents $e_Y$ are weak corings and the right idempotents~$f_Y$ are weak rings. Indeed, these examples are idempotent (co)rings. More generally, the left and right idempotents of a (not-necessarily-finite) smashing localization are idempotent (co)rings.
\end{Exa}

\begin{Rem}\label{rem:conservative-on-weak-rings}
	Let $\cat A$ be a symmetric monoidal additive category in which the tensor product functors $a \otimes -\colon\cat A\to\cat A$ are additive functors. Let $F\colon\cat A \to \cat B$ be any additive functor. If $F$ admits a left adjoint $L\colon\cat B \to \cat A$ then $F$ is conservative on objects of the form $L(b)\otimes w$ where $b\in \cat B$ is an arbitrary object and $w \in \cat A$ is a weak ring object. Indeed, if $F(L(b)\otimes w)=0$ then
		\[
			\cat A(L(b),L(b)\otimes w)=\cat B(b,F(L(b)\otimes w))=0.
		\]
	Thus, if $\eta\colon\unit \to w$ is a unit for the weak ring $w$ then the map $L(b)\xra{1\otimes \eta}L(b)\otimes w$ vanishes. This implies $L(b)\otimes w=0$ since the identity of $L(b)\otimes w$ factors  as 
		\[
			L(b)\otimes w \xra{1\otimes \eta \otimes 1}L(b)\otimes w\otimes w \to L(b)\otimes w
		\]
	and so must itself vanish. Similarly, if $F\colon\cat A\to\cat B$ admits a right adjoint $R\colon\cat B\to \cat A$ then $F$ is conservative on objects of the form $R(b)\otimes c$, where $b$ is arbitrary and $c$ is a weak coring.
\end{Rem}

\begin{Thm}\label{thm:img-of-spc}
	Let $f^*\colon\cat T \to \cat S$ be a geometric functor as in \cref{hyp:geometric}. Then for any weak ring object $w \in \cat S$,
		\[	
			\varphi(\Supp_{\cat S}(w))=\Supp_{\cat T}(f_*(w))
		\]
	provided either that
		\begin{enumerate}
			\item the detection property holds for $\cat S$; or
			\item the weak ring object $w$ is compact.
		\end{enumerate}
\end{Thm}

\begin{proof}
	Let $\cat P \in \Spc(\cat T^c)$ and recall from \cref{rem:fiber-is-support} that $\Supp(f^*(\gP))=\varphi^{-1}(\singP)$. For any object $w \in \cat S$ consider the following implications:
		\begin{equation}\label{eq:im-phi-imp1}
		\begin{aligned}
			\cat P \in \Supp_{\cat T}(f_*(w)) 
			\Longleftrightarrow&\; \gP\otimes f_*(w) \neq 0 \\
			\Longleftrightarrow&\; f_*(f^*(\gP)\otimes w)\neq 0\\
			\underbrace{\Longrightarrow}_{(\dagger)}&\; f^*(\gP)\otimes w \neq 0\\
			\underbrace{\Longleftarrow}_{(\ddagger)}&\; \Supp_{\cat S}(f^*(\gP)\otimes w) \neq \emptyset\\
			\Longleftrightarrow&\;\varphi^{-1}(\singP) \cap \Supp_{\cat S}(w) \neq \emptyset\\
			\Longleftrightarrow&\; \cat P \in \varphi(\Supp_{\cat S}(w)).
		\end{aligned}
		\end{equation}
	Note that the converse of the $(\ddagger)$ implication holds if $\cat S$ has the detection property. It also holds if $w$ is compact. Indeed, writing $\singP = Y_1 \cap Y_2^c$ and setting $Y_1' \coloneqq \varphi^{-1}(Y_1)$ and $Y_2' \coloneqq \varphi^{-1}(Y_2)$ observe that for any compact $x \in \cat S^c$, we have
		\begin{align*}
			\Supp(f^*(\gP) \otimes x) = \emptyset
			&\Longleftrightarrow \supp(x)\cap Y_1'  \subseteq Y_2' \\
			&\Longleftrightarrow e_{\supp(x)\cap Y_1'} \otimes f_{Y_2'} =0 \\
			&\Longleftrightarrow e_{\supp(x)}\otimes e_{Y_1'} \otimes f_{Y_2'} =0 \\
			&\Longleftrightarrow \Loco{e_{\supp(x)}} \otimes e_{Y_1'} \otimes f_{Y_2'} = 0\\
			&\Longleftrightarrow \Loco{x} \otimes e_{Y_1'} \otimes f_{Y_2'} = 0\\
			&\Longleftrightarrow x \otimes e_{Y_1'} \otimes f_{Y_2'} = 0.
		\end{align*}
	Here we have used \cite[Lemma 1.27]{bhs1} several times. In summary, we have established that $\Supp_{\cat T}(f_*(w)) \subseteq \varphi(\Supp_{\cat S}(w))$ provided the detection property holds or~$w$ is compact. It remains to prove that the converse of the~$(\dagger)$ implication holds under our assumption that $w$ is a weak ring. This was explained in \cref{rem:conservative-on-weak-rings}: the right adjoint $f_*$ is conservative on objects of the form $f^*(t)\otimes w$ with $w$ a weak ring.
\end{proof}

\begin{Cor}\label{cor:img-of-spc}
	We always have the equality
		\[ 
			\im \varphi = \Supp_{\cat T}(f_*(\unitS)).
		\]
\end{Cor}

\begin{proof}
	Apply \cref{thm:img-of-spc} to the ring object $w\coloneqq \unitS$ and note that $\Supp_{\cat S}(\unitS)=\Spc(\cat S^c)$ under our weakly noetherian assumption, e.g., by \cite[Lemma 2.18]{bhs1} or \cref{rem:W-intersection}.
\end{proof}

\begin{Rem}
	This result provides an unconditional formula for the image of the map on spectra induced by a geometric functor. It improves on \cite[Theorem~1.7]{Balmer18} which requires the right adjoint $f_*$ to preserve compactness. An analogous result for the homological spectrum is \cite[Theorem~5.12]{Balmer20_bigsupport}.
\end{Rem}

\begin{Exa}\label{exa:image-of-smashing}
	If $\cat T \to \cat S$ is a smashing localization with idempotent triangle $e \to \unitT \to f \to \Sigma e$ in $\cat T$, then the image of $\varphi\colon\Spc(\cat S^c)\to\Spc(\cat T^c)$ is $\Supp(f)$.
\end{Exa}

\begin{Exa}
    Let $\cat C$ be a rigidly-compactly generated symmetric monoidal stable $\infty$-category with associated homotopy category $\Ho(\cat C)$. If $\Ho(\cat C) \to \Ho(\Mod_{\cat C}(A))$ is extension-of-scalars with respect to a highly structured commutative algebra $A \in \CAlg(\cat C)$, then the image of $\Spc(\Ho(\Mod_{\cat C}(A))^c) \to \Spc(\Ho(\cat C)^c)$ is $\Supp(A)$.
\end{Exa}

\begin{Cor}\label{cor:surjective-for-conservative}
	If $f^*\colon\cat T\to \cat S$ is conservative then $\varphi\colon\Spc(\cat S^c)\to\Spc(\cat T^c)$ is surjective.
\end{Cor}

\begin{proof}
	Indeed, $0\neq \gP$ implies $0 \neq f^*(\gP)$ and thus $0 \neq f_*f^*(\gP)\simeq {f_*(\unitS)\otimes \gP}$ by \cref{rem:conservative-on-adjoint}. Therefore $\Supp(f_*(\unitS))=\Spc(\cat T^c)$ and we invoke \cref{cor:img-of-spc}.
\end{proof}

\begin{Rem}
	In contrast, the conservativity of
		\[ 
			f^*|_{\cat T^c}\colon \cat T^c \to \cat S^c
		\]
		(that is, conservativity of $f^*$ on compact objects) is equivalent to the image of $\varphi$ containing all the closed points; see \cite[Theorem 1.2]{Balmer18}. In fact, \cref{cor:surjective-for-conservative} and \cite[Theorem 1.4]{Balmer18} establish that if $f^*\colon\cat T \to \cat S$ is conservative then it satisfies a nilpotence theorem for morphisms between compact objects (see \emph{loc.\,cit.}~for more details). In fact, \cite[Theorem 2.25]{BCHNP1} deduces a stronger nilpotence theorem from the conservativity of $f^*$ in which only the source is assumed to be compact. Combined with \cite[Theorem 1.3]{Balmer18}, this gives another proof of \cref{cor:surjective-for-conservative}; see \cite[Corollary 2.26]{BCHNP1}.
\end{Rem}

\begin{Prop} \label{prop:f!conservative}
	If $f^!\colon\cat T\to \cat S$ is conservative then $f^*\colon \cat T \to \cat S$ is conservative.
\end{Prop}

\begin{proof}
	If $f^*(t) = 0$ then $f^!\ihom{t,t}\cong\ihom{f^*(t),f^!(t)} = 0$ by \eqref{eq:f!hom}. Hence $\ihom{t,t}=0$ so that $t=0$.
\end{proof}

\begin{Rem}
	We will later prove that the converses to \cref{cor:surjective-for-conservative} and \cref{prop:f!conservative} hold when $\cat T$ is stratified; see \cref{cor:conservative-iff-surjective}. Another situation in which the converses hold will be given in \cref{prop:weaklyfinite-surjective-conservative} below. As we shall see, these conservativity properties are closely related to the following notion, whose name is motivated by terminology used in \cite{Mathew16}:
\end{Rem}

\begin{Def}\label{def:weakly-descendable}
	We say that a geometric functor $f^*\colon \cat T \to \cat S$ is \emph{weakly descendable} if~$\unitT \in \Loco{f_*(\unitS)}$.
\end{Def}

\begin{Rem}\label{rem:consequence-of-wd}
	If $\unitT \in \Loco{f_*(\unitS)}$ then it follows from \eqref{eq:t@loc} and \eqref{eq:[loc,t]} that
		\[
			t \in \Loco{f_*f^*t} \quad\text{ and }\quad t \in \Coloco{f_*f^!t}
		\]
		for all $t \in \cat T$. Hence $f^*$ and $f^!$ are both conservative.
\end{Rem}

\begin{Rem}\label{rem:image-of-compact}
	For any compact object $x \in \cat S^c$, \cref{thm:img-of-spc} implies that
		\[ 
			\varphi(\supp(x)) = \Supp_{\cat T}(f_*(x\otimes x^\vee)).
		\]
	Indeed just recall that $x \otimes x^\vee\cong\ihom{x,x}$ is the endomorphism ring object and that $\supp(x)=\supp(x\otimes x^\vee)$. For an arbitrary object $s \in \cat S$, we can also consider the endomorphism ring object $\ihom{s,s}$. The theorem provides
		\[
			\varphi(\Supp_{\cat S}(\ihom{s,s})) = \Supp_{\cat T}(f_*\ihom{s,s})
		\]
	assuming $\cat S$ satisfies the detection property.
\end{Rem}

\begin{Rem}\label{rem:map-is-closed}
	Note that if $f_*(\unitS)$ is compact then \cref{cor:img-of-spc} implies that the image $\im \varphi=\Supp(f_*(\unitS))=\supp(f_*(\unitS))$ is closed. Moreover, if $f_*$ preserves all compact objects (not just $\unitS$) then \cref{rem:image-of-compact} shows that $\varphi$ maps every Thomason closed subset to a Thomason closed subset. It follows that $\varphi$ is a closed map (see, e.g., \cite[Theorem 5.3.3]{DickmannSchwartzTressl19}). This leads to the following terminology:
\end{Rem}

\begin{Def}\label{def:weaklyfinite}
	We will say that $f^*\colon \cat T \to \cat S$ is \emph{\weaklyfinite} if $f_*(\unitS)$ is compact and \emph{\finite} if $f_*$ preserves compact objects. Note that if~$\cat S$ is monogenic then a {\weaklyfinite} morphism is automatically {\finite}.
\end{Def}

\begin{Rem}
	The above terminology is given more for grammatical and linguistic convenience rather than deep mathematical significance. We do not claim that it captures the correct geometric notion of ``closed morphism'' in tensor triangular geometry. For example, a closed immersion $\Spec(R/I)\hookrightarrow\Spec(R)$ of affine schemes induces a geometric functor $\Der(R)\to \Der(R/I)$ which is often not (weakly or strongly) closed in the sense of \cref{def:weaklyfinite}; see \cite[Example 7.8]{Sanders19}.
\end{Rem}

\begin{Rem}\label{rem:gnduality}
	In the terminology of \cite{BalmerDellAmbrogioSanders16}, a geometric functor $f^*$ is {\finite} if and only if it satisfies \emph{Grothendieck--Neeman duality} (or GN-duality, for short). As explained in \cite{BalmerDellAmbrogioSanders16}, a morphism $f^*$ satisfies GN-duality if and only if it preserves products if and only if it has a left adjoint if and only if the trio of functors $f^* \dashv f_* \dashv f^!$ extends on both sides to a sequence of five adjoints:
		\begin{equation}\label{eq:fiveadjoints}
			f_! \dashv f^* \dashv f_* \dashv f^! \dashv f_{(-1)}.
		\end{equation}
	There is then an analogue of \cref{thm:img-of-spc} for the left adjoint~$f_!$:
\end{Rem}

\begin{Thm}\label{thm:img-of-spc-left}
	Let $f^*\colon\cat T\to\cat S$ be a geometric functor as in \cref{hyp:geometric}. If $f^*$ is {\finite} then for any weak coring object $w \in \cat S$,
		\[
			\varphi(\Supp_{\cat S}(w))=\Supp_{\cat T}(f_!(w))
		\]
	provided either that
	\begin{enumerate}
		\item the detection property holds for $\cat S$; or
		\item the weak coring object is compact.
	\end{enumerate}
\end{Thm}

\begin{proof}
	One proceeds through the argument as in the proof of \cref{thm:img-of-spc} replacing $f_*$ with $f_!$. Note that the left projection formula holds by \cite[(3.11)]{BalmerDellAmbrogioSanders16}. To obtain the right-to-left implication of the analogue of $(\dagger)$, namely that $f^*(\gP)\otimes w \neq 0$ implies $f_!(f^*(\gP)\otimes w) \neq 0$ one again uses the fact explained in \cref{rem:conservative-on-weak-rings} that the left adjoint $f_!$ is conservative on objects of the form $f^*(t)\otimes w$ for $w$ a weak coring.
\end{proof}

\begin{Rem}
	According to the ur-Wirthmüller isomorphism of \cite[(3.10)]{BalmerDellAmbrogioSanders16}, $f_!(s) \simeq f_*(s \otimes \omega_f)$ for any $s \in \cat S$. In particular, applied to the weak coring $s\coloneqq \unitS$, \cref{thm:img-of-spc-left} implies that $\im \varphi = \Supp(f_*(\omega_f))$. Thus, $f_*(\omega_f)$ and $f_*(\unitS)$ have the same support if $f^*$ is {\finite}. These objects can be very different for a general geometric functor $f^*$. For example, if~$f^*$ is a smashing localization as in \cref{exa:image-of-smashing} then $f_*(\unitS)=f$ while $f_*(\omega_f) =\ihom{f,\unitT}$. Of course, it is rare for a smashing localization to be {\finite}. On the other hand, according to \cite{Sanders22} a necessary condition for $f^*$ to be finite \'{e}tale is that  $f_*(\unitS) \simeq f_*(\omega_f)$. We have shown that these two objects have the same support for any {\finite} functor.
\end{Rem}

\begin{Prop}\label{prop:weaklyfinite-surjective-conservative}
	Suppose $f^*\colon\cat T\to\cat S$ is {\weaklyfinite} (\cref{def:weaklyfinite}). The following are equivalent:
	\begin{enumerate}
		\item $f^*$ is weakly descendable (\cref{def:weakly-descendable});
		\item $f^!$ is conservative;
		\item $f^*$ is conservative;
		\item $\varphi\colon\Spc(\cat S^c)\to\Spc(\cat T^c)$ is surjective.
	\end{enumerate}
\end{Prop}

\begin{proof}
	The implication $(a) \Rightarrow (b)$ follows from \cref{rem:consequence-of-wd}, $(b) \Rightarrow (c)$ is \cref{prop:f!conservative} and $(c) \Rightarrow (d)$ is \cref{cor:surjective-for-conservative}. For $(d)\Rightarrow (a)$ note that the surjectivity of $\varphi$ implies that $\SuppT(f_*(\unitS))=\Spc(\cat T^c)$ by \cref{cor:img-of-spc} and~$f_*(\unitS)$ is compact by assumption. Thus, by the classification of thick ideals of $\cat T^c$, $\unitT \in \thicko{f_*(\unitS)} \subseteq \Loco{f_*(\unitS)}$.
\end{proof}

\section{Base change for support and cosupport}\label{sec:base-change}

We now turn to base change formulas that are valid for arbitrary objects. The cost will be some additional hypotheses on our functors or on our categories.

\begin{Hyp}\label{hyp:geometric-base-change}
	We continue to let $f^*\colon\cat T\to \cat S$ denote an arbitrary geometric functor and assume both $\Spc(\cat T^c)$ and $\Spc(\cat S^c)$ are weakly noetherian. We write
		\[
			\varphi\colon\Spc(\cat S^c)\to\Spc(\cat T^c)
		\]
	for the induced map on spectra.
\end{Hyp}

\begin{Prop}\label{prop:general-object_f*}
	Suppose $f^*\colon\cat T \to \cat S$ is as above (\cref{hyp:geometric-base-change}) and let $s \in \cat S$.
	\begin{enumerate}
		\item If $\cat S$ has the detection property, then
			\[ 
				\SuppT(f_*(s)) \subseteq \varphi(\SuppS(s))
			\]
			with equality when the functor $f_*$ is conservative.
		\item If $\cat S$ has the codetection property, then
			\[ 
				\CosuppT(f_*(s)) \subseteq \varphi(\CosuppS(s))
			\]
			with equality when the functor $f_*$ is conservative.
	\end{enumerate}
\end{Prop}

\begin{proof}
	The set-up is the same as the proof of \cref{thm:img-of-spc}. Let $\cat P \in \Spc(\cat T^c)$. For part $(a)$ consider the implications \eqref{eq:im-phi-imp1} displayed in the proof of \cref{thm:img-of-spc}. Note that the left-to-right direction of $(\ddagger)$ holds if $\cat S$ has the detection property. Hence, going from left-to-right we obtain the inclusion in part $(a)$. On the other hand, the right-to-left direction of $(\dagger)$ holds if $f_*$ is conservative.

	For part $(b)$, consider the analogous series of implications:
		\begin{equation}
		\begin{aligned}
			\cat P \in \CosuppT(f_*(s)) 
			\Longleftrightarrow&\; \ihom{\gP,f_*(s)} \neq 0 \\
			\Longleftrightarrow&\; f_*\ihom{f^*(\gP),s}\neq 0\\
			\underbrace{\Longrightarrow}_{(\dagger)}&\; \ihom{f^*(\gP),s} \neq 0\\
			\underbrace{\Longleftarrow}_{(\ddagger)}&\; \CosuppS(\ihom{f^*(\gP),s}) \neq \emptyset\\
			\Longleftrightarrow&\;\varphi^{-1}(\singP) \cap \CosuppS(s) \neq \emptyset\\
			\Longleftrightarrow&\; \cat P \in \varphi(\CosuppS(s)).
		\end{aligned}
		\end{equation}
	Again, the converse of $(\ddagger)$ holds when $\cat S$ has the codetection property, and the converse of $(\dagger)$ holds if $f_*$ is conservative.
\end{proof}

\begin{Rem}
	If the right adjoint $f_*$ is not conservative then there exist objects for which the inequalities of \cref{prop:general-object_f*} are strict.
\end{Rem}

\begin{Rem}\label{rem:cons}
	It is a well-known exercise (recall \cref{rem:replace-with-compact}) that the right adjoint~$f_*$ is conservative if and only if $f^*$ sends a set of compact generators of $\cat T$ to a set of (compact) generators of $\cat S$, in other words, if and only if $\cat S^c=\thick\langle f^*(\cat T^c)\rangle$. For example, if $\cat S$ is monogenic then every geometric functor $f^*\colon\cat T \to \cat S$ has a conservative right adjoint.
\end{Rem}

\begin{Rem}\label{rem:weaklyconservative}
	The right adjoint $f_*$ of a geometric functor is always ``weakly conservative'' in the following sense: For any nonzero $s \in \cat S$ there always exists a compact object $c \in \cat S^c$ such that $f_*(c\otimes s)\neq 0$. Indeed, since $\cat S$ is compactly generated, if $s \neq 0$ then $\cat S(c,s) \neq 0$ for some compact $c \in \cat S^c$. Replacing~$c$ by its dual, we can assert that there exists a compact $c \in \cat S^c$ such that $\cat S(\unitS,c\otimes s)\neq 0$. Since $\unitS=f^*(\unitT)$, adjunction implies $\cat T(\unitT,f_*(c\otimes s))\neq 0$ hence $f_*(c\otimes s)\neq 0$. One way of appreciating this phenomenon is to recognize that while the kernel of~$f_*$ is a localizing subcategory, it need not be an ideal. Indeed, the above argument shows that the largest localizing ideal contained in it, namely
		\[ 
			\SET{s \in \cat S}{f_*(c\otimes s)=0 \;\forall c \in \cat S^c}
		\]
	is always the zero ideal. In other words, the right adjoint $f_*$ of a geometric functor is conservative precisely when the kernel of $f_*$ is an ideal.
\end{Rem}

\begin{Prop}\label{prop:phi-f*}
	Suppose $f^*\colon\cat T\to \cat S$ is as above (\cref{hyp:geometric-base-change}) and let $t \in \cat T$.
	\begin{enumerate}
		\item	We always have inclusions
				\[ 
					\varphi(\SuppS(f^*(t))) \subseteq \SuppT(t \otimes f_*(\unitS)) \subseteq \SuppT(t) \cap \im \varphi.
				\]
		\item	If $\cat S$ has the detection property, then
				\[ 
					\varphi(\SuppS(f^*(t))) = \SuppT(t \otimes f_*(\unitS)).
				\] 
		\item	If $\cat S$ has the detection property and $\cat T$ is stratified, then
				\[ 
					\varphi(\SuppS(f^*(t))) = \SuppT(t) \cap \im \varphi.
				\]
		\item	If $f^*$ is conservative, then
				\[ 
					\SuppT(t \otimes f_*(\unitS)) = \SuppT(t),
				\]
				hence
				\[ 
					\varphi(\SuppS(f^*(t)) = \SuppT(t)
				\]
				if $\cat S$ also has the detection property.
	\end{enumerate}
\end{Prop}

\begin{proof}
	Let $\cat P \in \Spc(\cat T^c)$ and write $\singP = Y_1 \cap Y_2^c$ with $Y_1$ and $Y_2$ Thomason. Recall \cref{rem:fiber-is-support}. For any object $t \in \cat T$ consider the following implications:
		\begin{equation}\label{eq:proof-of-f*}
		\begin{aligned}
			\cat P \in \varphi(\SuppS(f^*(t)))
			\Longleftrightarrow&\;\varphi^{-1}(\singP) \cap \SuppS(f^*(t)) \neq \emptyset\\
			\Longleftrightarrow&\; \SuppS(f^*(\gP)\otimes f^*(t)) \neq \emptyset\\
			\underbrace{\Longrightarrow}_{(\dagger)}&\; f^*(\gP)\otimes f^*(t) \neq 0\\
			\Longleftrightarrow&\; f^*(\gP \otimes t) \neq 0\\
			\underbrace{\Longleftarrow}_{(\ddagger)}&\; f_*(f^*(\gP\otimes t))\neq 0\\
			\Longleftrightarrow&\; \gP\otimes t \otimes f_*(\unitS) \neq 0\\
			\Longleftrightarrow&\; \cat P \in \SuppT(t\otimes f_*(\unitS)). 
		\end{aligned}
		\end{equation}
	Note that the converse of $(\ddagger)$ holds by \cref{rem:conservative-on-adjoint}. Hence, going from left to right establishes the first inclusion in part $(a)$, while the second inclusion in part $(a)$ is due to \cref{cor:img-of-spc} and \cite[Remark 2.12(e)]{bhs1}. This establishes $(a)$. Now the converse of $(\dagger)$ holds when we have the detection property. Hence going from right to left gives $(b)$. Part $(c)$ follows from part $(b)$ by invoking the tensor product formula (\cite[Theorem~8.2]{bhs1}), which holds because $\cat T$ is stratified, and \cref{cor:img-of-spc}: $\SuppT(t\otimes f_*(\unitS))=\SuppT(t) \cap \SuppT(f_*(\unitS))=\SuppT(t) \cap \im \varphi$. For part $(d)$ it suffices to show that $\SuppT(t) \subseteq \SuppT(t \otimes f_*(\unitS))$ when $f^*$ is conservative. Indeed if $f^*$ is conservative then so is the composite $f_*f^*$ (\cref{rem:conservative-on-adjoint}), hence $0 \neq \gP \otimes t$ implies $0 \neq f_*f^*(\gP \otimes t) \simeq \gP \otimes t \otimes f_*(\unitS)$.
\end{proof}

\begin{Rem}\label{rem:img-f*-compact}
	If $\cat S$ has the detection property then the equality
		\[ 
			\varphi(\SuppS(f^*(t))) = \SuppT(t) \cap \im \varphi
		\]
	also holds (without further hypotheses on $f^*$ or $\cat S$) if the object $t$ is compact, or a left idempotent $e_Y$, or a right idempotent $f_Y$, or an object $\gP$. Indeed, in these cases the inclusion $\SuppT(t \otimes f_*(\unitS)) \subseteq \SuppT(t) \cap \SuppT(f_*(\unitS))$ is an equality by \cite[Lemma 2.13 and Lemma 2.18]{bhs1}.
\end{Rem}

\begin{Prop}\label{prop:phi-f!}
	Suppose $f^*\colon\cat T\to \cat S$ is as above (\cref{hyp:geometric-base-change}) and let $t \in \cat T$.
	\begin{enumerate}
		\item We always have inclusions
			\[ 
				\varphi(\CosuppS(f^!(t))) \subseteq \CosuppT(\ihom{f_*(\unitS),t}) \subseteq \CosuppT(t) \cap \im \varphi.
			\]
		\item If $\cat S$ has the codetection property, then
			\[
				\varphi(\CosuppS(f^!(t))) = \CosuppT(\ihom{f_*(\unitS),t}).
			\] 
		\item If $\cat S$ has the codetection property and $\cat T$ is stratified, then
			\[ 
				\varphi(\CosuppS(f^!(t))) = \CosuppT(t) \cap \im \varphi.
			\]
		\item If $f^!$ is conservative, then
			\[ 
				\CosuppT(\ihom{f_*(\unitS),t}) = \CosuppT(t),
			\]
			hence
			\[ 
				\varphi(\CosuppS(f^!(t)) = \CosuppT(t)
			\]
			if $\cat S$ also has the codetection property.
	\end{enumerate}
\end{Prop}

\begin{proof}
	The proof is similar to the proof of \cref{prop:phi-f*}. Let $\cat P \in \Spc(\cat T^c)$ and write $\singP = Y_1 \cap Y_2^c$ with $Y_1$ and $Y_2$ Thomason. For any object $t \in \cat T$ consider the following implications analogous to those of \eqref{eq:im-phi-imp1}:
		\begin{equation}\label{eq:proof-of_f!}
		\begin{aligned}
			\cat P \in \varphi(\CosuppS(f^!(t)))
			\Longleftrightarrow&\;\varphi^{-1}(\singP) \cap \CosuppS(f^!(t)) \neq \emptyset\\
			\Longleftrightarrow&\; \CosuppS(\ihom{f^*(\gP),f^!(t)}) \neq \emptyset\\
			\underbrace{\Longrightarrow}_{(\dagger)}&\; \ihom{f^*(\gP),f^!(t)} \neq 0\\
			\underbrace{\Longleftarrow}_{(\ddagger)}&\; f_*\ihom{f^*(\gP),f^!(t)}\neq 0\\
			\Longleftrightarrow&\; \ihom{\gP, f_*f^!(t)})\neq 0\\
			\Longleftrightarrow&\; \ihom{\gP, \ihom{f_*(\unitS),t}}\neq 0\\
			\Longleftrightarrow&\; \cat P \in \CosuppT(\ihom{f_*(\unitS),t}).
		\end{aligned}
		\end{equation}
	Here we have invoked the adjunction isomorphisms  $f_*\ihom{f^*(a),b}\cong\ihom{a,f_*(b)}$ and $f_*f^!(t)\cong f_*\ihom{\unitS,f^!(t)} \cong \ihom{f_*(\unitS),t}$ established in \cite{BalmerDellAmbrogioSanders16}. Moreover, the isomorphism $f^!\ihom{a,b}\cong\ihom{f^*(a),f^!(b)}$ from \eqref{eq:f!hom} shows that $\ihom{f^*(\gP),f^!(t)}$ is in the essential image of $f^!$ and hence the converse of ($\ddagger$) holds by \cref{rem:conservative-on-adjoint}. Hence, going from left to right we obtain the first inclusion in part $(a)$, while the second inclusion is a standard property of cosupport \eqref{eq:cosupp_hom} together with \cref{cor:img-of-spc}. This establishes~$(a)$. Now the converse of $(\dagger)$ holds when we have codetection property. Hence going from right to left gives $(b)$. Part $(c)$ follows from part $(b)$ by invoking the hom formula (\cref{thm:strat_cosupport}), which uses that $\cat T$ is stratified, and \cref{cor:img-of-spc}. For part $(d)$ it suffices to show that  $\CosuppT(t) \subseteq \CosuppT(\ihom{f_*(\unitS),t})$ when $f^!$ is conservative. Indeed, if $\ihom{\gP,t} \neq 0$ then $f_* f^! \ihom{\gP,t} \neq 0$ by \cref{rem:conservative-on-adjoint}. That is, $0\neq f_* f^!\ihom{\gP,t} \cong f_*\ihom{f^*(\gP),f^!(t)} \cong \ihom{\gP,f_*f^!(t)}$ and recall $f_*f^!(t) \cong \ihom{f_*(\unitS),t}$.
\end{proof}

\begin{Cor}\label{cor:push-phi}
	Suppose $f^*$ is {\weaklyfinite} (\cref{def:weaklyfinite}). Then:
	\begin{enumerate}
		\item	If $\cat S$ has the detection property then
				\[
					\varphi(\SuppS(f^*(t)) = \SuppT(t) \cap \im \varphi
				\]
				for every $t \in \cat T$.
		\item	If $\cat S$ has the codetection property then
				\[
					\varphi(\CosuppS(f^!(t)) = \CosuppT(t) \cap \im \varphi
				\]
				for every $t\in \cat T$.
	\end{enumerate}
\end{Cor}

\begin{proof}
	This follows from parts $(a)$ and $(b)$ of \cref{prop:phi-f*} and \cref{prop:phi-f!} together with the half-tensor and half-hom theorems (\cite[Lemma 2.18]{bhs1} and \cref{prop:halfhom}).
\end{proof}

\begin{Rem}\label{rem:AV-trivial-inclusions}
	It follows from \cref{prop:phi-f*}$(a)$ and \cref{prop:phi-f!}$(a)$ that the following inclusions hold for any $t \in \cat T$:
		\begin{align*}
			\SuppS(f^*(t)) &\subseteq \varphi^{-1}(\SuppT(t)), \text{ and}\\
			\CosuppS(f^!(t)) &\subseteq \varphi^{-1}(\CosuppT(t)).
		\end{align*}
	Our next goal is to establish that these are equalities when $\cat T$ is stratified; see \cref{cor:globalavruninscott} below.
\end{Rem}

\begin{Lem}\label{lem:supp-of-f_*-gP}
	Let $\cat Q\in \Spc(\cat S^c)$. For any $s \in \cat S$ we have
		\[ 
			\SuppT(f_*(s\otimes \gQ))\subseteq \{\varphi(\cat Q)\}.
		\]
\end{Lem}

\begin{proof}
	Suppose $\cat P \in \SuppT(f_*(s\otimes\gQ))$. Then
		\[
			f_*(f^*(\gP)\otimes s\otimes \gQ) \simeq \gP \otimes f_*(s\otimes \gQ) \neq 0
		\]
	so that $f^*(\gP)\otimes \gQ\neq 0$. That is, $\cat Q \in \SuppS(f^*(\gP))=\varphi^{-1}(\singP)$ by \cref{rem:fiber-is-support}. Thus $\cat P =\varphi(\cat Q)$.
\end{proof}

\begin{Thm}[Local Avrunin--Scott identities]\label{thm:localavruninscott}
	Let $f^*\colon\cat T \to \cat S$ be a geometric functor as in \cref{hyp:geometric-base-change} and let $t \in \cat T$. Suppose the localizing ideals $\GammaP \cat T$ are minimal for all $\cat P \in \SuppT(t)$. Then we have:
		\begin{align}
			\SuppS(f^*(t)) &= \varphi^{-1}(\SuppT(t)); \text{ and}\label{eq:localAV-f*}\\
			\CosuppS(f^!(t)) &= \varphi^{-1}(\CosuppT(t)).\label{eq:localAV-f!}
		\end{align}
\end{Thm}

\begin{proof}
	The $\subseteq$ inclusions always hold (\cref{rem:AV-trivial-inclusions}). To establish the $\supseteq$ inclusion in \eqref{eq:localAV-f*}, suppose $\cat Q \in \varphi^{-1}(\SuppT(t))$ so that $\varphi(\cat Q) \in \SuppT(t)$. Since $\GammaphiQ \cat T$ is minimal, it follows that $\gphiQ\in \Loco{t}$. Hence $f^*(\gphiQ) \in \Loco{f^*(t)}$. Therefore 
		\[
			\cat Q \in \varphi^{-1}(\{\varphi(\cat Q)\})=\SuppS(f^*(\gphiQ))\subseteq \SuppS(f^*(t)),
		\]
	where the first equality is because of \cref{rem:fiber-is-support}. It remains to establish the~$\supseteq$ inclusion in \eqref{eq:localAV-f!}. Suppose $\cat Q \not\in \CosuppS(f^!(t))$. We will show that $\cat Q \not\in \varphi^{-1}(\CosuppT(t))$. By \cref{rem:weaklyconservative} and \cref{lem:supp-of-f_*-gP}, there exists a compact $c\in \cat S^c$ such that $\SuppT(f_*(c\otimes \gQ)) = \{\varphi(\cat Q)\}$. Since $\GammaphiQ \cat T$ is minimal, we conclude that 
		\begin{equation}\label{eq:gvarphiQin}
			g_{\varphi(\cat Q)} \in \Loco{f_*(c\otimes \gQ)}.
		\end{equation}
	Now, since $\cat Q \not\in \CosuppS(f^!(t))$ by hypothesis, we have $\ihom{\gQ,f^!(t)}=0$. Hence $\ihom{c\otimes \gQ,f^!(t)}=0$ and so 
		\[
			\ihom{f_*(c\otimes \gQ),t}\simeq f_*\ihom{c\otimes \gQ,f^!(t)} = 0.
		\]
	Therefore, $\ihom{g_{\varphi(\cat Q)},t}=0$ by \eqref{eq:gvarphiQin}. That is, $\varphi(\cat Q) \not\in \CosuppT(t)$.
\end{proof}

\begin{Cor}[Global Avrunin--Scott identities]\label{cor:globalavruninscott}
    If $\cat T$ is stratified, then for any $t \in \cat T$ we have:
		\begin{align}
			\SuppS(f^*(t)) &= \varphi^{-1}(\SuppT(t)); \text{ and}\label{eq:AV-f*}\\
			\CosuppS(f^!(t)) &= \varphi^{-1}(\CosuppT(t)).\label{eq:AV-f!}
		\end{align}
\end{Cor}
\begin{proof}
Indeed, \cite[Theorem 4.1]{bhs1} shows that stratification of $\cat T$ gives minimality at every point of its spectrum, so the result follows from \cref{thm:localavruninscott}.
\end{proof}

\begin{Rem}
	In a tensor triangular setting, the Avrunin--Scott identities for support and cosupport were studied (under more restrictive hypotheses) in \cite[Proposition 3.14]{BarthelCastellanaHeardValenzuela19}. They originate in the work of Avrunin--Scott \cite{avruninscott} on support varieties for representations of finite groups.
\end{Rem}

\begin{Rem}
    If we assume that the right adjoint $f_*$ is conservative, then we can prove the Avrunin--Scott identities of \cref{cor:globalavruninscott} under slightly weaker hypotheses on the category $\cat T$ by modifying the proof of \cite[Proposition 3.14]{BarthelCastellanaHeardValenzuela19}; namely, the identity for $f^*$ holds if $\cat T$ has the tensor product formula, while the identity for $f^!$ holds if $\cat T$ has the Hom formula. Recall that the former is implied by stratification (\cite[Theorem 8.2]{bhs1}) and the latter is equivalent to stratification (\cref{thm:strat_cosupport}). 
\end{Rem}

\begin{Cor}\label{cor:conservative-iff-surjective}
	Let $\cat T$ be stratified. For $f^*\colon\cat T \to \cat S$ as in \cref{hyp:geometric-base-change}, the following conditions are equivalent:
	    \begin{enumerate}
			\item $f^*$ is weakly descendable (\cref{def:weakly-descendable});
	        \item $f^!$ is conservative;
	        \item $f^*$ is conservative;
	        \item $\varphi\colon\Spc(\cat S^c)\to \Spc(\cat T^c)$ is surjective.
	    \end{enumerate}
\end{Cor}

\begin{proof}
	The implication $(a) \Rightarrow (b)$ follows from \cref{rem:consequence-of-wd}, $(b) \Rightarrow (c)$ is \cref{prop:f!conservative} and $(c) \Rightarrow (d)$ is \cref{cor:surjective-for-conservative}. For $(d)\Rightarrow (a)$ note that the surjectivity of $\varphi$ implies that $\SuppT(f_*(\unitS))=\Spc(\cat T^c)$ by \cref{cor:img-of-spc}. If $\cat T$ is stratified, this implies $\unitT \in \Loco{f_*(\unitS)}$.
\end{proof}

\begin{Exa}
	Let $(R,\mathfrak m,k)$ be a commutative noetherian local ring whose maximal ideal $\mathfrak m$ is nilpotent (that is, $R$ is a commutative artinian local ring). \Cref{cor:conservative-iff-surjective} implies that the functor $\Der(R) \to \Der(k)$ is conservative. Hence, if $M$ is an arbitrary flat $R$-module, we conclude that $M/\mathfrak m M=0$ implies~$M=0$. This is a well-known variant of Nakayama's Lemma which replaces finite generation of the module $M$ with flatness together with nilpotence of the ideal $\mathfrak m$; cf.~\cite[Theorem~7.10]{Matsumura89}.
\end{Exa}

\begin{Exa}\label{exa:keller}
	Let $A$ be the ring denoted $A$ in \cite{Keller94b}. It is a non-discrete valuation domain of rank 1 whose value group is $\bbZ[1/\ell]\subset \bbQ$; see \cite[Theorem~II.3.8]{FuchsSalce01}. The quotient $R\coloneqq A/xA$ by any nontrivial principal ideal is then a non-noetherian local ring whose spectrum is a single point. The map to the residue field $f\colon R\to k$ induces a functor $f^*\colon \Der(R)\to\Der(k)$ which is surjective on spectra. However, the maximal ideal of $R$ satisfies $\mathfrak m^2=\mathfrak m$ and is flat as an $R$-module. Hence $M \coloneqq \mathfrak m$ is a nonzero $R$-module which is annihilated by $f^*$. We conclude that $\Der(R)$ cannot be stratified, since this would contradict \cref{cor:conservative-iff-surjective}. Alternatively, it follows from the work of Bazzoni--{\Stovicek} (cf.~\cite[Example 5.24]{BazzoniStovicek17}) that $\Der(R)\to \Der(k)$ is a smashing localization which has no finite acyclics. Hence $\Der(R)$ does not satisfy the telescope conjecture and we can alternatively conclude that it is not stratified by invoking \cite[Theorem 9.11]{bhs1}. 
\end{Exa}

\begin{Exa}
	Let $(R,\mathfrak m,k)$ be a commutative noetherian local ring and let~$\widehat{R}$ denote the $\mathfrak m$-adic completion of $R$. The induced functor $\Der(R)\to\Der(\widehat{R})$ is conservative. Indeed, $R\to \widehat{R}$ is faithfully flat, so $\Spec(\widehat{R})\to\Spec(R)$ is surjective.
\end{Exa}

\section{Descending the local-to-global principle}\label{sec:descending-LGP}

We now provide a version of descent for the local-to-global principle.

\begin{Prop}\label{prop:descend-LGP}
	Let $f^*\colon\cat T\to \cat S$ be as above (\cref{hyp:geometric-base-change}). Assume that $\cat S$ satisfies the local-to-global principle. Then
		\[ 
			t \in \Loco{t\otimes \gP \mid \cat P\in \SuppT(t)}
		\]
	for any object $t \in \cat T$ such that $t \in \Loco{f_*f^*t}$.
\end{Prop}

\begin{proof} 
    By the local-to-global principle in $\cat S$, we have
		\begin{align*} \Loco{f^* \gP} &=
				\Loco{f^*\gP \otimes \gQ \mid \cat Q \in \varphi^{-1}(\singP)} \\
			&= \Loco{g_{\varphi^{-1}(\singP)} \otimes \gQ \mid \cat Q \in \varphi^{-1}(\singP)} \\
			&= \Loco{ \gQ \mid \cat Q \in \varphi^{-1}(\singP)} 
		\end{align*}
	for any $\cat P \in \Spc(\cat T^c)$, where the last equality uses \cref{rem:W-intersection}. Hence 
		\begin{equation}\label{eq:temp}
			\Loco{f^*t \otimes f^*\gP}=\Loco{f^*t \otimes \gQ \mid \cat Q \in \varphi^{-1}(\singP)}
		\end{equation}
    by \cite[Lemma 3.6]{bhs1}, for example. Then, by the local-to-global principle for $f^*t$, we have
		\begin{align*}
			f^*t &\in \Loco{f^*t \otimes \gQ \mid \cat Q \in \SuppS(f^*t)}\\
			&\subseteq \Loco{f^*t \otimes \gQ \mid \cat Q \in \varphi^{-1}(\SuppT(t))} &\text{(\cref{rem:AV-trivial-inclusions})}\\
			&= \Loco{f^*t \otimes f^*\gP \mid \cat P \in \SuppT(t)} &\text{(by \eqref{eq:temp}).}
		\end{align*}
	Using \eqref{eq:f_*-inclusion}, we obtain
		\[ 
			f_*f^*t \in \Loco{t\otimes \gP \mid \cat P \in \SuppT(t)}
		\]
    which establishes the claim.
\end{proof}

\begin{Rem}\label{rem:nice-LGP-descent}
	It follows from \cref{cor:img-of-spc} that objects satisfying $t \in \Loco{f_*f^*t}$ have their support contained in the image of $\varphi$. For some functors, the converse holds: If $\SuppT(t) \subseteq \im \varphi$ then $t \in \Loco{f_*f^*t}$. This is the case for the following two examples:
	\begin{enumerate}
		\item A finite localization $f^*\colon\cat T \to \cat T(U)$ under the assumption that $\cat T$ has the detection property. In this case, $\im \varphi = U$ and $f_*f^*t \cong t$ for all $t$ which are supported in $U$.
	    \item A {\weaklyfinite} functor $f^*\colon\cat T\to \cat S$ under the assumption that $\cat T$ has the detection property. This follows from \cite[Lemma~3.7]{bhs1}.
	\end{enumerate}
	Intuitively we can think of \cref{prop:descend-LGP} as saying that the local-to-global principle partially descends along $f^*$ to objects supported in the image of $\varphi$. However, note that if $f^*\colon\cat T \to \cat S$ is a functor which is not conservative and yet for which $\varphi$ is surjective (such as the functor described in \cref{exa:keller}) then there exists an object $t \in \cat T$ with $\SuppT(t) \subseteq \im\varphi$ and yet with $t \not\in \Loco{f_*f^*t}$, so the above intuition is not completely accurate.
\end{Rem}

\begin{Rem}\label{rem:detection-descends}
	It follows from \cref{rem:AV-trivial-inclusions} that the following hold:
	\begin{enumerate}
		\item If $f^*\colon\cat T\to \cat S$ is conservative then the detection property descends from $\cat S$ to~$\cat T$.
		\item If $f^!\colon\cat T\to \cat S$ is conservative then the codetection property descends from $\cat S$ to $\cat T$.
	\end{enumerate}
	Since the codetection property is equivalent to the (co)local-to-global principle (by \cref{thm:LGP-equiv}) we see that the local-to-global principle descends from $\cat S$ to $\cat T$ whenever~$f^!$ is conservative.
\end{Rem}

\begin{Exa}
	If $f^*\colon\cat T\to \cat S$ is fully faithful then the local-to-global principle descends from $\cat S$ to $\cat T$. This follows from \cref{prop:descend-LGP} since $t \cong f_*f^* t$ for all $t \in \cat T$. It also follows from \cref{rem:detection-descends} since if $f^*$ is fully faithful then $f^!$ is conservative. Indeed, in general $f_*f^!(t) \cong \ihom{f_*(\unitS),t}$, which becomes $f_*f^!(t) \cong t$ when $f^*$ is fully faithful since $f_*(\unitS)\cong\unitT$ (see, e.g., \cite[Remark~4.13]{Sanders22}).
\end{Exa}

\begin{Exa}\label{ex:localtoglobaldescent}
	If $f^*\colon\cat T\to \cat S$ is conservative and {\weaklyfinite} then the local-to-global principle descends from $\cat S$ to $\cat T$. This follows from \cref{prop:weaklyfinite-surjective-conservative} and \cref{rem:detection-descends}. It can also be obtained as an application of \cref{prop:descend-LGP}.
\end{Exa}

\section{Local cogeneration}\label{sec:local-cogeneration}

We have seen in \cref{sec:perfect-generation} that a compactly generated category is perfectly cogenerated by the Brown--Comenetz duals of its compact objects. We now explain how these can be used to construct suitable \emph{local} cogenerators in tt-geometry. This will be an important ingredient in our bootstrap result for descending costratification in \cref{sec:bootstrap}.

\begin{Prop}\label{prop:local-generators}
	Let $\cat T$ be a rigidly-compactly generated tt-category.
	\begin{enumerate}
		\item Let $Y\subseteq \Spc(\cat T^c)$ be a Thomason subset. The subcategory $\ihom{e_Y,\cat T}$ is perfectly cogenerated by $\SET{ I_c }{\supp(c) \subseteq Y}$. In particular,
			\[
				\ihom{e_Y,\cat T}=\Coloc\langle I_c \mid \supp(c) \subseteq Y\rangle.
			\]
		\item Let $Y_1, Y_2 \subseteq \Spc(\cat T^c)$ be Thomason subsets and consider the weakly visible subset $W \coloneqq Y_1 \cap Y_2^c$. The subcategory $\LambdaW\cat T$ is perfectly cogenerated by $\SET{\LambdaW I_c}{\supp(c) \subseteq Y_1}$. In particular,
			\[
				\LambdaW \cat T = \Coloc\langle \LambdaW I_c \mid \supp(c) \subseteq Y_1\rangle.
			\]
	\end{enumerate}
\end{Prop}

\begin{proof}
	Let $\cat K\coloneqq \cat T^c$ and let $\cat K_Y \coloneqq\SET{x \in \cat K}{\supp(x)\subseteq Y}$. By \cref{exa:BC-duality} and \cref{prop:Top-is-perfectly-generated}, $\cat T$ is perfectly cogenerated by $\SET{I_c}{c \in \cat K} = \cat K \otimes I_\unit$. Recall from \cref{lem:key-observation} that $\ihom{e_Y,\cat T} = \Coloc\langle \cat K_Y \otimes \cat T\rangle$. Hence \cref{lem:perfectly-generated} (applied to~$\cat T\op$) implies that $\ihom{e_Y,\cat T}$ is perfectly cogenerated by $\cat K_Y \otimes \cat K \otimes I_\unit = \cat K_Y\otimes I_\unit = \SET{I_c}{ c \in \cat K_Y}$. This establishes~$(a)$.

	We obtain part $(b)$ by applying \cref{prop:perfect-conservation} to part $(a)$. The functor 
		\begin{equation}\label{eq:e-down-to-g}
			\ihom{f_{Y_2},-}\colon\ihom{e_{Y_1},\cat T}\to \ihom{\gW,\cat T}
		\end{equation}
	preserves products. Hence, since the domain category is perfectly cogenerated (by part $(a)$), we know it must have a left adjoint. Indeed, as observed in \cref{rem:stalk-costalk-equivalence}, it has the fully faithful left adjoint
		\[
			-\otimes f_{Y_2}\colon\ihom{\gW,\cat T}\to \ihom{e_{Y_1},\cat T}.
		\]
		But this functor is naturally isomorphic to 
		\[
			\ihom{e_{Y_1},-}\colon\ihom{\gW,\cat T}\to \ihom{e_{Y_1},\cat T}
		\]
	which evidently preserves products. Thus, the functor \eqref{eq:e-down-to-g} has a conservative left adjoint, which itself preserves products. Hence we can invoke \cref{prop:perfect-conservation} to conclude that $\LambdaW\coloneqq\ihom{\gW,\cat T}$ is perfectly cogenerated by 
		\[
			\SET{ \ihom{f_{Y_2},I_c} }{{\supp(c)\subseteq Y_1}}.
		\]
	Note that $\ihom{f_{Y_2},I_c}=\ihom{\gW,I_c}=\LambdaW I_c$ since $I_c \in \ihom{e_{Y_1},\cat T}$ already.
\end{proof}

\begin{Rem}
	The statement in part $(a)$ of \cref{prop:local-generators} is the special case of part $(b)$ when $Y_2 \coloneqq \emptyset$. Indeed, as explained in the proof, if $\supp(c) \subseteq Y$ then $I_c \cong \ihom{e_Y,I_c} = \Lambda^Y I_c$. We have formulated the proposition as we have, since the special case is used to prove the more general statement.
\end{Rem}

\begin{Rem}
	The reader may find it interesting to compare with \cite[Prop.~5.4]{BensonIyengarKrause12}.
\end{Rem}

\begin{Rem}
	The compact generators of $e_Y \otimes \cat T = \Loc\langle c \mid \supp(c) \subseteq Y\rangle$ similarly pushdown to a set of perfect generators of $\GammaW \cat T = \Loc\langle \GammaW(c) \mid {\supp(c) \subseteq Y_1}\rangle$. Indeed, one uses the second adjunction displayed in \cref{rem:stalk-costalk-equivalence}, observing that the fully faithful right adjoint $\ihom{f_{Y_2},-}$ is naturally isomorphic to $e_{Y_1}\otimes -$, which preserves coproducts, whence one can invoke \cref{prop:perfect-conservation}.
\end{Rem}

\begin{Prop}\label{prop:stalk-trio}
	Let $f^*\colon\cat T\to \cat S$ be as above (\cref{hyp:geometric-base-change}) and let $W \subseteq \Spc(\cat T^c)$ be a weakly visible subset. The trio of adjoints $f^* \dashv f_* \dashv f^!$ induces a trio of adjoints
		\[\begin{tikzcd}[column sep=scriptsize]
			\GammaW\cat T \ar[d,"f^*"',shift right=1.25ex] \ar[r,phantom,"\cong"description] & \LambdaW\cat T\ar[d,"f^!",shift left=1.25ex]\\
			\GammainvW\cat S \ar[u,"f_*"',shift right=1.25ex]\ar[u,phantom,"\dashv"]\ar[r,phantom,"\cong"description]& \LambdainvW\cat S\ar[u,"f_*",shift left=1.25ex]\ar[u,phantom,"\dashv"]
		\end{tikzcd}\]
	where the middle square commutes and uses the equivalences of \cref{rem:stalk-costalk-equivalence}.
\end{Prop}

\begin{proof}
	It follows from the definitions and the standard isomorphisms from \cite{BalmerDellAmbrogioSanders16} together with \cref{rem:fiber-is-support} that the functors $f^*, f_*$ and $f^!$ restrict to the four functors in the statement, and one verifies directly that we have the two displayed adjunctions $f^* \dashv f_*$ and $f_* \dashv f^!$. One then checks that the middle square commutes from the definitions of the ``stalk-costalk equivalences'' in \cref{rem:stalk-costalk-equivalence}.
\end{proof}

\section{Bootstrap for costratification}\label{sec:bootstrap}

We now provide descent techniques for establishing costratification.

\begin{Lem}\label{lem:bcdualbasechange}
	Let $f^*\colon \cat T \to \cat S$ be a geometric functor. For any $d \in \cat T^c$, there is an isomorphism $I_{f^*d} \cong f^!I_d$. 
\end{Lem}

\begin{proof}
	For any $s \in \cat S$, the defining property of the Brown--Comenetz dual (\cref{def:compact-BC}) together with adjunction provides natural isomorphisms
	\begin{align*}
		\cat S(s,I_{f^*d}) & \cong \Hom_{\bbZ}(\cat S(f^*d,s),\bbQ/\bbZ) \\
		& \cong \Hom_{\bbZ}(\cat T(d,f_*s),\bbQ/\bbZ) \\
		& \cong \cat T(f_*s,I_d) \\
		& \cong \cat S(s,f^!I_d)
	\end{align*}
	and we summon Yoneda.
\end{proof}

\begin{Thm}\label{thm:detectingcomin}
	Let $f^*\colon \cat T \to \cat S$ be a geometric functor as in \cref{hyp:geometric-base-change} and consider a prime $\cat P \in \im \varphi$. Assume that 
	\begin{enumerate}
		\item $\cat T$ is stratified; and 
		\item $\LambdaQ \cat S$ is a minimal colocalizing coideal of $\cat S$ for all $\cat Q \in \varphi^{-1}(\singP)$. 
	\end{enumerate}
	Then $\LambdaP \cat T$ is a minimal colocalizing coideal of $\cat T$.
\end{Thm}

\begin{proof}
	Let $t\in \LambdaP\cat T$ be a nonzero object. Hence $\CosuppT(t) =\singP$ by the codetection property (which holds by \cref{thm:LGP-equiv}). Since $\cat T$ is stratified, we have
		\[
			\CosuppS(f^!(t))=\varphi^{-1}\CosuppT(t)=\varphi^{-1}(\singP)
		\]
	by \cref{cor:globalavruninscott}. Then, since $\cat S$ has cominimality at all primes $\cat Q \in \varphi^{-1}(\singP)$, we have
		\begin{align*}
			 \Coloco{\Lambda^{\cat Q}\cat S \mid \cat Q \in \varphi^{-1}(\singP)} \subseteq \Coloco{f^!(t)}. 
		\end{align*}
	Now, $\Coloco{f^!(t)} = \Coloc\langle \ihom{\cat S,f^!(t)}\rangle$ and so, since $f_*$ preserves products, we obtain an inclusion
		\[
			f_*\LambdaQ \cat S \subseteq \Coloc\langle f_*\ihom{\cat S,f^!(t)}\rangle.
		\]
	Using the adjunction isomorphism \cite[(2.18)]{BalmerDellAmbrogioSanders16}, we conclude that
		\begin{equation}\label{eq:detectingcomin:1}
			f_*\LambdaQ \cat S\subseteq \Coloc\langle \ihom{f_*\cat S,t}\rangle \subseteq \Coloco{t}
		\end{equation}
	for each $\cat Q \in \varphi^{-1}(\singP)$. 

	Combining \cref{lem:supp-of-f_*-gP} and \cref{rem:weaklyconservative}, there exists $c \in \cat S^c$ (possibly depending on $\cat Q \in \varphi^{-1}(\singP)$) with $\supp(f_*(\GammaQ c)) = \singP$. On the one hand, stratification of $\cat T$ shows that $\GammaP \unit \in \Loco{f_*\GammaQ c}$. Hence by \eqref{eq:[loc,t]} we have for every $d \in \cat T^c$:
		\[
			\LambdaP I_d \cong \ihom{\GammaP\unit,I_d} \in \Coloco{\ihom{f_*\GammaQ c,I_d}}.
		\]
	On the other hand, there are isomorphisms
		\begin{align*}
			f_*(\LambdaQ I_{c^{\vee} \otimes f^*(d)}) & \cong f_*\ihom{\GammaQ \unit, c^{\vee} \otimes I_{f^*d}} \\
			& \cong f_*\ihom{\GammaQ c, I_{f^*d}} \\
			& \cong \ihom{f_*\GammaQ c, I_d},
		\end{align*}
	where the fourth isomorphism uses \cref{lem:bcdualbasechange}. Combining these two observations with \eqref{eq:detectingcomin:1}, we see that 
		\[
			\LambdaP I_d  \in \Coloco{f_*(\LambdaQ I_{c^{\vee} \otimes f^*(d)})} \subseteq \Coloco{t}.
		\]
	\Cref{prop:local-generators} then implies that $\LambdaP\cat T \subseteq \Coloco{t}$. This establishes cominimality of $\cat T$ at $\cat P$.
\end{proof}

\begin{Cor}[Bootstrap]\label{cor:detectingcomin}
	Suppose $f^*\colon \cat T \to \cat S$ is a geometric functor as above (\cref{hyp:geometric-base-change}) such that
	\begin{enumerate}
		\item $\cat T$ is stratified;
		\item $\cat S$ is costratified; and
		\item $f^*$ is conservative. 
	\end{enumerate}
	Then $\cat T$ is costratified. 
\end{Cor}

\begin{proof}
	By \cref{thm:equiv-costrat} and \cref{thm:LGP-equiv}, we need to establish the minimality of $\LambdaP \cat T$ for each $\cat P \in \Spc(\cat T^c)$. By \cref{cor:surjective-for-conservative}, conservativity of $f^*$ implies that $\varphi$ is surjective. Therefore, \cref{thm:detectingcomin} applies at all points of~$\Spc(\cat T^c)$.
\end{proof}

\begin{Rem}\label{rem:detectingcomin}
    In the situation of \cref{cor:detectingcomin}, assuming $(a)$, condition~$(c)$ is equivalent to the statement that the map $\varphi$ is surjective, and also equivalent to $f^!$ being conservative, see \cref{cor:conservative-iff-surjective}.
\end{Rem}

\begin{Rem}\label{rem:bootstrap-empowerment}
	\Cref{thm:detectingcomin} becomes especially useful in combination with descent results for stratification. If $\cat S$ is costratified then it is also stratified (\cref{thm:costrat_implies_strat}) and one can then try to apply one of the stratification descent techniques to show that $\cat T$ is stratified. \Cref{cor:detectingcomin} implies that, under the mild hypothesis that the induced map on spectra is surjective, costratification then descends as well. Thus we view it as a bootstrap technique. We give three instances of this idea below, suggestively called
    \begin{itemize}
        \item Zariski descent (\cref{thm:costrat_cover});
        \item Quasi-finite descent (\cref{cor:quasifinite-costrat}); and
        \item Nil-descent (\cref{cor:nilpotentdescent}).
    \end{itemize}
	The following variant of the bootstrap theorem will also be useful:
\end{Rem}

\begin{Thm}\label{prop:bootstrap_for_algebraic_categories}
	Suppose that there exist geometric functors $f^*_{i} \colon \cat T \to \cat S_{i}$ (\cref{hyp:geometric-base-change}) to costratified categories $\cat S_{i}$ such that the induced maps $\varphi_i$ on spectra are jointly surjective: 
		\[
			\bigcup_i\im \varphi_i = \Spc(\cat T^c). 
		\]
	Then $\cat T$ is stratified if and only if $\cat T$ is costratified. 
\end{Thm}

\begin{proof}
	That costratification for $\cat T$ implies stratification for $\cat T$ is \cref{thm:costrat_implies_strat}. For the converse, we apply \cref{thm:detectingcomin} to the functors $f^*_{i} \colon \cat T \to \cat S_{i}$. Because~$\cat S_{i}$ is costratified, the theorem implies that $\cat T$ satisfies cominimality at~$\cat P$ for every $\cat P \in \im \varphi_i$. Varying over $i$, we see that cominimality holds at all primes of $\Spc(\cat T^c)$. Since $\cat T$ is stratified by assumption the local-to-global principle holds, and hence so does the colocal-to-global principle (\cref{thm:LGP-equiv}). Therefore $\cat T$ is costratified by \cref{thm:equiv-costrat}.
\end{proof}

\begin{Rem}
    In \cite[Theorem 8.11]{Stevenson13} and \cite[Corollary 5.5]{bhs1}, it was shown that stratification is---in an appropriate sense---a Zariski-local property of a rigidly-compactly generated tt-category. We complement this result by establishing the analogous statement (\cref{thm:costrat_cover}) for costratification.
\end{Rem}

\begin{Prop}\label{prop:colocal_finite_localization}
    Let $\cat T$ be a rigidly-compactly generated tt-category with $\Spc(\cat T^c)$ weakly noetherian. 
        \begin{enumerate}
            \item If $\cat T$ satisfies the (co)local-to-global principle, then so does each finite localization $\cat T(V)$. 
            \item If $\Spc(\cat T^c) = V_1 \cup \ldots \cup V_r$ is a finite cover by complements of Thomason sets such that $\cat T(V_i)$ satisfies the (co)local-to-global principle, then so does $\cat T$.
        \end{enumerate}
\end{Prop}

\begin{proof}
	The statements for the local-to-global principle were proven in \cite[Corollary 3.13]{bhs1} and \cite[Proposition 3.17]{bhs1}, respectively. We established that the local-to-global principle and the colocal-to-global principle are equivalent in \cref{thm:LGP-equiv}.
\end{proof}

\begin{Rem}
	Stevenson proves in \cite[Proposition 8.4]{Stevenson13} that the finite localization functor $\cat T\to\cat T(Y^c)$ associated with a Thomason subset $Y \subseteq \Spc(\cat T^c)$ induces an equivalence of stalks
		\[
			\GammaP\cat T \xrightarrow{\sim} \GammaP(\cat T(Y^c))
		\]
	for any $\cat P \in Y^c$. His result was stated under the assumption that $\Spc(\cat T^c)$ is noetherian, but the argument works when the space is weakly noetherian. The corresponding statement for costalks holds as well:
\end{Rem}

\begin{Prop}\label{prop:zariskicostalks}
	Suppose $Y\subseteq \Spc(\cat T^c)$ is a Thomason subset and let $\cat P \in Y^c$. We have an equivalence of costalks
		\[
			\LambdaP\cat T \xrightarrow{\sim} \LambdaP(\cat T(Y^c)).
		\]
\end{Prop}

\begin{proof}
	Let $f^*\colon\cat T\to \cat T(Y^c)$ denote the finite localization functor. As in \cref{prop:stalk-trio}, the adjunction $f_* \dashv f^!$ restricts to an adjunction
		\[
			f_*\colon \LambdaP(\cat T(Y^c)) \adjto \LambdaP\cat T\noloc f^!.
		\]
	The left adjoint is fully faithful (being the restriction of a fully faithful functor). Hence it remains to prove that the counit $f_*f^!(x)\to x$ is an isomorphism for each $x \in \LambdaP\cat T$. Indeed, $f_*f^!(\LambdaP t)\simeq \ihom{f_Y,\LambdaP t} \simeq \LambdaP t$ for all $t \in \cat T$, which gives the desired claim.
\end{proof}

\begin{Prop}\label{prop:minim_zariski}
    Let $U\subseteq \Spc(\cat T^c)$ be the complement of a Thomason subset. If~$\cat P \in U$ is a weakly visible point, then $\cat T$ satisfies cominimality at $\cat P \in \Spc(\cat T^c)$ if and only if the localized category $\cat T(U)$ satisfies cominimality at $\cat P \in \Spc(\cat T(U)^c)$.
\end{Prop}

\begin{proof}
	Let $f^*\colon\cat T\to \cat T(U)$ denote the localization functor. \Cref{prop:zariskicostalks} establishes that the induced adjunction $f_*:\LambdaP\cat T(U) \adjto \LambdaP\cat T:f^!$ of \cref{prop:stalk-trio} is an adjoint equivalence. We thus have an inclusion-preserving bijection between the colocalizing subcategories of $\LambdaP\cat T(U)$ and the colocalizing subcategories of $\LambdaP\cat T$ given by $\cat C \mapsto \cat C'\coloneqq f_*(\cat C)$ with inverse $\cat C' \mapsto \cat C\coloneqq f^!(\cat C')$. Moreover, since the inclusion of the costalk $\LambdaP\cat T \hookrightarrow \cat T$ preserves products, the colocalizing subcategories of $\LambdaP\cat T$ are precisely the colocalizing subcategories of $\cat T$ which are contained in $\LambdaP\cat T$ --- and similarly for the costalk $\LambdaP\cat T(U)\hookrightarrow \cat T(U)$. We claim that under this correspondence $\cat C \mapsto \cat C'$ of colocalizing subcategories, $\cat C$ is a coideal of $\cat T(U)$ if and only if $\cat C'$ is a coideal of $\cat T$. This follows from the adjunction isomorphisms $\ihomT{t,f_*(s)}\simeq f_*\ihomTU{f^*(t),s}$ and $f_*\ihomTU{s,f^!(t)}\simeq \ihomT{f_*(s),t}$. This establishes an inclusion-preserving correspondence between the colocalizing coideals of $\cat T(U)$ contained in $\LambdaP\cat T(U)$ and the colocalizing coideals of $\cat T$ contained in $\LambdaP\cat T$. Hence $\LambdaP \cat T(U)$ is minimal among colocalizing coideal of $\cat T(U)$ if and only if $\LambdaP \cat T$ is minimal among colocalizing coideals of $\cat T$.
\end{proof}

\begin{Cor}\label{cor:zariski_descent_minimality}
    The category $\cat T$ satisfies cominimality at a point $\cat P \in \Spc(\cat T^c)$ if and only if the local category $\cat T/\langle \cat P \rangle$ satisfies cominimality at its unique closed point. 
\end{Cor}

\begin{proof}
	Apply \cref{prop:minim_zariski} with $U=\gen(\cat P)$.
\end{proof}

\begin{Cor}[Zariski descent]\label{thm:costrat_cover}
    Let $\cat T$ be a rigidly-compactly generated tt-category with $\Spc(\cat T^c)$ weakly noetherian and which satisfies the local-to-global principle. Suppose $\Spc(\cat T^c) = \bigcup_{i \in I}V_i$ is a cover by complements of Thomason subsets. Then~$\cat T$ is costratified if and only if each finite localization $\cat T(V_i)$ is costratified. 
\end{Cor}

\begin{proof}
	Each $\cat T(V_i)$ has the colocal-to-global principle by \cref{prop:colocal_finite_localization} and \cref{thm:LGP-equiv}. Hence it suffices to consider cominimality (\cref{thm:equiv-costrat}) and we can invoke \cref{prop:minim_zariski}.
\end{proof}

\begin{Rem}
	We end this section with two descent techniques for stratification which empower our bootstrap theorem (cf.~\cref{rem:bootstrap-empowerment}). The first is a modification of the argument in \cite[Section~2.2.2]{barthel2021rep1} establishing finite \'etale descent, which in turn is a generalization of \cite[Theorem 6.4]{bhs1}.
\end{Rem}

\begin{Thm}[Quasi-finite descent]\label{thm:etaledescent-strat}
    Let $f^*\colon \cat T \to \cat S$ be a geometric functor as in \cref{hyp:geometric-base-change}. Suppose $f^*$ is {\finite} and $\varphi$ is surjective with discrete fibers. If $\cat S$ is stratified then so is $\cat T$.
\end{Thm}

\begin{proof}
	Since~$f^*$ is {\weaklyfinite}, the surjectivity assumption on~$\varphi$ implies that both $f^!$ and~$f^*$ are conservative (recall \cref{prop:weaklyfinite-surjective-conservative}). Hence the local-to-global principle descends from $\cat S$ to $\cat T$ by \cref{ex:localtoglobaldescent} and it suffices to check minimality at each prime in $\cat T$. The hypotheses on the functor $f^*$ are local in the target (cf.~\cite[Proposition~1.30]{bhs1}) and hence it suffices to assume that $\cat T$ is local and check minimality at the unique closed point $\mfrak m \in \Spc(\cat T^c)$.

	Since the Thomason closed subset $\varphi^{-1}(\{\mfrak m\})$ is discrete (by hypothesis), an elementary topological argument verifies that the fiber $\varphi^{-1}(\{\mfrak m\})$ consists of finitely many visible closed points. Now consider an object $t \in \cat T$. Recall that $\mfrak m \in \SuppT(t)$ if and only if $\mfrak m \in \CosuppT(t)$ (\cref{thm:general-min}). We claim that if $\mfrak m \in \SuppT(t)$ then $\varphi^{-1}(\{\mfrak m\}) \subseteq \SuppS(f^*(t)) \cap \CosuppS(f^!(t))$. To this end, let $\cat Q \in \varphi^{-1}(\{\mfrak m\})$. Since~$\cat Q$ is a visible closed point, there is some $z \in \cat S^c$ such that $\supp(z) = \{\cat Q\}$. Moreover, since the right adjoint $f_*$ is weakly conservative (\cref{rem:weaklyconservative}) there exists some $x \in \cat S^c$ such that $f_*(z\otimes x)\neq 0$; but $\supp(z\otimes x)=\{\cat Q\}$, so replacing $z$ by $z\otimes x$ if necessary we may assume without loss of generality that $f_*(z) \neq 0$. Note that $\emptyset \neq \Supp(f_*(z)) \subseteq \varphi(\Supp(z))=\{\mfrak m\}$ by \cref{prop:general-object_f*} so that $\Supp(f_*(z)) = \{\mfrak m\}$. Since $f_*(z)$ is compact, $\Supp(t\otimes f_*(z)) = \Supp(t) \cap \Supp(f_*(z)) = \{\mfrak m\}$ by \cite[Lemma 2.18]{bhs1}. In particular, $t \otimes f_*(z) \neq 0$ so that $f^*(t) \otimes z\neq 0$ and thus $\cat Q \in \SuppS(f^*(t))$. This establishes that $\varphi^{-1}(\{\mfrak m\}) \subseteq \SuppT(t)$. Now consider $\CosuppS(f^!(t))$. Again, since $f_*(z)$ is compact we have $\CosuppT(\ihom{f_*(z),t}) = \SuppT(f_*(z)) \cap \CosuppT(t) = \{\mfrak m\}$ by \cref{prop:halfhom}. Hence $f_*\ihom{z,f^!(t)}=\ihom{f_*(z),t}\neq 0$ so that $\ihom{z,f^!(t)} \neq 0$ and thus $\cat P \in \CosuppT(f^!(t))$. 

	In summary, if $t_1,t_2\in \cat T$ are objects with $\mfrak m$ contained in their (co)support then
		\begin{align*}
			\emptyset \neq \varphi^{-1}(\{\mfrak m\}) &\subseteq \SuppS(f^*(t_1)) \cap \CosuppS(f^!(t_2)) \\
													  &= \CosuppS(\ihom{f^*(t_1),f^!(t_2)}),
		\end{align*}
	where the equality uses the assumption that $\cat S$ is stratified and applies \cref{thm:strat_cosupport}. Therefore, $f^!\ihom{t_1,t_2} = \ihom{f^*(t_1),f^!(t_2)} \neq 0$ and hence $\ihom{t_1,t_2} \neq 0$. This establishes minimality at $\mfrak m$ by \cref{lem:minimality_crit}.
\end{proof}

\begin{Cor}\label{cor:quasifinite-costrat}
	If $f^*$ is {\finite}, $\varphi$ is surjective with discrete fibers, and $\cat S$ is costratified, then $\cat T$ is costratified as well.
\end{Cor}

\begin{proof}
	Since $\cat S$ is costratified, it is also stratified by \cref{thm:costrat_implies_strat}. It then follows from \cref{thm:etaledescent-strat} that $\cat T$ is also stratified. We can now invoke \cref{cor:detectingcomin} since $f^*$ is conservative due to \cref{prop:weaklyfinite-surjective-conservative}.
\end{proof}

\begin{Rem}
	The name of \cref{thm:etaledescent-strat} is motivated by algebraic geometry. A finite type morphism $f\colon X\to Y$ of schemes is quasi-finite in the sense of \cite[D\'{e}finition 6.2.3]{EGAII} if and only if it has discrete fibers.
\end{Rem}

\begin{Rem}
    The next result is inspired by work of Shaul and Williamson \cite{shaulwilliamson2020lifting} in the context of BIK-stratification and BIK-costratification. The following tt-geometric version strengthens \cite[Theorem 2.24]{barthel2021rep1}.
\end{Rem}

\begin{Thm}[Nil-descent]\label{thm:nilpotentdescent}
	Let $f^*\colon\cat T\to\cat S$ be a geometric functor as in \cref{hyp:geometric-base-change}. Suppose $f^!$ is conservative and $\varphi$ is injective. If $\cat S$ is stratified then so is~$\cat T$.
\end{Thm}

\begin{proof}
	Since $f^!$ is conservative, $\cat T$ inherits the local-to-global principle from~$\cat S$ (\cref{rem:detection-descends}). To establish that $\cat T$ is stratified we will use \cref{thm:strat_cosupport}. To this end, let $t_1,t_2\in \cat T$ be objects with $\ihom{t_1,t_2}=0$. Our goal is to prove that $\Supp(t_1)\cap\Cosupp(t_2)=\emptyset$. We have $0=f^!\ihom{t_1,t_2}=\ihom{f^* t_1,f^! t_2}$ by \eqref{eq:f!hom}, hence
		\[
			\Supp(f^* t_1)\cap\Cosupp(f^! t_2) =\emptyset
		\]
    since $\cat S$ is assumed stratified. The assumption that $\varphi$ is injective then implies
		\begin{equation}\label{eq:nil-descent-equation} 
			\varphi(\Supp(f^*t_1))\cap \varphi(\Cosupp(f^! t_2)) = \emptyset.
		\end{equation}
	Since $\cat S$ has detection and codetection, \cref{prop:phi-f*} and \cref{prop:phi-f!} imply that \eqref{eq:nil-descent-equation} is the same as 
		\[
			\Supp(t_1\otimes f_*(\unitS)) \cap {\Cosupp(\ihom{f_*\unit,t_2}) = \emptyset}.
		\]
	Since $f^!$ is conservative, $f^*$ is also conservative by \cref{prop:f!conservative}. \Cref{prop:phi-f*} and \cref{prop:phi-f!} then imply that $\Supp(t_1\otimes f_*(\unitS))=\Supp(t_1)$ and $\Cosupp(\ihom{f_*(\unitS),t_2})=\Cosupp(t_2)$, which completes the proof.
\end{proof}

\begin{Cor}\label{cor:nilpotentdescent}
	If $f^!$ is conservative and $\varphi$ is injective (hence bijective) and $\cat S$ is costratified, then $\cat T$ is costratified as well.
\end{Cor}

\begin{proof}
	Note that $f^*$ is also conservative by \cref{prop:f!conservative}. The claim then follows from \cref{thm:costrat_implies_strat}, \cref{thm:nilpotentdescent}, and \cref{cor:detectingcomin}.
\end{proof}

\begin{Rem}
	If $f^*$ is {\weaklyfinite} and $\varphi$ is a bijection then we can invoke \cref{thm:nilpotentdescent} and \cref{cor:nilpotentdescent} because of \cref{prop:weaklyfinite-surjective-conservative}.
\end{Rem}

\newpage
\part{Applications and examples}\label{part:applications-examples}

We now turn to applications and examples. We show that a pure-semisimple tt-category (in particular, any tt-field) is costratified (\cref{thm-pure-semisimple-costratified}) and that any affine weakly regular tt-category is costratified (\cref{thm:costratification_for_weakly_affine}). We also show that the derived category of a (topologically weakly noetherian) quasi-compact and quasi-separated scheme is costratified if and only if it is stratified (\cref{thm:qc_shseaves}). In representation theory, we show that the category of $k$-linear representations~$\Rep(G,k)$ and the derived category of \mbox{$k$-linear} permutation modules $\mathrm{DPerm}(G,k)$ is costratified for any finite group~$G$ and field $k$ (\cref{thm:rep_g_k_costratifcation}, \cref{thm:dpermstratification} and \cref{thm:rep_e_k_costratifcation}). In homotopy theory, we give a new proof that the category of $E(n)$-local spectra is costratified (\cref{thm:hovey-strickland}) and prove that certain cochain algebras have costratified derived categories (\cref{thm:cochains}). We establish that the category of rational $G$-spectra is costratified for any compact Lie group~$G$ (\cref{thm:rational-spectra}) and show that, for any finite group~$G$, the category of spectral $G$-Mackey functors valued in a commutative algebra $\mathbb{E} \in \CAlg(\Sp)$ is costratified whenever $\Der(\mathbb{E})$ itself is costratified and has a noetherian spectrum (\cref{thm:equivariant_costratification}). As a special case, this shows that Kaledin's category of derived Mackey functors is costratified (\cref{cor:derived_mackey_costratified}).

\section{Tensor triangular examples}\label{sec:ttexamples}

\subsection*{Pure-semisimple categories and tt-fields}\label{ssec:ttfieldexamples}

\begin{Lem}\label{lem:coideal=ideal}
	Let $\cat T$ be a rigidly-compactly generated tt-category with weakly noetherian spectrum. Suppose that every colocalizing coideal of $\cat T$ is also a localizing ideal. Then $\cat T$ is stratified if and only if it is costratified. 
\end{Lem}

\begin{proof}
	In light of \cref{thm:LGP-equiv} and \cref{thm:costrat_implies_strat}, it suffices to show that if minimality holds at a point $\cat P\in \Spc(\cat T^c)$ then cominimality also holds at~$\cat P$. Under the equivalence $\LambdaP\cat T\cong \GammaP\cat T$ of \cref{rem:stalk-costalk-equivalence}, the thick subcategories $\cat C$ of $\LambdaP \cat T$ correspond bijectively with the thick subcategories $\cat C'$ of $\GammaP\cat T$. Moreover, one readily checks that $\cat T^d \otimes \cat C \subseteq \cat C$ if and only if $\cat T^d \otimes \cat C' \subseteq \cat C'$. By our hypothesis, the costalk $\LambdaP\cat T$ is a localizing ideal of $\cat T$. In particular, the inclusion $\LambdaP\cat T \hookrightarrow \cat T$ preserves coproducts. It follows that if~$\cat C \subseteq \LambdaP\cat T$ is a localizing ideal of~$\cat T$ then its corresponding $\cat C'\subseteq \GammaP\cat T$ is also a localizing ideal of $\cat T$. Armed with these observations, suppose $\cat C \subsetneq \LambdaP\cat T$ is a proper colocalizing coideal. By hypothesis, it is a localizing ideal of $\cat T$, hence corresponds to a proper localizing ideal $\cat C' \subsetneq\GammaP\cat T$. Minimality at $\cat P$ implies that $\cat C'=0$ and hence~$\cat C=0$, so that we have cominimality at $\cat P$.
\end{proof}

\begin{Def}[Beligiannis \cite{Beligiannis00}, Krause \cite{Krause00}]\label{def:puresemisimple}
	A compactly generated triangulated category $\cat T$ is said to be \emph{pure-semisimple} if every pure monomorphism splits, where $f\colon x\to y$ is a pure monomorphism if the induced map $\cat T(c,x)\to\cat T(c,y)$ is a monomorphism for all compact objects $c\in \cat T^c$. This is equivalent to a host of other conditions, for example that any object of $\cat T$ is a coproduct of compact objects, that~$\cat T$ has all filtered colimits, or that~$\cat T$ is phantomless; see \cite[Theorem~9.3]{Beligiannis00} and \cite[Theorem~2.10]{Krause00}.
\end{Def}

\begin{Prop}\label{prop-balmer-spectrum-pure-semisimple}
	If $\cat T$ is a rigidly-compactly generated pure-semisimple tt-category then $\Spc(\cat T^c)$ is a finite discrete space.
\end{Prop}

\begin{proof}
	We first prove that every Thomason subset $Y \subseteq \Spc(\cat T^c)$ is both open and closed. Since $\cat T$ is pure-semisimple, we can write $f_Y = \coprod_{i\in I} x_i$ as a coproduct of compact objects. Note that each $x_i$ is contained in $\ker(-\otimes e_Y)=\Loco{f_Y}$; see \cref{rem:smashing-recollement}. Since this is a localizing ideal, it also contains the duals $x_i^{\vee}$, and hence
		\[
			\cat T(f_Y,\Sigma e_Y) = \prod_{i \in I} \cat T(x_i, \Sigma e_Y) =\prod_{i\in I}\cat T(\unit,\Sigma e_Y \otimes x_i^{\vee})=0.
		\]
	Thus \cref{prop:tate} establishes that $Y$ is both open and closed.

	Next we establish that the specialization order is trivial (cf.~\cref{rem:hausdorff}). To this end, suppose there exists an inclusion of primes $\cat P \subseteq \cat Q$ with $\cat P \neq \cat Q$. Then $\cat Q \not\subseteq \cat P$, i.e., $\cat Q \not\in \overbar{\singP}$. Hence there exists a Thomason closed subset $Z$ which contains $\cat P$ but does not contain $\cat Q$. This is a contradiction, since the Thomason subset $Z$ is open (by the first part of the proof) and hence closed under generalization.

	It follows that $\singP = \gen(\cat P)$ is open since it is the complement of a Thomason subset. Since every point is open, the topology is discrete. Moreover, a discrete quasi-compact space is necessarily finite.
\end{proof}

\begin{Thm}\label{thm-pure-semisimple-costratified}
	If $\cat T$ is a rigidly-compactly generated pure-semisimple tt-category, then $\Spc(\cat T^c)$ is a finite discrete space and $\cat T$ is costratified and stratified. 
\end{Thm}

\begin{proof}
	\Cref{prop-balmer-spectrum-pure-semisimple} establishes that the Balmer spectrum is finite and discrete. We first show that $\cat T$ is stratified.  By \cite[Theorem 3.21]{bhs1} the local-to-global principle holds for $\cat T$. We establish minimality using \Cref{lem:minimality_crit}. To that end, fix $\cat P \in \Spc(\cat T^c)$ and consider nonzero objects $t_1,t_2 \in \GammaP \cat T$. Since $\cat T$ is pure-semisimple, we have $t_1 = \coprod_{i \in I}x_i$ and $t_2 = \coprod_{j \in J}y_j$ for nonzero compact objects $x_i,y_j$. Since $x_i$ (respectively, $y_j$) is a retract of $t_1$ (respectively, $t_2$), we see that each $x_i$ and $y_j$ is in $\GammaP \cat T$ as well. Thus $\supp(x_i) = \supp(y_j) =  \{ \cat P \}$ and hence $\ihom{x_i,y_j} \ne 0$. It follows that $\ihom{t_1,t_2} \ne 0$, as required. 

	For any collection of objects $\{t_i\}_{i \in I}$ in a pure-semisimple tt-category $\cat T$, the canonical map $\coprod_{i\in I} t_i \to \prod_{i \in I} t_i$ from the coproduct to the product splits; see \cite[Theorem~9.3]{Beligiannis00}. In particular, any colocalizing subcategory of $\cat T$ is localizing. Moreover, since a colocalizing subcategory of $\cat T$ is a coideal if and only if it is closed under tensoring with dualizable objects (\cref{exa:submodules-of-Top}), we conclude that any colocalizing coideal is a localizing ideal. Hence we can invoke \cref{lem:coideal=ideal} to conclude that $\cat T$ is costratified.
\end{proof}

\begin{Exa}
	Let $A$ be an Artin algebra which is derived equivalent to a hereditary algebra of Dynkin type. Then the homotopy category $\KInj{A}$ of complexes of injective $A$-modules is pure-semisimple \cite[Proposition 4.4]{Zhe13pp} and hence stratified and costratified by \Cref{thm-pure-semisimple-costratified}.
\end{Exa}

\begin{Rem}
	As a consequence of \Cref{prop-balmer-spectrum-pure-semisimple}, we see that an infinite product of nontrivial pure-semisimple tt-categories cannot be pure-semisimple. This should be compared with the following observation: An infinite product of nontrivial pure-semisimple commutative rings is never pure-semisimple. Indeed, a pure-semisimple ring is always artinian \cite[Theorem 4.4]{Chase60} and, in particular, its Zariski spectrum is always finite discrete. We note also that if $\Der(R)$ is pure-semisimple in the sense of \Cref{def:puresemisimple} then $R$ is pure-semisimple \cite[Corollary 7.2]{GarkushaPrest05}. We thus see that for a field $k$, neither $\prod_{i = 1}^{\infty} \Der(k)$ nor $\Der(\prod_{i = 1}^{\infty}k)$ is pure-semisimple, even though~$\Der(k)$ is pure-semisimple.
\end{Rem}

\begin{Exa}\label{rem:ttfield}
	In \cite[Definition 1.1]{BalmerKrauseStevenson19}, Balmer--Krause--Stevenson define a \emph{tt-field} to be a rigidly-compactly generated tt-category $\cat F$ for which every object $X \in \cat F$ is a coproduct $X \simeq \coprod_{i \in I}x_i$ of compact-rigid objects $x_i \in \cat F^c$ and for which each object $X \in \cat F$ is $\otimes$-faithful ($X \otimes f = 0 \implies f = 0$). A tt-field $\cat F$ is pure-semisimple and $\Spc(\cat F^c)=\{\ast\}$ is a single point; see \cite[Proposition 5.1]{BalmerKrauseStevenson19}.
\end{Exa}

\begin{Exa}\label{ex:puresemisimple}
	The stable module category $\cat T = \StMod(kC_{p^n})$ is a pure-semisimple triangulated category which, if $n \ge 2$, is not a tt-field; see \cite[Example~5.11]{BalmerKrauseStevenson19}.
\end{Exa}

\begin{Cor}\label{cor:stratttfields}
	Any tt-field $\cat F$ is both stratified and costratified. 
\end{Cor}

\subsection*{Affine weakly regular tt-categories}\label{ssec:ttregularexamples}

\begin{Def}[Dell'Ambrogio--Stanley \cite{DellAmbrogioStanley16}] \label{def:affine_weakly_regular}
    A tensor-triangulated category~$\cat T$ is said to be \emph{affine weakly regular} if it satisfies the following two conditions:
		\begin{enumerate}
			\item (affine) $\cat T$ is compactly generated by its tensor unit $\unit$.
			\item (weakly regular) The graded endomorphism ring $R \coloneqq \Hom_{\cat T}^*(\unit,\unit)$ is a graded noetherian ring concentrated in even degrees, and for every homogeneous prime ideal $\mathfrak p$ of $R$, the maximal ideal of the local ring~$R_{\mathfrak p}$ is generated by a (finite) regular sequence of homogeneous non-zero-divisors.  
		\end{enumerate}
	The first axiom ensures that $\cat T$ is a rigidly-compactly generated tt-category and that every (co)localizing subcategory is a (co)ideal (cf.~\cref{rem:monogenic}).
\end{Def}

\begin{Rem}\label{Rem:residue_field_affine}
    Given an affine weakly regular tt-category we set
		\[
		\pi_*(X) \coloneqq \Hom^*_{\cat T}(\unit,X). 
		\]
    For any prime ideal $\mathfrak p \in \Spec^h(R)$, there is a residue field object $K(\mathfrak p) \in \cat T$ with the property that 
		\[
			\pi_*(K(\mathfrak p)) \simeq \kappa(\mathfrak p)
		\]
	where $\kappa(\mathfrak p) \coloneqq R_{\mathfrak p}/\mathfrak pR_{\mathfrak p}$ denotes the algebraic residue field; see \cite[\S 3]{DellAmbrogioStanley16}. This object plays a role analogous to $g_{\mathfrak p}$; in particular, $\Loc\langle g_{\mathfrak p}\rangle = \Loc\langle K(\mathfrak p)\rangle$. The following theorem states the main results of \cite{DellAmbrogioStanley16}:
\end{Rem}

\begin{Thm}[Dell'Ambrogio--Stanley]
    Let $\cat T$ be an affine weakly regular tt-category. There is a homeomorphism $\Spc(\cat T^c) \cong \Spec^h(R)$ and $\cat T$ is stratified.
\end{Thm}

\begin{Lem}\label{lem:retract}
	For any set $I$ and integers $\{d_i\}_{i \in I}$, the natural map 
		\begin{equation}\label{eq:coprod-to-prod}
		\begin{aligned}
			\coprod_{i \in I}\Sigma^{d_i}K(\mathfrak p) \longrightarrow \prod_{i \in I} \Sigma^{d_i}K(\mathfrak p)
		\end{aligned}
		\end{equation}
	is a split monomorphism.
\end{Lem}

\begin{proof}
	Applying $\pi_*$ to the morphism \eqref{eq:coprod-to-prod} gives the monomorphism of graded \mbox{$\kappa (\mathfrak p)$-modules} $\bigoplus_{i \in I} \Sigma^{d_i}\kappa(\mathfrak p) \to \prod_{i\in I} \Sigma^{d_i}\kappa(\mathfrak p)$. We can extend the standard basis of $\bigoplus_{i\in I} \Sigma^{d_i}\kappa(\mathfrak p)$ to a basis $\cat B$ of $\prod_{i\in I} \Sigma^{d_i}\kappa(\mathfrak p)$. As $\prod_{i\in I} \Sigma^{d_i}K(\mathfrak p) \simeq {K(\mathfrak p)\otimes \prod_{i\in I} \Sigma^{d_i}\unit},$ Proposition 3.6 of \cite{DellAmbrogioStanley16} provides an isomorphism
		\[
			\coprod_{b\in \cat B}\Sigma^{|b|}K(\mathfrak p) \simeq \prod_{i \in I} \Sigma^{d_i}K(\mathfrak p)
		\]
	where $|b|$ denotes the homological degree of $b$. Any set-theoretic splitting of $I \subseteq \cat B$ then lifts to a left-inverse of \eqref{eq:coprod-to-prod}. 
\end{proof}

\begin{Thm}\label{thm:costratification_for_weakly_affine}
	Let $\cat T$ be an affine weakly regular tt-category. Then $\cat T$ is costratified. 
\end{Thm}

\begin{proof}
	The proof follows \cite[Theorem 10.3]{BensonIyengarKrause12}. The colocal-to-global principle holds by \Cref{cor:noetherian_local_global}, so we only need to establish cominimality at every 
	${\mathfrak p \in \Spec^h(R)}$. To that end, let $0 \neq t \in \Lambda^{\mathfrak p}\cat T$. Then, since $\cat T$ is stratified we have
		\[
			\Coloc\langle t\rangle = \Coloc\langle\ihom{g_{\mathfrak p},t}\rangle = \Coloc\langle \ihom{K(\mathfrak p),t}\rangle
		\]
	where the last step uses that $\Loc\langle g_{\mathfrak p}\rangle = \Loc\langle K(\mathfrak p)\rangle$ and \cref{lem:eqlococosupp}. Now a similar argument to \cite[Proposition 3.6]{DellAmbrogioStanley16} shows that 
		\[
			\ihom{K(\mathfrak p),t} \simeq \coprod\nolimits_{\alpha} \Sigma^{|\alpha|} K(\mathfrak p)
		\]
	where $\alpha$ runs through a graded $\kappa(\mathfrak p)$-vector space basis of $\pi_*\ihom{K(\mathfrak p),t}$. It follows that $K(\mathfrak p) \in \Coloc\langle \ihom{K(\mathfrak p),t}\rangle$. Moreover, by \cref{lem:retract} we see that $\coprod_{\alpha} \Sigma^{|\alpha|} K(\mathfrak p)$ is a retract of $\prod_{\alpha} \Sigma^{|\alpha|} K(\mathfrak p)$, hence we have 
		\[
			\ihom{K(\mathfrak p),t} \simeq \coprod\nolimits_{\alpha} \Sigma^{|\alpha|} K(\mathfrak p) \in \Coloc\langle K(\mathfrak p)\rangle.
		\]
	Therefore, we obtain
		\[
			\Coloc\langle t\rangle = \Coloc\langle \ihom{K(\mathfrak p),t}\rangle= \Coloc\langle K(\mathfrak p)\rangle.
		\]
	If $\Coloc\langle t\rangle$ were a proper subcategory of $\Lambda^{\mathfrak p}\cat T$, then any object $s$ in the complement would have to satisfy $\Coloc\langle t\rangle = \Coloc\langle s\rangle$, which is absurd. Therefore, $\Lambda^{\mathfrak p}\cat T$ is minimal as a colocalizing subcategory of $\cat T$, as required.
\end{proof}

\subsection*{Finite products of tt-categories}

\begin{Exa}\label{Exa:finite_product_of_costratified_categories}
	Let $\{ \cat T_i \}_{i=1}^n$ be  a finite collection of rigidly-compactly generated tensor-triangulated categories. Their product $\prod_{i=1}^n \cat T_i$ is again rigidly-compactly generated and $(\prod_{i=1}^n \cat T_i)^c=\prod_{i=1}^n \cat T_i^c$. Moreover, 
		\[
			\Spc( \prod_{i=1}^n \cat T_i^c) \cong \coprod_{i=1}^n \Spc( \cat T_i^c)
		\]
	where the right-hand side has the disjoint union topology. Indeed, each prime ideal of $\prod_{i=1}^n \cat T_i^c$ is of the form $\prod_{i=1}^n \cat P_i$, where for some $k$, $\cat P_k$ is a prime ideal of~$\cat T_k^c$, and for $i \ne k$, $\cat P_i = \cat T_i^c$. Thus $\prod_{i=1}^n \cat T_i$ has (weakly) noetherian spectrum if and only if each~$\cat T_i$ has (weakly) noetherian spectrum. Moreover, by Zariski descent (\cref{prop:colocal_finite_localization,cor:zariski_descent_minimality}), $\prod_{i=1}^n \cat T_i$ is costratified if and only if each~$\cat T_i$ is costratified.
\end{Exa}

\begin{Exa}\label{exa:double-T}
	Let $\cat T$ be any rigidly-compactly generated tt-category. Let $\cat T\times_{\bbZ_2} \cat T$ denote the product category $\cat T\times \cat T$ with the $\bbZ_2$-graded tensor-triangulated structure described in \cite[Example~4.23]{Sanders22}. It is again rigidly-compactly generated and the inclusion $a \mapsto (a,0)$ is a fully faithful geometric functor $f^*\colon\cat T \hookrightarrow \cat T\times \cat T$. The projection onto the first coordinate is both left and right adjoint to $f^*$. Hence~$f^!=f^*$. The prime ideals of $(\cat T \times \cat T)^c = \cat T^c \times \cat T^c$ are $\cat P \times \cat P$ for $\cat P$ a prime ideal of $\cat T^c$, and the induced map
		\[
			\Spc(f^*)\colon\Spc(\cat T^c \times \cat T^c) \to \Spc(\cat T^c) 
		\]
	given by $\cat P \times \cat P \mapsto \cat P$ is a homeomorphism. Under this identification, one readily checks that $\Supp_{\cat T \times \cat T}( (t_0,t_1) ) = \Supp_{\cat T}(t_0) \cup \Supp_{\cat T}(t_1)$ and $\Cosupp_{\cat T \times \cat T}( (t_0,t_1) ) = \Cosupp_{\cat T}(t_0) \cup \Cosupp_{\cat T}(t_1)$. It is then straightforward to establish that if $\cat T$ is costratified then $\cat T \times_{\bbZ_2} \cat T$ is costratified. The converse follows from \cref{thm:nilpotentdescent} and \cref{cor:detectingcomin}.
\end{Exa}

\section{Algebraic examples}\label{sec:algebraicexamples}

\subsection*{Commutative rings and schemes}\label{ssec:alggeoexamples}

\begin{Prop}\label{prop:costratification_for_dr}
	Let $R$ be a commutative ring with $\Spec(R)$ weakly noetherian. Then $\Der(R)$ is stratified if and only if it is costratified. 
\end{Prop}

\begin{proof}
	Recall that $\Spc(\Der(R)^c) \cong \Spec(R)$ by Thomason's theorem \cite{Thomason97}; see \cite[Example 1.36]{bhs1}. For each $\mathfrak p \in \Spec(R)$, we consider the base-change functor $f^*_{\mathfrak p} \colon \Der(R) \to \Der(\kappa(\mathfrak p))$, where $\kappa(\mathfrak p)$ denotes the residue field at~$\mathfrak p$. This is a geometric functor whose target is a tt-field and the induced map on spectra sends the unique point $\ast$ to $\mathfrak p$. The result then follows from \Cref{prop:bootstrap_for_algebraic_categories} and \Cref{cor:stratttfields}. 
\end{proof}

\begin{Exa}\label{ex:costratified_noetherian_ring}
    For any commutative noetherian ring $R$, we then deduce the main result of \cite{Neeman11} from the original \cite{Neeman92a}: $\Der(R)$ is costratified. 
\end{Exa}

\begin{Exa}
	Let $R$ be an absolutely flat ring which is not noetherian. Stevenson \cite{Stevenson14,Stevenson17} proves that $R$ is semi-artinian $\iff$ the local-to-global principle for $\Der(R)$ holds $\iff$ $\Der(R)$ is stratified. By \Cref{prop:costratification_for_dr}, this is also equivalent to~$\Der(R)$ being costratified. For example, this applies to the subring of $\prod_{\mathbb{N}} \bbF_p$ consisting of those sequences which are eventually constant; cf.~\cite[Example~3.25]{bhs1}.
\end{Exa}

\begin{Rem}\label{rem:Dqc(X)}
	We can extend \cref{prop:costratification_for_dr} to derived categories of schemes. For a quasi-compact and quasi-separated scheme $X$, let $\Derqc(X)$ denote the derived category of complexes of $\cat O_X$-modules with quasi-coherent cohomology. It is rigidly-compactly generated and its subcategory of rigid-compact objects $\Derqc(X)^c=\Derperf(X)$ is the derived category of perfect complexes. A fundamental result concerning the Balmer spectrum is that $\Spc(\Derperf(X)) \cong X$. See \cite[Theorem~6.3]{Balmer05a}, \cite[Theorem~9.5]{BuanKrauseSolberg07}, and~\cite{Thomason97}. 
\end{Rem}

\begin{Thm}\label{thm:qc_shseaves}
	Let $X$ be a quasi-compact and quasi-separated scheme which is topologically weakly noetherian. The derived category $\Derqc(X)$ is stratified if and only if it is costratified.
\end{Thm}

\begin{proof}
	One direction is \Cref{thm:costrat_implies_strat}. For the converse, suppose that $\Derqc(X)$ is stratified. It satisfies the local-to-global principle by \cite[Theorem 4.1]{bhs1} and hence the colocal-to-global principle by \cref{thm:LGP-equiv}. Now take an open affine cover of $X$ by subsets $V_i = \Spec(A_i)$. Each of these affine schemes is also topologically weakly noetherian (cf.~\cite[Remark 2.6]{bhs1}) and we have the implications $\Derqc(X)$ is stratified $\iff$ each $\Der(A_i)$ is stratified (\cite[Corollary~5.5]{bhs1}) $\iff$ each $\Der(A_i)$ is costratified (\Cref{prop:costratification_for_dr}) $\iff \Derqc(X)$ is costratified (\Cref{thm:costrat_cover}). 
\end{proof}

\begin{Exa}
    If $X$ is noetherian, then it follows from \cite[Corollary~5.10]{bhs1} that $\Derqc(X)$ is stratified and hence is also costratified. For such $X$, the costratification of $\Derqc(X)$ has also been obtained in recent work of Verasdanis \cite{Verasdanis22bpp}.
\end{Exa}

\subsection*{Derived categories of representations}\label{ssec:reptheoryexamples}

\begin{Def}
	For a finite group $G$ and commutative ring $R$, we let
		\[
			\Rep(G,R) \coloneqq \Ind \Fun(BG,\Perf(R))
		\]
	denote the derived $\infty$-category of $R$-linear $G$-representations, where $\Ind$ denotes ind-completion (\cite[Section 5.3.5]{HTTLurie}), and let
		\[
			\StMod(G,R) \coloneqq \Ind \Fun(BG,\Perf(R))/\Perf(R[G]).
		\]
	By construction, both of these categories are rigidly-compactly generated symmetric monoidal stable $\infty$-categories; passage to their homotopy categories yields the corresponding rigidly-compactly generated tt-categories. Moreover, up to idempotent completion, the category of compact objects in~$\StMod(G,R)$ is obtained as a finite localization of the compact objects in~$\Rep(G,R)$. 
\end{Def}

\begin{Rem}
    If $R = k$ is a field, then the category~$\Rep(G, k)$ is equivalent to~$\KInj{k[G]}$, the homotopy category of unbounded complexes of injective $k[G]$-modules, as studied in \cite{BensonKrause2008Complexes}, and the category $\StMod(G,k)$ agrees with the usual stable module category.
\end{Rem}

\begin{Rem}
    The next result is originally due to Benson--Iyengar--Krause \cite[Theorem 11.6]{BensonIyengarKrause12} with the computation of the spectrum due to Benson--Carlson--Rickard \cite{BensonCarlsonRickard97}. We will give an alternative proof via our bootstrap theorem, relying on a result we prove at the end of \cref{ssec:galois} using Galois ascent as in \cite{Mathew15bpp}. 
\end{Rem}

\begin{Thm}[Benson--Iyengar--Krause]\label{thm:rep_g_k_costratifcation}
	For any finite group $G$ and field~$k$, the category $\Rep(G,k)$ is costratified with spectrum $\Spc(\Rep(G,k)^c) \cong \Proj H^*(G;k)$.
\end{Thm}

\begin{proof}
	By Chouinard's theorem, as given in \cite[Proposition 9.6]{BensonIyengarKrause11a}, the functor
		\[
			\Res^G_E \colon \Rep(G,k) \xrightarrow{} \prod_{E \le G} \Rep(E,k)
		\]
	given by the product of restriction functors is a conservative geometric functor. Moreover, the target category $\prod_{E \le G}\Rep(E,k)$ is costratified by \Cref{thm:rep_e_k_costratifcation} below and \cref{Exa:finite_product_of_costratified_categories}, and $\Rep(G,k)$ is stratified by \cite[Theorem 10.1]{BensonIyengarKrause11a}. It follows from \Cref{cor:detectingcomin} that $\Rep(G,k)$ is costratified as well. 
\end{proof}

\begin{Rem}
    The forthcoming \cite{BBIKP_stratification} will establish stratification of $\Rep(G,R)$ for all finite groups $G$ and all noetherian commutative rings $R$, and then deduce costratification via the bootstrap theorem. The computation of $\Spc(\Rep(G,R)^c)$ for any commutative ring $R$ is due to Lau \cite{Lau2021Balmer}.
\end{Rem}

\begin{Rem}
	There is an enlargement of the category $\Rep(G,R)$ given by the derived category of permutation modules $\mathrm{DPerm}(G,R)$, for any (pro)finite group $G$ and commutative ring $R$. For the construction of this category as well as its relation to Artin motives and Mackey functors, we refer to~\cite{BalmerGallauer2021pp}.
\end{Rem}

\begin{Thm}\label{thm:dpermstratification}
	The derived category of permutation modules $\mathrm{DPerm}(G,k)$ is costratified for any finite group $G$ and field $k$ of characteristic~$p$ dividing the order of $G$.
\end{Thm}

\begin{proof}
	In \cite{BalmerGallauer2022pp}, the authors construct geometric functors
		\[
			\check{\Psi}_H\colon \mathrm{DPerm}(G,k) \to \Rep(W_G(H),k)
		\]
	indexed by all conjugacy classes of $p$-subgroups $H$ of $G$, where $W_G(H)$ denotes the Weyl group of $H$ in $G$. In \cite[Theorem 5.12]{BalmerGallauer2022pp}, they prove that these functors are jointly conservative, while \cite[Theorem 8.11]{BalmerGallauer2022pp} establishes stratification for $\mathrm{DPerm}(G,k)$. It then follows from our bootstrap theorem \cref{cor:detectingcomin} together with \cref{thm:rep_g_k_costratifcation} that $\mathrm{DPerm}(G,k)$ is costratified as well.
\end{proof}

\section{Homotopical examples}\label{sec:homotopicalexamples}

\subsection*{Galois descent and ascent}\label{ssec:galois}

\begin{Def}\label{def:galois}
	Let $G$ be a compact Lie group and consider a morphism $f\colon A \to B$ of commutative ring spectra, where $B$ is equipped with an $A$-linear $G$-action. Following Rognes \cite{Rognes08}, the map $f$ is a \emph{Galois extension with Galois group $G$} (or simply a \emph{$G$-Galois extension}) if it satisfies the following two conditions:
		\begin{enumerate}
			\item the canonical map $A \to B^{hG}$ is an equivalence;
			\item the canonical map $B\otimes_A B \to \ihom{G_+,B}$ is an equivalence.
		\end{enumerate}
	A Galois extension $f\colon A \to B$ is said to be \emph{faithful} if $f^*\colon \Mod(A) \to \Mod(B)$ is conservative. 
\end{Def}

\begin{Rem}\label{rem:BCHNPdescent}
    In forthcoming joint work with Naumann and Pol \cite{BCHNPS_descent}, we investigate general descent properties for stratification along conservative geometric functors. Combined with our bootstrap theorem, we obtain Galois descent for costratification:
\end{Rem}

\begin{Prop}[Galois descent]\label{prop:galoisdescent}
    Let $f \colon A \to B$ be a faithful $G$-Galois extension for $G$ a compact Lie group. If $\Mod(B)$ is costratified, then so is~$\Mod(A)$.
\end{Prop}

\begin{proof}
	Since $\Mod(B)$ is costratified, it is also stratified by \cref{thm:costrat_implies_strat}. It then follows from Galois descent for stratification, proved in \cite{BCHNPS_descent}, that $\Mod(A)$ is also stratified. The bootstrap theorem (\cref{cor:detectingcomin}) thus gives the claim.
\end{proof}

\begin{Exa}
	There are numerous examples that this result applies to. For example, the complexification map $KO \to KU$ from real $K$-theory to complex $K$-theory is a faithful $C_2$-Galois extension \cite[Proposition 5.3.1]{Rognes08} and $\Mod(KU)$ is costratified by \Cref{thm:costratification_for_weakly_affine}. Hence, costratification descends, and we deduce that $\Mod(KO)$ is costratified. Similarly, if $E_n$ denotes the Lubin--Tate spectrum, then $\Mod(E_n)$ is costratified by \Cref{thm:costratification_for_weakly_affine} (see \cref{rem:e_theory_bousfield} below) and we deduce from \cite[Theorem 5.4.4]{Rognes08} and \cite[Proposition 3.6]{HeardMathewStojanoska2017Picard} that $\Mod(E_n^{hG})$ is costratified for any finite subgroup $G \subseteq \mathbb{G}_n$ of the Morava stabilizer group. 
\end{Exa}

\begin{Rem}
	Under stronger conditions on the Galois group, there is a converse to \cref{prop:galoisdescent} which we establish as  \cref{prop:galoisascent} below. The proof follows the strategy used in \cite{barthel2021rep2} to establish the analogous result for stratification. This in turn was inspired by ideas developed in \cite{Mathew15bpp}. First we need some general lemmas.
\end{Rem}

\begin{Not}
	Let $\cat T$ be a rigidly-compactly generated tt-category and write $\Colocidset{\cat T}$ for the class of colocalizing coideals of $\cat T$. Note that for any functor $F\colon \cat T \to \cat S$, we have a function
		\[
			\Colocidset{\cat T} \xrightarrow{F} \Colocidset{\cat S}
		\]
	which sends a colocalizing coideal $\cat C$ to the colocalizing coideal generated by~$F(\cat C)$.
\end{Not}

\begin{Rem}\label{rem:can-push}
	If $G\colon\cat S \to \cat R$ is a functor with the property that colocalizing coideals of $\cat R$ pull back to colocalizing coideals of $\cat S$, then $G(\Coloco{\cat E}) \subseteq \Coloco{G(\cat E)}$ for any collection of objects $\cat E \subseteq \cat S$. It follows that for any functor $F\colon \cat T \to \cat S$, the diagram
		\[\begin{tikzcd}
			\Colocidset{\cat T} \ar[rr,bend left=15,"G\circ F"] \ar[r,"F"'] & \Colocidset{\cat S} \ar[r,"G"'] & \Colocidset{\cat R}
		\end{tikzcd}\]
	commutes.
\end{Rem}

\begin{Exa}\label{exa:can-push}
	Let $f^*\colon\cat T\to \cat S$ be a geometric functor. The isomorphism~\eqref{eq:f!hom} implies that colocaling coideals pull back along the product-preserving exact functor $f^!\colon\cat T\to \cat S$. If $f_*$ is conservative, then the same is true for the functor $f_*\colon\cat S \to \cat T$. Indeed, given a colocalizing coideal~$\cat C$ of $\cat T$, we have
		\begin{align*}
			\ihom{\cat S,f_*^{-1}(\cat C)}
			&=\ihom{\Loc\langle f^*(\cat T^c)\rangle,f_*^{-1}(\cat C)} & (\text{\cref{rem:cons}})\\
			&\subseteq \Coloc\langle\ihom{f^*(\cat T^c),f_*^{-1}(\cat C)}\rangle & \eqref{eq:[loc,t]}\\
			&\subseteq \Coloc\langle f_*^{-1}(\cat C)\rangle
		= f_*^{-1}(\cat C)
		\end{align*}
	where the last inclusion uses the isomorphism $f_*\ihom{f^*(a),b}\simeq \ihom{a,f_*(b)}$ from \cite[(2.17)]{BalmerDellAmbrogioSanders16}.
\end{Exa}

\begin{Lem}\label{lem:descendableretract}
	Let $f^*\colon\cat T\to \cat S$ be a weakly descendable (\cref{def:weakly-descendable}) geometric functor whose right adjoint $f_*$ is conservative. Then the composite
		\[\begin{tikzcd}[ampersand replacement=\&]
			{\Colocidset{\cat T}} \& {\Colocidset{\cat S}} \& { \Colocidset{\cat T}}
			\arrow["{f^!}", from=1-1, to=1-2]
			\arrow["{f_*}", from=1-2, to=1-3]
		\end{tikzcd}\]
	is the identity.
\end{Lem}

\begin{proof}
	Since $f_*$ is conservative, \cref{rem:can-push} and \cref{exa:can-push} imply that the above composite sends a colocalizing coideal $\cat D$ of $\cat T$ to
		\[
			\Coloco{f_*f^!(\cat D)}=\Coloco{\ihom{f_*(\unitS),\cat D}}.
		\]
	Certainly $\Coloco{\ihom{f_*(\unitS),\cat D}} \subseteq \cat D$. On the other hand, if $f^*$ is weakly descendable, then 
		\begin{align*}
			\cat D = \ihom{\unitT,\cat D}\subseteq \ihom{\Loco{f_*(\unitS)},\cat D} 
				\subseteq \Coloco{\ihom{f_*(\unitS),\cat D}}
		\end{align*}
	by \eqref{eq:[loc,t]} and the proof is complete.
\end{proof}

\begin{Exa}\label{exa:descendable}
	A morphism $f\colon A \to B$ of commutative ring spectra is called \emph{descendable} if $\thickt{B} = \Mod(A)$ in $\Mod(A)$. Then the base-change functor $f^*\colon\Mod(A) \to \Mod(B)$ is weakly descendable in the sense of \cref{def:weakly-descendable} and its right adjoint is conservative. Hence \cref{lem:descendableretract} applies. For example, a faithful $G$-Galois extension $f\colon A\to B$ is descendable. This follows from the fact that $B$ is dualizable as an $A$-module (by \cite[Proposition 6.2.1]{Rognes08}). Since $B\otimes-$ is conservative on $\Mod(A)$, $B\otimes f_{\supp(B)} = 0$ implies $f_{\supp(B)}=0$ so that $\supp(B)=\Spc(\Mod(A)^c)$; hence $A \in \thickt{B}$ follows from the classification of thick ideals of compact objects. A more high-powered argument is given in \cite[Theorem 3.38]{Mathew16}.
\end{Exa}

\begin{Lem}\label{lem:galoisascentbijection}
	Suppose $f\colon A \to B$ is a faithful $G$-Galois extension with $G$ a connected compact Lie group. Then $f^!$ induces a bijection
		\[
			f^!\colon \Colocidset{\Mod(A)} \xrightarrow{\cong} \Colocidset{\Mod(B)}.
		\]
\end{Lem}

\begin{proof}
	The following pushout square of commutative ring spectra induces a commutative diagram of forgetful functors, as displayed on the right:
		\[\begin{tikzcd}[ampersand replacement=\&]
			A \& B \& {\Mod(B \otimes_AB)} \& {\Mod(B)} \\
			B \& {B\otimes_AB} \& {\Mod(B)} \& {\Mod(A).}
			\arrow["f"', from=1-1, to=2-1]
			\arrow["f", from=1-1, to=1-2]
			\arrow["{g_2}"', from=2-1, to=2-2]
			\arrow["{g_1}", from=1-2, to=2-2]
			\arrow["{(g_2)_*}"', from=1-3, to=2-3]
			\arrow["{f_*}"', from=2-3, to=2-4]
			\arrow["{(g_1)_*}", from=1-3, to=1-4]
			\arrow["{f_*}", from=1-4, to=2-4]
		\end{tikzcd}\]
		The latter square is horizontally right adjointable, i.e., the corresponding Beck--Chevalley transformation $(g_2)_*\circ g_1^! \to f^!\circ f_*$ is a natural equivalence. Hence the following diagram is commutative
		\[\begin{tikzcd}[ampersand replacement=\&]
			{\Mod(B)} \& {\Mod(B \otimes_AB)} \\
			{\Mod(A)} \& {\Mod(B)}.
			\arrow["{g_1^!}", from=1-1, to=1-2]
			\arrow["{f_*}"', from=1-1, to=2-1]
			\arrow["{(g_2)_*}", from=1-2, to=2-2]
			\arrow["{f^!}"', from=2-1, to=2-2]
		\end{tikzcd}\]
	By \cref{rem:can-push} and \cref{exa:can-push}, it induces a commutative square
		\begin{equation*}
			\begin{tikzcd}[ampersand replacement=\&]
			{\Colocidset{\Mod(B)}} \& {\Colocidset{\Mod(B \otimes_AB)}} \\
			{\Colocidset{\Mod(A)}} \& {\Colocidset{\Mod(B)}}.
			\arrow["{g_1^!}", from=1-1, to=1-2]
			\arrow["{f_*}"', from=1-1, to=2-1]
			\arrow["{(g_2)_*}", from=1-2, to=2-2]
			\arrow["{f^!}"', from=2-1, to=2-2]
		\end{tikzcd}\end{equation*}
	The morphism $f$ is descendable (\cref{exa:descendable}) and consequently $g_1$ and $g_2$ are descendable, too. Hence, by \cref{lem:descendableretract}, the maps $f^!$ and $g_1^!$ are (split) injective while $f_*$ and $(g_2)_*$ are (split) surjective. The assumption on~$G$ guarantees that the canonical morphism $h\colon {B\otimes_AB \to B}$ is descendable as well (see \cite[Proposition~3.36]{Mathew16}), so $h_*$ is (split) surjective using \cref{lem:descendableretract} once more. Furthermore, $h \circ g_2 \simeq \id$ on $B$, so $(g_2)_* \circ h_* = \id$ as maps on $\Colocidset{\Mod(B)}$. This shows that~$h_*$ is also injective, hence both $(g_2)_*$ and $h_*$ are in fact bijections. It follows from the commutative square above that $f_*$ is a bijection as well.
\end{proof}

\begin{Prop}[Galois ascent]\label{prop:galoisascent}
    Let $f\colon A \to B$ be a faithful $G$-Galois extension with $G$ a connected compact Lie group. Then $f^*$ induces a homeomorphism
		\[
			\varphi\colon \Spc(\Mod(B)^c) \xrightarrow{\cong} \Spc(\Mod(A)^c).
		\]
    Moreover, if $\Mod(A)$ is costratified, then so is $\Mod(B)$. 
\end{Prop}

\begin{proof}
	The statement about the Balmer spectra is part of \cite{BCHNPS_descent}, so assume that $\Mod(A)$ is costratified. We have a diagram
		\[
			\begin{tikzcd}
				\Colocidset{\Mod(A)} \ar[r,"f^!"] \ar[d,"\Cosupp"',"\cong"] & \Colocidset{\Mod(B)} \ar[d,"\Cosupp"] \\
				\mathcal{P}(\Spc(\Mod(A)^c)) \ar[r,"\varphi^{-1}"',"\cong"] & \mathcal{P}(\Spc(\Mod(B)^c)),
			\end{tikzcd}
		\]
	which commutes by \cref{cor:globalavruninscott} and \cref{thm:costrat_implies_strat}. The top horizontal map is a bijection by \cref{lem:descendableretract}, hence so is  the right vertical map. In other words, $\Mod(B)$ is costratified.
\end{proof}

\begin{Rem}
	As a consequence we can complete our proof of \cref{thm:rep_g_k_costratifcation} by giving the following alternative proof of \cite[Theorem 11.6]{BensonIyengarKrause12}:
\end{Rem}

\begin{Thm}[Benson--Iyengar--Krause]\label{thm:rep_e_k_costratifcation}
	Let $E$ be an elementary abelian $p$-group and $k$ a field of characteristic $p$. Then the category $\Rep(E,k)$ is costratified, with spectrum $\Spc(\Rep(E,k)^c) \cong \Proj H^*(E;k)$.
\end{Thm}

\begin{proof}
	Let $E$ be an elementary abelian subgroup of rank $r\ge 1$. By \cite[Proposition~3.9]{Mathew15bpp}, the inclusion $E \cong (\bbZ/p)^{\times} \subseteq (S^1)^{\times r} \eqqcolon \bbT$ induces a faithful $\bbT$-Galois extension $f\colon k^{h\bbT} \to k^{hE}$. Note that 
		\[
			\pi_*k^{h\bbT} \cong k[x_1,\ldots,x_r]
		\]
	with all generators $x_i$ in degree 2, so $\Mod(k^{h\bbT})$ is an affine weakly regular tt-category. Therefore, $\Mod(k^{h\bbT})$ is costratified by \cref{thm:costratification_for_weakly_affine}. Since $f$ satisfies the assumptions of \cref{prop:galoisascent}, Galois ascent implies that $\Mod(k^{hE})$ is costratified as well. Finally, we observe that there is an equivalence of tt-categories
		\[
			\Mod(k^{hE}) \simeq \Rep(E,k),
		\]
	so $\Rep(E,k)$ is costratified, too. 
\end{proof}

\subsection*{Chromatic homotopy theory}\label{ssec:chromaticexamples}

\begin{Def}
	For a fixed prime number $p$, let $\cat S$ denote the $p$-local stable homotopy category, let $E_n$ denote the $n$th \emph{Lubin--Tate spectrum}, and let~$\cat S_{E_n}$ denote the category of $E_n$-local spectra. Recall that $E_n$ is a commutative ring spectrum Bousfield equivalent to the $n$th Johnson--Wilson spectrum $E(n)$. For further background material on chromatic homotopy theory, we refer the interested reader to \cite{BarthelBeaudry19pp}.
\end{Def}

\begin{Rem}\label{rem:e_theory_bousfield}
    We write $L_n \colon \cat S \to \cat S_{E_n}$ for the corresponding Bousfield localization functor. This is a smashing localization, and as such, we have $\cat S_{E_n} \simeq \Mod(L_nS^0)$ and the localization $L_n \colon \cat S \to \Mod(L_nS^0)$ is given by base change along $S^0 \to L_nS^0$. 
\end{Rem}

\begin{Rem}\label{rem:costratification_for_en}
    We have 
		\[
			\pi_*(E_n) \cong W(\bbF_{\hspace{-0.1em}p^n})[\![u_1,\ldots,u_n]\!][u^{\pm 1}]
		\]
    where $|u_i| = 0$ and $|u| = -2$. In particular, $\Mod(E_n)$ is affine weakly regular (\Cref{def:affine_weakly_regular}) and hence $\Mod(E_n)$ is costratified by \Cref{thm:costratification_for_weakly_affine}. 
\end{Rem}

\begin{Not}\label{not:En-primes}
    For each $0 \le h \le n$, we define
        \[ 
            \cat P_h \coloneqq \SET{x \in \cat S_{E_n}^c }{K(h)_*(x)=0},
        \]
    a prime ideal of $\cat S_{E_n}^c$. The following is a restatement of \cite[Theorem 6.9]{HoveyStrickland99}, as given in \cite[Proposition 3.5]{BarthelHeardNaumann20pp}.
\end{Not}

\begin{Thm}[Hovey--Strickland] \label{thm:hovey-strickland}
    The spectrum
        \[
            \Spc(\cat S_{E_n}^c) = \cat P_{n} - \cdots - \cat P_{1} - \cat P_0
        \]
    is a local irreducible space consisting of $n+1$ points, where closure goes to the left: $\overbar{\{\cat P_h\}} = \SET{\cat P_k}{h \le k \le n}$.
\end{Thm}

\begin{Rem}
	The following result recovers the classification of colocalizing subcategories of $\cat S_{E_n}$ given in \cite[Theorem 6.14]{HoveyStrickland99}.
\end{Rem}

\begin{Thm}[Hovey--Strickland]\label{thm:chromatic_costratification}
    The category of $E_n$-local spectra $\cat S_{E_n}$ is costratified. 
\end{Thm}

\begin{proof}
	Consider the geometric functor $f^* \colon \cat S_{E_n} \to \Mod(E_n)$ given by base-change. It is a consequence of the Hopkins--Ravenel smash product theorem \cite[Chapter~8]{Ravenel92} that this functor is conservative (see also \cite[Prop.~3.18 and Thm.~4.17]{Mathew16}). Moreover, $\Mod(E_n)$ is costratified (\Cref{rem:costratification_for_en}) and $\cat S_{E_n}$ is stratified by \cite[Theorem 10.14]{bhs1}.   By \Cref{cor:detectingcomin}, we conclude that $\cat S_{E_n}$ is costratified. 
\end{proof}

\begin{Rem}[Chromatic cosupport]\label{rem:chromatic_cosupport}
	In \cite[Proposition~10.12]{bhs1}, we established an isomorphism $g_{\cat P_k}\simeq M_k S^0$ between the Balmer--Favi idempotent and the fiber of $L_kS^0 \to L_{k-1}S^0$. It follows that 
		\[
		\begin{split}
			\Cosupp(t)	&= \{ \cat P_k \in \Spc(\cat S_{E(n)}^c) \mid \ihom{g_{\cat P_k},t} \ne 0 \} \\
						&= \{ k \in \{ 0,\ldots,n \} \mid \ihom{M_kS^0,t} \ne 0 \}. 
		\end{split}
		\]
	As recalled in \cref{Exa:chromatic_cosupport}, Hovey--Strickland define the chromatic cosupport by
		\[
			\cosuppa(t)= \SET{k \in \{0,\ldots, n\}}{\ihom{K(k),t} \ne 0}.
		\]
    We claim that $\cosuppa(t) = \Cosupp(t)$ for all $t \in \cat S_{E_n}$.  Indeed, suppose that $\ihom{M_kS^0,t} = 0$. Then 
		\[
			\ihom{K(k),\ihom{M_kS^0,t}} \simeq \ihom{K(k) \otimes M_kS^0,t} = 0
		\]
	but $K(k) \otimes M_kS^0 \simeq M_kK(k) \simeq K(k)$, so $\ihom{K(k),t} = 0$. Conversely, suppose that $\ihom{K(k),t} = 0$. The collection of $Y \in \cat S_{E_n}$ for which $\ihom{Y,t} = 0$ is a localizing subcategory which contains $K(k)$, and hence also contains $M_kS^0$ by \cite[Proposition 6.17]{HoveyStrickland99}. Therefore, $\ihom{M_kS^0,t} = 0$ as well. Therefore the Balmer--Favi notion of cosupport agrees with the usual version of chromatic cosupport. Alternatively, one can prove this using \cref{cor:Top-uniqueness} and \cite[Theorem 6.14]{HoveyStrickland99}. 
\end{Rem}

\subsection*{Cochain algebras}\label{ssec:cochainexamples}

\begin{Def}
	For a space $X$, we let $C^*(X;\Fp)$ denote the function spectrum $F(\Sigma^{\infty}_+X ,\Fp)$.
\end{Def}

\begin{Rem}
    In \cite{BarthelCastellanaHeardValenzuela22} we investigated when $\Mod(C^*(X;\Fp))$ is stratified in the sense of BIK. This used a certain category $\mathcal E(X)$ associated to $X$, whose objects are (isomorphism classes of) pairs $(V,\phi)$ consisting of an elementary abelian $p$-group~$E$ and a finite morphism $\phi \colon H^*(X;\Fp) \to H^*(BE;\Fp)$ of unstable algebras over the dual Steenrod algebra. We showed that these maps $\phi$ lift to maps on the level of cochains, and therefore we have a map of commutative ring spectra
		\begin{equation}\label{eq:maps_of_spaces_rector}
			\rho_X \colon C^*(X;\Fp) \to \prod_{(E,\phi) \in \mathcal{E}(X)} C^*(BE;\Fp). 
		\end{equation}
\end{Rem}

\begin{Def}
	A $p$-good\footnote{in the sense of Bousfield--Kan \cite{BousfieldKan1972}.} connected topological space $X$ with noetherian mod $p$ cohomology is said to satisfy \emph{Chouinard’s condition} if induction and coinduction along the map \eqref{eq:maps_of_spaces_rector} are conservative. 
\end{Def}

\begin{Thm}\label{thm:cochains}
	Let $X$ be a $p$-good connected space with noetherian mod $p$ cohomology. Then the following are equivalent:
	\begin{enumerate}
		\item $\Mod(C^*(X;\Fp))$ is stratified. 
		\item $\Mod(C^*(X;\Fp))$ is costratified. 
		\item $X$ satisfies Chouinard's condition. 
	\end{enumerate}
	If any of these holds, then $\Spc(\Mod(C^*(X;\Fp))^c) \cong \Spec^h(H^*(X;\Fp))$.
\end{Thm}

\begin{proof}
	We have that $(a)$ is equivalent to $(c)$ by \cite[Theorem 5.12]{BarthelCastellanaHeardValenzuela22} and $(b)$ implies $(a)$ by \Cref{thm:costrat_implies_strat}. These also imply claim $(b)$. Assume then that $(a)$, and hence $(c)$, hold and consider the base change functor
		\[
			f^* \colon \Mod(C^*(X;\Fp)) \to \prod_{(E,\phi) \in \mathcal{E}(X)} \Mod(C^*(BE;\Fp)).
		\]
	Since the product is finite (which is a consequence of \cite[p.~194]{Rector1984Noetherian}), the target category is costratified by  \cite[Theorem 11.6]{BensonIyengarKrause12} and \Cref{Exa:finite_product_of_costratified_categories}. Moreover, $f^*$ is a conservative geometric functor and $\Mod(C^*(X;\Fp))$ is stratified. Now apply \Cref{cor:detectingcomin}. 
\end{proof}

\begin{Exa}\label{exa:3-connected-cover}
    Let $X = S^3\langle 3 \rangle$ denote the 3-connected cover of $S^3$. In this case, there is a single non-trivial object in the category $\mathcal{E}(S^3\langle 3 \rangle)$, namely the pair coming from the composite $B\bbZ/p \to BS^1 \to S^3\langle 3 \rangle$. In \cite[Example 5.16]{BarthelCastellanaHeardValenzuela19} it is shown that 
		\[
			\Mod(C^*(S^3\langle 3 \rangle; \Fp)) \to \Mod(C^*(B\bbZ/p;\Fp))
		\]
    satisfies Chouinard's condition. We deduce that $\Mod(C^*(S^3\langle 3 \rangle; \Fp))$ is costratified. Note that in our previous work, we were unable to prove this; see \cite[Section~4.5]{BarthelCastellanaHeardValenzuela19} and the discussion therein. 
\end{Exa}

\begin{Exa}
    If $X$ is a noetherian $H$-space, then $X$ satisfies Chouniard's condition by \cite[Theorem 5.15]{BarthelCastellanaHeardValenzuela19}. Therefore, $\Mod(C^*(X;\Fp))$ is costratified.
\end{Exa}

\subsection*{Equivariant homotopy theory}\label{ssec:equivariantexamples}

\begin{Not}
	Let $\Sp_G$ denote the $\infty$-category of genuine $G$-spectra. For a commutative algebra $\mathbb{E} \in \CAlg(\Sp_G)$, we write $\Mod_{\Sp_G}(\mathbb{E})$ for the \mbox{$\infty$-category} of modules and $\Der_G(\mathbb{E})$ for the associated homotopy category. 
\end{Not}

\begin{Thm}\label{thm:equivariant_costratification}
	Let $G$ be a finite group and let $\mathbb{E} \in \CAlg(\Sp_G)$. Suppose that the following conditions hold for all subgroups $H \le G:$
    \begin{enumerate}
        \item $\Spc(\Der(\Phi^H\bbE)^c)$ is noetherian; and
        \item $\Der(\Phi^H\bbE)$ is costratified.
    \end{enumerate}
	Then $\Der_G(\bbE)$ is costratified and $\Spc(\Der_G(\bbE)^c)$ is noetherian.
\end{Thm}
\begin{proof}
	We first observe that by \cref{thm:costrat_implies_strat}, condition $(b)$ implies that each $\Der(\Phi^H\bbE)$ is stratified. Then \cite[Theorem 3.33]{BCHNP1} shows that $\Der_G(\bbE)$ is stratified with noetherian spectrum $\Spc(\Der_G(\bbE)^c)$. In order to show that it is costratified, we use the bootstrap theorem (\cref{cor:detectingcomin}).  Indeed, let 
		\[
			\Phi \colon \Der_G(\bbE) \xrightarrow{\{\Phi^H\}_H} \prod_{H \le G} \Der(\Phi^H\bbE)
		\]
	denote the product of geometric fixed point functors. By  \cite[Proposition~3.25]{BCHNP1}, $\Phi$ is a conservative geometric functor. By assumption, each $\Der(\Phi^H\bbE)$ is costratified, and hence so is the product (\cref{Exa:finite_product_of_costratified_categories}) and we have already seen that $\Der_G(\bbE)$ is stratified. Now apply \Cref{cor:detectingcomin}. 
\end{proof}

\begin{Rem}
    The generalized Quillen stratification theorem proven in \cite[Theorem 4.3]{BCHNP1} describes the underlying set of $\Spc(\Der_G(\bbE)^c)$ in terms of the spectra of the geometric fixed points $\Der(\Phi^H\bbE)$ along with the Weyl group actions. 
\end{Rem}

\begin{Rem}
	Let $\mathcal{F}$ denote a family of subgroups of $G$. Recall that $\mathbb{E}\in \CAlg(\Sp_G)$ is $\mathcal{F}$-nilpotent if $\mathbb{E}$ is in the thick ideal of $\Sp_G$ generated by $\SET{G/H_+}{H \in \mathcal{F}}$. In this case, we have $\Phi^H \bbE = 0$ whenever $\bbE \not \in \mathcal{F}$ by \cite[Theorem 6.41]{MathewNaumannNoel17}, and in particular it suffices to check the conditions of \Cref{thm:equivariant_costratification} for $H \in \mathcal{F}$.   There is always a minimal such family $\mathcal{F}$ known as the \emph{derived defect base} of $E$. See \cite{mnn2} for a computation of the derived defect base of many equivariant spectra. 
\end{Rem}

\begin{Exa}
	There is a canonical geometric functor $\triv_G\colon \Sp \to \Sp_G$ that sends a spectrum to the corresponding $G$-spectrum with trivial \mbox{$G$-action.} See \cite[Section 3]{PatchkoriaSandersWimmer22} for further details. For $\bbE \in \CAlg(\Sp)$, let $\bbE_G\coloneqq \triv_G \bbE \in \CAlg(\Sp_G)$ denote the corresponding commutative algebra in genuine $G$-spectra. For each $H \le G$, we have that $\Phi^H\bbE_G \simeq \bbE$. In particular, \Cref{thm:equivariant_costratification} implies the following:
\end{Exa}

\begin{Cor}\label{cor:costratified_spectral_mackey}
	Let $\bbE \in \CAlg(\Sp)$ be a commutative ring spectrum such that $\Spc(\Der(\bbE)^c)$ is noetherian. If $\Der(\bbE)$ is costratified, then $\Der_G(\bbE_G)$ is costratified for any finite group G.
\end{Cor}

\begin{Cor}\label{cor:derived_mackey_costratified}
	For any finite group $G$ and discrete commutative ring $R$, the category of derived Mackey functors $\Der_G(\HR_G)$ is costratified.
\end{Cor}

\begin{proof}
	Using \cref{ex:costratified_noetherian_ring}, we can apply \cref{cor:costratified_spectral_mackey} to $\bbE=\HR$.
\end{proof}

\begin{Rem}
    By \cite[Proposition 4.9]{PatchkoriaSandersWimmer22} and \cite[Proposition 14.3]{bhs1}, there is an equivalence of symmetric monoidal $\infty$-categories
		\[
		   \Mack{G}{\bbE} \simeq \Mod_{\Sp_G}(\bbE_G)
		\]
		where $\Mack{G}{\bbE}$ denotes Barwick's category of spectral $G$-Mackey functors with $\bbE$-coefficients \cite{Barwick17}. In particular, the previous result shows that $HR$-valued spectral Mackey functors are costratified for any discrete noetherian commutative  ring $R$. For $R = \bbZ$, the corresponding spectrum was determined completely in \cite{PatchkoriaSandersWimmer22}. 
\end{Rem}

\begin{Exa}
	Taking $\bbE = L_nS^0$ (\cref{rem:e_theory_bousfield}) we obtain the category $\SH_{G,E_n} \coloneqq \Der_G(\bbE_G)$ of $E_n$-local spectral Mackey functors; see \cite[Example 13.15]{bhs1}. Equivalently, this is the category of spectral Mackey functors with coefficients in $\cat S_{E_n}$; see \cite[Example 14.4]{bhs1}. As explained in \emph{loc.\,cit.}, the underlying set of the spectrum of $\SH_{G,E_n}$ is known,  thanks to \cref{thm:hovey-strickland}.
\end{Exa}

\begin{Cor}
	For any finite group $G$, prime number $p$, and $0 \le n < \infty$, the category of $E_n$-local spectral Mackey functors $\SH_{G,E_n}$ is costratified.
\end{Cor}

\begin{proof}
	This follows from \Cref{cor:costratified_spectral_mackey} with $\bbE = L_nS^0$, using \Cref{thm:chromatic_costratification}.
\end{proof}

\begin{Rem}
    Given a non-equivariant spectrum $\bbE \in \CAlg(\Sp)$, there is another way to produce an equivariant spectrum, namely by taking the associated Borel equivariant spectrum $\borel{G}{\bbE} \in \CAlg(\Sp_G)$; see \cite[Section~6.3]{MathewNaumannNoel17}. 
\end{Rem}

\begin{Thm}\label{thm:borel-E_n}
	For any finite group $G$, the category $\Der_G(\borel{G}{E_n})$ is costratified by $\Spc(\Der_G(\borel{G}{E_n})^c) \cong \Spec(E_n^0(BG))$. 
\end{Thm}

\begin{proof}
	It is shown in \cite[Theorem 6.14]{BCHNP1} that $\Phi^H(\borel{G}{E_n})$ is nonzero only when $H$ is an abelian $p$-group of rank at most $n$, in which case its homotopy groups are a regular noetherian even-periodic ring. In particular, each $\Der(\Phi^H(\borel{G}{E_n}))$ is costratified with noetherian spectrum by \Cref{thm:costratification_for_weakly_affine}. By \Cref{thm:equivariant_costratification}, $\Der_G(\borel{G}{E_n})$ is costratified by $\Spc(\Der_G(\borel{G}{E_n})^c)$, which is homeomorphic to $\Spec(E^0(BG))$ by \cite[Lemma 7.3]{BCHNP1}.
\end{proof}

\begin{Rem}
	Our next example comes from the category of modules associated to equivariant complex $K$-theory $KU_G$ for a finite group $G$; see \cite{Segal68b}. We let $R(G)$ denote the complex representation ring associated to $G$. 
\end{Rem}

\begin{Thm}
	For any finite group $G$, the category $\Der_G(KU_G)$ is costratified by $\Spc(\Der_G(KU_G)^c) \cong \Spec(R(G))$. 
\end{Thm}

\begin{proof}
	This is similar to the previous proof. It is shown in \cite[Lemma 8.6]{BCHNP1} that $\Phi^H(KU_G)$ is nonzero only when $H$ is a cyclic subgroup of $G$, in which case its homotopy groups are a regular noetherian even-periodic ring. In particular, each $\Der(\Phi^H(KU_G))$ is costratified with noetherian spectrum by \Cref{thm:costratification_for_weakly_affine}. By \Cref{thm:equivariant_costratification}, $\Der_G(KU_G)$ is costratified by $\Spc(\Der_G(KU_G)^c)$, which is homeomorphic to $\Spec(R(G))$ by \cite[Lemma 8.11]{BCHNP1}.
\end{proof}

\begin{Rem}\label{rem:SHGQ}
	For a compact Lie group $G$, let $\SH_{G,\bbQ}$ denote the stable homotopy category of rational $G$-equivariant spectra. This is a rigidly-compactly generated tt-category which is stratified; see \cite{Greenlees19_rational} or \cite[Theorem 12.22]{bhs1}. Our final goal is to prove that it is costratified.
\end{Rem}

\begin{Rem}
	Recall that, by definition, $L$ is \emph{cotoral} in $K$ if $L$ is a normal subgroup of $K$ and $K/L$ is a torus. Moreover, for each $H \le G$ we have geometric fixed point functors
		\[
			\Phi^H \colon \SH_{G,\bbQ} \to \SH_{\bbQ}
		\]
	which are jointly conservative (for example, \cite[Proposition 3.3.10]{Schwede18_global}). Since the spectrum of $\SH_{\bbQ} \simeq \Der(\bbQ)$ is a single point, we obtain a prime ideal $\mathfrak{p}_H \in \Spc(\SH_{G,\bbQ})$ for each $H \le G$. Up to conjugacy, these turn out to be all the prime ideals: 
\end{Rem}

\begin{Thm}[Greenlees]\label{thm:spectrumrationalgspectra}
    Let $G$ be a compact Lie group. Then as a set 
        \[
            \Spc(\SH_{G,\bbQ}^c) = \SET{ \mathfrak{p}_H}{(H)  \text{ conjugacy class of closed subgroups in } G}.
        \]
    The specialization order is determined by cotoral inclusions:
        \[
            \mathfrak{p}_K \subseteq \mathfrak{p}_H  \text{ if and only if $K$ is conjugate to a subgroup cotoral in $H$.}
        \]
    The topology on $\Spc(\SH_{G,\bbQ}^c)$ is the ``Zariski topology on the $f$-topology'' of~\cite{Greenlees98_rational}. 
\end{Thm}

\begin{Rem}
    The space $\Spc(\SH_{G,\bbQ}^c)$ is weakly noetherian by \cite[Lem.~12.12]{bhs1}.
\end{Rem}
 
\begin{Thm}\label{thm:rational-spectra}
    For any compact Lie group $G$, the category of rational $G$-spectra is costratified. 
\end{Thm}

\begin{proof}
We apply bootstrap to the collection of geometric fixed point functors $\Phi^H \colon \SH_{G,\bbQ} \to \SH_{\bbQ} \simeq \Der(\bbQ)$. The target category is costratified (\Cref{cor:stratttfields}), while $\SH_{G,\bbQ}$ is stratified (\cref{rem:SHGQ}). Therefore, by \Cref{prop:bootstrap_for_algebraic_categories} and \Cref{thm:spectrumrationalgspectra} we deduce that $\Spc(\SH_{G,\bbQ})$ is costratified. 
\end{proof}

\section{Open questions}\label{sec:open-questions}

We collect here some questions we do not know the answer to.

\begin{Que}
	Does the detection property always hold? (See \cref{rem:tantalizing}.)
\end{Que}

\begin{Que}
	Does stratification imply costratification? (See \cref{rem:does-strat-imply-costrat}.)
\end{Que}

\begin{Que}
	What is the cosupport of $\gP$? More precisely, in what generality is it true that $\Cosupp(\gP) = \gen(\cat P)$? (See \cref{lem:cosupp-of-gP}.)
\end{Que}

\begin{Que}
	Does the (co)detection property always hold for weak (co)rings?
\end{Que}

\begin{Que}
	A rigidly-compactly generated tt-category $\cat T$ satisfies the telescope conjecture if it is stratified with generically noetherian spectrum; cf.~\cite[Theorem~9.11]{bhs1}. Can this result be improved if $\cat T$ is costratified?
\end{Que}

\begin{Que}
	Does there exist a geometric functor $f^*\colon \cat T \to \cat S$ between rigidly-compactly generated tt-categories such that $f^*$ is conservative but $f^!$ is not? (See \cref{cor:conservative-iff-surjective} and the results cited in its proof.) 
\end{Que}

\begin{Que}
	Is it true that if $f^*\colon\cat T\to\cat S$ is any geometric functor and~$\cat S$ satisfies the local-to-global principle, then $\cat T$ satisfies the local-to-global principle for objects $t\in \cat T$ with $\Supp(t) \subseteq \im \varphi$? It suffices to prove it under the additional assumption that $f_*$ is conservative. (See \cref{sec:descending-LGP}.)
\end{Que}

\begin{Que}
	Does the local-to-global principle hold for an object $t \in \cat T$ if its support $\Supp(t)$ is noetherian?
\end{Que}

\begin{Que}
	Is there a characterization of when the local-to-global principle holds purely in terms of the topology of $\Spc(\cat T^c)$?
\end{Que}

\begin{Que}
	Does stratification \emph{always} descend along a conservative geometric functor $f^*\colon \cat T \to \cat S$ between rigidly-compactly generated tt-categories? If so, our bootstrap theorem (\cref{cor:detectingcomin}) implies that costratification always descends as well.
\end{Que}

\addtocontents{toc}{\vspace{\normalbaselineskip}}
\bibliographystyle{alphasort}\bibliography{bibliography}

\newcommand{\etalchar}[1]{$^{#1}$}
\begin{thebibliography}{BCH{\etalchar{+}}23b}

\bibitem[AR94]{AdamekRosicky94}
Ji\v{r}\'{\i} Ad\'{a}mek and Ji\v{r}\'{\i} Rosick\'{y}.
\newblock {\em Locally presentable and accessible categories}, volume 189 of
  {\em London Mathematical Society Lecture Note Series}.
\newblock Cambridge University Press, Cambridge, 1994.

\bibitem[ATJLL97]{AlonsoJeremiasLipman97}
Leovigildo Alonso~Tarr\'{\i}o, Ana Jerem\'{\i}as~L\'{o}pez, and Joseph Lipman.
\newblock Local homology and cohomology on schemes.
\newblock {\em Ann. Sci. \'{E}cole Norm. Sup. (4)}, 30(1):1--39, 1997.

\bibitem[AS82]{avruninscott}
George~S. Avrunin and Leonard~L. Scott.
\newblock Quillen stratification for modules.
\newblock {\em Invent. Math.}, 66(2):277--286, 1982.

\bibitem[BS21]{BalchinStevenson21pp}
Scott Balchin and Greg Stevenson.
\newblock Big categories, big spectra.
\newblock Preprint, 30~pages, available online at
  \href{https://arxiv.org/abs/2109.11934}{arXiv:2109.11934}. To appear in
  \emph{Journal of Topology}, 2021.

\bibitem[Bal05]{Balmer05a}
Paul Balmer.
\newblock The spectrum of prime ideals in tensor triangulated categories.
\newblock {\em J. Reine Angew. Math.}, 588:149--168, 2005.

\bibitem[Bal10]{BalmerICM}
Paul Balmer.
\newblock Tensor triangular geometry.
\newblock In {\em International {C}ongress of {M}athematicians, Hyderabad
  (2010), {V}ol. {II}}, pages 85--112. Hindustan Book Agency, 2010.

\bibitem[Bal18]{Balmer18}
Paul Balmer.
\newblock On the surjectivity of the map of spectra associated to a
  tensor-triangulated functor.
\newblock {\em Bull. Lond. Math. Soc.}, 50(3):487--495, 2018.

\bibitem[Bal20]{Balmer20_bigsupport}
Paul Balmer.
\newblock Homological support of big objects in tensor-triangulated categories.
\newblock {\em J. \'{E}c. polytech. Math.}, 7:1069--1088, 2020.

\bibitem[BDS16]{BalmerDellAmbrogioSanders16}
Paul Balmer, Ivo Dell'Ambrogio, and Beren Sanders.
\newblock Grothendieck--{N}eeman duality and the {W}irthm\"uller isomorphism.
\newblock {\em Compos. Math.}, 152(8):1740--1776, 2016.

\bibitem[BF11]{BalmerFavi11}
Paul Balmer and Giordano Favi.
\newblock Generalized tensor idempotents and the telescope conjecture.
\newblock {\em Proc. Lond. Math. Soc. (3)}, 102(6):1161--1185, 2011.

\bibitem[BG21]{BalmerGallauer2021pp}
Paul {Balmer} and Martin {Gallauer}.
\newblock {Permutation modules, Mackey functors, and Artin motives}.
\newblock Preprint, 30~pages, available online at
  \href{https://arxiv.org/abs/2107.11797}{arXiv:2107.11797}. To appear in
  \emph{Proceedings of ICRA 2020}., 2021.

\bibitem[BG22]{BalmerGallauer2022pp}
Paul {Balmer} and Martin {Gallauer}.
\newblock {The tt-geometry of permutation modules. Part I: Stratification}.
\newblock Preprint, 34~pages, available online at
  \href{https://arxiv.org/abs/2210.08311}{arXiv:2210.08311}, 2022.

\bibitem[BKS19]{BalmerKrauseStevenson19}
Paul Balmer, Henning Krause, and Greg Stevenson.
\newblock Tensor-triangular fields: ruminations.
\newblock {\em Selecta Math. (N.S.)}, 25(1):Paper No. 13, 36, 2019.

\bibitem[BS17]{BalmerSanders17}
Paul Balmer and Beren Sanders.
\newblock The spectrum of the equivariant stable homotopy category of a finite
  group.
\newblock {\em Invent. Math.}, 208(1):283--326, 2017.

\bibitem[{Bar}21]{barthel2021rep1}
Tobias {Barthel}.
\newblock {Stratifying integral representations of finite groups}.
\newblock Preprint, 36~pages, available online at
  \href{https://arxiv.org/abs/2109.08135}{arXiv:2109.08135}, 2021.

\bibitem[{Bar}22]{barthel2021rep2}
Tobias {Barthel}.
\newblock {Stratifying integral representations via equivariant homotopy
  theory}.
\newblock Preprint, 19~pages, available online at
  \href{https://arxiv.org/abs/2203.14946}{arXiv:2203.14946}, 2022.

\bibitem[BB19]{BarthelBeaudry19pp}
Tobias Barthel and Agn\`{e}s Beaudry.
\newblock Chromatic structures in stable homotopy theory.
\newblock In {\em Handbook of {H}omotopy {T}heory}, pages 165--222. Chapman and
  Hall/CRC, 2019.

\bibitem[BBI{\etalchar{+}}23]{BBIKP_stratification}
Tobias {Barthel}, Dave {Benson}, Srikanth~B. {Iyengar}, Henning {Krause}, and
  Julia {Pevtsova}.
\newblock Lattices over group algebras and stratification.
\newblock In preparation, 2023.

\bibitem[BCH{\etalchar{+}}23a]{BCHNP1}
Tobias Barthel, Nat\`alia Castellana, Drew Heard, Niko Naumann, and Luca Pol.
\newblock Quillen stratification in equivariant homotopy theory.
\newblock Preprint, 56~pages, available online at
  \href{https://arxiv.org/abs/2204.01082}{arXiv:2301.02212}, 2023.

\bibitem[BCH{\etalchar{+}}23b]{BCHNPS_descent}
Tobias Barthel, Natalia Castellana, Drew Heard, Niko Naumann, Luca Pol, and
  Beren Sanders.
\newblock Descent in tensor triangular geometry.
\newblock In preparation, 2023.

\bibitem[BCHV19]{BarthelCastellanaHeardValenzuela19}
Tobias Barthel, Nat\`alia Castellana, Drew Heard, and Gabriel Valenzuela.
\newblock Stratification and duality for homotopical groups.
\newblock {\em Adv. Math.}, 354:106733, 61, 2019.

\bibitem[BCHV22]{BarthelCastellanaHeardValenzuela22}
Tobias Barthel, Nat\`alia Castellana, Drew Heard, and Gabriel Valenzuela.
\newblock On stratification for spaces with {N}oetherian {${\rm mod}\, p$}
  cohomology.
\newblock {\em Amer. J. Math.}, 144(4):895--941, 2022.

\bibitem[BHN22]{BarthelHeardNaumann20pp}
Tobias Barthel, Drew Heard, and Niko Naumann.
\newblock On conjectures of {H}ovey-{S}trickland and {C}hai.
\newblock {\em Selecta Math. (N.S.)}, 28(3), 2022.

\bibitem[BHS21]{bhs2}
Tobias Barthel, Drew Heard, and Beren Sanders.
\newblock Stratification and the comparison between homological and tensor
  triangular support.
\newblock Preprint, 18~pages, 2021.

\bibitem[BHS23]{bhs1}
Tobias Barthel, Drew Heard, and Beren Sanders.
\newblock Stratification in tensor triangular geometry with applications to
  spectral {M}ackey functors.
\newblock Preprint, 59~pages. To appear in \emph{Cambridge Journal of
  Mathematics }, 2023.

\bibitem[BHV18]{BarthelCastellanaHeardValenzuela18}
Tobias Barthel, Drew Heard, and Gabriel Valenzuela.
\newblock Local duality in algebra and topology.
\newblock {\em Adv. Math.}, 335:563--663, 2018.

\bibitem[Bar17]{Barwick17}
Clark Barwick.
\newblock Spectral {M}ackey functors and equivariant algebraic {$K$}-theory
  ({I}).
\newblock {\em Adv. Math.}, 304:646--727, 2017.

\bibitem[B{\v{S}}17]{BazzoniStovicek17}
Silvana Bazzoni and Jan {\v{S}}{\v{t}}ov\'{\i}\v{c}ek.
\newblock Smashing localizations of rings of weak global dimension at most one.
\newblock {\em Adv. Math.}, 305:351--401, 2017.

\bibitem[Bel00]{Beligiannis00}
Apostolos Beligiannis.
\newblock Relative homological algebra and purity in triangulated categories.
\newblock {\em J. Algebra}, 227(1):268--361, 2000.

\bibitem[BCR97]{BensonCarlsonRickard97}
Dave Benson, Jon~F. Carlson, and Jeremy Rickard.
\newblock Thick subcategories of the stable module category.
\newblock {\em Fund. Math.}, 153(1):59--80, 1997.

\bibitem[BIK08]{BensonIyengarKrause08}
Dave Benson, Srikanth~B. Iyengar, and Henning Krause.
\newblock Local cohomology and support for triangulated categories.
\newblock {\em Ann. Sci. \'Ec. Norm. Sup\'er. (4)}, 41(4):573--619, 2008.

\bibitem[BIK11a]{BensonIyengarKrause11a}
Dave Benson, Srikanth~B. Iyengar, and Henning Krause.
\newblock Stratifying modular representations of finite groups.
\newblock {\em Ann. of Math. (2)}, 174(3):1643--1684, 2011.

\bibitem[BIK11b]{BensonIyengarKrause11b}
Dave Benson, Srikanth~B. Iyengar, and Henning Krause.
\newblock Stratifying triangulated categories.
\newblock {\em J. Topol.}, 4(3):641--666, 2011.

\bibitem[BIK12]{BensonIyengarKrause12}
David~J. Benson, Srikanth~B. Iyengar, and Henning Krause.
\newblock Colocalizing subcategories and cosupport.
\newblock {\em J. Reine Angew. Math.}, 673:161--207, 2012.

\bibitem[BK08]{BensonKrause2008Complexes}
David~John Benson and Henning Krause.
\newblock Complexes of injective {$kG$}-modules.
\newblock {\em Algebra Number Theory}, 2(1):1--30, 2008.

\bibitem[Boa70]{Boardman70}
J.~M. Boardman.
\newblock Stable homotopy theory is not self-dual.
\newblock {\em Proc. Amer. Math. Soc.}, 26:369--370, 1970.

\bibitem[Bou79]{Bousfield79}
A.~K. Bousfield.
\newblock The localization of spectra with respect to homology.
\newblock {\em Topology}, 18(4):257--281, 1979.

\bibitem[BK72]{BousfieldKan1972}
A.~K. Bousfield and D.~M. Kan.
\newblock {\em Homotopy limits, completions and localizations}.
\newblock Lecture Notes in Mathematics, Vol. 304. Springer-Verlag, Berlin-New
  York, 1972.

\bibitem[BC76]{BrownComenetz76}
Edgar~H. Brown, Jr. and Michael Comenetz.
\newblock Pontrjagin duality for generalized homology and cohomology theories.
\newblock {\em Amer. J. Math.}, 98(1):1--27, 1976.

\bibitem[BKS07]{BuanKrauseSolberg07}
Aslak~Bakke Buan, Henning Krause, and {\O}yvind Solberg.
\newblock Support varieties: an ideal approach.
\newblock {\em Homology, Homotopy Appl.}, 9(1):45--74, 2007.

\bibitem[CGR14]{CasacubertaGutierrezRosicky14}
Carles Casacuberta, Javier~J. Guti\'{e}rrez, and Ji\v{r}\'{\i} Rosick\'{y}.
\newblock Are all localizing subcategories of stable homotopy categories
  coreflective?
\newblock {\em Adv. Math.}, 252:158--184, 2014.

\bibitem[Cha60]{Chase60}
Stephen~U. Chase.
\newblock Direct products of modules.
\newblock {\em Trans. Amer. Math. Soc.}, 97:457--473, 1960.

\bibitem[DS16]{DellAmbrogioStanley16}
Ivo Dell'Ambrogio and Donald Stanley.
\newblock Affine weakly regular tensor triangulated categories.
\newblock {\em Pacific J. Math.}, 285(1):93--109, 2016.
\newblock Erratum available at
  \url{https://math.univ-lille1.fr/~dellambr/affreg_erratum.pdf}.

\bibitem[DST19]{DickmannSchwartzTressl19}
Max Dickmann, Niels Schwartz, and Marcus Tressl.
\newblock {\em Spectral spaces}, volume~35 of {\em New Mathematical
  Monographs}.
\newblock Cambridge University Press, Cambridge, 2019.

\bibitem[DG02]{DwyerGreenlees02}
W.~G. Dwyer and J.~P.~C. Greenlees.
\newblock Complete modules and torsion modules.
\newblock {\em Amer. J. Math.}, 124(1):199--220, 2002.

\bibitem[DGI06]{DwyerGreenleesIyengar06}
W.~G. Dwyer, J.~P.~C. Greenlees, and S.~Iyengar.
\newblock Finiteness in derived categories of local rings.
\newblock {\em Comment. Math. Helv.}, 81(2):383--432, 2006.

\bibitem[FS01]{FuchsSalce01}
L\'{a}szl\'{o} Fuchs and Luigi Salce.
\newblock {\em Modules over non-{N}oetherian domains}, volume~84 of {\em
  Mathematical Surveys and Monographs}.
\newblock American Mathematical Society, Providence, RI, 2001.

\bibitem[GP05]{GarkushaPrest05}
Grigory Garkusha and Mike Prest.
\newblock Triangulated categories and the {Z}iegler spectrum.
\newblock {\em Algebr. Represent. Theory}, 8(4):499--523, 2005.

\bibitem[Gre98]{Greenlees98_rational}
J.~P.~C. Greenlees.
\newblock Rational {M}ackey functors for compact {L}ie groups. {I}.
\newblock {\em Proc. London Math. Soc. (3)}, 76(3):549--578, 1998.

\bibitem[Gre01]{Greenlees01}
J.~P.~C. Greenlees.
\newblock Tate cohomology in axiomatic stable homotopy theory.
\newblock In {\em Cohomological methods in homotopy theory ({B}ellaterra,
  1998)}, volume 196 of {\em Progr. Math.}, pages 149--176. Birkh\"auser,
  Basel, 2001.

\bibitem[Gre19]{Greenlees19_rational}
J.~P.~C. Greenlees.
\newblock The {B}almer spectrum of rational equivariant cohomology theories.
\newblock {\em J. Pure Appl. Algebra}, 223(7):2845--2871, 2019.

\bibitem[GM92]{GreenleesMay92}
J.~P.~C. Greenlees and J.~P. May.
\newblock Derived functors of {$I$}-adic completion and local homology.
\newblock {\em J. Algebra}, 149(2):438--453, 1992.

\bibitem[Gro61]{EGAII}
A.~Grothendieck.
\newblock \'{E}l\'{e}ments de g\'{e}om\'{e}trie alg\'{e}brique. {II}. \'{E}tude
  globale \'{e}l\'{e}mentaire de quelques classes de morphismes.
\newblock {\em Inst. Hautes \'{E}tudes Sci. Publ. Math.}, (8):222, 1961.

\bibitem[Har66]{Hartshorne66}
Robin Hartshorne.
\newblock {\em Residues and duality}.
\newblock Lecture Notes in Mathematics, No. 20. Springer-Verlag, Berlin-New
  York, 1966.
\newblock Lecture notes of a seminar on the work of A. Grothendieck, given at
  Harvard 1963/64, With an appendix by P. Deligne.

\bibitem[HMS17]{HeardMathewStojanoska2017Picard}
Drew Heard, Akhil Mathew, and Vesna Stojanoska.
\newblock Picard groups of higher real {$K$}-theory spectra at height {$p-1$}.
\newblock {\em Compos. Math.}, 153(9):1820--1854, 2017.

\bibitem[Hoc69]{Hochster69}
M.~Hochster.
\newblock Prime ideal structure in commutative rings.
\newblock {\em Trans. Amer. Math. Soc.}, 142:43--60, 1969.

\bibitem[HS98]{HopkinsSmith98}
Michael~J. Hopkins and Jeffrey~H. Smith.
\newblock Nilpotence and stable homotopy theory. {II}.
\newblock {\em Ann. of Math. (2)}, 148(1):1--49, 1998.

\bibitem[HPS97]{HoveyPalmieriStrickland97}
Mark Hovey, John~H. Palmieri, and Neil~P. Strickland.
\newblock Axiomatic stable homotopy theory.
\newblock {\em Mem. Amer. Math. Soc.}, 128(610), 1997.

\bibitem[HS99]{HoveyStrickland99}
Mark Hovey and Neil~P. Strickland.
\newblock Morava {$K$}-theories and localisation.
\newblock {\em Mem. Amer. Math. Soc.}, 139(666):viii+100, 1999.

\bibitem[Joh82]{Johnstone82}
Peter~T. Johnstone.
\newblock {\em Stone spaces}, volume~3 of {\em Cambridge Studies in Advanced
  Mathematics}.
\newblock Cambridge University Press, Cambridge, 1982.

\bibitem[Kaw02]{Kawasaki02}
Takesi Kawasaki.
\newblock On arithmetic {M}acaulayfication of {N}oetherian rings.
\newblock {\em Trans. Amer. Math. Soc.}, 354(1):123--149, 2002.

\bibitem[Kel94]{Keller94b}
Bernhard Keller.
\newblock A remark on the generalized smashing conjecture.
\newblock {\em Manuscripta Math.}, 84(2):193--198, 1994.

\bibitem[KP17]{KockPitsch17}
Joachim Kock and Wolfgang Pitsch.
\newblock Hochster duality in derived categories and point-free reconstruction
  of schemes.
\newblock {\em Trans. Amer. Math. Soc.}, 369(1):223--261, 2017.

\bibitem[Kra00]{Krause00}
Henning Krause.
\newblock Smashing subcategories and the telescope conjecture---an algebraic
  approach.
\newblock {\em Invent. Math.}, 139(1):99--133, 2000.

\bibitem[Kra01]{Krause01}
Henning Krause.
\newblock On {N}eeman's well generated triangulated categories.
\newblock {\em Doc. Math.}, 6:121--126, 2001.

\bibitem[Kra02]{Krause02}
Henning Krause.
\newblock A {B}rown representability theorem via coherent functors.
\newblock {\em Topology}, 41(4):853--861, 2002.

\bibitem[Kra10]{Krause10}
Henning Krause.
\newblock Localization for triangulated categories.
\newblock In {\em Triangulated categories}, volume 375 of {\em London Math.
  Soc. Lecture Note Ser.}, pages 161--235. Cambridge Univ. Press, Cambridge,
  2010.

\bibitem[KS19]{KrauseStevenson19}
Henning Krause and Greg Stevenson.
\newblock The derived category of the projective line.
\newblock In {\em Spectral structures and topological methods in mathematics},
  EMS Congress Reports, pages 275--297. European Mathematical Society, 2019.

\bibitem[{Lau}21]{Lau2021Balmer}
Eike {Lau}.
\newblock {The Balmer spectrum of certain Deligne-Mumford stacks}.
\newblock Preprint, 34~pages, available online at
  \href{https://arxiv.org/abs/2101.01446}{arXiv:2101.01446}, 2021.

\bibitem[Lur09]{HTTLurie}
Jacob Lurie.
\newblock {\em Higher topos theory}, volume 170 of {\em Annals of Mathematics
  Studies}.
\newblock Princeton University Press, Princeton, NJ, 2009.

\bibitem[Mat15]{Mathew15bpp}
Akhil Mathew.
\newblock Torus actions on stable module categories, {P}icard groups, and
  localizing subcategories.
\newblock Preprint, 20~pages, available online at
  \href{https://arxiv.org/abs/1512.01716}{arxiv:1512.01716}, 2015.

\bibitem[Mat16]{Mathew16}
Akhil Mathew.
\newblock The {G}alois group of a stable homotopy theory.
\newblock {\em Adv. Math.}, 291:403--541, 2016.

\bibitem[MNN17]{MathewNaumannNoel17}
Akhil Mathew, Niko Naumann, and Justin Noel.
\newblock Nilpotence and descent in equivariant stable homotopy theory.
\newblock {\em Adv. Math.}, 305:994--1084, 2017.

\bibitem[MNN19]{mnn2}
Akhil Mathew, Niko Naumann, and Justin Noel.
\newblock Derived induction and restriction theory.
\newblock {\em Geom. Topol.}, 23(2):541--636, 2019.

\bibitem[Mat78]{Matlis78}
Eben Matlis.
\newblock The higher properties of {$R$}-sequences.
\newblock {\em J. Algebra}, 50(1):77--112, 1978.

\bibitem[Mat89]{Matsumura89}
Hideyuki Matsumura.
\newblock {\em Commutative ring theory}, volume~8 of {\em Cambridge Studies in
  Advanced Mathematics}.
\newblock Cambridge University Press, Cambridge, second edition, 1989.
\newblock Translated from the Japanese by M. Reid.

\bibitem[Nee92a]{Neeman92a}
Amnon Neeman.
\newblock The chromatic tower for {$D(R)$}.
\newblock {\em Topology}, 31(3):519--532, 1992.

\bibitem[Nee92b]{Neeman92b}
Amnon Neeman.
\newblock The connection between the {$K$}-theory localization theorem of
  {T}homason, {T}robaugh and {Y}ao and the smashing subcategories of
  {B}ousfield and {R}avenel.
\newblock {\em Ann. Sci. \'Ecole Norm. Sup. (4)}, 25(5):547--566, 1992.

\bibitem[Nee96]{Neeman96}
Amnon Neeman.
\newblock The {G}rothendieck duality theorem via {B}ousfield's techniques and
  {B}rown representability.
\newblock {\em J. Amer. Math. Soc.}, 9(1):205--236, 1996.

\bibitem[Nee01]{Neeman01}
Amnon Neeman.
\newblock {\em Triangulated categories}, volume 148 of {\em Annals of
  Mathematics Studies}.
\newblock Princeton University Press, 2001.

\bibitem[Nee10]{Neeman10}
Amnon Neeman.
\newblock Derived categories and {G}rothendieck duality.
\newblock In {\em Triangulated categories}, volume 375 of {\em London Math.
  Soc. Lecture Note Ser.}, pages 290--350. Cambridge Univ. Press, Cambridge,
  2010.

\bibitem[Nee11]{Neeman11}
Amnon Neeman.
\newblock Colocalizing subcategories of {$\mathbf D(R)$}.
\newblock {\em J. Reine Angew. Math.}, 653:221--243, 2011.

\bibitem[PSW22]{PatchkoriaSandersWimmer22}
Irakli Patchkoria, Beren Sanders, and Christian Wimmer.
\newblock The spectrum of derived {M}ackey functors.
\newblock {\em Trans. Amer. Math. Soc.}, 375(6):4057--4105, 2022.

\bibitem[PSY14]{PortaShaulYekutieli14}
Marco Porta, Liran Shaul, and Amnon Yekutieli.
\newblock On the homology of completion and torsion.
\newblock {\em Algebr. Represent. Theory}, 17(1):31--67, 2014.
\newblock [Erratum ibid 18(5):1401--1405, 2015].

\bibitem[PSY15]{PortaShaulYekutieli15}
Marco Porta, Liran Shaul, and Amnon Yekutieli.
\newblock Cohomologically cofinite complexes.
\newblock {\em Comm. Algebra}, 43(2):597--615, 2015.

\bibitem[Rav92]{Ravenel92}
Douglas~C. Ravenel.
\newblock {\em Nilpotence and periodicity in stable homotopy theory}, volume
  128 of {\em Annals of Mathematics Studies}.
\newblock Princeton University Press, Princeton, NJ, 1992.

\bibitem[Rec84]{Rector1984Noetherian}
D.~L. Rector.
\newblock Noetherian cohomology rings and finite loop spaces with torsion.
\newblock {\em J. Pure Appl. Algebra}, 32(2):191--217, 1984.

\bibitem[Rog08]{Rognes08}
John Rognes.
\newblock Galois extensions of structured ring spectra. stably dualizable
  groups.
\newblock {\em Mem. Amer. Math. Soc.}, 192(898):viii+137, 2008.

\bibitem[San19]{Sanders19}
Beren Sanders.
\newblock The compactness locus of a geometric functor and the formal
  construction of the {A}dams isomorphism.
\newblock {\em J. Topol.}, 12(2):287--327, 2019.

\bibitem[San22]{Sanders22}
Beren Sanders.
\newblock A characterization of finite \'{e}tale morphisms in tensor triangular
  geometry.
\newblock {\em \'{E}pijournal G\'{e}om. Alg\'{e}brique}, 6:Art. 18, 25, 2022.

\bibitem[San17]{BillySanders2017pp}
William~T. Sanders.
\newblock Support and vanishing for non-{N}oetherian rings and tensor
  triangulated categories.
\newblock Preprint, 33~pages, available online at
  \href{https://arxiv.org/abs/1710.10199}{arXiv:1710.10199}, 2017.

\bibitem[SWW17]{SatherWagstaffWicklein17}
Sean Sather-Wagstaff and Richard Wicklein.
\newblock Support and adic finiteness for complexes.
\newblock {\em Comm. Algebra}, 45(6):2569--2592, 2017.

\bibitem[Sch18]{Schwede18_global}
Stefan Schwede.
\newblock {\em Global homotopy theory}, volume~34 of {\em New Mathematical
  Monographs}.
\newblock Cambridge University Press, Cambridge, 2018.

\bibitem[Seg68]{Segal68b}
Graeme Segal.
\newblock Equivariant {$K$}-theory.
\newblock {\em Inst. Hautes \'{E}tudes Sci. Publ. Math.}, (34):129--151, 1968.

\bibitem[Sha79]{Sharp79}
Rodney~Y. Sharp.
\newblock Necessary conditions for the existence of dualizing complexes in
  commutative algebra.
\newblock In {\em S\'{e}minaire d'{A}lg\`ebre {P}aul {D}ubreil 31\`eme
  ann\'{e}e ({P}aris, 1977--1978)}, volume 740 of {\em Lecture Notes in Math.},
  pages 213--229. Springer, Berlin, 1979.

\bibitem[SW20]{shaulwilliamson2020lifting}
Liran {Shaul} and Jordan {Williamson}.
\newblock {Lifting (co)stratifications between tensor triangulated categories}.
\newblock Preprint, 20~pages, available online at
  \href{https://arxiv.org/abs/2012.05190}{arXiv:2012.05190}. To appear in
  \emph{Israel Journal of Mathematics}, 2020.

\bibitem[Ste13]{Stevenson13}
Greg Stevenson.
\newblock Support theory via actions of tensor triangulated categories.
\newblock {\em J. Reine Angew. Math.}, 681:219--254, 2013.

\bibitem[Ste14a]{Stevenson14}
Greg Stevenson.
\newblock Derived categories of absolutely flat rings.
\newblock {\em Homology Homotopy Appl.}, 16(2):45--64, 2014.

\bibitem[Ste14b]{Stevenson14b}
Greg Stevenson.
\newblock Subcategories of singularity categories via tensor actions.
\newblock {\em Compos. Math.}, 150(2):229--272, 2014.

\bibitem[Ste17]{Stevenson17}
Greg Stevenson.
\newblock The local-to-global principle for triangulated categories via
  dimension functions.
\newblock {\em J. Algebra}, 473:406--429, 2017.

\bibitem[Ste18]{Stevenson18}
Greg Stevenson.
\newblock A tour of support theory for triangulated categories through tensor
  triangular geometry.
\newblock In {\em Building bridges between algebra and topology}, Adv. Courses
  Math. CRM Barcelona, pages 63--101. Birkh\"{a}user/Springer, Cham, 2018.

\bibitem[Tho97]{Thomason97}
R.~W. Thomason.
\newblock The classification of triangulated subcategories.
\newblock {\em Compositio Math.}, 105(1):1--27, 1997.

\bibitem[Trl96]{Trlifaj96}
Jan Trlifaj.
\newblock Two problems of {Z}iegler and uniform modules over regular rings.
\newblock In {\em Abelian groups and modules ({C}olorado {S}prings, {CO},
  1995)}, volume 182 of {\em Lecture Notes in Pure and Appl. Math.}, pages
  373--383. Dekker, New York, 1996.

\bibitem[Ver21]{Verasdanis22pp}
Charalampos Verasdanis.
\newblock Stratification and the smashing spectrum.
\newblock Preprint, 29~pages, available online at
  \href{https://arxiv.org/abs/2204.01082}{arXiv:2204.01082}, 2021.

\bibitem[Ver22]{Verasdanis22bpp}
Charalampos Verasdanis.
\newblock Costratification and actions of tensor-triangulated categories.
\newblock Preprint, 26~pages, available online at
  \href{https://arxiv.org/abs/2211.04139v2}{arXiv:2211.04139v2}, 2022.

\bibitem[Zhe13]{Zhe13pp}
Han Zhe.
\newblock Homotopy categories of injective modules over derived discrete
  algebras.
\newblock Preprint, 19~pages, available online at
  \href{https://arxiv.org/abs/1308.2334}{arXiv:1308.2334}, 2013.

\bibitem[Zou23]{Zou23bpp}
Changhan Zou.
\newblock Manuscript in preparation, 2023.

\end{thebibliography}

\end{document}